\documentclass[11pt]{article}
\usepackage{amssymb,amsmath,amsthm}
\usepackage{upgreek}
\usepackage{yfonts}
\usepackage{bbm}
\usepackage{mathrsfs}
\usepackage{stmaryrd}
\usepackage{graphicx}
\usepackage{color}
\newtheorem{remark}{Remark}

\newtheorem{proposition}{Proposition}
\newtheorem{lemma}{Lemma}
\newtheorem{corollary}{Corollary}
\newtheorem{theorem}{Theorem}

\def\rest{\hskip 1pt{\hbox to 10.8pt{\hfill
\vrule height 7pt width 0.4pt depth 0pt\hbox{\vrule height 0.4pt width 7.6pt depth 0pt}\hfill}}}

\newcommand{\eps}{\varepsilon}

\def \rd {{\rm d}}
\def\B{{\mathbb B}}
\def\E{{\rm E}}
\def\Eps {{\rm E}_\varepsilon}
\def\Veps {{\mathbb V}_\varepsilon}
\def\C{{\mathbb C}}
\def\D{{\mathbb D}}

\def\R{{\mathbb R}}
\def\N{{\mathbb N}}
\def\S{{\mathbb S}}
\def \Cnrg{{\rm C}_{\rm nrg}}

\def \rL{{\rm L}}

\def \be{\vec {\bf e}}

\def \betheta {\vec {\bf  e}^{\,  \theta}}

\def \upsigmam{{\upsigma}_{\rm main}}
\def \upsigman{{\upsigma}_{\rm bis}}

\def \Cdec {{\rm C}_{\rm dec}}
\DeclareMathAlphabet{\mathpzc}{OT1}{pzc}{m}{it}
\def \bJ{ {\mathbbmss J }}
\def \bL{ {\mathbbmss L }}
\def \bN{ {\mathbbmss N }}

\def \bm{ {\mathbbm m }}
\def \bN{{ \mathbbmss N}}
\def\tbe{{\large \mathbbm e}}
\def \bJ{{ \mathbbmss J}}
\def \bL{{ \mathbbmss L}}
\def \bP {{\mathbbmtt P}}
\def \bH {{\mathbbmtt H}}

\def\QED{\hbox{${\vcenter{\vbox{
   \hrule height 0.4pt\hbox{\vrule width 0.4pt height 6pt
   \kern5pt\vrule width 0.4pt}\hrule height 0.4pt}}}$}\vspace{7pt}}

\begin{document}

\author{Fabrice BETHUEL\thanks{Sorbonne Université, CNRS, Université de Paris, Laboratoire Jacques-Louis Lions (LJLL), F-75005 Paris.  
} }
\title{Asymptotics for  two-dimensional vectorial Allen-Cahn systems}
\date{}
\maketitle

\begin{abstract}
The formation of codimension-one interfaces for multi-well gradient-driven problems  is well-known and established in the scalar case,   where the equation  is often referred 
to as the Allen-Cahn equation. The proofs rely for a large on a monotonicity formula for the energy density,  which is itself related to the vanishing of the so-called  discrepancy function. The vectorial case  in contrast is quite open.  This lack of results and insight  is to a large extend  related  to the absence  of known appropriate  monotonicity formula. In this paper, we focus on the  \emph{elliptic case in two dimensions},  and introduce  methods, relying on the analysis of the partial differential equation,  which allow  to circumvent    the lack of monotonicity formula for the energy density.   
In the  last part of the paper, we recover a \emph{new monotonicity formula} which relies on  a \emph{new discrepancy relation}. These tools allow to extend to the vectorial case in two dimensions  most of the results obtained for the scalar case. We emphasize also some \emph{specific features} of the vectorial case.
\end{abstract}
\bigskip
\noindent
\section{Introduction}
\subsection{Statement of the  main result}
\label{statem}
 Let $\Omega$ be a   smooth bouned  domain in $\R^2$. In the present paper we investigate  asymptotic properties of  families of solutions $(u_\eps)_{\eps>0}$ of the systems of  equations having the general form  
 \begin{equation}
 \label{elipes}
  -\Delta u_\eps=-\eps^{-2}\nabla_u V(u_\eps)  {\rm \ in \ } \Omega \subset \R^2,  
  \end{equation}
as the parameter $\eps >0$ tends to zero.  
 The function $V$, usually termed the \emph{potential},  denotes   a smooth scalar   function on   $\R^k$, where $k \in \N$ is a given integer.  Given $\eps>0$, the function  $v_\eps$  represents    a function defined on  the domain $\Omega$ with values into the \emph{euclidian space $\R^k$}, so that equation \eqref{elipes} is  a \emph{system of $k$ scalar partial differential equations}  for each of the components of the map $u_\eps$. The equation \eqref{elipes} and its parabolic version have been have been introduced  as models in the physics and material literature (see e.g. \cite{bronret} and the references therein, in particular \cite{bech}).
 
    Equation \eqref{elipes}   corresponds to the  Euler-Lagrange equation   of the  energy functional  ${\E_\eps}$ which is defined for a function $ u:\Omega \mapsto \R^k$ by the formula
    \begin{equation}
 \label{glfunctional}
 \E_\eps(u)= \int_\Omega e_\eps(u)=\int_{\Omega} \eps \frac{\vert \nabla u \vert ^2}{2}+\frac{1}{\eps} V(u).
 \end{equation}
   We  assume that the potential $V$ is bounded below, so that we may impose, without loss of generality and changing possibly $V$ by a suitable  additive  constant,  that
   \begin{equation}
   \label{infimitude}
   \inf V=0.
   \end{equation}
 We introduce the  set $\Sigma$ of minimizers of $V$,  sometimes called the vacuum manifold, that is the  subset of $\R^k$ defined 
     $$\Sigma\equiv \{ y \in \R^k, V(y)=0 \}.$$
   Properties of solutions to \eqref{elipes} crucially depend on the nature of $\Sigma$. In this paper, we  will assume that the 
    vacuum manifold is finite, with at least two distinct elements, so that   \\
    
    \noindent
   $\displaystyle{ (\text{H}_1) \ \ \ \ \  
 \Sigma=\{\upsigma_1, ..., \upsigma_q\},\ q\geq 2, \ \upsigma_i \in \R^k, \ \forall i=1,...,q.
 }$ \\
 
 \noindent
  We impose furthermore a condition on the behavior of $V$ near its zeroes, namely:
    
    \smallskip
    \noindent
${(\text{H}_2)}$   {\it  The matrix $\nabla^2V(\upsigma_i)$ is positive definite at each point $\upsigma_i$ of $\Sigma$, in other words, if $\lambda_i^-$ denotes its smallest eigenvalue, then  $\lambda_i^->0$. We denote by $\lambda_i^+$ its largest eigenvalue. }\\

\noindent
Finally, we also impose a growth conditions at infinity: 

\smallskip
\noindent
  ${(\text{H}_3)}$ {\it There exists  constants  $\upalpha_\infty >0$ and  $R_\infty >0$  such that  
  \begin{equation}
  \label{condinfty}
  \left\{
   \begin{aligned}
   y\cdot\nabla V( y )&\geq \upalpha_\infty \vert y \vert ^2, \ \hbox {if }  \vert y \vert >R_\infty {\rm \  and \ }  \\
 V(x) \to &+\infty {\rm \ as  \  }  \vert x \vert \to + \infty. 
 \end{aligned}
 \right.
 \end{equation}   
 }
 \smallskip
\noindent
A potential $V$ which fullfils  conditions ${(\text{H}_1)}$, ${(\text{H}_2)}$ and ${(\text{H}_3)}$   is  termed  throughout  the paper  a potential with multiple  equal depth wells (see Figure \ref{potent}). 
\smallskip

 A  typical  example is provided  in the scalar case $k=1$ 
 by the potential, often termed \emph{Allen-Cahn} or \emph{Ginzburg-Landau} potential,
 \begin{equation}
 \label{glexemple}
  V(u)=\frac{(1-u^2)^2}{4}, 
  \end{equation}  
  whose  infimum equals $0$ and whose minimizers are $+1$ and $-1$, 
  so that   $\Sigma=\{+1,-1\}.$  It is used as  an elementary model for \emph{phase transitions} for  materials with  two equally  preferred states,   the   minimizers $+1$ and $-1$ of the potential $V$.  
  
  \smallskip
  
  \begin{figure}[h]
\centering
\includegraphics[height=7cm]{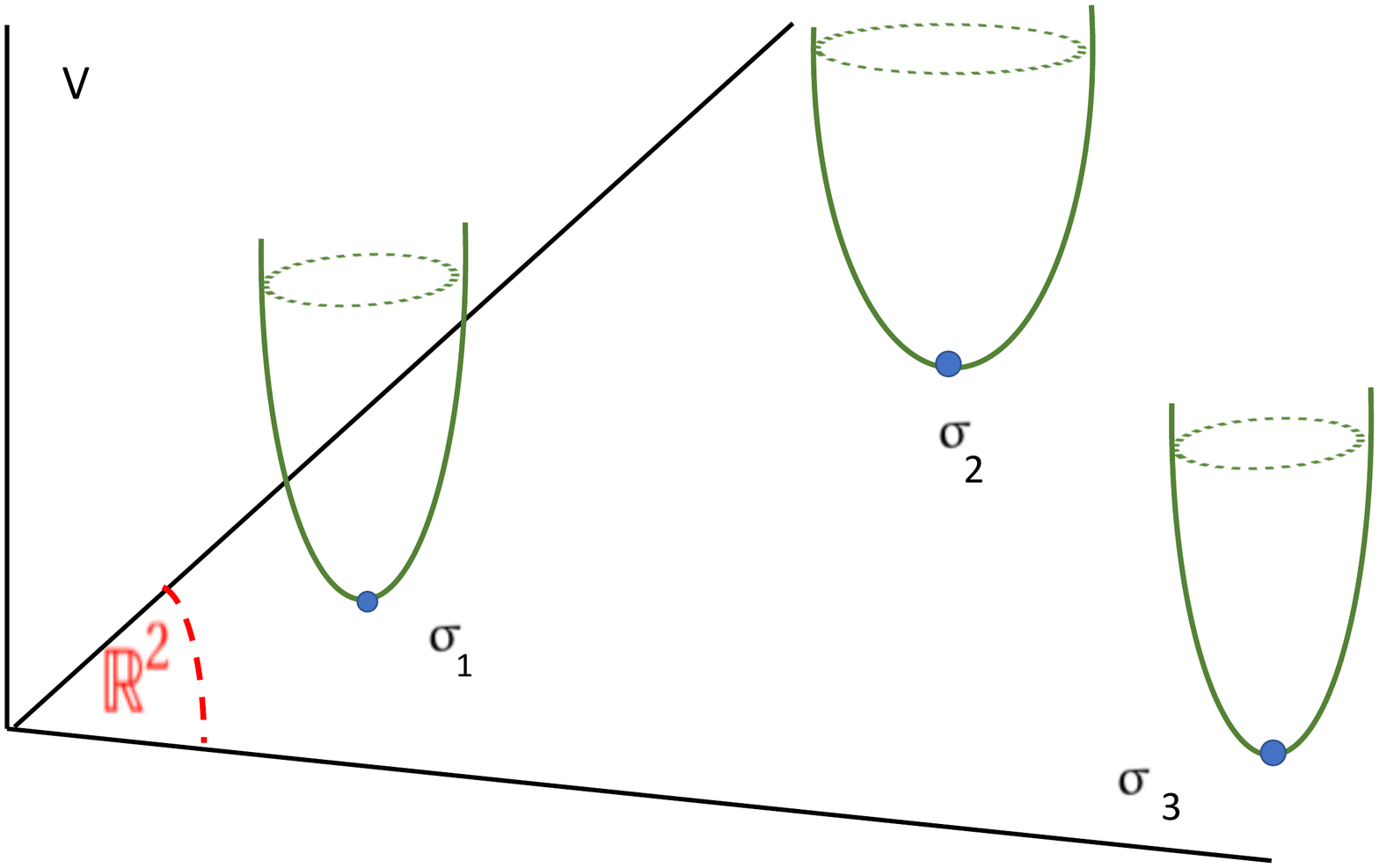}
\caption{  {\it Graph of a potential with several minimizers.}}
\label{potent}
\end{figure}
Important efforts have been devoted so far to the study  of  solutions of the stationary \emph{Allen-Cahn} equations, i.e. solutions to \eqref{elipes} for the  potentials similar to \eqref{glexemple},   or to  the corresponding parabolic evolution equations,  in the asymptotic limit $\eps \to 0$, in arbitrary dimension $N$ of the domain $\Omega$. The  mathematical  theory for this question is now well advanced and may be considered as satisfactory. The results found there    provide a sound mathematical foundation  to the intuitive idea that      the domain $\Omega$ decomposes  into regions  where the solution takes  values  either close to  $+1$ or  close to  $-1$, the   regions being separated by interfaces  of width of order $\eps$. These interfaces,    are expected to converge  to hypersurfaces of codimension 1, which are a shown to be   \emph{generalized minimal surfaces}  in the stationary case, or \emph{moved by mean curvature} for the parabolic evolution equations. 

  Several of the arguments rely on   \emph{integral methods} and \emph{energy estimates}. For instance in \cite{ilmanen},  T.Ilmanen proved  convergence   to   motion  by mean curvature  in the \emph{weak sense of Brakke},   a notion  relying on  the language,  concepts  and methods of\emph{ geometric measure theory}. In the elliptic case considered in this paper, convergence to minimal surfaces was established by Modica and Mortola in their celebrated paper \cite{mortadela},  F. Hutchinson and Y. Tonegawa in \cite{hutchtone} established related results for non-minimizing solutions in \cite{hutchtone}. More refrences will be provided in Subsection \ref{comparse}.

  \begin{remark}
 \label{mini}
 {\rm  The case of \emph{minimizing solutions} was treated  in the vectorial case  in   by Baldo, on one hand (see \cite{baldo}),  and Fonseca and Tartar on the other (see \cite{fontar}),  whre  they   obtained quite similar results to \cite{mortadela}. The approaches rely on ideas from Gamma convergence, and \emph{do not rely on monotonicity formulas},  as for general stationary solutions or solutions of the corresponding evolution equations in the scalar case.
 }
 \end{remark}
 
\medskip     
  The purpose of  the present  paper is to show that,  to  a  some  extend,   the results obtained in the scalar case, can be transposed to the vectorial case for  potentials $V$ which fulfill  conditions ${(\text{H}_1)}$, ${(\text{H}_2)}$ and ${(\text{H}_3)}$, that is potentials with multiple  equal depth wells, if we restrict ourselves to \emph{two dimensional domains}. Let us emphasize that, prior to the present paper,   \emph{no  monotonicity  formula  similar to \eqref{monotonie} was   known in the vectorial case}, so that different arguments have to be worked out. Several of them  rely strongly on some specificities of dimension two.
    
  \smallskip
  We assume that we are given a constant ${\rm M}_0>0$ and a family $(u_\eps)_{0<\eps\leq 1}$ of solutions to the equation \eqref{elipes} for the corresponding value of the parameter $\eps$, satisfying the natural energy bound 
   \begin{equation}
  \label{naturalbound}
  \Eps(u_\eps) \leq{\rm  M}_0,   \    \forall \eps >0.
      \end{equation}
      Assumption \eqref{naturalbound}  is rather  standard in the field, since it corresponds to the energy magnitude required for the creation of $(N-1)$-dimensional interfaces.  Our main result is the following:
    
  \begin{theorem}
 \label{maintheo}  
 Let $(u_{\eps_n})_{n \in \N}$  be  a sequence of solutions to \eqref{elipes}  satisfying \eqref{naturalbound}. 
 There exist a subset $\mathfrak S_\star$ of  $\Omega$ and a subsequence of $(\eps_n)_{n \in \N} $, still denoted $(\eps_n)_{n \in \N}$ for sake of simplicity,
such that the following properties hold: 
\begin{enumerate}
\item[i)] $\mathfrak S_\star$  is  a closed  1 dimensional \emph{rectifiable}  subset of $\Omega$ 
 such that 
 \begin{equation}
 \label{herbert}
 \mathcal H^1(\mathfrak S_\star)\leq \rm C_{\rm H}\, {\rm M}_0,
 \end{equation}
  where ${\rm C}_{\rm H}$ is a constant depending only on the potential $V$.
\item[ii)]   Set $\mathfrak U_\star=\Omega\setminus  \mathfrak S_\star$, and   let $(\mathfrak U_\star^i)_{i \in I}$ be  the connected components of $\mathfrak U_\star$.   For each $i\in I$ there exists an element $\upsigma_i \in \Sigma$ such that 
\begin{equation}
\label{convers}
u_ {\eps_n}  \to \upsigma_i  {\rm \ uniformly  \ on  \ every  \ compact \ subset \  of \ } \mathfrak U_\star  {\rm \ as \ } n \to +\infty.
\end{equation}
\end{enumerate}
  \end{theorem}
  
 Similar to the results obtained for \emph{the scalar case},  Theorem 1 expresses, \emph{for the vectorial case in dimension two},  the fact that the domain can be decomposed into subdomains,  where, for $n$ large,  the maps $u_{\eps_n}$ takes values close to an element of the vacuum set $\Sigma$ (see Figure \ref{partition}). This subdomains  which are separated by a closed one-dimensional subdomain,  the set $\mathfrak S_\star$,  on which the map  $u_{\eps_n}$ might possibly undergo a transition from one element of $\Sigma$ to  another.  Notice that Theorem \ref{maintheo}  result extends also to \emph{non-minimizing} solutions the results\footnote{This result  hold however in arbitrary dimension and yield stronger, in particular minimizing,  properties for $\mathfrak S_\star$.} of \cite{baldo,fontar} (see Remark \ref{mini}).

 \smallskip
 An important property of the set $\mathfrak S_\star$ stated in Theorem \ref{maintheo} is its rectifiability. Recall that  a Borel set $\mathcal S\subset \R^2$, and
is rectifiable of dimension 1  if its one-dimensional Hausdorff dimension is locally finite, and if there there is a countable family of $C^1$ one-dimensional submanifolds of $\R^2$ which cover $\mathcal H^1$ almost all of $\mathcal S$. Rectifiability of $\mathcal S$
  implies in particular, that  the set $\mathcal  S$ has \emph{an approximate tangent line} at $\mathcal H^1$-almost every point $x_0 \in \mathcal S$. More precisely, there exists a set $\mathfrak  A_\star $ with $\mathcal H^1(\mathfrak A_\star)=0$  such that, if $x_0 \in \mathfrak S_\star\setminus  \mathfrak A_\star$, then  we have
  \begin{equation}
  \label{densitusone}
 {\underset {r \to 0}\lim} \frac{\mathcal H^1 (\mathfrak S_\star(\D^2(x_0, r))}{2r}=1, 
  \end{equation}
   and   there exists  a unit vector  $\vec e_{x_0}$ (depending on  the point $x_0$)  with the following property:  For \emph{any  number}  $\uptheta >0$ we have
 \begin{equation}
 \label{tangent}
 {\underset {r\to 0}\lim} 
 \frac{ 
 \mathcal H^1 \left( \mathcal S_\star \cap\left( \D^2\left(x_0, r\right) \setminus   \mathcal C_{\rm one}\left(x_0, \vec e_{x_0}, \uptheta \right) \right) \right)
 }
{r}=0, 
 \end{equation}
 where, for a unit vector $\vec e$ and $\uptheta>0$, the set  $\mathcal C_{\rm one}\left(x_0, \vec e, \uptheta  \right)$ is the cone given by
 \begin{equation}
\label{conalpha}
\mathcal C_{\rm one}\left(x_0, \vec e, \uptheta  \right)=
\left \{ y \in \R^2,   \vert  \vec e^\perp \cdot  (y-x_0)  \vert \leq   \tan\uptheta \vert  \vec e \cdot  (y-x_0) \vert\right\},
 \end{equation}
 $\vec e^\perp$ being a unit vector  orthonormal to $\vec e$ (see e.g. \cite{simon}). A point $x_0\in \mathfrak S_\star \setminus \mathfrak A_\star$  is termed a \emph{regular point} of $\mathfrak  S_\star$.
 
 \smallskip
 In the minimizing case, it is established in \cite{baldo,fontar} that the interface $\mathfrak S_\star$ is a  co-dimension one  minimal surface, which hence reduces, in dimension two, to \emph{the union of segments}.  Our next result  shows that, in dimension two, the same kind of result holds for \emph{non-minimizing solutions}.   
 
 To order to state the result, and since the \emph{notion of minimality}  is also  related in our context to the presence of \emph{densities} of measures, we specify first which measures  we have in mind.  To that aim,  we    introduce a limiting measure for the potential term:  Consider the positive measure $\upzeta_\eps$  defined on 
 $\Omega$ by
 \begin{equation}
 \label{boulga00} 	
 \upzeta_\eps \equiv \frac{  V(u_{\eps})}{\eps}  {\rm d}\, x,  {\rm \  so \  that  \ } \,  \upzeta_\eps (\Omega) \leq M_0.
 \end{equation}
 Since the family $(\upzeta_\eps)_{\eps>0}$ is uniformly  bounded,   passing possibly   to a further subsequence, we have the convergence  
   \begin{equation}
   \label{boulga}
 \upzeta_{\eps_n} \equiv \frac{  V(u_{\eps_n})}{\eps_n}  {\rm d}\, x \rightharpoonup  \upzeta_\star,  {\rm \ in \ the \ sense \ of \  measures \ on  \  }\Omega,  {\rm \ as \ } n \to + \infty, 
 \end{equation}
  
\begin{theorem}
 \label{segmentus}  There exists a set $\mathfrak E_\star \subset \mathfrak S_\star$ such that $\mathcal H^1(\mathfrak E_\star)=0$, such that $\mathfrak A_\star\subset \mathfrak E_\star$ and such that, for $x_0 \in \mathfrak S_\star \setminus \mathfrak E_\star$, the set   $\mathfrak S_\star$ is  locally near $x_0$  a segment. More precisely,  there exists a unit vector $\vec e_{x_0}$ and a radius $r_0>0$, depending on $x_0$,  such that
 \begin{equation}
 \label{segmentus0}
  \mathfrak S_\star \cap  \D^2(x_0, r_0) =\left (x_0-r_0\vec e_{x_0},  x_0+r_0\vec e_{x_0}\right).
  \end{equation}
 Moreover the restriction of the measure $\upzeta_\star$ to $\D^2(x_0, r_0)$ is proportional to the $\mathcal H^1$ measure of  $\displaystyle{\left (x_0-r_0\vec e_{x_0},  x_0+r_0\vec e_{x_0}\right)}$, that is is there exists  a number $c_{x_0}>0$,  depending on $x_0$, such  that
\begin{equation} 
\label{constantitude}
 \upzeta_\star\rest\D^2(x_0, r_0) =c_{x_0}\left( \mathcal H^1\rest \left (x_0-r_0\vec e_{x_0},  x_0+r_0\vec e_{x_0}\right) \right).
 \end{equation}
 The number $c_{x_0}$ are bounded below, that is, there exists a constant $\upeta_0 ({\rm d}(x))>0$, depending only on $V$, $M_0$ and $\rd(x_0)\equiv {\rm dist} (x,  \partial \Omega)$ such that such that
 \begin{equation}
 \label{belowitude}
 c_{x_0} \geq \upeta_0({\rm d}(x) )   {\rm \  for \ any \ } x_0 \in \mathfrak S_{\star} \setminus \mathfrak E_\star.
 \end{equation}
 \end{theorem}
 
 Notice that, as a consequence of \eqref{belowitude},  for any $x_0 \in  \mathfrak S_{\star} \setminus \mathfrak E_\star$ , the  one dimensional  density 
 $\Uptheta_\star$ defined by
 \begin{equation}
{\tiny \Uptheta}_\star (x) =\underset{r \to 0} \liminf \frac{\upzeta_\star\left(\D^2(x, r)\right)}{2r} 
 \end{equation}
 is bounded below by $\upeta_0(\rd (x))$, hence away from zero,  and is locally  constant, equal to $c_{x_0}= \Uptheta_\star (x_0)$. \\

   Theorem  \ref{segmentus} expresses \emph{local stationarity  properties} of the set $\mathfrak S_\star$ and the measure $\upzeta_\star$.  As we will discuss later, the set $\mathfrak S_\star$ may have singularities, and hence $\mathfrak E_\star$ is not empty. However, more global stationary properties are also available. In order to state these properties,  the abstract language of \emph{varifolds} is the most appropriate. An important  preliminary step is to establish that the measure $\upzeta_\star$ concentrate on the set $\mathfrak S_\star$, i.e. its restriction the $\Omega\setminus \mathfrak S_\star$ vanishes (see Theorem \ref{maintheo}), and that it  is \emph{absolutely continuous}  with respect to the $\mathcal H^1$-measure on $\mathfrak S_\star$  (see Theorem \ref{absolute}).  In particular, this property implies that the measure $\upzeta_\star$ is completely detrmined by the set $\mathfrak S_\star$ and the density ${ \Uptheta_\star}$, and we have
   \begin{equation}
   \label{densitos}
   \upzeta_\star=\Uptheta_\star(\mathcal H^1\rest \mathfrak S_\star)
   = \Uptheta_\star \rd \lambda, {\rm \  where \ }  \rd  \lambda=\mathcal H^1\rest \mathfrak S_\star.
   \end{equation}
   We have:
    \begin{theorem}
   \label{varifoltitude} The rectifiable one-varifold ${\rm \bf V}(\mathfrak S_\star, \Uptheta_\star)$  corresponding to the measure $\upzeta_\star$ is  stationary.
   \end{theorem} 
 
 \begin{figure}[h]
\centering
\includegraphics[height=7.5cm]{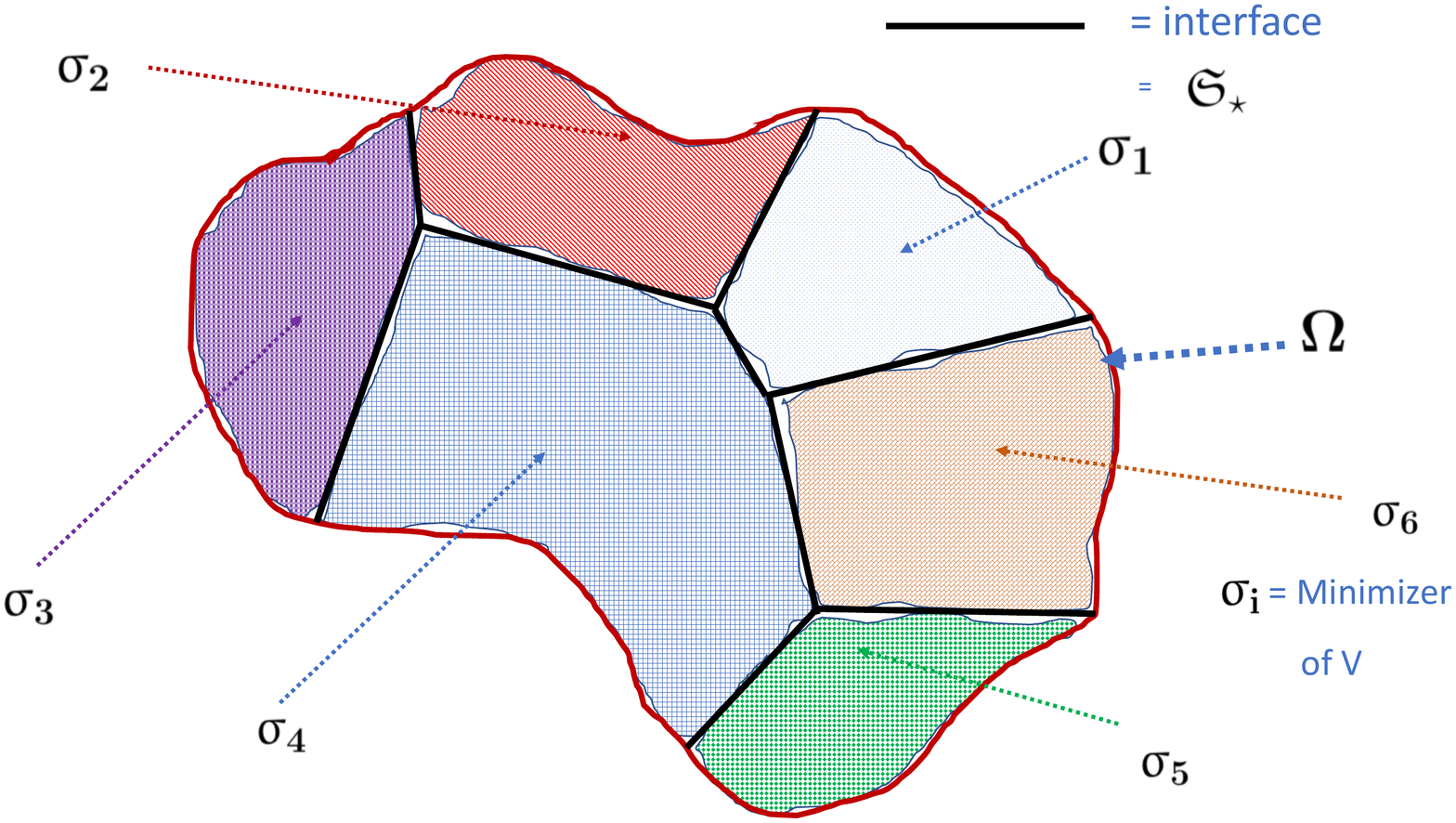}
\caption{  {\it The domain $\Omega$ is divided in subdomains where  $u_\eps$ is nearly constant. The  interfaces are  union of segments.}}
\label{partition}
\end{figure}
 
 
 The theory of varifolds has been developed in the context of minimal surfaces, but it turns out to be also an important tool  in the study of singular limits (see e.g. \cite{simon} for a general presentation of the theory of varifolds). 
 The fact  that ${\rm \bf V}(\mathfrak S_\star, \Uptheta_\star)$  is a stationary varifold  is equivalent to the following statement: Given  any smooth vector field  $\vec X \in C_c (\Omega, \R^2)$ on $\Omega$ with compact support, the following identity holds
 \begin{equation}
 \label{nary}
 \int_{\Omega} {\rm div}_{_{T_x\mathfrak S_\star} }  \vec X \rd \upzeta_\star=0. 
 \end{equation} 
 Here, for $x \in \mathfrak S_\star \setminus \mathfrak E_\star$, the number  $\displaystyle{{\rm div}_{_{T_x\mathfrak S_\star}}\vec X(x) }$  is defined by
 \begin{equation}
 {\rm div}_{_{T_x\mathfrak S_\star}}\vec X(x)=
\left (\vec e_x\cdot \vec {\nabla}\vec X(x)\right) \cdot\vec e_x, {\rm \ for \ } x \in \mathfrak S_\star.  
\end{equation}

 The structure on one-dimensional varifolds with  densities bounded away from zero was thoroughly investigated by Allard and Almgren in \cite{allardalm}.  They showed that such varifolds have a graph structure and are the sum of segments with densities. Theorem \ref{segmentus} \emph{may therefore be deduced from Theorem \ref{varifoltitude} } invoking the results of Section 3 in \cite{allardalm}. In the present paper, we provide however a simple self-contained proof, based on several results which are worked out independently. 
 
  One-dimensional varifolds may have singularities,  which are characterized by the fact that the density is not constant in their neighborhood.  The simplest example of such a singular varifold in the whole plane with  a singularity at $0$ is provided by the union of a finite numbers of distinct half-lines, intersecting at the origin, with appropriate  constant densities.   More precisely,  consider an integer $d >2$, and let $\vec e_1,  \vec e_2, \ldots, \vec e_d$ be $d$  distinct unit vectors in $\R^2$. Set 
  \begin{equation}
  \label{junctions} 
  \mathcal S_\star=\underset{i=1} {\overset d \cup}
  \mathbbm H_i,{\rm \ where \  for \ } i=1, \ldots, d,  {\rm \ we \ set \ } \mathbbm H_i=  \left \{ t\vec e_i, t \geq 0 \right \},  
  \end{equation}
  and let $\theta_1, \ldots, \theta_d$  be $d$ positive numbers.  If $\theta_i$ represents the density $\Theta$ of  $\mathcal S_\star$ on $ \mathbbm H_i$ (which is hence constant there), then ${\bold V} (\mathcal S_\star, \Theta)$  is a stationary one-dimensional rectifiable varifold if and only if
  \begin{equation}
  \label{stationaritat} 
 \underset{i=1} {\overset d \sum }\theta_i \vec e_i=0.
  \end{equation}
Singularities $x_0$ which behave \emph{locally} as \eqref{junctions}-\eqref{stationaritat} are termed of \emph{finite type}.
It turns out that singularities of finite type   appear in the asymptotics of the vectorial Allen Cahn equation, \emph{even in the minimizing case}, and are actually an intrinsic  part in the problem. A first trivial example is provided by an uncoupled system of two scalar Allen-Cahn equation, taking for instance as a potential $V: \R^2 \to \R$ the potential 
$\displaystyle{V(u_1, u_2)=\frac{1}{4} \left[(1-u_1)^2+(1-u^2)^2  \right] }$. For this potential,  the map  $u_\eps$ defined on $\R^2$ by 
$$u_\eps(x_1, x_2)=(\tanh  \left(\frac{x_1}{\sqrt{2\eps}}\right) , \tanh  \left(\frac{x_2}{\sqrt{2\eps}}\right), {\rm \    for  \ }  (x_1, x_2) \in \R^2, $$
is a solution to \eqref{elipes} on the whole plane. The limiting  interface $\mathfrak S_\star$  for $\eps \to 0$ is  then given as the union   of the lines $x_1=0$ and $x_2=0$, so that $0$ is a singularity where these lines cross with right angles.   One may  actually  construct similar examples where the angle between the two lines  is arbitrary.

\smallskip
A more involved example  is constructed in \cite{brongui}, where   a sequence of minimizing solutions is constructed on the entire plane, for a potential with three minimizers and equilateral symmetry. The set $\mathfrak S_\star$  then consists  of three half lines with equal angles and equal densities, yield a singularity at zero with \emph{triple junction}(see Figure \ref{karmelit}).  The appearance of  triple junctions in general minimizing problems is discussed in \cite{sternziem} and analyzed through Gamma-convergence results.

\begin{remark} {\rm   Singularities of finite type have also be constructed as limits of  scalar Allen-Cahn problem (see \cite{delpinopac, guiliu}). In these constructions, the number $d$ of half-lines in \eqref{junctions} is even.

}
\end{remark}

 Besides singularities with a locally finite sum of segments as in \eqref{junctions}, an example of  a singularity  of a  stationary varifold  with an  \emph{infinite  complexity} is produced in \cite{allardalm}. It is however shown in \cite{allardalm} that the occurence of such singularities is ruled out if the set of densities is discrete. As we will see later,  there are examples of potential such that the possible set of densities is infinite, so that  of singularities of \emph{infinite  type} cannot be excluded a priori in the limits of solutions to \eqref{elipes}. 

\begin{figure}[h]
\centering
\includegraphics[height=7cm]{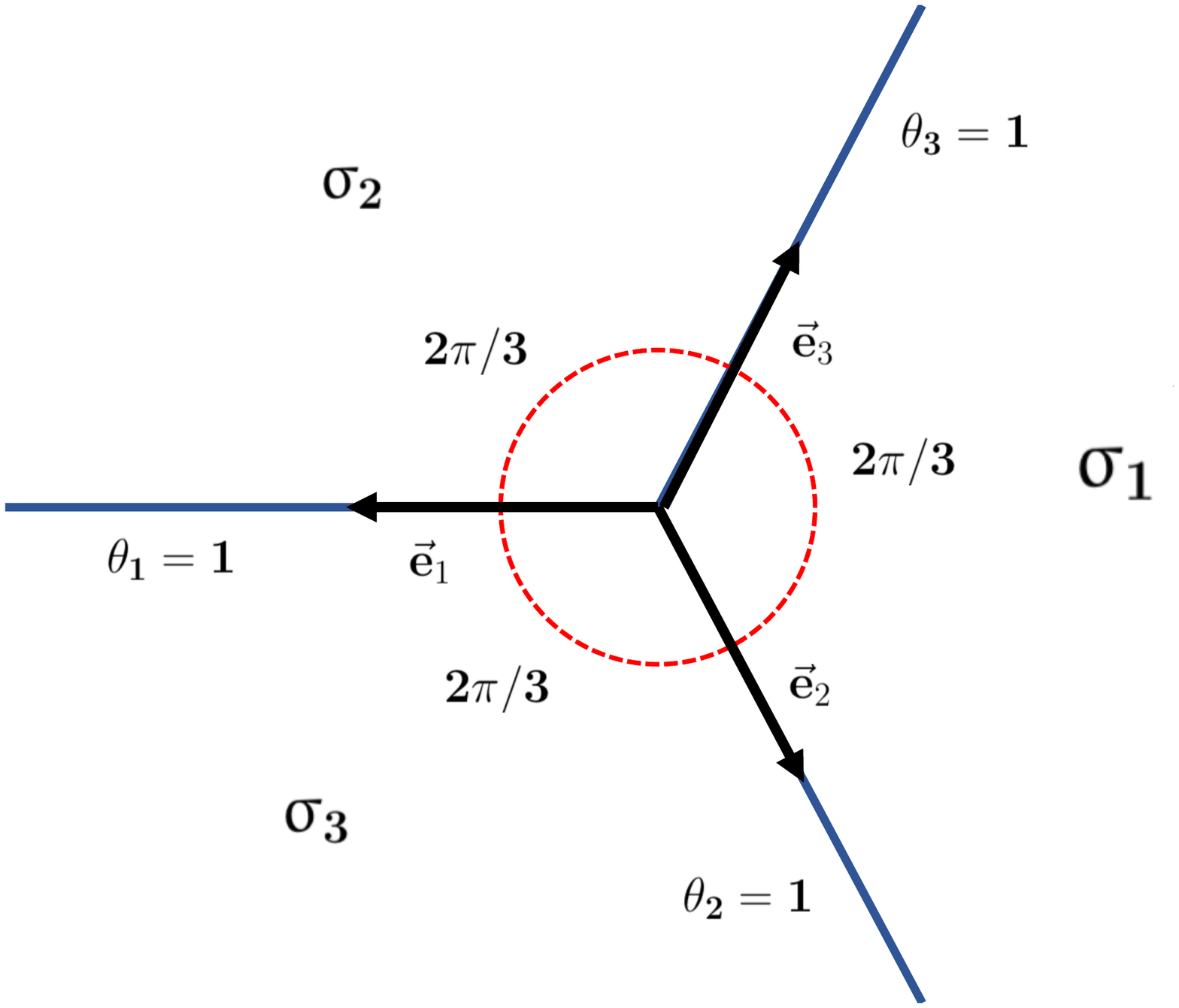}
\caption{  {\it Example of a triple junction, as in \cite{brongui}.}}
\label{karmelit}
\end{figure}
 
  \subsection{Comparing results  in the scalar and vectorial cases} 
  
  Although the results stated in Theorems \ref{maintheo}, \ref{segmentus} and \ref{varifoltitude}   for the vectorial Allen-Cahn equation  are somewhat parallel with the results obtained  so far  in the literature for  the scalar case, it is worthwhile to stress  some major differences between the scalar and the vectorial case. \\
  
\noindent  
  {\bf The one-dimensional case}
  
  \noindent
  Distinct behaviors are already observed for the one-dimensional case.  
  Indeed, for   $\Omega=\R$,    equation \eqref{elipes} reduces   to the ordinary differential equation
\begin{equation}
	\label{heterocline}
	-\frac{{\rm d}^2 w_\eps}{{\rm d} s
		^2}=\eps^{-2}\nabla_w V(w_\eps) {\rm \ on \ } \R.
\end{equation} 
Finite energy solutions necessarily connect at $\pm \infty$ two  minimizers {$\upsigma^-$} and {$\upsigma^+$}: They are called 
{\emph{profiles}} or {\emph{heteroclinic connections}}, if {$\upsigma^-\not = \upsigma^+$}. Multiplying \eqref{heterocline} by $w_\eps$, we are led to the conservation law 
 \begin{equation}
 \label{conserve}
 \frac{d}{dx} \left(\frac{1}{\eps} V(w_\eps)-\eps\frac{\vert \dot w_\eps\vert^2}{2}\right)=0, 
 \end{equation}
 so that for \emph{profiles}  one derives  the identity
 \begin{equation}
 \label{first-order}
 \eps  \vert \dot w_\eps\vert=\sqrt{V(w_\eps)} {\rm \ on \ } \R.
 \end{equation}
   In the  \emph{scalar} case, the first order equation \eqref{first-order} is easily integrated by separation of variables, so  that profiles  connect only \emph{nearby minimizers} $\upsigma^-$ and $\upsigma^+$ of the potential, and  are essentially unique,  up to translations and symmetries. For instance, in the case of the   \emph{Allen-Cahn potential} \eqref{glexemple}, the solution  is given up to translation and symmetry, by  
$\displaystyle{
w_\eps(s)= \tanh \left(\frac{s} {\sqrt{2}\, \eps} \right)}$, for $s \in \R$.

The situation is very different in the \emph{vectorial case}, since relation \eqref{first-order} is less constraining: Under additional assumptions on the potential $V$, one may find several profiles connecting two minimizers of the potential (see e.g \cite{Alifusco} and references therein). The search for such solution is still an active field of research (see for instance \cite{Alibetu,  Zuster, monsant}).  
As we will see next, the genuine non-uniqueness of one-dimensional profiles is a first source of  important difference also in the higher dimensional case, in particular concerning the conservation law \eqref{first-order}. 

\medskip  
\noindent
{\bf The higher dimensional case}

\noindent
The higher-dimensional theory in the scalar case is rather advanced  and  a very satisfactory theory has been set up in any dimension $N\geq 2$. As mentioned,  the existence of a $(N-1)$-dimensional  set $\mathfrak S_\star$ is established in \cite{ilmanen, hutchtone}. Moreover,  it is shown  there that the  $(N-1)$-rectifiable  set $\mathfrak S_\star$,  equipped with the energy density corresponding to the measure $\upnu_\star$ defined in
\eqref{choux}   is a \emph{stationary rectifiable  varifold}.
 The results in  \cite{hutchtone}  embody  the intuitive idea that locally, the equation reduces to a one-dimensional problem. More precisely,  
 typically, in dimension two, the expected situation reduces,  \emph{locally near some point}  $x_0$,  to the case
\begin{equation}
\label{quoique}
u_\eps(x)\underset{x \to x_0} \simeq w_\eps (x_2),  {\rm \ with \ }  x=(x_1, x_2) \in \R^2, 
\end{equation} 
where the coordinates are chosen so that the tangent to $\mathfrak S_\star$ at $x_0$ has equation $x_2=0$, and where $w_\eps$ stands for a solution to the one-dimension problem \eqref{heterocline}   (see Figure \ref{fonction}). Notice that the possibility of gluing of several such solutions is not excluded, but we will not discuss this here.   Ultimately, the results in  \cite{ilmanen, hutchtone} provide a rather simple picture of the solutions. They involve a minimal surface, the solution may be represented  as a one-dimensional profiles glued to the surface in the transversal direction, so that one is tempting to write the correspondance
\begin{equation}
\label{sim}
 {\rm \ solutions \ to \   \eqref{elipes} }  \sim {\rm \ minimal \ surface} + {\rm glued \ profiles}.
 \end{equation}
 The general structure of solution is hence fairly well understood (see Figure \ref{laglue}). 
As a matter of fact, the correspondance goes to some extend in either way, since, conversely, given a minimal surface, one may construct  solutions  to the scalar Allen-Cahn equation having the previous behavior (see \cite {pacarwei}).  This should be also connected with the famous De Giorgi conjecture (\cite{degiorg}) (see \cite{savin}, and references therein).
 \begin{figure}[h]
\centering
\includegraphics[height=7cm]{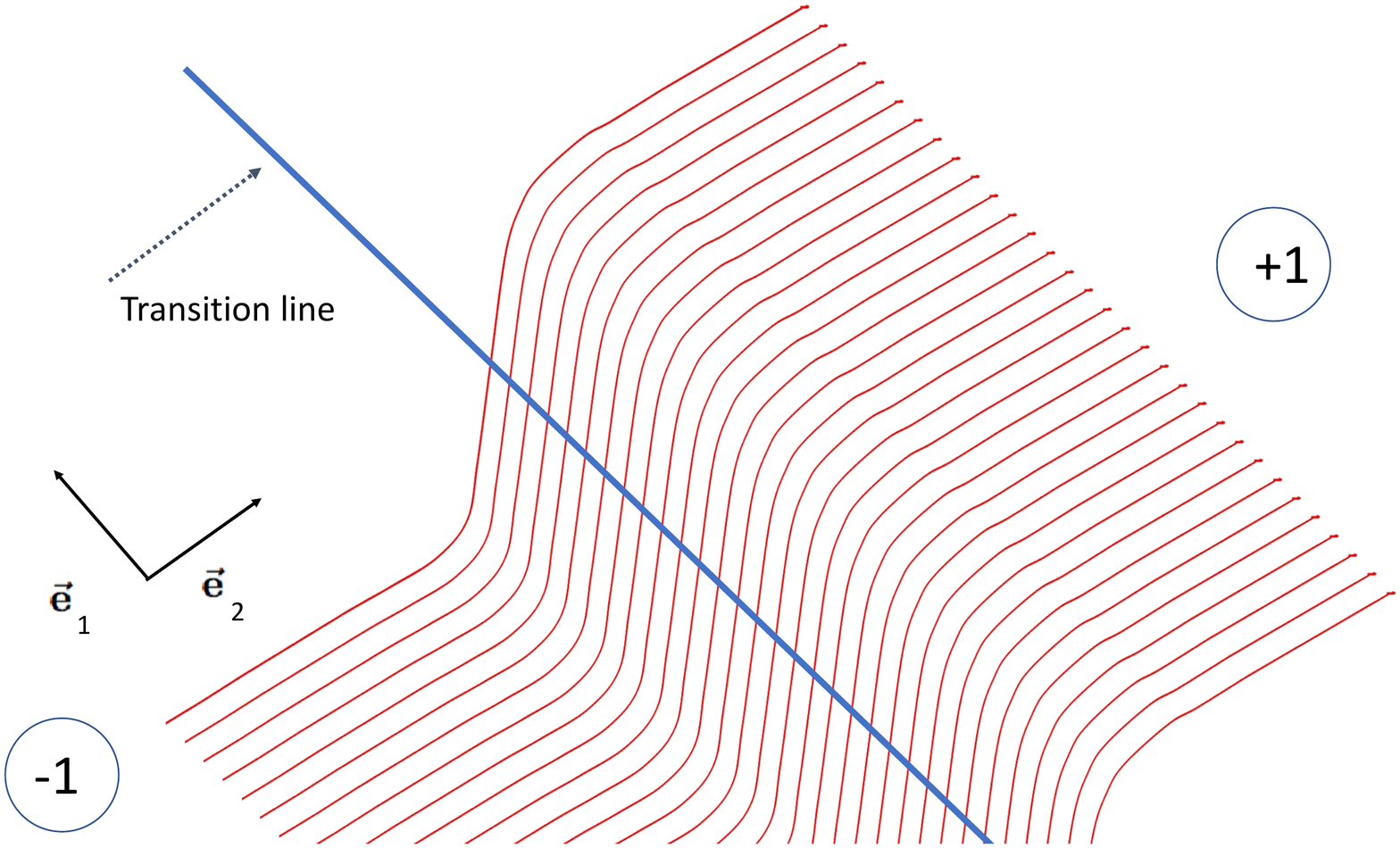}
\caption{  {\it Interface near a regular point $x_0$ in the scalar case, with an Allen-Cahn type potential.}}
\label{laglue}
\end{figure}

The picture  in the   vectorial case is more complex. Firstly, as we have already seen, the set of one dimensional profiles is much larger, it could perhaps be even infinite.  Besides this, there are solutions which  \emph{cannot be reduced to one dimensional profiles},   in view of results  in \cite{alabrogui} and \cite{bethol}, and are hence genuinely multi-dimensional, so that a property similar the \eqref{quoique}  or \eqref{sim} \emph{cannot not be expected in full generality}. 

In \cite{bethol}, it is shown that, under specific conditions on the potential $V$, one may \emph{mountain-pass  solutions} to $\displaystyle{-\Delta u=\nabla_u V(u)}$  on the cylinder 
 $\Lambda_\rL=[-\rL, \rL] \times \R$ provided $L>0$ is sufficiently large, 
 with periodic boundary conditions in the $x_1$ direction, namely such that 
 \begin{equation}
 \label{grouicper}
 u(-\rL, x_2)=u(\rL, x_2) {\rm \ and \ } \frac{\partial u}{\partial x_1}(-\rL, x_2)= \frac{\partial u}{\partial x_1 }(\rL, x_2), {\rm \ for \ any \ }  x_2 \in \R.
 \end{equation}
The solution  obtained in \cite{bethol}  is \emph{not a one-dimensional profile}, since one may show that there are   also \emph{tangential  contributions}: Indeed, we have 
 \begin{equation}
  \label{bla}
 \frac{\partial u}{\partial x_1}\not  =  0  {\rm \ on \ } \Lambda_\rL.
  \end{equation}
One then considers the   scaled map on $\R^2$ defined for $x=(x_1, x_2) $ by 
\begin{equation}
\label{glasnost}  
 u_\eps(x)=u\left(\frac{x-N\eps\be_1}{\eps}\right),  {\rm  \ if  \ }   x_1 \in [N\eps, (N+1)\eps]
 \end{equation}
 which solves \eqref{elipes} on $\R^2$ (see Figure \ref{pseudoP}).  Moreover, it follows from \eqref{bla}, that for {the transversal derivative}, we have 
 \begin{equation}
  \eps  \left  \vert  \frac{\partial u_\eps}{\partial x_1}  \right  \vert^2  \rightharpoonup \upmu_{\star, 1,1} \not =0,
 {\rm \ where \ }  \upmu_{\star, 1, 1}=c \mathcal H^1(D) {\rm \ with \ } D=\left\{(x_1, 0), x_1\in \R \right\}, 
 \end{equation}
  for some constant $c>0$.  Finally it can be shown that the set of densities obtained for such solution is infinite, choosing various values for the constant $\rL>0$.

\begin{figure}
 \centering
 \includegraphics[height=6cm]{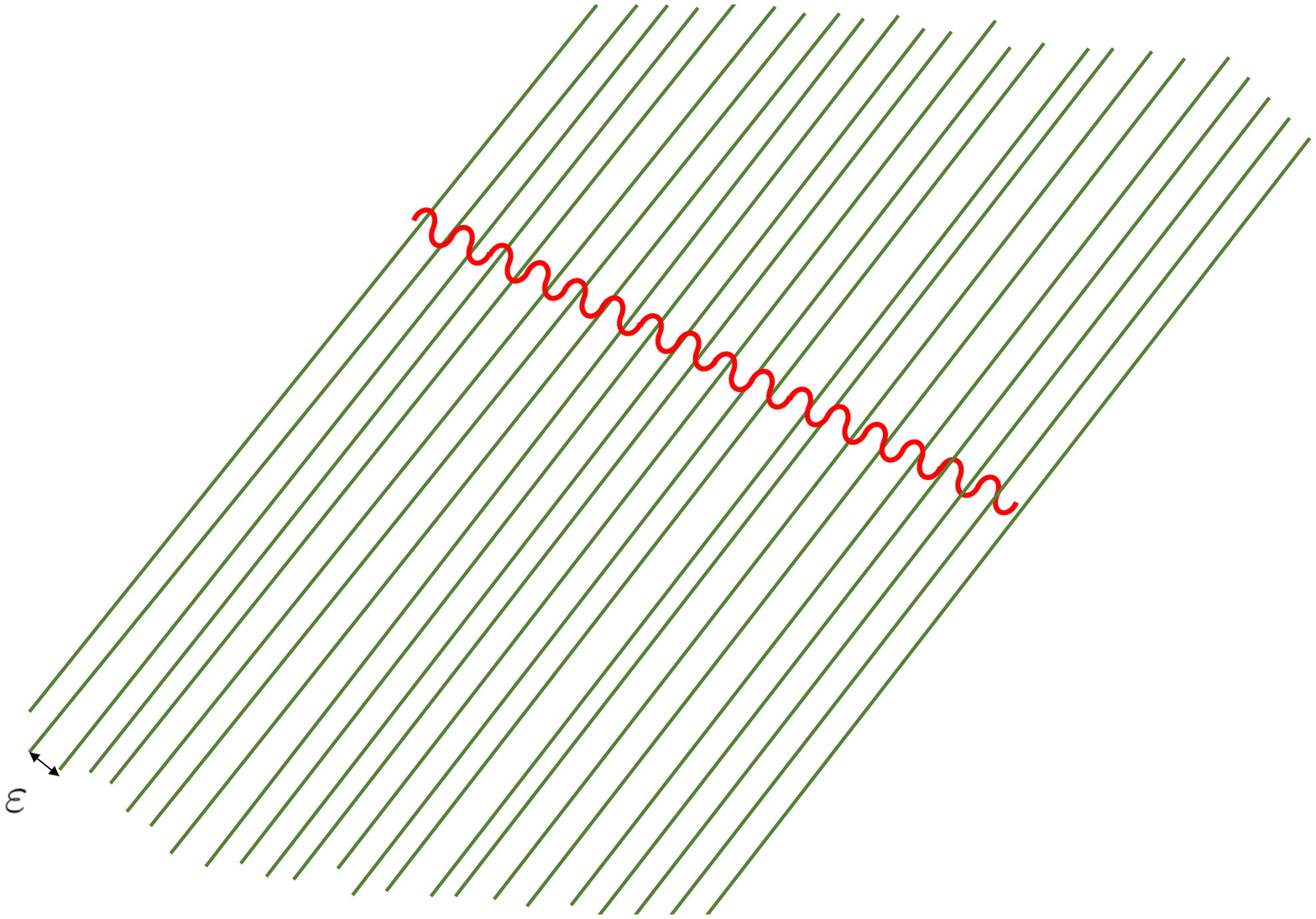}
	\caption{\it Interface with a periodic pseudo-profile}
	\label{pseudoP} 
	\end{figure}


 \subsection{Comparing   the methods in the scalar  and vectorial cases} 
 \label{comparse}
 \noindent
{\bf Monotonicity for the energy in the scalar case} 

\noindent
A large part  of  the arguments developed for    the scalar theory, as well as actually in the present paper, rely on integral estimates, starting with the energy, but also the integral of the potential. In the present context, we set  for an arbitrary subdomain $G \in \Omega$, 
  \begin{equation}
  \label{energyu}
  \Eps\left(u_\eps, G\right)=\int_{\mathcal U} e_\eps(u)  {\rm d} x  {\rm \ and \ } \Veps(u, G)=\frac{1}{\eps}\int_{\mathcal U} V(u)  {\rm d} x.
  \end{equation}
  Monotonicity formula play a distinguished role in the field. We recall that the monotonicity formula
   \begin{equation*}
  \label{monotoniee}
  \frac{d}{dr} \left(\frac{1}{r^{N-2} }  \Eps\left(u_\eps, \B^N(x_0, r)\right) \right) \geq 0,  {\rm \ for \ any \ } x_0 \in \Omega,  
  \end{equation*} 
   holds for arbitrary potentials, and is relevant if one wants to establish  concentration on $N-2$ dimensional sets, as it occurs in Ginzburg-Landau theory. If one wants instead to establish concentration on $N-1$ dimensional sets, then the stronger monotonicity formula 
    \begin{equation}
  \label{monotonie}
  \frac{d}{dr} \left(\frac{1}{r^{N-1} }  \Eps\left(u_\eps, \B^N(x_0, r)\right) \right) \geq 0,  {\rm \ for \ any \ } x_0 \in \Omega,  
  \end{equation} 
  seems more appropriate. As a matter of fact,  we have, in dimension $N=2$,   the identity (see Remark \ref{monotoinou})
  \begin{equation}
   \label{monantoine0}
   \frac{d}{dr}\left( \frac{\Eps\left(u_\eps, \D^2(r)\right)}{r}\right)=\frac{1}{r^2}\int_{\D^2(r)} \xi_\eps(u_\eps){\rm d} x  +
   \frac{1}{r} \int_{\S^1(r)}\vert \frac{\partial u_\eps}{\partial r} \vert ^2 {\rm d} \ell, 
   \end{equation}
   where $\xi_\eps(u_\eps)$  denotes the discrepancy function given  by
    \begin{equation}
 \label{discretpanse}
 \xi_\eps (u_\eps)= \frac{1}{\eps} V(u_\eps)-\eps \frac{\vert \nabla  u \vert^2}{2}.
 \end{equation}
   Notice that, in view of \eqref{first-order}, the discrepancy function vanishes for one-dimensional profiles, a property which allows to compute solution in the scalar case as seen before.
   
   Formula \eqref{monotonie} has been established in  \cite{ilmanen}  in the \emph{scalar case}. The proof provided in \cite{ilmanen} relies strongly on  \emph{the positivity} of the  \emph{discrepancy}  function $\xi_\eps$, 
a property obtained  thanks to the maximum principle. The fact that $\xi_\eps$ is positive for \emph{scalar solutions} of \eqref{elipes} was observed first by L. Modica in \cite{modica} for entire solutions.  It is actually  proved in \cite{hutchtone} that the discrepancy $\xi_\eps$ vanishes asymptotically as $\eps \to 0$.

Inequality \eqref{monotonie} is the cornerstone of the  scalar theory,  as developed in \cite{ilmanen, hutchtone}. It yields both upper and lower bounds for the concentration of the energy. A large part of the arguments deal with properties of  limiting measures, obtained as   $\eps \to 0$.  Instead of the measure  $\upzeta_\star$ which appears both in Theorem \ref{maintheo} and Theorem \ref{segmentus}, obtained as a limit of the potential (see \eqref{boulga00} and \eqref{boulga}), the central tool in the scalar case is the corresponding measure for the full energy.  More precisely,  consider  the family $(\upnu_\eps)_{0<\eps\leq 1} $ of measures defined on $\Omega$  by
  \begin{equation}
  \label{mesure}
  \upnu_\eps \equiv e_\eps(u_\eps) \, {\rm d}\, x    {\rm  \ on  \ } \Omega, 
    \end{equation}
so that,    in view of \eqref{naturalbound}, the total mass of the measures is bounded by ${\rm M}_0$, that is 
   $\upnu_\eps (\Omega) \leq {\rm M}_0.$
 By compactness, there exists    a decreasing subsequence $(\eps_n)_{n\in \N}$  tending to $0$ and a limiting measure $\upnu_\star$    on $\Omega$ with $\upnu_\star (\Omega) \leq {\rm M}_0$, such that 
  \begin{equation}
  \label{choux}
  \upnu_{\eps_n}\rightharpoonup \upnu_\star  {\rm \ in \ the \ sense \ of \ measures  \ on \ }
  \Omega  {\rm \ as \ } n\to +\infty.
  \end{equation}  
  A first straightforward consequence of the  monotonicity formula \eqref{monotonie} for the energy   is  that   the one-dimensional density of the measure $\upnu_\star$ is bounded from above. This property then implies  that the the concentration set $\mathfrak S_\star$ of $\upnu_\star$ has \emph{at least dimension one}.
  Combining the monotonicity \eqref{monotonie} with a weak form of the clearing-out property, similar to Proposition \ref{brioche} in the present paper, the monotonicity formula yields also a \emph{lower bound} on the density of $\upnu_\star$ 
which is hence  bounded away from zero. This property  implies  that  the concentration set $\mathfrak S_\star$ of $\upnu_\star$ has \emph{at most dimension one}, hence its dimension is \emph{exactly one}. The previous discussion therefore  the concentration property  of $\upnu_\star$ is a direct consequence of \eqref{monotonie}. 

Notice also that  the previous arguments show that the  measure $\upnu_\star$ is \emph{absolutely continuous with respect to $\rd \lambda$}, the $\mathcal H^{N-1}$ measure on $\mathfrak S_\star$, so that one may write $\upnu_\star=\tbe\,  \rd \lambda$, where $\tbe\, $ is a integrable function on $\mathfrak S_\star$. Going to the limit $\eps\to 0$ in \eqref{discretpanse}, we obtain, since $\xi_\eps \to 0$ as
 $\eps \to 0$, 
\begin{equation}
\label{asympdiscr}
2\upzeta_\star=\upnu_\star, 
\end{equation}
 a relation which in some sense extends \eqref{first-order} to the high-dimensional setting. We will see, in contrast,  that  relation \eqref{asympdiscr} \emph{does not extend} to the vectorial case.

 \begin{remark}{\rm   As already mentioned, it has been proven in \cite{ilmanen, hutchtone} that, in the sclar case, the  rectifiable  varifolds 
 ${\bf V} (\mathfrak S_\star, \mathbbm e)$  corresponding to the measure $\upnu_\star$ is stationary. In view of relation \eqref{asympdiscr}  this implies that the rectifiable varifold  ${\bf V} (\mathfrak S_\star, \upzeta_\star)$  corresponding to the measure $\upzeta_\star$ is also stationary, which is hence consistent with Theorem \ref{varifoltitude} of the present paper.}
 \end{remark} 
 
 \bigskip
 \noindent
{\bf Circumventing  lack of  monotonicity  for the energy in the two-dimensional vectorial case}. 

\noindent
Concerning the vectorial case,  non-negativity  of the discrepancy as well
 as the monotonicity formula are known to fail for some solutions of the \emph{Ginzburg-Landau system}, so that the question whether they might still hold  under some possible additional conditions on the potential  or the solution itself is widely open to our knowledge (see \cite{Ali1} for a discussion  of these issues and for additional references). 
 
 \smallskip
In  order to circumvent the  lack of monotonicity formula for the energy,   we have to work out new results on the level of solutions to  PDE (termed in the paper the \emph{$\eps$-level}), which will be present in Subsection \ref{newpde}. The clearing-out result given in Theorem \ref{clearingoutth}  is central in our analysis: It  implies, as in the scalar case, that the set $\mathfrak S_\star$ has dimension at most one. Combining with several other results for the PDE, we are able to deduce most of the properties developed in Theorem \ref{maintheo}.

 For the proofs of Theorems \ref{segmentus} and \ref{varifoltitude}, the fact that the measures $\upzeta_\star$ and $\upnu_\star$ are absolutely continuous with respect to the $\mathcal H^1$ measure of $\mathfrak S_\star$ is \emph{crucial}.  We will show, in the last part of this paper:
 
 \begin{theorem}
  \label{absolute}
   The measures $\upnu_\star$ and  $\upzeta_\star$   have support on the set $\mathfrak S_\star$ defined in Theorem \ref{maintheo}, and   are  absolutely continuous with respect to $\rd\lambda=\mathcal H^1 \rest \mathfrak S_\star$, the one-dimensional Hausdorff measure on $\mathfrak S_\star$. 
       Let $\mathbbm e_\star$ and $\Uptheta_\star$ denote the densities of $\upnu_\star$ and $\upzeta_\star$ with respect to $\rd \lambda$ respectively, so that
    $\upnu_\star=\mathbbm e_\star \rd \lambda$ and $\upzeta_\star=\Uptheta_\star \rd \lambda$.  We have the inequalities, for $x \in \mathfrak S_\star$, 
    \begin{equation}
    \label{ondanse}
    \left\{
    \begin{aligned}
   &\upeta_1 \leq \mathbbm e_\star(x)\leq K_{\rm dens} \left({\rm d}(x)\right) \Uptheta(x) ,  
      {\rm \ and } \\
   & \Uptheta_\star(x)\leq \frac{M_0}{{\rm d}(x)}, 
    \end{aligned}
    \right.
    \end{equation}
 where $\upeta_1>0$ is some constant depending only on $V$,  ${\rm d}(x)={\rm dist}(x, \partial \Omega)$ and 
$K_{\rm dens} \left({\rm d}(x)\right)$ denotes a constant depending only on $V$, $M_0$ and ${\rm d}(x)$.
\end{theorem} 
 
 Notice that we have also the straightforward inequality $\upzeta_\star \leq \upnu_\star$,  so that $\Uptheta_\star\leq \mathbbm e$.
 It follows from the inequalities \eqref{ondanse} that the densities $\mathbbm e_\star$ and $\Uptheta_\star$ are locally bounded from above and away from zero.

  \bigskip
 \noindent
 {\bf A new discrepancy relation}
 
 \noindent
  Our arguments require to split the energy, in particular the  gradient term,  into its components, leading to   several other measures.     For a given orthonormal basis $(\be_1, \be_2)$,  we  consider, for $i, j=1, 2$,  the quadratic gradient terms $\eps {u_{\eps}}_{x_i}\cdot  {u_{\eps}}_{x_j}$, and pass  to the limit $\eps\to 0$, extracting possibly a further subsequence 
 \begin{equation}
 \label{boulga2}
 \eps {u_{\eps_n}}_{x_i}\cdot  {u_{\eps_n}}_{x_j}\rightharpoonup  \upmu_{\star, i, j}  {\rm \ in \ the \ sense \ of \  measures \ on  \  }\Omega,  {\rm \ as \ } n \to + \infty, {\rm for \ }  i, j=1, 2.
   \end{equation}
 where  $\upmu_{\star, i, j}$  denotes a bounded (signed) Radon measure on $\Omega$. Notice that
 \begin{equation}
 \label{citrus}
-2\upnu_\star \leq  \upmu_{i, j} \leq2 \upnu_\star  {\rm \ and \ } \upmu_{\star, j, i}=\upmu_{\star, i, j}.
 \end{equation}
In the scalar case, the fact that solutions essentially reduce to the one-dimensional profile with respect to the transversal direction, also implies the vanishing of the tangential  contributions to the gradient terms. More precisely, we may write,  in view of Theorem \ref{absolute} since $\upnu_\star$ is absolutely continuous with respect to $\rd \lambda$  
\begin{equation}
 \label{gdansk0}
 \upmu_{\star, i, j}={\bm}_{\star, i, j} \rd \lambda, 
 \end{equation}
where ${\bm}_{\star, i, j}$ is an integrable function on $\mathfrak S_\star$. The definition and values of $\bm_{\star, i, j}$ strongly depend on the the choice of orthonormal frame. In order to derive some more intrinsic objects, we may work in a \emph{moving frame}
associated to $\mathfrak S_\star$. More precisely if  $x_0 \in \mathfrak S_\star \setminus \mathfrak E_\star$, and if   the  orthonormal frame $(\be_1, \be_2)$ is  choosen so that $\be_1=\vec e_{x_0}$, then we set
\begin{equation}
\label{parallel}
\bm_{\star, \perp, \perp}(x_0)=\bm_{\star, 2,2}(x_0), \   \,  
\bm_{\star, \parallel, \parallel}(x_0)=\bm_{\star, 1,1}(x_0)   {\rm \ and \ } 
\bm_{\star, \perp, \parallel}(x_0)=\bm_{\star, 1,2}(x_0),
\end{equation}
  and define the measures 
  \begin{equation}
  \label{mespara}
  \upmu_{\star, \perp, \perp}=\bm_{\star, \perp, \perp}  \rd \lambda, \ 
   \upmu_{\star, \parallel, \parallel}=\bm_{\star, \parallel, \parallel}  \rd \lambda, \  {\rm \ and \ } 
    \upmu_{\star, \perp, \parallel}=\bm_{\star, \perp, \parallel}  \rd \lambda. 
  \end{equation}
In the scalar case, the fact that  the tangential contributions vanish (see \cite{hutchtone}) can be expressed as
  \begin{equation}
 \label{transvanish}
 \left\{
 \begin{aligned}
  \upmu_{\star, \parallel, \parallel}(x_0)&=0   {\rm \ when \ } k=1  \  ({\rm i.e. \ in \ the \ scalar \ case})  {\rm \ and \ }  \\
 \upmu_{\star, \perp, \parallel}&=0   {\rm \ when \ } k=1  \  ({\rm i.e. \ in \ the \ scalar \ case}).
  \end{aligned}
  \right.
 \end{equation}
 On the other hand,  vanishing of the discrepancy leads to (see \cite{hutchtone} once more)
 \begin{equation}
\label{vanishdisc}
2\upzeta_\star=\upmu_{\star, \perp,  \perp},  {\rm \ when \ } k=1  \  ({\rm i.e. \ in \ the \ scalar \ case}). 
 \end{equation}
It turns out that the relation \eqref{vanishdisc}  does not hold \emph{in general for the vectorial case}. 
Indeed,    for the map constructed in \cite{bethol} and given  in \eqref{glasnost}, we have  $\upnu_{\star, \parallel, \parallel}\not =0$,  so that the first relation in \eqref{transvanish} is not \emph{satisfied}. We will see later that  the second one is always  satisfied, whereas  the discrepancy relation \eqref{vanishdisc} \emph{is not, in general}.   Our next result provides a generalization of \eqref{vanishdisc} for the vectorial case.
  
  \begin{theorem} 
\label{discrepvec}  We have the identities 
\begin{equation}
\label{discrepvec11}
2\upzeta_\star=\upmu_{\star, \perp,  \perp}-\upmu_{\star, \parallel, \parallel}  \  {\rm \ and \  }
\upmu_{\star, \perp,  \parallel}=0.
\end{equation}
\end{theorem} 
Notice that, in view of identities \eqref{transvanish}, the discrepancy identity \eqref{vanishdisc} appears as a special case of \eqref{discrepvec11}.

\bigskip
\noindent
{\bf Recovering monotonicity} 

\smallskip
\noindent
So far, we have introduced in Theorems \ref{maintheo}, \ref{segmentus},  \ref{varifoltitude}, \ref{absolute} and \ref{discrepvec} the main 
results of this paper. As mentioned, many arguments have to be carried out without monotonicity formula, in particular Theorem \ref{maintheo}. However, in order to obtain the proofs of Theorems \ref{segmentus} to \ref{discrepvec}, we rely ultimately  on a \emph{new monotonicity formula}, which we describe at the end of this subsection.

\smallskip
  Before doing so, let us emphasize that,  in order to prove Theorem \ref{absolute} and several intermediate results,  we  rely  in an essential way on Lebesgue's decomposition theorem  for measures which assert that measures at hand can be  decomposed into an absolutely continuous part and a singular part with respect to the one-dimensional measure  Hausdorff measure on $\mathfrak S_\star$. More precisely,  we  decompose the measure $\upzeta_\star$ and 
  $\upnu_\star$ as 
\begin{equation}
\label{lebesgue}
\upnu_\star=\upnu_\star^{s}+ \upnu_\star^{ac}, {\rm \ and  \ }  \upzeta_\star=\upzeta_\star^{s}+ \upzeta_\star^{ac}
\end{equation}
where the measures $\upnu_\star^{ac}$ and $\upzeta_\star^{ac}$  are absolutely continuous with respect to the measure  
$\mathcal H^1 \rest \mathfrak S_\star$, that is 
  $$
  \upnu_\star^{ac} \ll\mathcal H^1 \rest \mathfrak S_\star  {\rm \ and \ }   \upzeta_\star^{ac} \ll\mathcal H^1 \rest \mathfrak S_\star,   
  $$
 and
\begin{equation}
\label{singularmes}
  \upnu_\star^{s} \perp\upnu_\star^{ac}  {\rm \ and \ }  \upzeta_\star^{s} \perp  \upzeta_\star^{ac}. 
  \end{equation}
We are then in position to write, prior to the proof of Theorem \ref{absolute}, 
  $ \upnu_\star^{ac}=\mathbbm e_\star \rd \lambda$  and $\upzeta_\star^{ac}=\Uptheta \rd \lambda.$ 
  An important intermediate step in the paper, is a preliminary version of Theorem \ref{discrepvec} (see Proposition \ref{discrepvec6})
established only for \emph{the absolutely continuous } parts of the measures.
   \\

 In order to show that $\upnu_\star^{s}=\upzeta_\star^{s}=0$, the cornerstone of the argument is an \emph{alternate  differential inequality} for solutions of the \eqref{elipes}. We have indeed, for any $x_0\in \Omega$  such that $D(x_0, r) \subset \Omega$  (see Subsection \ref{monotoinou} for the proof), the differential relation
  \begin{equation}
   \label{monotonio0}
  \frac{1}{\eps^2} \frac{d}{dr}\left( \frac{V\left(u_\eps, \D^2(x_0, r)\right)}{ r}\right)=  
  \frac {1}{4 r}\int_{\partial \D^2(x_0,  r) } 
  \left(\frac{2}{\eps^2}V(u_\eps)-
 r^2 \left \vert\frac{\partial  u_\eps}{\partial \theta}\right \vert^2 +\left\vert \frac{\partial  u_\eps}{\partial r} \right\vert^2
    \right) {\rm d} \tau. 
\end{equation}
Although this does not transpire from the formula above, we will see that the right hand side has, in an asymptotic limit $\eps \to 0$,  an appropriate sign, yielding monotonicity for the measure $\upzeta_\star$: As a matter of fact, it turns out that the function
$\displaystyle {\upzeta_\star (\D^2(x_0, r))\slash r}$ is non-decreasing (see Proposition \ref{psycho}). This yields  
 and upper bound for the density of $\upzeta_\star$, so that the singular part vanishes.

In the next subsections, we provide more details on the structure of  the proof.

  \subsection{Elements in the proof of Theorem \ref{maintheo}: PDE  analysis }
  \label{newpde}
  As  mentioned, many of our main results,  dealing with the limiting measures,   are derived from corresponding results at the $\eps$-level for the map $u_\eps$, for given $\eps >0$, which rely on PDE methods.  We describe nexts these  PDE  results.

  \subsubsection{Scaling invariance of the equation} 
  \label{squale}
  
  As a first  preliminary remark,    we notice the invariance of the equation by translations as well as scale changes, which  plays an important role in our later arguments. Given  any fixed  $r>0$  and $\eps>0$,  we introduce the corresponding  \emph{scaled parameter} $\displaystyle{ \tilde \eps=\frac{\eps}{r} }$. For a given map $u_\eps : \D^2(x_0, r) \to \R^k$, we   consider  the  \emph{scaled (and translated)  map}  $ \tilde u_{\eps}$ defined on the unit disk $\D^2$  by 
  $$ \tilde u_{\eps}(x)=u_\eps ( rx+ x_0)), \forall x \in \D^2.$$
  If the map $u_\eps$  is a solution to \eqref{elipes}, when  the map $\tilde u_{\eps}$ is a solution to \eqref{elipes} with  the parameter $\eps$ changed into $\tilde \eps$.  The scale invariance of the energy  is given by  the relation 
   \begin{equation}
   \label{scaling0}
   e_{\tilde \eps}(\tilde u_\eps)(x)= r e_\eps(u)(rx+ x_0), \, \forall x \in \D^2.
   \end{equation}
   Integrating this identity, we obtain  the integral relations 
   \begin{equation}
   \label{scalingv}
   {\E_\eps}\left(u_\eps, \D^2(r)\right)= {r} {\E}_{\tilde \eps}\left(\tilde u_\eps, \D^2(1)\right)  {\rm \ and \ } 
   \mathbb V_\eps \left(u_\eps, \D^2(r)\right)= {r} \mathbb V_{\tilde \eps} \left(\tilde u_\eps, \D^2(1)\right),
      \end{equation}
   where we have made  use of the notation \eqref{energyu}.
It follows from the previous discussion that the parameter $\eps$ as well as  the energy $\E_\eps$  behave, according to the previous scaling laws, essentially  as lengths. In this  loose  sense,  inequality \eqref{scalingv} shows  that the quantity $\eps^{-1} E_\eps$ is scale invariant.
  
  \subsubsection{ The $\eps$-clearing-out Theorem}
  
  We next provide clearing-out results for solutions of the PDE \eqref{elipes}. 
    In view of the   assumptions  ${(\text{H}_1)}$, ${(\text{H}_2)}$ and ${(\text{H}_3)}$ on the potential $V$,  we may choose some  constant  $\upmu_0>0$ sufficiently small so that
    \begin{equation}
    \label{kiv}
    \left\{
    \begin{aligned}
 B^k(\upsigma_i, 2\upmu_0)\cap \B^k(\upsigma_j, 2\upmu_0) &= \emptyset 
 {\rm \ for \  all  \ } i\neq j  {\rm \  in \ }  \{1,\cdots,q\} {\rm \ and \  such  \   that  \ }  \\
 \frac{1}{2}\lambda_i^-{\rm Id} \leq  
 \nabla^2 V (y) &\leq 2\lambda_i^+{\rm Id}   \  \ 
  {\rm \ for \  all  \ } i\in \{1,\cdots,q\} {\rm \  and  \ }  y \in B(\upsigma_i,2\upmu_0).
 \end{aligned}
  \right. 
\end{equation}
We then have:

  \begin{theorem}
 \label{clearingoutth}
  Let  $0<\eps \leq 1$ and  $u_\eps$  be a solution of 
  \eqref{elipes}  on $\D^2$. There exists some constant $\upeta_1>0$ such that 
  if 
  \begin{equation}
  \label{petitou}
   \Eps(u_\eps, \D^2) \leq  2\upeta_1,
   \end{equation}
  then  there exists some $\upsigma \in \Sigma$ such that 
  \begin{equation}
  \label{benkon}
  \vert u_\eps(x)-\upsigma  \vert  \leq \frac{\upmu_0}{2}, {\rm \ for \ every \ } x \in \D^2(\frac34), 
  \end{equation}
  where $\upsigma_0$ is defined in \eqref{kiv}. Moreover, we have the energy estimate, for  some constant 
  $\Cnrg>0$ depending only on the potential $V$
 \begin{equation}
 \label{engie}
\E_\eps\left(u_\eps, \D^2\left(\frac 58\right)\right) \leq \Cnrg \, \eps E_\eps(u_\eps, \D^2). 
 \end{equation}
 \end{theorem}
  
  Theorem \ref{clearingoutth}  is the main PDE result of the present paper: It paves the way to the concentration of measures on   set 
  $\mathfrak S_\star$, and will be used to show that its dimensional is \emph{at most one}.
 The main ingredient in the proof of Theorem \ref{clearingoutth} is provided by the following  estimate:
  
  \begin{proposition}
  \label{sindec}
   Let  $0<\eps \leq 1$ and  $u_\eps$  be a solution of 
  \eqref{elipes}  on $\D^2$. There exists a constant  $\Cdec>0$  depending only on $V$ such that
   \begin{equation}
   \label{bien}
    \int_{\D^2(\frac {9}{16})}e_\eps (u_\eps) {\rm d} x \leq  \Cdec  \left[ 
    \left(\int_{\D^2} e_\eps (u_\eps){\rm d} x\right)^{\frac 32}+ 
   \eps  \int_{\D^2} e_\eps (u_\eps){\rm d} x
    \right]. 
   \end{equation}
  \end{proposition}
  
 From a technical point of view,  Proposition \ref{sindec} is perhaps \emph{the main new ingredient} provided by the present paper: When both $\E_\eps(u_\eps)$ and $\eps$ are small, it provides a fast decay of  the energy on smaller balls.

 \smallskip
  
  Combining the result \eqref{bien} of proposition \ref{sindec}  with the scale invariance properties of the equation given in  subsection \ref{squale},  we obtain corresponding results for arbitrary discs $\D^2(x_0, r)$.  Indeed, applying  Proposition \ref{sindec}  to the map $\tilde u_{\eps}$ with parameter $\tilde \eps=\eps\slash r$ and  expressing the corresponding inequality \eqref{bien} we obtain,  provided $\tilde \eps \leq 1$, i.e. $\eps\leq  r$,
  $$
  \E_{\tilde \eps}\left(\tilde u_\eps, \D^2\left(x_0, \frac{9}{16}\right)\right)\leq \Cdec \left[  \E_{\tilde \eps}\left(\tilde u_\eps\right)^{\frac32}+ \tilde   \eps \E_{\tilde \eps}\left( \tilde u_\eps\right) \right].
  $$
  Since ${\E}_{\tilde \eps}\left(\tilde u_\eps\right)=r^{-1}{\E_\eps}\left(u_\eps, \D^2(r)\right)$ and  
  $\E_{\tilde \eps}\left(\tilde u_\eps, \D^2(x_0, {9}\slash{16})\right)=r^{-1}{\E_\eps}\left(u_\eps, \D^2(9r\slash 16)\right)$
   we are led, provided $\eps \leq  r$, to the inequality
  \begin{equation}
  \label{bienscale}
 \E_\eps\left(u_\eps, \D^2\left(x_0, \frac{9r}{16}\right)\right)  \leq  \Cdec
  \left[\frac{1} {\sqrt r}\, 
    {\left(\E_\eps \left(u_\eps, \D^2(x_0, r)\right) \right)}^{\frac 32}+ 
  \frac{ \eps}{r} \E_\eps \left(u_\eps, \D^2(x_0, r)\right)
    \right]. 
    \end{equation}

 Iterating this decay estimate  on concentric discs centered at $x_0$, and combining with elementary properties of the solution $u_\eps$, we eventually obtain  the proof of Theorem \ref{clearingoutth}.

 \smallskip
 Invoking once more  the scale invariance properties of the equation given in  subsection \ref{squale}, the scaled version of Theorem \ref{clearingoutth} writes then  as follows: 
 
 \begin{proposition} 
 \label{cestclair}
 Let $x_0 \in \Omega$ and $0<\eps \leq r$ be given,  assume that $\D^2(x_0, r) \subset \Omega$ and let $u_\eps$ be a solution of \eqref{elipes} on $\Omega$. If
 \begin{equation}
 \label{donnez}
 \frac{\E_\eps\left (u_\eps, \D^2\left(x_0, r\right)\right)}{r} \leq  \upeta_1, 
 \end{equation}
 then there exist some $\upsigma \in \Sigma$ such that 
  \begin{equation}
  \label{benkono}
  \left\{
  \begin{aligned}
  \vert u_\eps(x)-\upsigma  \vert & \leq \frac{\upmu_0}{2}, {\rm \ for \ }  x \in \D^2(x_0, \frac{3r}{4}) {\rm \ and }\\\
 \E_\eps\left(u_\eps, \D^2\left(x_0, \frac{ 5r}{8}\right)\right)& \leq \Cnrg \, \frac{\eps}{r} E_\eps\left(u_\eps, \D^2\left(x_0, r\right)\right).
  \end{aligned}
  \right.
  \end{equation}
\end{proposition}
 The proof of Proposition \ref{cestclair} is  a straightforward consequence of Theorem \ref{clearingoutth} and the scaling  properties  given in  subsection \ref{squale}, in particular identities \eqref{scalingv}. 

\subsubsection{Other results at the $\eps$-level}
\smallskip
The analysis of  the limiting measures require some other ingredients, in particular concerning the interplay between the measure $\upzeta_\eps$ and $\upnu_\eps$, leading to the relations \eqref{ondanse} on the limiting densities. 

The connectedness of $\mathfrak S_\star$ also requires results at the $\eps$-level, in particular we will rely on  Proposition \ref{pave}. 

\smallskip
We next present the main tools for handling the measures and the concentration set $\mathfrak S_\star$.

\subsection{Elements in the proof of Theorem \ref{maintheo}: construction of  $\mathfrak S_\star$ and clearing out for the  measure  $\upnu_\star$}  
\label{addi}

  The set $\mathfrak S_\star$  in Theorem \ref{maintheo} and Proposition \ref{tangentfort} is obtained as a concentration set of the energy.  The properties stated in Theorem \ref{maintheo} are, for a large part,  consequences of the two results  we present next.  These results are deduced from  corresponding properties of solutions to \eqref{elipes}, and presented in the previous subsection.
  
     \smallskip 
  The first result we present nexxt represents  a classical form of a  clearing-out result for the measure $\upnu_\star$ and leads directly to the fact that energy concentrates on sets which are \emph{at most} one-dimensional.  
  \begin{theorem} 
\label{claire}
Let $x_0 \in \Omega$ and $r>0$ be given such that $\D^2(x_0, r) \subset \Omega$.  There exists a constant $\upeta_1>0$ such that, if we have 
\begin{equation}
\frac{\upnu_\star \left(\overline{\D^2(x_0, r)} \right)}{r} < \upeta_1, {\rm \ then   \  it  \  holds \  }
\displaystyle{\upnu_\star \left(\overline{\D^2(x_0,\frac{r}{2}}) \right)=0.}
\end{equation}
\end{theorem}

  The  previous  statement leads   to consider the one-dimensional  lower density   of the measure $\upnu_\star$ defined, for $x_0 \in \Omega$,  by 
$$\tbe_\star (x_0)={\underset{r \to 0} \liminf } \frac{\upnu_\star \left(\overline{\D^2(x_0, r)} \right)}{r},   $$
 so that $\tbe_\star(x_0) \in [0, +\infty]$. We define the set $\mathfrak S_\star$ as the \emph{ concentration set}  for  the measure $\upnu_\star$. More precisely,  we set 
\begin{equation}
\label{mathfrakSstar}
\mathfrak S_\star=\{ x \in \Omega, \tbe_\star (x_0)\geq\upeta_1\}, 
\end{equation}
where $\upeta_1>0$ is the constant provided by Theorem \ref{claire}. 
The fact that $\mathfrak S_\star$ is   closed of finite one-dimensional Hausdorff measure is then a rather direct consequence of the clearing-out property for the measure $\upnu_\star$ stated in Theorem \ref{claire}.

\begin{remark}
{\rm   Let us emphasize  once more  that the previous definition of $\mathfrak S_\star$ directly leads, by construction and in view of Theorem \ref{claire} to concentration of the measure $\upnu_\star$ and $\upzeta_\star$  on the set $\mathfrak S_\star$ and also a \emph{lower bound} on the density of $\upnu_\star$.   The \emph{upper bound} requires different arguments, in particular a monotonicity formula.
}
\end{remark}
\smallskip
The connectedness properties of $\mathfrak S_\star$ stated in Theorem \ref{maintheo}, part ii) require  a different type of clearing-out result.   Its statement involves general  regular subdomains $\mathcal U \subset \Omega$,  and,  for $\updelta>0$,  the related sets  (see Figure \ref{bordurier})
\begin{equation}
\label{Udelta}
\left\{
\begin{aligned}
 \mathcal U_\updelta&=\left \{  x \in \Omega,  {\rm dist}(x, \mathcal U)\leq \updelta\right\} \supset
  \mathcal U {\rm \ and \ }  \\
 \mathcal V_\updelta&=\mathcal U_\delta \setminus \mathcal U
 =\left \{  x \in \Omega, 0\leq  {\rm dist}(x, \mathcal U)\leq \updelta\right\}. 
 \end{aligned}
\right.
\end{equation}

\begin{theorem}
\label{bordurer} 
 Let $\mathcal U\subset \Omega$ be a open  subset of $\Omega$ and $\updelta>0$ be given. If we have  
\begin{equation}
\label{clearstream}
 \upnu_\star  (\mathcal V_\delta)=0  , {\rm \ then   \  it  \  holds \  }
\upnu_\star \left(\overline{\mathcal U} \right)=0.
\end{equation}
\end{theorem}
In other terms, if the  measure $\upnu_\star$ vanishes in some neighborhood of the boundary $\partial \mathcal U$, then it vanishes on 
$\overline {\mathcal U}$. This result will allow us to establish  connectedness properties of $\mathfrak S_\star$.   For instance, we will prove the following \emph{local connectedness property}:

\begin{proposition}
\label{connective}
Let $x_0\in \Omega$, $r>0$ such that $\D^2(x_0, 2r) \subset \Omega$. There exists a radius $\rho_0\in (r, 2r)$ such that $\mathfrak S_\star \cap \D^2(x_0, \rho_0)$ contains a finite union of path-connected components. 
\end{proposition}

\medskip
  \begin{figure}
 \centering
 \includegraphics[height=6cm]{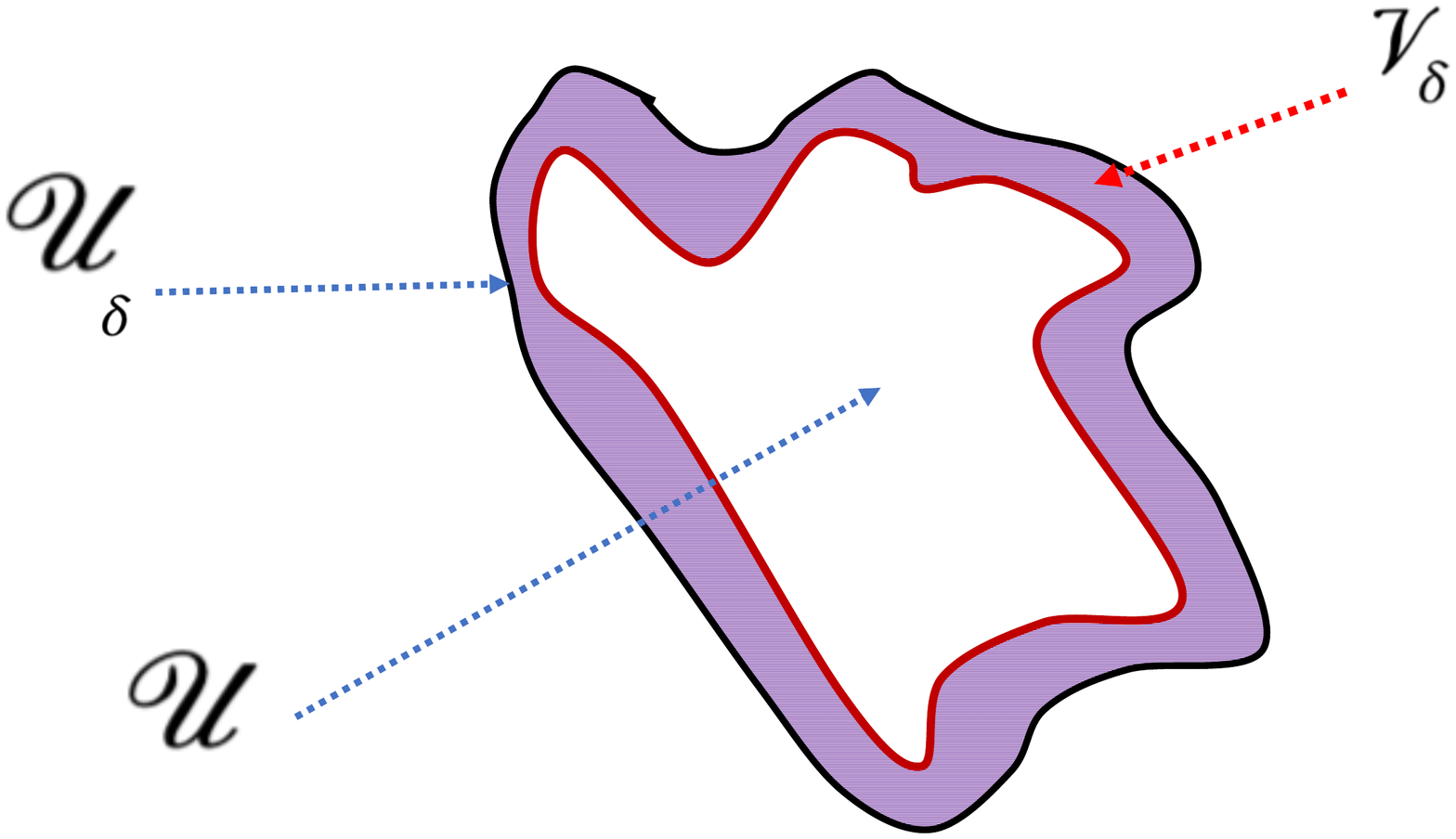}
	\caption{\it The sets $\mathcal U_\delta$ and $\mathcal V_\delta$.}
	\label{bordurier} 
	\end{figure}

 This connectedness property in Proposition \ref{connective}   implies   the  rectifiability  of $\mathfrak S_\star$, invoking classical  results on continua  of bounded one-dimensional Hausdorff measure (see e.g \cite{falconer}).   The proof of Theorem \ref{maintheo} is  then a combination of the results in Theorem \ref{claire}  and Proposition \ref{claire}. 
 
 \smallskip
     For  the  set  $\mathfrak S_\star$ given by Theorem \ref{maintheo}, the approximate tangent line  property \eqref{tangent}  can actually be strengthened as follows (see Figure \ref{comique}):
 
 \begin{proposition}
 \label{tangentfort} 
 Let $x_0$ be a regular point of $\mathfrak S_\star$. Given any $\Uptheta>0$ there exists  a radius $R_{\rm cone}(\uptheta, x_0)$ such that 
 \begin{equation}
 \label{radius}
\mathfrak S_\star \cap  \D^2\left(x_0, r\right) \subset  \mathcal C_{\rm one}\left(x_0, \vec e_{x_0}, \uptheta  \right), 
 {\rm  \  for \   any \  }  0<r\leq R_{\rm cone}(\uptheta, x_0).
 \end{equation}
 \end{proposition} 
 
  \begin{figure}
 \centering
 \includegraphics[height=6cm]{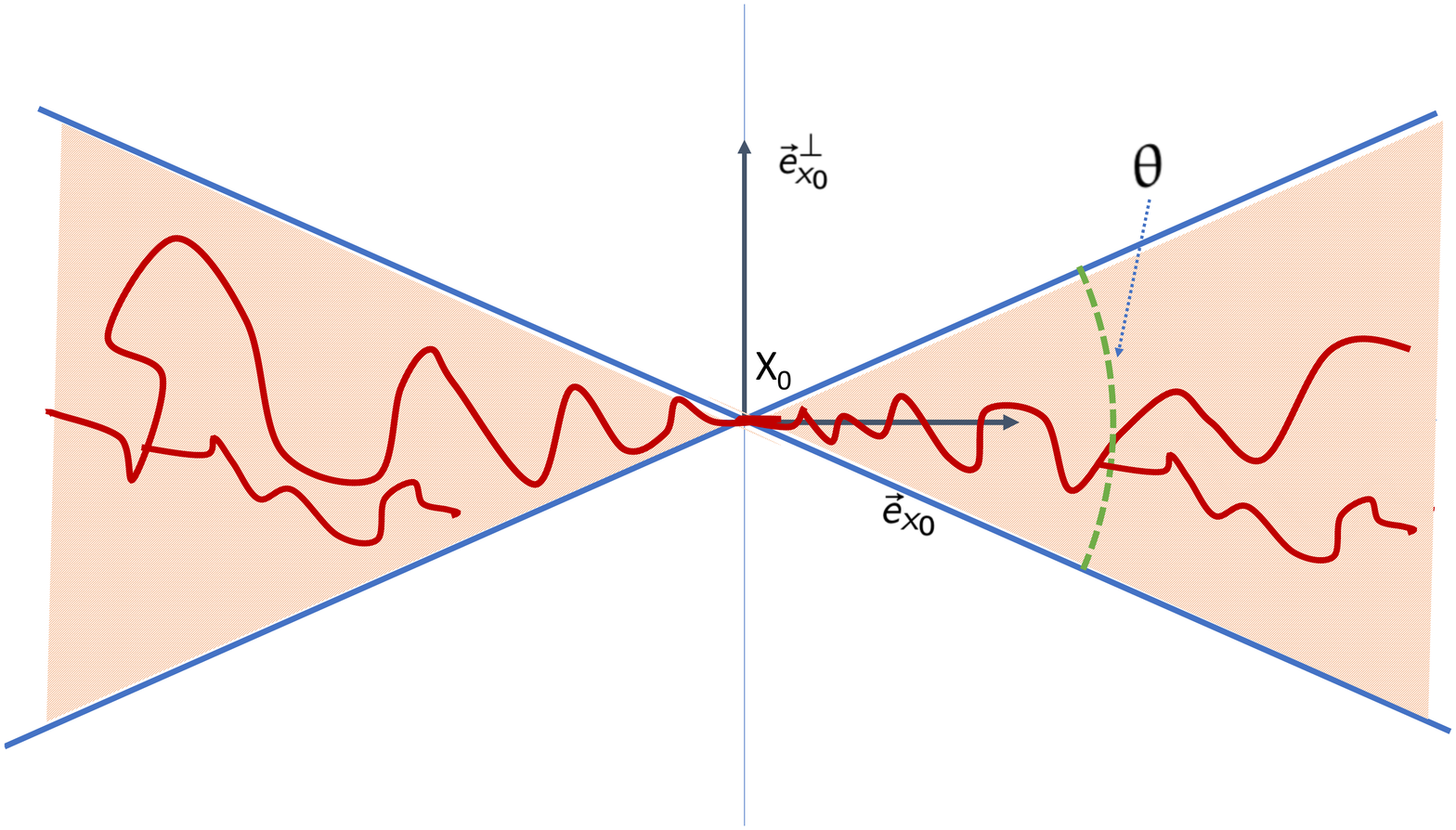}
	\caption{\it  The tangent cone property, as given in Proposition \ref{tangentfort}.}
	\label{comique} 
	\end{figure}

\subsection{A useful tool: The limiting Hopf differential $\omega_\star$}
\label{hopfite} 
   We introduce the complex-valued measure referred to as the \emph{limiting Hopf differential} 
   \begin{equation}
\label{schmil0}
\omega_\star=\left( \upmu_{\star, 1, 1} -\upmu_{\star, 2, 2}\right) -2i\upmu_{\star, 1, 2}, 
\end{equation}
where  the measures $\upmu_{i, j}$  have been defined in \eqref{boulga2}. Since the measures $\upnu_{\star, i, j}$ depend on the choice of orthonormal basis,  the expression of the Hopf differential also strongly depends on this choice.
The measures $\upzeta_\star$ and $\omega_\star$  are  strongly  related in view of our   next result.

\begin{lemma}
\label{holomorphitude}
We have, in the sense of distributions
\begin{equation}
\label{distributude}
\frac{\partial \omega_\star}{\partial \bar  {z} }=2\frac{\partial \upzeta_\star}{\partial{z}} \ {\rm \ in  \ }  \mathcal D'(\Omega).
\end{equation}
\end{lemma}

Relation \eqref{distributude} is the two-dimensional analog of the conservation law \eqref{first-order}. It  expresses the fact that the   energy of the solution $u_\eps$  is stationary  with respect to variations of the domain.  Since the measures $\upnu_\star$ and $\upzeta_\star$ are supported by $\mathfrak S_\star$,  identity \eqref{distributude} also expresses a stationary condition, when integrated against a test function,  for the set $\mathfrak S_\star$ and the measure 
$\upnu_\star$.   As  a matter of fact, identity \eqref{distributude}, is \emph{the starting point of the proofs of Theorems \ref{segmentus}, 
\ref{varifoltitude}, \ref{absolute} and \ref{discrepvec}.}

\medskip
Taking the real and imaginary parts of this relation, we  obtain, in the sense of distributions, the \emph{modified Cauchy-Riemann relations} 
\begin{equation}
\label{cauchise}
\left\{
\begin{aligned}
\frac{\partial}{\partial x_2} (2\upmu_{\star, 1, 2})&=\frac{\partial}{\partial x_1} \left(2\upzeta_\star -\upmu_{\star, 1, 1}+ \upmu_{\star, 2, 2} \right)  {\rm \ and \  } \\
\frac{\partial}{\partial x_1} \left( 2\upmu_{\star, 1, 2} \right) &=\frac{\partial}{\partial x_2}
\left(2\upzeta_\star+ \upmu_{\star, 1, 1}- \upmu_{\star, 2, 2}\right), 
 \end{aligned}
\right.
\end{equation}
the second relation being in some sense  the  closest  to \eqref{first-order}.

\smallskip
Our next result involve the decomposition of the measures into absolutely continuous parts with respect to $\rd \lambda=\mathcal H^1\rest \mathfrak S_\star$ and singular part and describe properties of the absolutely continuous part.
Besides \eqref{lebesgue}, we may also decompose the measures $\upmu_{\star, i, j}$, writing
\begin{equation}
\label{lebesguebis}
\upmu_{\star, i, j}=\upmu_{\star, i, j} ^{s}+ \upmu_{\star, i, j} ^{ac} {\rm \ with \ }  \upmu_{\star, i, j}^{s} \perp\upmu_{\star, i, j} ^{ac}, 
\end{equation}
where the measures $\upmu_{\star, i, j}^{ac}$   is absolutely continuous with respect to the measure  
$\rd \lambda=\mathcal H^1 \rest \mathfrak S_\star$.
 Relations \eqref{lebesgue} and \eqref{lebesguebis}  imply that there exists a set $\mathfrak B_\star\subset \mathfrak S_\star$ such that 
$\displaystyle{\mathcal H^1 \rest \mathfrak S_\star (\mathfrak B_\star)=0}$ and
\begin{equation}
\label{perpendicularite}
\upnu_\star^{s}(\mathfrak S_\star\setminus\mathfrak B_\star)=0, \,  \upzeta_\star^{s}(\mathfrak S_\star\setminus\mathfrak B_\star)=0, {\rm \  and  \ }  
\upmu_{\star, i, j}^{s}(\mathfrak S_\star\setminus\mathfrak B_\star)=0, {\rm \ for \ } i, j=1, 2. 
\end{equation}
Since the measures $\upzeta_\star$, $\upnu_\star$ and $\upmu_{\star, i, j}$ are absolutely continuous with respect to $\rd \lambda$ there exists functions ${\Uptheta}_\star$, $\tbe_\star$ and ${\bm}_{\star, i, j}$ defined on $\mathfrak S_\star$, such that we have
 \begin{equation}
 \label{gdansk00}
 \upzeta_\star^{ac}={\Uptheta}_\star\,\rd \lambda, \upnu_\star^{ac}={\tbe}_\star\,\rd \lambda, {\rm \ and \ }
  \upmu_{\star, i, j}^{ac}={\bm}_{\star, i, j} \rd \lambda, 
 \end{equation}

Besides $\mathfrak A_\star$ and $\mathfrak B_\star$, we introduce a third class of exceptional points, the set $\mathfrak C_\star$,  defined  the complementary of the set of Lebesgue points for the densities of the measures $\upmu_\star^{ac}, \upzeta_\star^{ac}, \upmu_{\star, i, j}^{ac}$ with respect to $\rd\lambda=\mathcal H^1 \rest \mathfrak S_\star$.
The set $\mathfrak S_\star\setminus \mathfrak C_\star$,  then corresponds to the intersection of the set of Lebesgue's points of the functions ${\Uptheta}_\star, {\tbe}$ and ${\bm}_{\star, i, j}$.
  We consider the union of all exceptional points 
 \begin{equation}
 \label{bordeaux}
 \mathfrak E_\star =\mathfrak A_\star \cup \mathfrak B_\star \cup \mathfrak C_\star, 
\end{equation} 
which is precisely the set appearing in Theorems \ref{maintheo} and \ref{segmentus}.

\begin{proposition} 
\label{discrepvec6} 
Let $x_0 \in \mathfrak S_\star \setminus \mathfrak E_\star$. Assume that  the  orthonormal frame $(\be_1, \be_2)$ is  choosen so that $\be_1=\vec e_{x_0}$.
We have the identity, for the functions $\Uptheta$, $\bm_{i, j} $ defined in \eqref{gdansk00}
\begin{equation}
\label{discrepvec66}
2{\Uptheta}_\star(x_0)={\bm}_{\star, 2, 2}(x_0)-{\bm} _{\star, 1, 1}(x_0)  {\rm \ and \ } {\bm}_{\star, 1, 2}(x_0)=0.
\end{equation} 
\end{proposition}

Let $\omega_\star^{ac}=\left( \upmu_{\star, 1, 1}^{ac} -\upmu_{\star, 2, 2}^{ac}\right) -2i\upmu_{\star, 1, 2}^{ac} $ denote the absolutely continuous part of of $\omega_\star$ with respect to $\rd \lambda$.  The previous result yields, after change of orthonormal basis: 

\begin{lemma}
\label{starac}
For a given orthonormal basis $(\be_1, \be_2)$,  we  have the identity
\begin{equation}
\label{starac2}
\omega_\star^{ac}=-2\exp (-2 i \upgamma_{\star})\upzeta_\star^{ac}=-2 (\cos 2\upgamma_\star-i \sin 2\upgamma_\star) \upzeta_\star^{ac}, 
\end{equation}
where, $\upgamma_\star(x)\in [-\frac{\pi}{2}, \frac{\pi}{2}]$ denotes  for $x  \in \mathfrak S_\star\setminus \mathfrak E_\star$, the 
angle between $\be_1$ and $\vec e_{x_0}.$
\end{lemma} 

\begin{remark}{\rm  Changing possibly $\vec e_{x_0}$ into $-\vec e_{x_0}$ we may indeed always choose $\upgamma_\star(x_0)$ in an interval of length $\pi$, here  $[-\frac{\pi}{2}, \frac{\pi}{2}]$.
}
\end{remark}

\smallskip
We present next some arguments involved  in the proof of Proposition  \ref{discrepvec6}.
We  work near a \emph{regular} point $x_0=(x_{0, 1}, x_{0, 2}) \in \mathfrak S_\star \setminus \mathfrak E_\star$, where $\mathfrak E_\star$ is defined in \eqref{bordeaux}, and choose the orthonormal basis so that $\be_1=\vec e_{x_0}$. In the neighborhood of the point  $x_0$, the measure $\upnu_\star$ hence concentrates near the line $x_2=x_{0, 2}$, and we may follow the approach of     \cite{ambroson},  eliminating  the derivatives according to the transversal direction, that is eliminating the $x_2$-variable,  in order to  obtain a one-dimensional problem: For that purpose, we integrate  along  "vertical" lines. The general idea would be to consider integrals of the form 
$$
I_{i, j}(s)=\int_{(x_{0, 2}-3\slash 4r)}^{(x_{0, 2}+3\slash 4r)} \upmu_{\star, i, j}(s, x_{0,2}) \,  \rd x_2 {\rm \ or \ } W(s)=\int_{(x_{0, 2}-3\slash 4r)}^{(x_{0, 2}+3\slash 4r)} \upzeta_{\star, i, j}(s, x_{0,2}) \rd x_2.
$$
However, since at this stage of our argument we don't know that the measures are absolutely continuous with respect to the Lebesgue measure, one has to be a little more careful in order to define the previous integrals. 
To that aim, we introduce  for $s>0$,  the segments   $\mathcal I_r(s)=[s-r, s+r]=\B^1(s, r)$ and  the square   $\displaystyle{Q_r(x_0)   = \mathcal I_r(x_{0, 1}) \times  \mathcal I_r(x_{0, 2})}$ and consider the \emph{localized} measures 
$$ 
\tilde \upmu_{\star, i, j}={\bf 1}_{Q_r(x_0)}\upmu_{\star, i, j} {\rm \ and \ }  \tilde \upzeta_{\star}={\bf 1}_{Q_r(x_0)}\upzeta_{\star, i, j}.
$$
We introduce also   the orthogonal projection $\mathbbmtt P$ onto the tangent line ${\rm D}_{x_0}^1=\{ x_0+s\be_1, s\in \R \}$, and the \emph{pushforward measures}   on ${\rm D}_{x_0}^1$ of the localized  measures we have introduced so far, namely the measures on ${\rm D}_{x_0}^1$
\begin{equation}
\label{mesproj}
\tilde \upmu_{\star, i, j}^{x_1}=\mathbbmtt P_{\sharp} \left( \tilde \upmu_{\star, i, j}\right) {\rm \ and \ } 
\tilde \upzeta_{\star, i, j}^{x_1}=\mathbbmtt P_{\sharp} \left( \tilde \upzeta_{\star}\right), 
\end{equation}
defined for every Borel set $A$ of ${\rm D}_{x_0}^1$ by
\begin{equation*}
\left\{
\begin{aligned}
&\tilde \upmu_{\star, i, j}^{x_1}(A)=\upmu_{\star, i, j} \left (\bP^{-1}(A)\cap Q_r(x_0)\right)=\upmu_{\star, i, j} \left ((A \times \R)\cap Q_r(x_0)\right) {\rm \ and \ } \\
&\tilde \upzeta_{\star}^{x_1}(A)=\upzeta_{\star} \left ((A \times \R)\cap Q_r(x_0)\right).
\end{aligned} 
\right.
\end{equation*} 
We  then consider the  measure $\mathbbmss L_{x_0, r}$ defined  on $\mathcal I_r(x_{0, 1})$ by 
\begin{equation}
\label{caminare}
\left\{
\begin{aligned}
\bL_{x_0, r} &=\bP_{\sharp}  \left(2\tilde \upzeta_\star  + \tilde \upmu_{\star, 1, 1} -\tilde \upmu_{\star, 2, 2}\right)
=2\tilde \upzeta_{\star, i, j}^{x_1} -\tilde \upmu_{\star, 1, 1}^{x_1} +\tilde \upmu_{\star, 2, 2}^{x_2} {\rm \ and \ } \\
\bN_{x_0, r}&=
\mathbbmtt P_{\sharp}  \left(2\tilde \upzeta_\star  + \tilde \upmu_{\star, 1, 1} -\tilde \upmu_{\star, 2, 2}\right)=
2\tilde \upzeta_\star^{x_1}  + \tilde \upmu_{\star, 1, 1}^{x_1} -\tilde \upmu_{\star, 2, 2}^{x_1}.
\end{aligned}
\right.
\end{equation}

  Mutiplying \eqref{holomorphitude} by appropriate test functions and integrating, we are led to the  somewhat remarkable properties of these measures, expressed   in Propositions \ref{letitwave}, \ref{assos}, \ref{constantin} and \ref{diff}, leadig to the completion of the proof of Proposition \ref{discrepvec6}.
 


\subsection{Monotonicity for $\upzeta_\star$ and its consequences} 
The next important step in the proofs of Theorem \ref{segmentus}, \ref{varifoltitude}, \ref{absolute} and \ref{discrepvec} is to show that the singular part of all measures introduced so far vanish. We first establish this statement for the measure $\upzeta_\star$.  Our argument involves a new ingredient, the monotonicity formula for $\upzeta_\star$, which actually directly yields the absolute continuity of $\upzeta_\star$ with respect to $\mathcal H^1\rest \mathfrak S_\star$.  

\begin{proposition} 
\label{psycho}
Let $x_0 \in \Omega$, let $\rho >0$  be such  that $\D^2(x_0, \rho)\subset \Omega$.  If    $0< 0<r_0\leq r_1 \leq \rho$, then we have 
  \begin{equation}
\label{monomaniaque}
\frac{\upzeta_\star (\D^2(x_0, r_1))}{r_1}\geq\frac{\upzeta_\star (\D^2(x_0, r_0))}{\rho_0}.
 \end{equation}
For every $x_0 \in \Omega$  the  quantity $\displaystyle{\frac{\upzeta_\star (\D^2(x_0, r))}{r}}$ has a limit when  $r \to 0$ and we have the estimate
\begin{equation}
\label{dansitus}
\Theta_\star (x_0)={\underset {r \to 0} \lim} \frac{\upzeta_\star (\D^2(x_0, r))}{r}\leq \frac{\upzeta_\star (\D^2(x_0, \rho))}{\rho} \leq 
\frac{M_0} {d(x, \partial \Omega)}. 
\end{equation}
The measure $\upzeta_\star$ is hence absolutely continuous with respect to the $\mathcal H^1$-measure on $\mathfrak S_\star$.
\end{proposition}

    The starting point of the  proof of Proposition \ref{psycho} is  the monotonicity  formula \eqref{monotonio0} for the potential term $V$, which may be written, after  integration, for  a solution $u_\eps$ of \eqref{elipes} on $\Omega$  and $0<r_0<r_1\leq \rho$ such that $D(x_0, \rho)\subset \Omega$ 
\begin{equation}
\label{degrace}
\frac{\upzeta_\eps(\D^2(x_0, r_1))}{r_1}-\frac{\upzeta_\eps( \D^2(x_0, r_0))}{r_0}=\int_{\D^2(x_0, r)\setminus \D^2(x_0, \rho_0)}\frac{1}{4r}
\rd \mathscr N_\eps, 
\end{equation}
where we have set 
$$\displaystyle{\mathscr N_\eps=\left (\frac{2}{\eps}V(u_\eps)-\eps r^2\left\vert\frac{\partial  u_\eps}{\partial \theta}\right \vert^2 +\eps \left\vert \frac{\partial  u_\eps}{\partial r} \right\vert^2\right) \rd x},
$$
and 
where $(r, \theta)$ denote  radial coordinates so that $x_1-x_{0, 1}=r\cos \theta$ and $x_2-x_{0, 2}=r \sin \theta$.
Passing to the limit $\eps\to 0$ in identity \eqref{degrace}, we are led to :

\begin{lemma}
\label{passing}
Let $x_0 \in \Omega$, let $\rho >0$ and assume that $\D^2(x_0, \rho)\subset \Omega$. For almost every radii $0<r_0<r_1\leq \rho$, we have the identity
\begin{equation}
\label{monomaniaque}
\frac{\upzeta_\star (\D^2(x_0, r_1))}{r_1}-\frac{\upzeta_\star (\D^2(x_0, r_0))}{\rho_0}=
\int_{\D^2(r_1)\setminus \D^2(\rho_0)}\frac{1}{4r} \rd \mathscr N_\star
 \end{equation}
where $\mathscr N_\star=2 \upzeta_\star-r^2  \upmu_{\star, \theta, \theta}+ \upmu_{\star, r, r}$, with 
\begin{equation}
\label{radicaux}
\left\{
\begin{aligned}
\upmu_{\star, r, r}&=\cos^2 \theta   \upmu_{\star, 1, 1} + \sin^2 \theta \upmu_{\star, 2, 2}+2\sin\theta \cos \theta \upmu_{\star, 1, 2} 
{\rm \ and \ } \\
r^{-2} \upmu_{\star, \theta, \theta}&=\sin^2 \theta   \upmu_{\star, 1, 1} +
 \cos^2 \theta \upmu_{\star, 2, 2}-2\sin \theta  \cos \theta\upmu_{1, 2}.
 \end{aligned}
\right.
 \end{equation}
\end{lemma}
Notice that we may  verify that 
\begin{equation*}
\left\vert \frac{\partial  u_{\eps_n}}{\partial r} \right\vert^2  \underset{n \to +\infty} \to \upmu_{\star, r, r}  {\rm \ and \ }
\left\vert \frac{\partial  u_{\eps_n}}{\partial \theta} \right\vert^2  \underset{n \to +\infty} \to \upmu_{\star, \theta , \theta}  {\rm \ as \ measures}.
\end{equation*}

The next step in the proof of Proposition \ref{psycho} is the fact that, as a  consequence of Proposition \ref{discrepvec6}, the absolutely continuous part of $\mathcal N_\star$ is non-negative.  We  them perform a  few manipulation which  allow to get rid of the singular part in \eqref{monomaniaque},  and lead to the completion of the proof of Proposition \ref{psycho}.

\smallskip
In order to proof that  $\upnu_\star$ is also  absolutely continuous  with respect to $\rd \lambda$, we will invoke the fact that $\upnu_\star$ is "dominated"  by the measure $\upzeta_\star$. Whereas the reverse statement is straighforward, since we have the inequality $\upzeta_\star \leq \upnu_\star$,  the fact that $\upnu_\star$ is "dominated"  by the measure $\upzeta_\star$ is a consequence of several  estimates at the $\eps$-level, requiring some PDE analysis  (in particular Proposition \ref{lheure}). Theorem \ref{discrepvec}  is an  then an  direct  consequence of Theorem \ref{absolute} and Proposition \ref{discrepvec6}. 
\subsection{On the proof on Theorem \ref{segmentus} and \ref{varifoltitude}}
\label{ontheproof}
 The proof of Theorem \ref{varifoltitude} is a direct consequence of Lemma \ref{holomorphitude} combined with Theorem \ref{discrepvec}. 
Theorem \ref{segmentus} could be deduced from Theorem \ref{varifoltitude} combined with the result of \cite{allardalm}, but we provide here a self contained proof.

\subsection{Open questions and conclusion}
\label{openconc} 
As already mentioned, one of the main unsolved open problems in the present paper, i.e. in the two dimensional elliptic context,  is the existence or not of singularities of \emph{"infinite type"}  in the limiting varifold.  If such singularities  do exist, their  actual construction may   turn out to  be extremely difficult.

\smallskip 

Although the paper focuses exclusively on the two-dimension case,  it is quite tempting to believe that a large part of the results might extend to higher dimensions.  However,  it is not clear how the arguments presented in this paper, in paper  the PDE part,   can be adapted in higher dimensions. 
Indeed, as the previous presentation hopefully shows, many of our arguments rely on the fact that we work in two dimensions., and do not seem to have natural extensions in higher dimensions.

\smallskip
Another challenging problem is the related  parabolic two-dimensional equation, which has already attracted attention  (see e.g.\cite{bronret} or more recently \cite{lauxsim}). One might express the hope that some of the methods introduced in this paper extend also to this case.

\subsection{Plan of the paper}
  The outline of the paper merely follows the description given in Subsections \ref{newpde}  to \ref{ontheproof}.  As a matter of fact, the presentation of arguments is divided into three parts. Part I  is a preliminary part which presents various properties of the energy functional  and consists of a single section,  Section \ref{engie2}. It  presents some consequences of the energy bound, starting with  estimates  on one-dimensional sets, as well as consequences of the co-area formula.  Part II, which runs from Section \ref{pde} to Section \ref{solde},   gathers all properties of solutions to the PDE  \eqref{elipes},   including standard one. For a large part, in both parts,  special  emphasis is put on energy estimates on  \emph{level sets of some appropriate simple scalar functions} (see \eqref{doublevi}).  Section \ref{challengedata} presents the proof of Proposition \ref{sindec}.  
  The last part, Part III, describes the properties of the limiting set $\mathfrak S_\star$  and the limiting measures measures, and contains therefore the proofs to the main results of the paper.

 \tableofcontents
 
   \numberwithin{theorem}{section} \numberwithin{lemma}{section}
\numberwithin{proposition}{section} \numberwithin{remark}{section}
\numberwithin{corollary}{section}
\numberwithin{equation}{section}

 \begin{center}{ \Large \bf
Part I : Properties of the functional} 
\end{center}
\addcontentsline{toc}{section}{Part I : Properties of the functional}
 \section{First consequences of the  energy bounds}   
 \label{engie2} 
    The results in this section  are  based on  variants of an idea of Modica and Mortola (see \cite{mortadela}),  adapted to the vectorial case  as in \cite{baldo, fontar}. we also present some applications of the co-area formula in connection with the functional.  The results in this section  apply to maps having a suitable bound on their   energy $\Eps$, of the type of the bound \eqref{naturalbound}. We stress in particular $BV$ type bounds obtained under these energy bound.  
 None of the results in this section \emph{involves the PDE}. 
 We start with simple consequences of assumptions ${(\text{H}_1)}$, ${(\text{H}_2)}$ and ${(\text{H}_3)}$  for the potential  with multiple  equal depth wells (see Figure \ref{potent}). 

\subsection {Properties of the potential}
\label{potentiel}
 It follows  from the definition of $\upmu_0$ and  property \eqref{kiv}  that we have the following behavior near the points of the vacuum manifold  $\Sigma$:
 
\begin{proposition} 
\label{potto}
 For any $i=1, \ldots, q$  and any   $y \in \B^k(\upsigma_i, 2\upmu_0)$, we have the local bound
\begin{equation}
\label{glutz}
\left\{
\begin{aligned}
\frac{1}{4}\lambda_i^-\vert y-\upsigma_i \vert^2 & \leq V(y) \leq  \lambda_i^+\vert y-\upsigma_i \vert^2
  \\
  \frac{1}{2}\lambda_i^-\vert y-\upsigma_i \vert^2 & \leq  \nabla V(y) \cdot (y-\upsigma_i) \leq  2\lambda_i^+\vert y-\upsigma_i \vert^2, 
    \end{aligned}
  \right.
 \end{equation}   
Choosing possibly an even smaller constant $\upmu_0$, we  have
   \begin{equation}
   \label{extrut}
   V(y) \geq  \upalpha_0\equiv\frac{1}{2} \lambda_0 \upmu_0^2 {\rm \ on \ } 
   \R^k  \setminus  \underset{i=1}{\overset {q} \cup} \B^k(\upsigma_i, \upmu_0), 
  \end{equation}
  where we have set  $\lambda_0=\inf\{\lambda_i^-, i=1, \ldots, q\}$
  \end{proposition}
  The proof relies on a straightforward integration of \eqref{kiv} and we therefore omit it.
 Proposition \ref{potto}  shows that  the potential   $V$  essentially behaves  as  a positive definite quadratic  function near points  of  the vacuum manifolds $\Sigma$.   This will be used throughout as a guiding thread. Proposition \ref{potto}  leads to a first  elementary observation: 
  
  \begin{lemma}
  \label{watson} Let $ y \in \R^k$ be such that $V(y)<\upalpha_0$. Then there exists some point $\upsigma\in \Sigma$ such that 
  $$\displaystyle{ \vert y-\upsigma  \vert \leq \upmu_0.}$$
  Moreover, we have the upper bound
  $$ \vert y-\upsigma  \vert  \leq \sqrt{ 4\lambda_0^{-1} V(y)}.$$
  \end{lemma}

  We next turn to the behavior at infinity. For that purpose, we introduce the radius 
\begin{equation}
\label{grandr0}
{\rm R}_0=\sup\{\vert \upsigma  \vert  , \upsigma \in \Sigma\}
\end{equation}
and study the properties of $V$ on the  set  $\R^k\setminus \B^k(2{\rm R}_0)$.

\begin{proposition} 
\label{barre}
 There exists a constant $\upbeta_\infty >0$ such that 
\begin{equation}
\label{upbetainfty}
V(y) \geq \upbeta_\infty \vert y \vert^2  {\rm \ for \ any \ } y {\rm \ such \ that \ }  \vert y \vert  \geq  2 {\rm R}_0.  
\end{equation}
\end{proposition}
\begin{proof}
Integrating assumption ${\rm H}_3$ we obtain that, for some constant $C_\infty>0$, we have
   \begin{equation}
   \label{cahors}
   V(y) \geq \frac{\upalpha_\infty  \vert y \vert^2}{2}-C_{\infty}, {\rm \ for \ any \ } y \in \R^k. 
   \end{equation}
It follows that
\begin{equation}
\label{tramoo}
 V(y) \geq  \frac{\upalpha_\infty  \vert y \vert^2}{4}, {\rm \ provided   \  } \vert y \vert  \geq  {\rm  R'}_0 \equiv \sup \left\{ 2\sqrt{\frac{C_\infty}{\upalpha_\infty}},\,  4R_0\right\}.
 \end{equation}
On the other hand, by assumption, we have 
$$\frac{V(y)}{\vert y \vert^2} >0 {\rm \ for \ } y \in \overline{\B^k({\rm R'}_0)\setminus \B^k(2{\rm R}_0)}, $$
so that, by compactness, we deduce that there exist some constant $\upalpha'_\infty>0$, such that
$$V(y)  \geq  \upalpha'_\infty \vert y \vert^2   {\rm \ for \ } y \in \overline{\B^k(2{\rm R'}_0)\setminus \B^k(2{\rm R}_0)}. $$
Combining the last inequality with \eqref{tramoo},  the conclusion follows, choosing 
$\displaystyle{\upbeta=\inf \{ \frac{\upalpha_\infty}{4}, \upbeta'_\infty\}  }$.
\end{proof}
\subsection{Modica-Mortola type inequalities}  
\label{momo} 
   
  Let $\upsigma_i$ be an arbitratry  element in $\Sigma$.   We consider  the function  $\chi_i: \R^k \to \R^+$  defined by
  $$ \chi_{_i}(y)=\varphi (\vert y-\upsigma_i  \vert) {\rm \ for \ } y \in \R^k, $$
where $\varphi$ denotes a function  $\varphi : [0, +\infty [ \to \R^+ $ such that $0\leq \varphi'  \leq 1 $  and 
    $$  \varphi(t)=t {\rm \ if \  } 0\leq t \leq \frac{\upmu_0}{2}  {\rm \ and \ }  \varphi (t)=\frac{3\upmu_0 }{4}
     {\rm \ if \  }  t\geq \upmu_0.$$
  Given a function $u: \Omega \to \R^k$ we finally define the \emph{scalar} function $w_i$  on $\Omega$  as
  \begin{equation}
  \label{doublevi}
    w_i(x)=\chi_i ( u(x)), \forall x  \in \Omega. 
 \end{equation}
 First properties of the map $w_i$ are summarized in the next Lemma.

   \begin{lemma} Let $w_i$ be as above. We have    
   \begin{equation}
      \label{monster}
      \left\{
      \begin{aligned}
   w_i (x)&=\vert u(x)-\upsigma_i  \vert,   {\rm \ if \ }\vert u(x)-\upsigma_i \vert \leq  \frac{\upmu_0}{2},\\
    w_i(x)&= \frac{3\upmu_0 }{4},   {\rm \ hence  \  }  \nabla w_i= 0 \ {\rm \ if \ }\vert u(x)-\upsigma_i \vert  \geq \upmu_0, \\
    \vert \nabla w^i \vert &\leq \vert \nabla u \vert {\rm \  on  \ }  \Omega, 
    \end{aligned}
    \right.
    \end{equation}
and 
 \begin{equation}
      \label{bvbound}
    \vert   \nabla (w_i)^2\vert
     \leq 4\sqrt{\lambda_0}^{-1} J (u) (x), 
           \end{equation}
            where we have set 
           \begin{equation}
           \label{jetski}
           J(u)= \vert \nabla u \vert \sqrt{V(u)}.
           \end{equation}
\end{lemma}
\begin{proof}
Properties \eqref{monster} is a straightforward consequence of the definition \eqref{doublevi}. For \eqref{bvbound}, we  notice that, in view of \eqref{monster}, we may restrict ourselves to the case $u(x) \in \B^k(\upsigma_i, \upmu_0)$, since otherwise $\nabla w_i=0$, and inequality \eqref{bvbound} is hence straightforwardly satisfied. In that case, it follows from \eqref{watson}, we have 
$$
 \vert  w_i (x) \vert \leq  \vert u(x)-\upsigma_i\vert \leq  \sqrt{ 4\lambda_0^{-1} V(u(x))},  {\rm \ for \ all \ } x   {\rm \ such \ that \ }  u(x) \in \B^k(\upsigma_i, \upmu_0), 
$$
 so that 
 \begin{equation}
      \label{bvbounda}
    \vert   \nabla (w_i)^2(x)\vert=2\left\vert w_i  (x) \right\vert.
    \left\vert \nabla \left \vert  w_i (x) \right\vert \right \vert \,  
    \leq  2 \vert \nabla u \vert \sqrt{ 4\lambda_0^{-1} V(u(x))} 
     \leq 4\sqrt{\lambda_0}^{-1} J (u) (x), 
           \end{equation}
 and the proof  is complete.
  \end{proof}

\begin{lemma} 
\label{ab0}
 We have, for any $x \in \Omega$, the inequality
   \begin{equation}
          \label{ab}
         J(u(x)) 
        \leq 
         e_\eps(u(x)). 
\end{equation}
\end{lemma}
\begin{proof}
We have, by definition of the energy $e_\eps(u)$,
 \begin{equation}
          \label{ab}
         J(u(x)) =(\sqrt{\eps} \vert \nabla u(x) \vert ). \sqrt{\eps^{-1}V(u(x)) } 
          \end{equation}
          We invoke next the inequality $\displaystyle{ab\leq \frac{1}{2}(a^2+b^2)}$ to obtain
          $$J(u(x) \leq \frac12 \left (  \eps\vert \nabla u (x)\vert^2+ \eps^{-1} V(u(x))\right), $$ 
          which yields the desired result. 
      \end{proof}
    \subsection{The one-dimensional case}

     In dimension $1$  estimate \eqref{bvbound} directly leads to uniform bound on $w_i$, as expressed in our next result. 
 For that purpose, we consider, for $r>0$, the  circle $\S^1(r)=\{x\in \R^2, \vert x \vert =r\}$ and maps $u: \S^1(r) \to \R^k$.

     \begin{lemma} 
     \label{valli}
     Let $0< \eps\leq 1$ and $\eps<r\leq1$ be given. There exists a constant ${\rm C}_{\rm unf}>0$, depending only on $V$,  such that, for any   given  map    $u:\mathbb S^1(r) \to \R^k$,  there exists an element $\upsigmam \in \Sigma$ such that 
      \begin{equation}
      \label{bornuni}
            \vert u(\ell)-\upsigmam \vert \leq  {\rm C}_{\rm unf}\sqrt{\int_{\S^1(r)} \frac 12 (J(u(\ell))+r^{-1} V(u(\ell))){\rm d} \ell}, \\
\  \   {\rm \ for \ all \ } \ell \in \S^1(r),
      \end{equation}
    and hence 
      \begin{equation}
      \label{bornunif}
     \vert u(\ell)-\upsigmam \vert    \leq  {\rm C}_{\rm unf}\sqrt{\int_{\S^1 (r)} e_\eps(u){\rm d}\ell}.  
      \end{equation}
       \end{lemma}
      
      \begin{proof} 
      By the mean-value formula, there exists some point $\ell_0 \in \S^1(r)$ such that 
      \begin{equation}
      \label{buena}
      V(u(\ell_0)) = \frac{1}{2\pi r} \int_{\S^1(r) } V(u(\ell)) {\rm d} \ell.
      \end{equation}
   We distinguish two cases.  
   
   \smallskip
   \noindent
   {\bf  Case 1.} {\it The function $u$ satisfies additionnally the estimate }
   \begin{equation}
    \label{val}
   \frac{1}{2\pi r} \int_{\S^1(r)} V(u(\ell))  \, {\rm d} \ell <  \upalpha_0, 
     \end{equation}
{\it  where $\upalpha_0$ is the constant introduced in Lemma \ref{watson}}.
    Then,  we deduce from inequality \eqref{val}  that
   \begin{equation*}
      \label{buena}
      V(u(\ell_0)) \leq \frac{1}{2\pi r}  \int_{\S^1(r)} V(u(\ell))  \, {\rm d} \ell < \upalpha_0.
            \end{equation*}
      It follows from Lemma \ref{watson} that there exists some $\upsigmam \in \Sigma$ such that 
     $$ \vert u(\ell_0))-\upsigmam  \vert^2  \leq 4 \lambda_0^{-1} V(u(\ell_0) )\leq 
      \frac {2\lambda_0^{-1}}{\pi r} \int_{\S^1(r)} V(u)  {\rm d} \ell. $$
    On the other hand,  we deduce,  integrating  the bound \eqref{bvbound},  that, for any $\ell \in \S^1(r)$, we have 
    $$ \vert \left\vert u-\upsigmam\right\vert^2(\ell)-\left \vert u-\upsigmam\right\vert ^2(\ell_0))\vert \, {\rm d} \ell
     \leq
      4\sqrt{\lambda_0^{-1}} \int_{\S^1(r)} J(u).$$
    Combining the two previous estimates,  we obtain the desired result  in case 1, using the fact that $\eps \leq 1$ and provided the constant $ {\rm C}_{\rm unf}$ satisfies the bound
    $$ {\rm C}_{\rm unf}^2 \geq  4\sqrt{\lambda_0^{-1}}+2\lambda_0^{-1}.$$
   
   \bigskip
   \noindent
   {\bf  Case 2}. {\it  Inequality \eqref{val} does not hold}. 
   In that case,  we have hence 
  \begin{equation}
  \label{trouville}
\frac{1}{2\pi r}  \int_{\S^1(r)} V(u(\ell))  {\rm   d} \ell \geq \upalpha_0.
   \end{equation}
 We consider  the number $\rm R_0=\sup\{\vert \upsigma  \vert  , \upsigma \in \Sigma\}$,   introduced in definition \eqref{cahors} of the proof of Proposition  \ref{barre},   and discuss next three subcases.
 
 \smallskip
 \noindent
 {\it Subcase 2a : For any $\ell \in \S^1(r)$,   we have
 $$ u(\ell) \in \B^k(2 \rm R_0).$$}
Then, in this case,  for any  $\upsigma \in \Sigma$, we have
\begin{equation}
\label{welldone}
\vert u(\ell)- \upsigma\vert^2 \leq 9{\rm R}_0^2=\left( \frac{9{\rm R}_0^2}{\upalpha_0}\right)\upalpha_0 \leq   
\left( \frac{9{\rm R}_0^2}{\upalpha_0}\right)\frac{1}{2\pi r}  \int_{\S^1(r)} V(u(\ell))  {\rm   d} \ell.
\end{equation}
 Hence,  inequality \eqref{bornuni} is immediately satisfied, whatever the choice of $\upsigma_{\rm main}$,  provided we impose the additional condition
 \begin{equation}
 \label{fayat}
  {\rm C}_{\rm unf}^2 \geq \frac{9{\rm R}_0^2}{2\upalpha_0}.
\end{equation}

\smallskip
 \noindent
 {\it Subcase 2b : There exists some  $\ell_1 \in \S^1(r)$, and some $\ell_2\in \S^1(r)$ such that,  we have}
 $$ u(\ell_1) \in \B^k(2 {\rm R}_0)  {\rm \ and \ }  u(\ell_2) \not  \in \B^k(2 \rm R_0).$$
Let $\ell \in \S^1(r)$. If $u(\ell) \in \B^k(2\rm R_0)$, then we argue as in subcase 2a, so that we obtain  inequality  \eqref{welldone} as before,  hence \eqref{bornuni}, and we are done. Otherwise, by continuity, there exists some $\ell' \in \S^1(r)$ such that $u(\ell') \in \partial \B^k(2{\rm R}_0)$ and such for  any point $a\in \mathcal C(\ell, \ell')$ we have   $u(a) \not \in \B^k(2 \rm R_0)$,   where $\mathcal C(\ell, \ell')$ denotes the arc on $\S^1(r)$ joining counterclockwise the points  $\ell$ and $\ell'$.   We have, by integration, using  the fact that $\vert u(a)\vert \geq 2{\rm R}_0$ for $a\in \mathcal C(\ell, \ell')$ together with inequality \eqref{upbetainfty}, 
\begin{equation*}
\begin{aligned}
\vert u(\ell)\vert^2-\vert u(\ell')\vert ^2 &\leq 2\int_\ell^{\ell'} \vert u(a) \vert \cdot \vert \nabla u(a) \vert  \,  {\rm d} a  \\
&\leq  {\rm R}_0^{-1}\int_\ell^{\ell'} \vert u(a) \vert^2 \cdot \vert \nabla u(a) \vert  \,\rd a  \\
 &\leq   \frac{1}{{\upbeta_\infty} {\rm R_0}  }\int_\ell^{\ell'} V(u(a)) \vert \nabla u(a) \vert  \,  {\rm d} a
 \leq  \frac{1}{{\upbeta_\infty} {\rm R}_0 } \int_{\S^1(r)} J(u(a))   {\rm d} a.
 \end{aligned}
\end{equation*}
Since $\vert u (\ell') \vert=2{\rm R}_0$, we obtain, for any $\upsigma \in \Sigma$, 
\begin{equation*}
\begin{aligned}
  \vert u(\ell)- \upsigma\vert^2  &\leq 2 \left(\vert u(\ell)\vert^2 + \vert \upsigma \vert^2 \right)\leq 2 \left(\vert u(\ell)\vert^2 + {\rm R}_0^2 \right) \\
   &\leq 2 \left( \frac{1}{{\upbeta_\infty}  {\rm R}_0 } \int_{\S^1(r)} J(u(a))   {\rm d} a +{\rm R}_0^2 + \vert u(\ell') \vert^2  \right)  \\
   &\leq  \left( \frac{2}{{\upbeta_\infty}  {\rm R}_0} \int_{\S^1(r)} J(u(a))   {\rm d} a + 10{\rm R}_0^2   \right) \\
   &\leq   \left( \frac{2}{{\upbeta_\infty} {\rm R}_0 } \int_{\S^1(r)} J(u(a))   {\rm d} a + 10\frac{{\rm R}_0^2 }{\upalpha_\infty}\upalpha_\infty   \right) \\
  & \leq   \left( \frac{2}{{\upbeta_\infty} {\rm R}_0 } \int_{\S^1(r)} J(u(a))   {\rm d} a +
   \frac{10{\rm R}_0^2 }{2\pi\upalpha_0 r}  \int_{\S^1(r)} V(u(\ell))  {\rm   d} \ell, 
    \right) \\
   \end{aligned}
\end{equation*}
so that the conclusion \eqref{bornuni} follows for any choice of $\upsigmam \in \Sigma$ , imposing  again  an appropriate lower bound on   the constant ${\rm C}_{\rm unf}$.

  \medskip
 \noindent
 {\it Subcase 2c :  For any $\ell \in \S^1(r)$,   we have
 $$  \vert  u(\ell) \vert \geq 2{ \rm R}_0 .$$}
  Let $\ell_0$ satisfy \eqref{buena}, so that, in view of Proposition \ref{barre}
  \begin{equation*}
  \vert u(\ell_0)\vert^2 \leq  \frac{1}{\upbeta_\infty}  V(u(\ell_0))=  
  \frac{1}{\upbeta_\infty } \left( \frac{1}{2\pi r}\int_{\S^1(r)} V(u(\ell)\right).
  \end{equation*}
  We obtain hence, for any arbitrary $\upsigma \in \Sigma$ 
 \begin{equation}
 \label{monge}
 \begin{aligned}
  \vert u(\ell_0)- \upsigma\vert^2 &\leq 2 \left(\vert u(\ell_0)\vert^2 + \vert \upsigma \vert^2 \right)\leq 
   \frac{2}{\upbeta_\infty} \left( \frac{1}{2\pi r} \int_{\S^1(r)} V(u(\ell)) { \rm d} \ell +  {\rm R}_0^2 \upbeta_\infty\right ) \\
   &\leq \frac{2}{\upbeta_\infty} \left( \frac{1}{2\pi r} \int_{\S^1(r)} V(u(\ell) ){ \rm d} \ell + 
   \upalpha_0\left( \frac{  {\rm R}_0^2 \upbeta_\infty}{\upalpha_0}\right)\right ) \\
   &\leq  \  \frac{1}{\pi \upbeta_\infty  } \left(  1+ \left(\frac{2{\rm R}_0^2 \upbeta_\infty}{\upalpha_0}\right) \right)
   \left( r^{-1} \int_{\S^1(r)} V(u(\ell)){ \rm d} \ell\right). 
    \end{aligned}
     \end{equation}
This yields again \eqref{bornuni} for an arbitrary choice of $\upsigmam \in \Sigma$ and imposing an additional  suitable lower bound on 
${\rm C}_{\rm unf}$.     

We have hence established  for upper bound \eqref{bornuni}  in all three possible  cases $2a, 2b$ and $2c$, for a suitable an arbitrary choice of $\upsigmam \in \Sigma$ and imposing an additional  suitable lower bound on 
${\rm C}_{\rm unf}$.      It is hence established in case $2$. Since we alreday establishes it in Case 1, the proof of \eqref{bornuni} is complete. 

\medskip 
Turning to inequality \eqref{bornunif}, we first observe that, since  by assumption $r\geq \eps$, we have 
\begin{equation}
\label{lepape}
r^{-1} \int_{\S^1(r)} V(u(\ell)){ \rm d} \ell \leq \int_{\S^1(r)} \eps^{-1}  V(u(\ell)){ \rm d} \ell
      \leq     \int_{\S^1(r)} e_\eps(u(\ell)){ \rm d} \ell.  
      \end{equation}
      Combining \eqref{bornuni} with \eqref{ab} and \eqref{lepape}, we obtain the desired result \eqref{bornunif}. 
     \end{proof}
    
   \subsection{Controlling the energy on  circles}
     \label{radamel}
    When working on two-dimensional disks, the tools  developed in the previous section allow to choose  radii with appropriate control on the energy, invoking a standard  mean-value argument. More precisely, we have: 
     
     \begin{lemma}  
     \label{moyenne}
     Let   $\eps \leq r_0< r_1\leq 1$  and $u: \D^2 \to \R^k$ be given. There exists a radius $\mathfrak r_\eps \in [r_0, r_1]$ such that     
     $$ \int_{\S^1(\mathfrak r_\eps)}e_\eps(u){\rm  d } \ell  \leq \frac{1}{r_1-r_0} \, \E_\eps(u, \D^2(r_1)). $$
     \end{lemma}
  The proof is based on a classical mean-value argument, therefore we omit it. 
  
  \smallskip
  In the sequel, we will often make use of Lemma \ref{moyenne} combined with  the uniform bounds obtained in dimension one. For instance, it  follows    from Lemma \ref{valli} that there exists 
  some point  $\upsigma_{\mathfrak r_\eps} \in \Sigma$, \emph{depending on $\mathfrak r_\eps$},  such that 
   \begin{equation}
      \label{bornuni2}
     \vert u(\ell)-\upsigma_{\mathfrak r_\eps}  \vert \leq \frac{ {\rm C}_{\rm unf}}{\sqrt{r_1-r_0}} \sqrt{{\E}_\eps(u, \D^2(r_1)}),  \  \   {\rm \ for \ all \ } \ell \in 
     \S^1(\mathfrak r_\eps).
      \end{equation}
 Moreover, it follows from \eqref{bvbounda} that 
 \begin{equation}
 \label{bvbound2}
 \int_{\S^1(\mathfrak r_\eps)} \vert J(u) \vert \leq \frac{1}{r_1-r_0} \int_{\D^2(r_1)} e_\eps(u_\eps) {\rm d}x. 
 \end{equation}
 
     \subsection {BV estimates and the coarea formula}
     \label{detroit}
The right-hand side  of estimate \eqref{bornunif}, in particular the term  involving  $J(u)$, may be interpreted  as a  $BV$ estimate (as in \cite{mortadela}). In dimension $1$, as expected, it yields  used a uniform estimates.  In higher dimensions of course, this is no longer true. Nevertheless our $BV$-estimates lead to useful estimates for  the measure of specific  level sets. In order to state the kind of results we have in mind,  we consider more generally an arbitrary    smooth function $\varphi:  \Omega\to \R$,  where $\Omega \subset \R^N$ is a arbitrary $N$-dimensional domain, and introduce, for a given  number $s \in \R$,  the  level set
$$\varphi^{-1} (s)=\{s \in \Omega, {\rm \ such \ that \ } \varphi(x)=s\}.$$
If $w$ is assumed to be sufficiently smooth, then Sard's theorem asserts that $w^{-1}(s)$ is a  regular submanifold of dimension $(N-1)$, for almost every $s\in \R$, and the coarea formula relates the integral of the total  measures of these level sets  to the $BV$-norm  through the formula     
     \begin{equation}
     \label{coarea}
     \int_{\R} {\mathcal H}^{N-1}\left(\varphi^{-1} (s)\right){\rm d}s=\int_\Omega  \vert  \nabla  \varphi  (x)\vert {\rm d}x. 
     \end{equation}
We specify this formula to  our needs in  the specific case $N=2$, $\Omega=\D^2(r)$, for some $r>\eps$,   and $\varphi=(w_i)^2: \Omega \to \R$,   where $i \in \{1, \ldots, q\}$ and   where $w_i: \Omega \to \R$ is the map constructed   in \eqref{doublevi} for a given $u: \Omega\to \R^k$. Combining \eqref{coarea} with \eqref{bvbound} and \eqref{ab}, we are led to the inequality, for the level sets $(w_i^2)^{-1} (s)\subset \Omega=\D^2(r)$, 
\begin{equation}
\label{coaforme}
\begin{aligned}
\int_{\R^+} {\mathcal L}\left((w_i^2)^{-1} (s)\right){\rm d}s&\leq 4 \sqrt{\lambda_0}^{-1}\int_{\D^2(r)}  J(u(x)){\rm d}x \\
&\leq  4 \sqrt{\lambda_0}^{-1}\int_{\D^2(r)}  e_\eps(u){\rm d}x=  4 \sqrt{\lambda_0}^{-1} \E_\eps\left(u, \D^2(r)\right), 
\end{aligned}
\end{equation}
where 
$\mathcal  L=\mathcal H^{1}$ denotes length. In most places, we will invoke this inequality jointly with a mean value argument.  This yields:

\begin{lemma} 
\label{claudio}
Let $u$, $w_i$ and $r>\eps$  be as above. Given  any  number $A>0$,     there exists some 
$\displaystyle{A_0  \in [ \frac{A}{2}, A]}$ such that $w_i^{-1} (s_0)$ is a regular curve in $\D^2(r)$ and such that 
\begin{equation}
{\mathcal L}\left(w_i^{-1} (A_0)\right)  \leq\frac{8}{\sqrt{\lambda_0}A^2}\int_{\D^2(r)}  e_\eps(u){\rm d}x
\leq \frac{8 \,  \E_\eps\left(u, \D^2(r) \right)} {\sqrt{\lambda_0}A^2}.
\end{equation}
\end{lemma}
\begin{proof}  In view of  Definition \ref{doublevi}, the map $w_i$ takes values in the interval $\displaystyle{[0, \frac{3\upmu_0}{4}]}$, so that 
$\displaystyle{w_i^{-1}(s) =\emptyset}$, if $\displaystyle{s>\frac{3\upmu_0}{4}}$. Hence, it remains only to consider the case $\displaystyle{A \leq  \frac{3\upmu_0}{4}}$.  We introduce to that aim  the domain $\displaystyle{\Omega_{i, A}=\{x\in \D^2(r),\frac{A}{2} \leq  \vert u(x)-\upsigma_i \vert \leq A\}}$. Using formula \eqref{coaforme} on this domain, we are led to the inequality 
\begin{equation*}
\int_{\frac {A^2} {4}} ^{A^2} {\mathcal L}((w_i^2)^{-1} (s) ){\rm d} s
\leq 4 \sqrt{\lambda_0}^{-1}\int_{\Omega_{i, A}}  e_\eps(u){\rm d}x
  \leq 4 \sqrt{\lambda_0}^{-1} \E_\eps\left(u, \D^2(r)\right).
\end{equation*}
The conclusion that follows by a mean-value argument. 
\end{proof}


  \subsection{Controlling uniform bounds on good circles  }
Whereas in subsection  \ref{radamel} we have selected radii with controlled energy for the map $u$,   in this subsection, we select radii with appropriate uniform bounds on $u$. 
 We assume throughout this subsection that we  are given a radius $\varrho \in [\frac 12, 1]$, a number 
  $\displaystyle{0<\upkappa <\frac{\upmu_0}{2}}$, a smooth map $u:\overline{ \D^2(\varrho)} \to \R^k$   and an element $\upsigma \in \Sigma$ such that 
   \begin{equation}
  \label{kappacite0}
   \vert u-\upsigma \vert  < \upkappa  {\rm \ on \ } \partial \D^2(\varrho).
  \end{equation}
  We introduce  the subset  $\mathcal I (u, \upkappa)$  of radii $\displaystyle{r \in  [\frac 12, \varrho]}$  such that 
\begin{equation}
\label{sunyu} 
\mathcal I (u, \upkappa) =\left \{ r \in [\frac{1}{2}, \varrho ] {\rm \ such \ that \ }
 \vert u  (\ell) -\upsigma \vert \leq  \upkappa, \, \forall \ell \in \S^1(r) \right \}. 
 \end{equation}
   We have: 

  \begin{proposition}   
  \label{jarre}
Assume that  \eqref{kappacite0} holds. Then, we have  the lower bound
 \begin{equation}
    \label{clamart}
    \vert \mathcal I (u, \upkappa) \vert \geq \varrho-\frac{9}{16}, 
    \end{equation}
    provided we have the lower bound on $\upkappa$
    \begin{equation} 
    \label{camembert}
 \upkappa^2  \geq  \frac{1}{32\sqrt{\lambda_0}} \E_\eps(u, \D^2(\varrho)). 
\end{equation}

    \end{proposition}
\begin{proof}
We consider the number $\displaystyle{A_0  \in [ \upkappa , 2\upkappa]} $ provided by Lemma \ref{claudio} with the choice $r=\varrho$ and $A=\upkappa$, so that $w^{-1} (A_0)\subset \D^2(\varrho)$ is smooth and 
\begin{equation*}
\mathcal L (w^{-1} (A_0)) \leq \frac{8E_\eps (u, \D^2(\varrho))}{4\sqrt{\lambda_0} \upkappa^2}=\frac{2E_\eps (u, \D^2(\varrho))}{\sqrt{\lambda_0} \upkappa^2}.
 \end{equation*}
If moreover \eqref{camembert} is satisfied, then we have
\begin{equation}
\label{brie}
\mathcal L (w^{-1} (A_0)) < \frac{1}{16}.
\end{equation}
We introduce the auxiliary set 
\begin{equation*}
\left\{
\begin{aligned}
\mathcal J (u, \upkappa) &=\{ r \in [\frac{1}{2}, \varrho],  {\rm \ such \ that \ }
 \vert u_\eps  (\ell) -\upsigma \vert < A_0, \, \forall \ell \in \S^1(r) \},   {\rm \ and \ } \\
 \mathcal Z (u, \upkappa) &=\{ r \in [\frac{1}{2}, \varrho],  {\rm \ such \ that \ }
 \vert u_\eps  (\ell) -\upsigma \vert > A_0, \, \forall \ell \in \S^1(r) \}.
   \end{aligned}
 \right.
  \end{equation*}
We first show that 
\begin{equation}
\label{euphrate0}
\displaystyle{\mathcal Z (u, \upkappa)= \emptyset}. 
\end{equation}
  Indeed,  assume by contradiction that \eqref{euphrate0} does not hold, so that there exists some  radius 
 $\frac 12 \leq r_0\leq \varrho$ in  $\mathcal Z (u, \upkappa)$.  In view of the definition of  $\mathcal Z(u, \upkappa)$, we have  therefore
 \begin{equation}
 \label{inte}
 \vert u_\eps-\upsigma \vert >A_0  {\rm \ on \ } \partial \D^2(r_0).
 \end{equation}
 On the  other hand,in view of assumption 
 \eqref{kappacite0},  we have $\vert u_\eps-\upsigma \vert <\upkappa <A_0$ on $\partial \D^2(\varrho)$.   Combining \eqref{inte} and \eqref{kappacite0}, it follows from the intermediate value theorem that there exists  some smooth domain $V$ such that  $u(x)=A_0$ for $ x \in \partial V$, so that $\partial V\subset w^{-1}(A_0)$, and hence is smooth,  and  such that 
  \begin{equation}
  \label{uniqlo}
  \D^2(r_0)\subset V\subset \D^2(\varrho).
  \end{equation}
    We deduce  from \eqref{uniqlo}    that, since by assumption $1\slash 2 \leq r \leq \varrho$,  
 $$
    \partial V\subset w^{-1}(A_0)               {\rm \ and \ } \mathcal L(\partial V) \geq 2\pi r \geq \pi , 
 $$
Hence, we obtain, in view of \eqref{uniqlo},   
$$
\mathcal L(w^{-1}(A_0) ) \geq \pi. 
 $$
 This however contradicts  inequality \eqref{brie} and hence establishes \eqref{euphrate0}. \\
 
 We next show that 
  \begin{equation}
 \label{carensac00}
  \vert \mathcal J (u, \upkappa) \vert  \geq 
  \varrho- \frac{9}{16}, 
  \end{equation}
For that purpose,   consider an arbitrary radius 
 $\frac 12 \leq r\leq \varrho$  such that $r \not \in  \mathcal   J (u, \upkappa)$ (see Figure \ref{coaire1}). It follows from  the definition of $\mathcal J(u, \upkappa)$ that there exists some $\ell_r \in \S^1(r)$ such that  $\vert u_\eps(\ell_r) -\upsigma \vert \geq  A_0$. We deduce  from identity  \eqref{euphrate0} and the intermediate value theorem that 
  $$ 
 w^{-1}(A_0)   \cap \S^1(r) \neq \emptyset,  \, \,  \forall r \not \in \mathcal J(u, \upkappa).
$$
This relation implies,  by Fubini's theorem,  that 
$$
\mathcal  L ( w^{-1}(A_0) ) \geq \left (\varrho-\frac 12\right)- \vert \mathcal J (u, \upkappa)\vert, 
 $$
 so that 
 \begin{equation}
 \label{carensac0}
  \vert \mathcal J (u, \upkappa) \vert \geq  \left (\varrho-\frac 12\right)-\mathcal  L ( w^{-1}(A_0) ) \geq 
  \varrho- \frac{9}{16}, 
  \end{equation}
where we made use of  estimate \eqref{brie}. This establishes \eqref{carensac00}. 
Since  $0< \upkappa \leq  A_0 $ by construction, we have 
$$\mathcal J (u, \upkappa) \subset \mathcal I (u, \upkappa),  {\rm so \ that \ }
\vert \mathcal J (u, \upkappa) \vert \leq \vert \mathcal I (u, \upkappa) \vert.
$$  
 Combining with inequality \eqref{carensac00}, we obtain   the desired inequality \eqref{clamart}. 
 \end{proof}
 
 \begin{figure}[h]
\centering
\includegraphics[height=7.5cm]{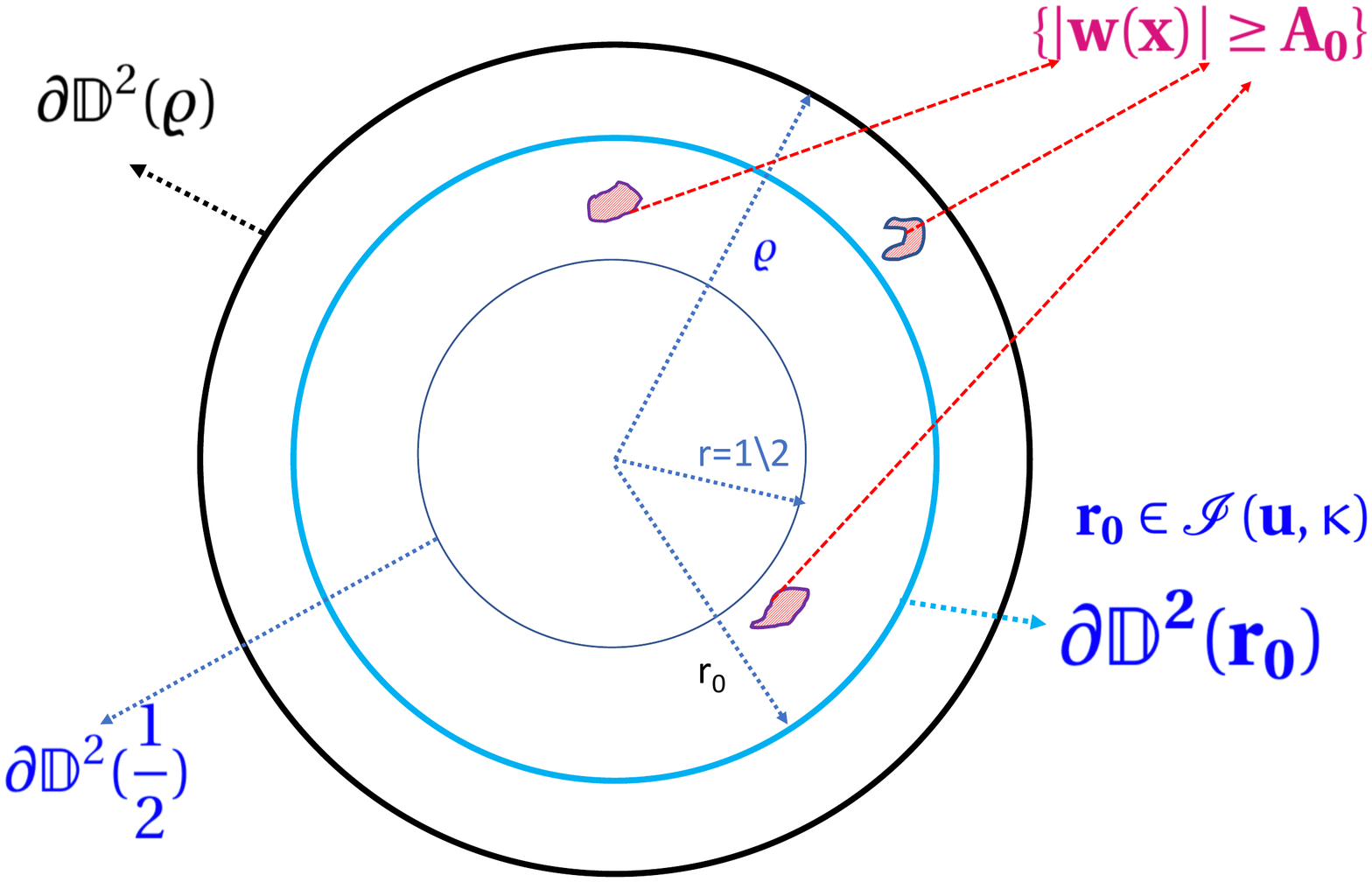}
\caption{  {\it  The circle $\partial \D^2(r_0)$ does not interset the level set $\mathcal L ( w^{-1}(A_0))$.}}
\label{coaire1}
\end{figure}
 \subsection{Revisiting the control of the energy on concentric circles}
Using  the results of  the previous section, we   work out variants of the Lemma \ref{moyenne}. For that purpose, given a radius $\varrho \in [\frac{3}{4}, 1]$, a number 
  $\displaystyle{0<\upkappa \leq \frac{\upmu_0}{2}}$, a smooth map $u:\overline{ \D^2(\varrho)} \to \R^k$   and an element $\upsigma \in \Sigma$ such that \eqref{kappacite0}  holds, we introduce the set 
  \begin{equation}
  \label{subsub}
  \Upsilon_\upsigma(u,  \varrho,  \upkappa)=\left \{ x\in \D^2(\varrho),{ \rm \ such \ that \ } \vert u(x)-\upsigma \vert \rm  \leq \upkappa\right\}.
  \end{equation}
The following result is a major tool  in the proof of our main results:

     \begin{lemma}  
     \label{remoyen}
     Let $u, \varrho$ and $\upkappa$ be as above  and assume that the bound \eqref{camembert} holds.  Assume that $\varrho \geq \frac 34$. There exists a radius $\displaystyle{ \uptau_\eps \in[\frac{5}{8}, \varrho]}$  such that $\S^1(\uptau_\eps) \subset \Upsilon_\upsigma(u,  \varrho,  \upkappa)$, i.e. 
     $$
     \vert u(\ell)-\upsigma) \vert \leq \upkappa,   {\rm  \ for \ any \ } \ell \in \S^1(\uptau_\eps), 
     $$
      and such that 
         $$ \int_{\S^1(\uptau_\eps)}e_\eps(u){\rm  d } \ell  \leq   \frac{1}{\varrho-\frac{11}{16}} \, \E_\eps(u, \Upsilon_\upsigma(u,  \varrho,  \upkappa)). $$
     \end{lemma}
     
\begin{proof} In view of definition  \eqref{subsub} of $  \Upsilon_\upsigma(u,  \varrho,  \upkappa)$ and the definition \eqref{sunyu} of $\mathcal I(u, \upkappa)$, we have 
 $ \S^1(r) \subset  \Upsilon_\upsigma(u,  \varrho,  \upkappa) $ for  any $r \in \mathcal I (u, \upkappa)$, so that, by Fubini's theorem, we have 
$$ 
\int_{\mathcal I (u, \upkappa)}  \left( \int_{\S^1(\varrho)}e_\eps(u_\eps){\rm d} \ell \right) {\rm d}\varrho \leq 
\int_{  \Upsilon_\upsigma(u,  \varrho,  \upkappa)} e_\eps(u_\eps){\rm d}x.
$$ 
Since we assume that the bound \eqref{camembert} holds, it follows from Proposition \ref{jarre} that 
$$ 
\vert \mathcal I (u, \upkappa) \vert \geq \varrho-\frac{9}{16} {\rm \ and  \ hence \ }
\vert \mathcal I (u, \upkappa) \cap   [\frac 58, \varrho] \vert \geq \varrho-\frac{11}{16}. 
$$
 Hence by a mean value argument that there exists some radius $\uptau_\eps   \in [\frac 58, \varrho] \cap \mathcal I_\eps$ such that 
 $$
  \int_{\S^1(\uptau_\eps )}e_\eps(u_\eps){\rm d} \ell \leq  \frac{1}{\varrho-\frac{11}{16}}\int_{ \Upsilon_\upsigma(u,  \varrho,  \upkappa)} e_\eps(u_\eps){\rm d}x,
 $$
 which is precisely  the conclusion. 
\end{proof}     

\noindent
{\bf Comment.}  The result above will be used in connection with the estimates for $u$ when $u$ is the solution to \eqref{elipes}. Thanks to the equation, we will be able to estimate the growth of $\E_\eps(u, \Upsilon_\upsigma(u,  \varrho,  \upkappa))$ with $\upkappa$. We will choose $\upkappa$ as small as possible to satify \eqref{camembert}, which amounts to choose of the magnitude of $\sqrt{\E_\eps(u)}$, as we will see in \eqref{cabrovski}.
\subsection{Gradient estimates on level sets}
  Given a arbitrary   smooth function $\varphi : \Omega \to \R$, where   $\Omega$  denotes a denote of $\R^N$, and an arbitrary integrable  function $f: \Omega \to \R$,  the coarea formula  \eqref{coarea} generalized as 
  \begin{equation}
     \label{coarea2}
     \int_{\R} \left (\int_{\varphi^{-1}(s)}  f(\ell){\rm d}\ell \right){\rm d} s=     \int_\Omega  \vert  \nabla  \varphi  (x)\vert  f(x){\rm d}x.
     \end{equation}
Given a smooth function $u: \Omega \to \R^k$,  we specify identity \eqref{coarea2}  with choices $\varphi =\vert u \vert$ and $f=\vert \nabla u \vert$: We are led to  the identity
 \begin{equation}
 \begin{aligned}
     \label{coarea3}
     \int_{\R} \left (\int_{\vert u \vert ^{-1}(s)} \vert \nabla u \vert (\ell){\rm d}\ell\right){\rm d} s&=     \int_\Omega  \vert  \nabla  u  (x)\vert. 
     \vert \nabla \vert u \vert  \vert {\rm d}x,  \\
     &\leq \int_\Omega  \vert  \nabla  u  (x)\vert^2{\rm d} x. 
    \end{aligned}
     \end{equation}
     We specify furthermore  this formula,  as in Subsection \ref{detroit},   for  a given map  $u$ defined on a disk $\D^2(r)$ and $w_i$ being  the corresponding  maps $w_i$ defined  on $\D^2(r)$ by formula  \eqref{doublevi}.  We introduce the subdomain 
       \begin{equation}
 \label{nablalala}
 \begin{aligned}
 \Theta (u, r)&=\left\{  x \in \D^2(r)  {\rm \ such \ that \ } u(x) \in \D^2(r) \setminus 
 \underset{i=1}{\overset {q} \cup} \B^k (\upsigma_i, \frac{\upmu_0}{2})
   \right \}   \\
   &=u^{-1}\left(\D^2(r) \setminus \underset{i=1}{\overset {q} \cup} \B^k(\upsigma_i, \frac{\upmu_0}{2})\right) 
   =  \underset{i=1}{\overset {q} \cup}  \Upsilon_{\upsigma_i} (u, r, \frac{\upmu_0}{2}).\\
  \end{aligned}
 \end{equation}
     We have: 
     
     \begin{lemma}
     \label{claudius}  
     Let $u$  be as above. 
  There exists   some  number   
  $\displaystyle{  \tilde \upmu  \in  [\frac{\upmu_0}{2},\upmu_0] }$, where $\upmu_0$ denotes the constant introduced Paragraph \ref{potentiel},  such that 
  \begin{equation}
  \label{claudius2}
\underset{i=1}{\overset{q}  \sum }  \int_{w_i^{-1}(\tilde \upmu)} \vert \nabla u \vert (\ell){\rm d}\ell  \leq  
  \frac{2}{\upmu_0} \int_{\Theta(u, r)} \vert \nabla u \vert^2 \leq \frac{4}{\upmu_0 \eps} \E_\eps(u, \Theta (u, r)).
  \end{equation}
        \end {lemma}
        
        \begin{proof}   It follows from identity \eqref{coarea3}, applied to $u-\upsigma_i$,  that 
     \begin{equation}
    \begin{aligned} 
     \underset{i=1}{\overset{q}  \sum }   \int_{\frac{\upmu_0}{2}}^{\upmu_0 }    \left( \int_{w_i^{-1}(s)} \vert \nabla u \vert (\ell){\rm d}\ell\right){\rm d} s&=   \int_{\frac{\upmu_0}{2}}^{\upmu_0 }  \underset{i=1}{\overset{q}  \sum }   \left( \int_{w_i^{-1}(s)} \vert \nabla u \vert (\ell){\rm d}\ell\right){\rm d} s \\
&\leq \int_{\Theta(u, r)} \vert \nabla u \vert^2  {\rm d} x. 
\end{aligned}
     \end{equation}
     We conclude once more by a mean-value argument.
        \end{proof}
        
  \bigskip
\begin{center}{ \Large \bf
Part II : PDE Analysis }
\end{center}
\addcontentsline{toc}{section}{Part II : PDE Analysis}

        \section{Some  properties of the PDE}
        \label{pde}
       In this section, we recall first several classical  properties of the solutions to the equation \eqref{elipes}.  We then provide some energy and potential estimates (see e.g. \cite{BBH}).
       
       \subsection{Uniform bound  through the maximum principle}
       The following uniform upper bound is standard:
       
       \begin{proposition}  
       \label{princours0}
       Let $u_\eps\in H^1(\Omega)$  be a solution of \eqref{elipes}. Then we have the  uniform bound bound, for $x\in \Omega$
\begin{equation}
\label{princours}
\vert u(x) \vert^2 \leq  \frac{4 {\rm C}_{\rm unf}}{{\rm dist}(x, \partial \Omega)} {\E}_\eps(u_\eps) +2 \sup \{ \vert \sigma \vert^2, \upsigma \in \Sigma \}.
\end{equation}
       \end{proposition}
       
       \begin{proof} 
       We argue as in \cite{BBH0}. We compute, using equation \eqref{elipes}
       \begin{equation}
       \label{multiplie}
       \begin{aligned}
       \Delta \vert u_\eps \vert^2&=u_\eps\cdot \Delta u_\eps+ \vert \nabla u_\eps \vert^2   
       =\eps^{-2} u_\eps\cdot \nabla  V_u( u_\eps)+ \vert \nabla u_\eps \vert^2  \\
     &  \geq \eps^{-2} u_\eps\cdot \nabla  V_u( u_\eps),   {\rm \ on \ } \Omega. 
       \end{aligned}
       \end{equation}
     On the other hand, it follows from assumption \eqref{condinfty}, see \eqref{cahors},  that there exists some constant $C_\infty \geq 0$ such that 
     \begin{equation}
     \label{convexita}
     y.\nabla  V( y )\geq \upalpha_\infty \vert y \vert^2- C_\infty  {\rm \ for \ any \ } y \in \R^N.
     \end{equation}
     Hence, combining \eqref{multiplie} and \eqref{convexita} we obtain the  inequality
     \begin{equation*}
    -\Delta \vert u_\eps \vert^2  +\upalpha_\infty \eps^{-2} \left( \vert u_\eps\vert^2 -\frac{ C_\infty}{\upalpha_\infty}\right) \leq 0 {\rm \ on \ } \Omega.
     \end{equation*}
     We introduce the function$W_\eps=\vert u_\eps \vert^2-\frac{C_\infty}{\upalpha_\infty}$. We are led to  the differential  inequality for $W_\eps$ 
       \begin{equation}
       \label{prince}
    -\Delta W_\eps   +\upalpha_\infty \eps^{-2} W_\eps \leq 0  {\rm \ on \ } \Omega. 
     \end{equation}
     Let $x\in \Omega$ and set $R_x={\rm dist}( x, \partial \Omega)$, so that $\D^2(x, R_x)\subset \Omega$. It follows from Lemma \ref{moyenne} and inequality \eqref{bornuni2} that there exists some radius $\displaystyle{\uptau \in [\frac{R_x}{2}, R_x]}$ and some element $\upsigma \in \Sigma$ such that 
      \begin{equation*}
      \label{bornunivert}
     \vert u_\eps(\ell)-\upsigma \vert \leq \frac{\sqrt{2} {\rm C}_{\rm unf}}{\sqrt{{R_x}}} \sqrt{{\E}_\eps(u_\eps, \D^2(R_x)})
    \leq  \frac{\sqrt{2} {\rm C}_{\rm unf}}{\sqrt{{R_x}}} \sqrt{{\E}_\eps(u_\eps)}), 
      \  \   {\rm \ for \ all \ } \ell \in 
     \S^1(\uptau).
      \end{equation*}
      Hence, we deduce  from the previous inequality  that 
      \begin{equation}
      \label{vulcain}
       W_\eps (\ell)  = \vert u_\eps (\ell) \vert^2 -\frac{C_\infty}{\upalpha_\infty}\leq  \frac{4 {\rm C}_{\rm unf}}{{R_x}} {\E}_\eps(u_\eps) +
      2\sup \{ \vert \sigma \vert^2\}-\frac{C_\infty}{\upalpha_\infty},
       \   { \rm \ for \ all \ } \ell \in    \S^1(x, \uptau),
    \end{equation}
where $\S^1(x, \uptau)=\{ \ell \in \R^2, \vert \ell-x\vert =\uptau\}$. Since $W_\eps$ satisfies inequality    \eqref{prince},  we may apply the maximum principle to assert that  
\begin{equation}
\label{balkan0}
W_\eps(y) \leq \sup\{ W_\eps (\ell), \ell \in \S^1(\uptau)\}  {\rm \ for \ } y \in \D^2(x, \uptau), 
   \end{equation} 
   so that, combining \eqref{vulcain} and \eqref{balkan0} and the definition of $R_x$, we obtain, 
   \begin{equation*}
    W_\eps(y)  \leq  \frac{4 {\rm C}_{\rm unf}}{{\rm dist}(x, \partial \Omega)} {\E}_\eps(u_\eps) +2 \sup \{ \vert \sigma \vert^2\}-\frac{c_\infty}{\upalpha_\infty}  \   { \rm \ for \ all \ } y \in    \D^2(x, \frac{R}{2}) 
  \end{equation*}
   Choosing $y=x$,    the conclusion follows.
       \end{proof}
       
       \subsection{Regularity and gradient bounds}       
    The next  result is  a  standard  a consequence of the smoothness of the potential, the regularity theory for the Laplacian and the maximum principle. 
    
    \begin{proposition}
    \label{classic}  Let $u_\eps\in H^1(\Omega)$  be a solution of \eqref{elipes} and $\delta>0$. Set $\mathcal O_\delta=\{x \in \Omega, {\rm dist} (x, \partial \Omega)\geq \delta\}$. Then $u_\eps$ is smooth on $\Omega$  and there exists a constant 
    $C_{\rm gd}\left(\Vert u \Vert_{L^\infty(\mathcal O_{\delta\slash 2})}, \delta\right)$, depending only on $V$, 
    $\Vert u \Vert_{L^\infty(\mathcal O_{\delta\slash2})}$ and $\delta$ such that
     \begin{equation}
   \label{classicic0}
 \vert    \nabla u_\eps  \vert (x) \leq \frac{C_{\rm gd}\left(\Vert u \Vert_{L^\infty(\mathcal O_\delta)}, \delta \right)}{\eps}, {\rm \ if \ } {\rm dist} (x, \partial \Omega)\geq \delta.
 \end{equation}
    \end{proposition}
   
 \begin{proof}  Estimate \eqref{classicic0} is a  consequence of Lemma A.1 of \cite{BBH0}. It  asserts that, if  $v$ is a solution on  some domain $\mathcal U$ of 
 $\R^n$ of  $-\Delta v=f$,  then we have the inequality
 \begin{equation}
 \label{bbhoo}
 \vert \nabla v \vert^2(x) \leq C\left (\Vert f \Vert_{L^{\infty}(\mathcal U)} \Vert v \Vert_{L^{\infty}(\mathcal U)}  +
 \frac{1}{{\rm dist } (x, \partial \mathcal U)^2} \Vert v \Vert_{L^{\infty}(\mathcal U)}^2\right),   {\rm \ for \ all \ } x \in \mathcal  U.
 \end{equation}
 We apply inequality \eqref{bbhoo} to the solution $u_\eps$, with  source term  $f=\eps^{-2}\nabla_u V (u_\eps) $  on the domain $\mathcal U= 
 \mathcal O_{\frac{\delta}{2}}$: This  yields \eqref{classicic0}.    we invoke the uniform estimates provided  by  Proposition \ref{princours}.  
 \end{proof}
 
 Whereas the result of Proposition \ref{classic} involves the uniform norm of $u_\eps$, the next results provides a related results, involving the energy 
 $\E_\eps(u_\eps)$.
 
 \begin{proposition}
 \label{milan}
 Let $u_\eps\in H^1(\Omega)$  be a solution of \eqref{elipes},  $\delta>0$,  $M>0$,  and assume that that $\E_\eps (u_\eps) \leq M$.
 There exists some constant $K_{\rm dr}(M, \delta)>0$,  depending only on the potential $V$, $M$ and $\delta$, such  that,
   \begin{equation}
   \label{classicic}
 \vert    \nabla u_\eps  \vert (x) \leq \frac{K_{\rm dr}(M, \delta)}{\eps}, {\rm \ if \ } {\rm dist} (x, \partial \Omega)\geq \delta.
 \end{equation}
\end{proposition}

 \begin{proof}
 we invoke the uniform estimates provided  by  Proposition \ref{princours}.  
  We have, indeed, in view of \eqref{princours0}, the uniform  upper bounds, for $u_\eps$ and $f=eps^{-2}\nabla_u V (u_\eps)$, 
 \begin{equation*}
 \left\{
 \begin{aligned}
 \vert u(x) \vert^2 &\leq  C\left ( \frac{M}{\delta}+1\right) ,   {\rm \ for \ } x  \in \mathcal O_{\frac{\delta}{2}} \\
 \vert f (x)\vert &\leq   \eps^{-2} C(M, \delta) ,   {\rm \ for \ } x  \in  \mathcal O_{\frac{\delta}{2}}. 
 \end{aligned}
 \right.
 \end{equation*}
 Combining   again these bounds with \eqref{bbhoo} and arguing as in Proposition \ref{classic},  we derive the conclusion.
 \end{proof}
   
   \subsection{Gradient term versus potential term: First estimates}
   Major ingredients in the proof of our main PDE result, namely Proposition \ref{sindec},  are provided  in  Proposition \ref{kappacity} and  Proposition \ref{borneo},  which we will state  below  and prove a little  later. They  roughly states that the \emph{total energy}, which involves both  a gradient term and  a potential terms,  can "essentially" be bounded by \emph{ the integral of the  sole  potential term}.  In order to derive these results,  we are led to divide domains into two regions: the region where the solution is near the set of potential wells  $\Sigma$, and  the region where it is far.  Whereas  the region where the solution is  \emph{near the potential wells} requires some further analysis, the region where the solution is far from the wells can be handled thanks to the results of the previous subsection, in particular the gradient bound described in Proposition \ref{classic}.
   
   Restricting ourselves to the  case $u_\eps$ is defined on  $\Omega=\D^2$, we introduce for $r>0$  the set 
 \begin{equation}
 \label{nablala}
 \begin{aligned}
 \Upxi_\eps(r)\equiv \Upxi( u_\eps, r)&=\left\{  x \in \D^2(r)  {\rm \ such \ that \ } u_\eps(x) \in \R^k \setminus 
 \underset{i=1}{\overset {q} \cup} \B^k(\upsigma_i, \frac{\upmu_0}{4})
   \right \}   \\
  & ={({u_\eps}_{_{\vert \D^2(r)}})}^{-1}\left(\R^k\setminus  \underset{i=1}{\overset {q} \cup} \B^k(\upsigma_i, \frac{\upmu_0}{4})
 \right).
    \end{aligned}
 \end{equation}
 The sets $\Upxi_\eps$ are aimed to describe region where the solution is \emph{ far from } $\Sigma$. Indeed, we have, by definition
 \begin{equation}
 \label{sofar} 
  {\rm dist} \left( u(x), \Sigma\right) \geq \frac{\upmu_0}{4} {\rm \ for \  } x \in \Upxi_\eps(r).
  \end{equation}
The integral of the energy on the set $\Upxi_\eps$ can be estimated by the integral of the  potential as follows:
 
 \begin{lemma}
\label{bornepote}
 $u_\eps\in H^1(\D^2)$  be a solution of \eqref{elipes}. There exist a constant $C_{\rm pt} \left(\Vert u \Vert_{L^\infty(\D^2(4\slash5))}\right)$  depending only on $V$ and 
 $\displaystyle{\Vert u \Vert_{L^\infty(\D^2(\frac45))}}$ such that 
 \begin{equation}
 \label{portos}
e_\eps(u_\eps)  \leq  C_{\rm  pt } \left(\Vert u \Vert_{L^\infty(\D^2(\frac45)}\right) \frac{V(u_\eps)}{\eps} \ 
{\rm  \ on \ } \Upxi_\eps (\frac 34). 
 \end{equation}
 Let $M>0$ and assume that $E(u_\eps) \leq M$. 
 There exists a constant $C_{\rm T}$ depending only on the potential $V$ and  on $M$ such that 
 \begin{equation}
 \label{portos2}
e_\eps(u_\eps)  \leq  C_{\rm  T } (M) \frac{V(u_\eps)}{\eps} \ 
{\rm  \ on \ } \Upxi_\eps (\frac 34). 
 \end{equation}
 \begin{proof} It follows from the definition of $\Upxi_\eps$  and in view of inequality \eqref{extrut} that 
 $$
 V(u_\eps(x))  \geq  \frac{\alpha_0}{16}, \, {\rm \ for \ } x \in \Upxi_\eps.
 $$
 Since, by definition $\Upxi_\eps \subset \D^2(4\slash 5)$, we have ${\rm dist}(x, \partial \D^2)=1\slash 5$, for $x\in \Upxi_\eps$.  We may therefore  invoke  inequality \eqref{classicic} of Proposition  \ref{classic} with $\delta=1\slash 20$, we obtain, for $x\in \Upxi_\eps$
 \begin{equation}
 \begin{aligned}
 \label{turin}
 \eps  \vert \nabla u_\eps \vert^2(x) &\leq C_{\rm gd}^2\left(\Vert u \Vert_{L^\infty(\D^2(4 \slash 5)}, {1}\slash{20} \right)\eps^{-1} \\
 &= \frac{\alpha_0}{4\eps} \left( \frac{4C_{\rm gd}^2}{\alpha_0}\right)
\leq   \left( \frac{4C_{\rm gd}^2}{\alpha_0}\right) \frac{V(u_\eps(x))}{\eps}. 
\end{aligned}
 \end{equation}
 Set ${\rm L}=\Vert u \Vert_{L^\infty(\D^2(4 \slash 5))}$.  Inequality \eqref{turin}   yields 
 $$e(u_\eps) \leq \left( \frac{2C_{\rm gd}({\rm L}, 1\slash20)^2}{\alpha_0}+1\right) \frac{V(u_\eps)}{\eps}.$$
 The conclusion \eqref{portos} follows choosing the constant   $C_{\rm pt}$ as $\displaystyle{C_{\rm pt}=  \left( \frac{4C_{\rm gd}({\rm L}, 1\slash 20) ^2}{\alpha_0}\right)}$. For \eqref{portos2}, we combine 
 \eqref{portos} with the uniform bound \eqref{princours}.
 \end{proof}
 \end{lemma}
 
    \subsection{The stress-energy tensor}
    The stress-energy tensor is an important tool in the analysis of singularly perturbed gradient-type problems. In dimension two, its expression is simplified thanks to complex analysis.  
      \begin{lemma}
      \label{canardwc}
Let $u_\eps$ be a solution of \eqref{elipes} on $\Omega$. Given any vector field $\vec X \in \mathcal{D}(\Omega,\R^2)$ we have
\begin{equation}
\label{canardwc1}
\int_{\Omega}   {A_\eps(u_\eps)}_{i, j}
 \cdot
\frac{\partial X_i}{\partial x_j}\,dx =0 {\rm \ where \ }
A_\eps(u_\eps)=e_\eps(u_\eps)\delta_{ij}-\eps
\frac{\partial u_\eps}{\partial x_i} \cdot\frac{\partial u_\eps}{\partial x_j}.
\end{equation}
  \end{lemma}
 The proof is standard (see \cite{BOS2}
 and references therein): It is derived multiplying the equation \eqref{elipes} by the function $\displaystyle{v=\sum X_i 
 \partial _i u_\eps}$ and integrating by parts on $\Omega$. The $2 \times 2$ stress-energy matrix $A_\eps$ may  be decomposed as
  \begin{equation}
  \label{matrixaeps}
 A_\eps\equiv A_\eps(u_\eps)= T_\eps(u_\eps)
  + \frac{V(u_\eps)}{\eps}  \,\text{I}_2\, ,
  \end{equation}
  where the matrix $T_\eps(u)$ is defined, for a map $u: \Omega \to \R^2$, by
  \begin{equation}\label{stresstensor}
  T_\eps(u)=\frac{\eps}{2} \left(
  \begin{array}{cc}
   |u_{x_2}|^2-|u_{x_1}|^2  &  -2u_{x_1}\cdot u_{x_2} \\
    -2u_{x_1}\cdot u_{x_2}  & |u_{x_1}|^2-|u_{x_2}|^2 \\
  \end{array}
  \right).
\end{equation}

\begin{remark}
\label{scratch}
{\rm  Formula \eqref{canardwc} corresponds to  the first variation of the energy when one performs  deformations of the domain induced by the diffeomorphism related to the vector field $\vec X$. More precisely, it can be derived from the fact that 
$$\frac{d}{dt} \E_\eps (u_\eps \circ \Phi_t)=0,$$
where, for $t \in \R$  $\Phi_t: \Omega \to \omega$ is a diffeomorphism such that 
$$\frac{d}{dt} \Phi_t(x)=\vec X(\Phi_t(x)),  \forall x \in \Omega.
$$
}
\end{remark}

In dimension two, one may use complex notation to obtain a simpler expression of  $\displaystyle{T_{ij}\, \frac{\partial
X_i}{\partial x_j}}$.  Setting  $\displaystyle{X= X_1+ iX_2}$ we consider the complex function $\omega_\eps: \Omega \to \C$ defined by 
\begin{equation}
\label{hopfique}
\omega_\eps=\eps \left( \vert {u_\eps}_{_{x_1}}\vert^2-\vert {u_\eps}_{_{x_2}} \vert^2-2i
{u_\eps}_{_{x_1}}\cdot {u_\eps}_{_{x_2}}\right) , 
\end{equation}
the quantity $\omega_\eps$ being usually termed the \emph{Hopf
differential} of $u_\eps$.
 We obtain the identities
\begin{equation*}
T_{ij}(u_\eps)\frac{\partial X_i}{\partial x_j} =
\mathrm{Re}\left( - \omega_\eps \frac{\partial X}{\partial
\bar{z} }\right) {\rm  \  and \ }
 \delta_{i, j} \frac{\partial X_i}{\partial x_j}=2\mathrm{Re}\left( \frac{\partial X}{\partial{z}}\right).
\end{equation*}
 Identity  \eqref{canardwc1} is turned into
 \begin{equation}
 \label{canardwc2}
 \int_{\Omega}\mathrm{Re}\left(  \omega_\eps \frac{\partial X}{\partial
\bar{z} }\right)=\frac{2}{\eps} \int_{\Omega}  {V(u_\eps)}\, \mathrm {Re} \left(\frac{\partial X}{\partial
{z} }\right)=\frac{1}{\eps} \int_{\Omega}  {V(u_\eps)}\, \mathrm {div} \, \vec X.
 \end{equation}
  
  \begin{remark}
  \label{scratch2}  {\rm  Recall that the Dirichlet energy is invariant by conformal transformation. Such transformation are locally obtained through vector-fields $\vec X$ which are holomorphic.  
  
  }
  \end{remark}
  
 \subsection{Pohozaev's identity on disks}     
 \label{pohodisk}
  Identity \eqref{canardwc2} allows to derive integral estimates of the potential $V(u_\eps)$ using a suitable choice of test vector fields. We restrict ourselves to the special case the domain is  $\Omega=\D^2(r)$, for some $r>0$. We notice that for the vector field  $X=z$, we have
   $$
   \frac{\partial X}{\partial\bar{z}}=0   {
   \rm \ and \ }
\frac{\partial X}{\partial{z}}=1.
$$
However $X=z$ is  not a test vector field, since in does not have compact support, so that we consider instead vector fields $X_\delta$ of the form 
$$X_\delta=z \varphi_\delta ({\vert z \vert })),$$
 where $0<\delta<\frac 12$ is a small parameter  and $\varphi_\delta$ is a scalar function defined on $[0, r]$ such that  
\begin{equation}
\varphi_\delta (s)=1 {\rm \ for \ } s \in [0, r-\delta) ,\ \,  \vert \varphi' (s)\vert \leq 2 \delta {\rm \ for \ } s\in [ r-\delta, r]{\rm \ and \ }  \varphi(s)=0
{\rm \ on \ } [r-\delta\slash 4, r],
\end{equation}
  so that $\varphi_\delta(r)=0$. A short  computation shows  that
$$
  \frac{\partial  \varphi_\delta ({\vert z \vert })}{\partial \bar z}=\frac{z}{2\vert z\vert} \varphi'_\delta (\vert z \vert)  {\rm \ and \ } 
   \frac{\partial  \varphi_\delta ({\vert z \vert })}{\partial  z}= \frac{\bar z}{2\vert z \vert} \varphi'_\delta (\vert z \vert),
   $$
so that 
$$
\frac{\partial  X_\delta }{\partial \bar z}=\frac{ z^2 }{2\vert z\vert} \varphi'_\delta (\vert z \vert)  {\rm \ and \ } 
   \frac{\partial  X_\delta}{\partial  z}= \frac{ \vert z \vert }{2} \varphi'_\delta (\vert z \vert)+\varphi_\delta (\vert z \vert) \in \R.
$$
We drop  the subscript $\eps$ and simply wrire $u=u_\eps$. Using polar coordinates $(r, \theta)$ such that $(x_1, x_2)=(r\cos \theta, r\sin \theta)$, we have 
$u_{x_1}=\cos \theta \, u_r-r^{-1}\, \sin \theta\,  u_\theta$  and $u_{x_2}=\sin \theta \,  u_r+r^{-1}\, \cos  \theta \, u_\theta$. After some computations, this  leads to the formula 
\begin{equation*}
\begin{aligned}
\omega_\eps &= \eps(\cos 2\theta - i\sin 2\theta)  \left[ (\vert u_r \vert^2-r^{-1} \vert u_\theta \vert^2)- 2i u_r. u_\theta)\right]  \\
&=\frac{{\bar z}^2}{\vert z \vert^2} \left[ (\vert u_r \vert^2-r^{-2} \vert u_\theta \vert^2)- 2i u_r. u_\theta)\right].
\end{aligned}
\end{equation*} 
  Combining the previous computations, we obtain
  \begin{equation}
  \label{wallonne}
  \left\{
  \begin{aligned}
{\mathrm {Re}}\left( \omega_\eps \frac{\partial X_\delta}{\partial\bar{z}}\right)
&=\frac \eps 2 \left( \vert u_r\vert^2-r^2 \vert u_\theta \vert^2\right) \vert z \vert \varphi'_\delta (\vert z \vert)   {\rm \ and \ }  \\
{\mathrm {Re}}\left( \frac{\partial X_\delta}{\partial{z}}\right)
&=\frac{1}{2} \vert z \vert \varphi_\delta '(\vert z \vert) + \varphi_\delta (\vert z \vert)  {\rm \ on \ } 
\D^2 (r).
\end{aligned}
 \right. 
  \end{equation}   
  We check that, as expected, we have 
    \begin{equation*}
   \frac{\partial X_\delta}{\partial{ \bar z}}=0 {\rm \ and \ }   \frac{\partial X_\delta}{\partial{z}}=1
 {\rm \ on \ }    \D^2( r-\delta).   
\end{equation*}
    Inserting these relations into \eqref{canardwc2} and passing to the limit $\delta\to 0$ yields the following identity, usually termed Pohozaev's identity:
   
   \begin{lemma}
    \label{poho}     Let $u_\eps$  be  a solution of \eqref{elipes} on $\D^2$. We have, for any  radius $0<r \leq 1$      
    \begin{equation}
   \label{poho1}
  \frac{1}{\eps^2} \int_{\D^2( r)} V(u_\eps) =\frac r 4\int_{\partial \D^2( r) } 
  \left(
  \left\vert \frac{\partial  u_\eps}{\partial \tau}\right \vert^2 -\left\vert \frac{\partial  u_\eps}{\partial r} \right\vert^2
  +\frac{2}{\eps^2}V(u_\eps)
  \right) {\rm d} \tau. 
    \end{equation}
      \end{lemma} 
      
  \begin{proof}  using the vector field $X_\delta$ in \eqref{canardwc2}, we obtain, in view of  identities  \eqref{wallonne}
  \begin{equation*}
  \frac{2}{\eps^2} \int_{\D^2(r)} V(u_\eps)\left[\frac{1}{2} \vert x \vert \varphi'_\delta (\vert x \vert) + \varphi_\delta (\vert x\vert)\right] \rd x = 
  \int_{\D^2(r)} 
\frac 12 \left( \vert u_r\vert^2-r^{-2} \vert u_\theta \vert^2\right) \vert x \vert \varphi'_\delta (\vert x \vert)  \rd x. 
  \end{equation*}
   so that 
   \begin{equation}
   \label{titeuf}
    \frac{2}{\eps} \int_{\D^2(r)} V(u_\eps) \varphi_\delta (\vert z \vert) \rd x=
   \frac 12  \int_{\D^2(r)} 
 \left( \vert u_r\vert^2-r^{-2} \vert u_\theta \vert^2 -\frac{2}{\eps}V(u_\eps) \right) \vert x \vert \varphi'_\delta (\vert x \vert)  \rd x. 
   \end{equation}
   Next we observe that 
   \begin{equation*}
   \left\{
   \begin{aligned}
   &\varphi_\delta (\vert \cdot \vert ) \to {\bf 1}_{\D^2(r)}   {\rm \ as  \ } \delta \to  0  \ {\rm \ in \  the  \ sense \ of \ measures}, {\rm \ and \ } \\
   & \vert \cdot  \vert \varphi'_\delta (\vert \cdot  \vert) \to -r\rd \tau  {\rm \ as  \ } \delta \to  0  {\rm \ in \ } \mathcal D'( \R^2), 
  \end{aligned} 
   \right.
   \end{equation*}
  where $\rd \tau$    denotes the $\mathcal H^1$ measure on $\S^1(r)$. the conclusion follows.
    
  \end{proof}    
      
      Identity \eqref{poho1} is central in the paper, in particular it leads to the monotonicity for $\upzeta_\star$.    This identity has the remarkable property that it yields an identity   of the integral of the potential \emph{inside}   the disk involving only    energy terms on  the  \emph{boundary}. A straightforward consequence of Lemma \ref{poho} is  the estimate: 
        
         \begin{proposition}
     \label{pascap0}
     Let $u_\eps$  be  a solution of \eqref{elipes} on $\D^2$. We have, for any $0<r\leq 1$ 
     \begin{equation}
     \label{pascap}
      \frac{1}{\eps} \int_{\D^2( r)} V(u_\eps) \leq  \frac r 2 \int_{\S^1(r)} e_\eps(u_\eps) {\rm d}\ell.
     \end{equation}
     \end{proposition}  
     Proposition \ref{pascap0} follows immediately from Lemma \ref{poho} noticing that the absolute value of the integrand on the left hand  side is bounded by $2\eps^{-1}e_\eps(u_\eps)$.
     
\smallskip
Besides Proposition \ref{pascap0}, we notice      that     Pohozaev's identity leads directly to remarkable  consequences:  For instance,  all solutions which are constant with values in $\Sigma$ on $\D^2(r)$ are necessarily constant. 
      \begin{remark}{\rm 
        The previous  results are  specific to dimension $2$, however the use of the stress-energy tensor yields other results in higher dimensions (for instance monotonicity formulas).   
    }
    \end{remark}
   \subsection{Proofs of the "monotonicity" formula  for $\upzeta_\eps$}
      \label{monotoinou}
   We provide here a proof of formula \eqref{monotonio0}, which is actually not a real monotonicity, since there is no evidence that the right hand side is non negative (only the asymptotic version is a monotonicity formula). The proof relies  on Lemma \ref{poho},  identity \eqref{poho1}. We have indeed, by Leibnitz rules 
   $$
   \frac{d}{dr}\left( \frac{\Veps\left(u_\eps, \D^2(r)\right)}{ r}\right)=
   -\frac{1}{r^2} \Veps\left(u_\eps, \D^2(r)\right)  +\frac{1}{r}  \frac{d}{dr}\left( \Veps\left(u_\eps, \D^2(r)\right)\right).
   $$
   By Fubini's theorem, we have 
   $$
   \frac{d}{dr}\left( \Veps\left(u_\eps, \D^2(r)\right)\right)=\frac{1}{\eps} \int_{\S^1(r)} V(u_\eps) \rd \tau,
   $$
   so that, combining the previous identities, we obtain 
    \begin{equation*}
    \label{borg}
    \begin{aligned}
     \frac{d}{dr}\left( \frac{\Veps\left(u_\eps, \D^2(r)\right)}{ r}\right)&= 
    -\frac{1}{r^2} \int_{\D^2(r)} \eps^{-1} V(u_\eps) \rd x +\frac{1}{r} \int_{\S^1(r)}\eps^{-1} V(u_\eps) \rd \tau \\
    &=\frac{1}{4r}\int_{\S^1(r)} (\eps \vert u_r \vert^2-\eps\vert u_\tau\vert^2-2V_\eps(u)) \rd \tau +\frac{1}{r} \int_{\S^1(r)}\eps^{-1} V(u_\eps) \rd \tau \\
    &=\frac{1}{4r}\int_{\S^1(r)} (\eps \vert u_r \vert^2-\eps\vert u_\tau\vert^2+2V_\eps(u)) \rd \tau 
   \end{aligned}
    \end{equation*}
    where we have used \eqref{poho1} for the second line. Hence, identity  \eqref{monotonio0} is established.

   \subsection {Proof of formula \eqref{monantoine0}}

    For the identity \eqref{monantoine0}, we have similarily 
    \begin{equation}
    \label{connors}
    \begin{aligned}
     \frac{d}{dr}\left( \frac{\Eps\left(u_\eps, \D^2(r)\right)}{r}\right)&=-\frac{1}{r^2} \int_{\D^2(r)} e_\eps(u_\eps) \rd x+ \frac{1}{r} \int_{\S^1(r)} e_\eps(u) \rd \tau \\\
     &=-\frac{1}{2r^2} \int_{\D^2(r)}  \eps\vert \nabla u \vert^2\rd x-\frac{1}{r^2} \int_{\D^2(r) }\eps^{-1} V(u_\eps)  \rd x \\
     &+\frac{1}{2r} \int_{\S^1(r)} ( \eps(\vert u_{\tau}\vert^2+\vert u_{r}\vert^2)+2\eps^{-1} V(u_\eps)) \rd \tau 
          \end{aligned}
    \end{equation}
    We  may decompose  $\eps\vert \nabla u \vert^2$ as  $\eps\vert \nabla u \vert^2=2\eps^{-1}V(u)- 2\xi_\eps(u)$, where  the discrepancy 
   $
   \xi_\eps(u_\eps)$
   is defined in \eqref{discretpanse}, so that the second line in \eqref{connors} may be written as
    \begin{equation} 
    \label{connors2} 
    -\frac{1}{2r^2} \int_{\D^2(r)}  \eps\vert \nabla u \vert^2\rd x-\frac{1}{r^2} \int_{\D^2(r) }\eps^{-1} V(u_\eps)  \rd x =
    \int_{\D^2(r)}\xi_\eps(u)-\frac{2}{r^2} \int_{\D^2(r) }\eps^{-1} V(u_\eps)  \rd x \\
    \end{equation} 
    Combining \eqref{connors}, \eqref{connors2} with \eqref{poho1}, we obtain a nice cancelation which yields \eqref{monantoine0}.
    
    \medskip

 \subsection{Pohozaev's type inequalities on general subdomain}
  We present in this subsection a related tool which will be of interest in the proof of Theorem \ref{bordurer}. We consider a solution $u_\eps$ of \eqref{elipes} on a general domain $\Omega$, a  subdomain $\mathcal U$ of $\Omega$ and for $\updelta >0$ the domain 
  $\mathcal U_\updelta$ introduced in \eqref{Udelta}.  As  a variant of Proposition \ref{pascap0}, we have:   
     
          \begin{proposition}
     \label{pascapit}
     Let $u_\eps$  be  a solution of \eqref{elipes} on $\Omega$. We have, for any $0<\updelta $ 
     \begin{equation}
     \label{pascap36}
      \frac{1}{\eps} \int_{\mathcal U_{\frac{\updelta}{2}} }V(u_\eps) {\rm d} x \leq  C(\mathcal U, \updelta)  \int_{\mathcal V_\updelta} e_\eps(u_\eps) {\rm d}x,
     \end{equation}
      where the constant $C(\mathcal U, \updelta) >0$ depends on $\mathcal U$, $\updelta$ and $V$.
     \end{proposition}
     
     The main difference with Proposition \ref{pascap0} is that, in the case of a disk,  the form of the $C(\mathcal U, \updelta) >0$  is determined more accurately. 
\begin{proof} [Proof of Proposition \ref{pascapit}] Turning back to identity \eqref{canardwc2},  we choose once more a  test vector field $\vec X_\updelta $ of the form $X_\updelta(z)=z  \chi _{_\delta} (z)$,  where the function $\chi_{_\delta} $ is a  smooth scalar positive function  such that 
$$
\chi_{_\delta} (z)=1 {\rm \ for \ } z \in {\mathcal U}_{\frac{\updelta}{2}}  {\rm \ and \ } \chi_{_\delta} (z)=0 {\rm \ for \ } z \in \R^2 \setminus \,  
{\mathcal U}_{\updelta}
$$
 so that $\nabla \chi_{_\updelta}=0$ on the set ${\mathcal U}_{\frac{\updelta}{2}}$ and hence 
 \begin{equation*}
 \frac{\partial X_\delta}{\partial
\bar{z}}=0  
{\rm \ and \ } 
\frac{\partial X_\delta}{\partial
{z}}=1 \,  
 {\rm \ on \ } {\mathcal U}_{\frac{\updelta}{2}}.
 \end{equation*}   
Inserting these relations into \eqref{canardwc2}, we are led  to inequality \eqref{pascap36}.
\end{proof}
     

 \section{Energy estimates}

\subsection{First energy estimates on levels sets close to $\Sigma$} 
In this subsection, we estimate the energy on domains where the solution is close to one of minimizers of the potential  $\upsigma \in \Sigma$.  Near such a  point,  the potential is locally convex, close to a quadratic potential. In such a situation,   solutions to the equation behave, at first order,  as solution to the \emph{linear equation} of the type
$$-\Delta v+ \eps^{-2} \nabla^2 V(\upsigma) \cdot v\simeq 0, $$ 
for which energy estimates can be obtained directly by multiplying the equation  by the solution itself and integration by parts, provided estimates on the boundary are available. 
More precisely,  we consider again for given $0<\eps \leq 1$ a solution  $u_\eps: \D^2 \to \R^k$  to \eqref{elipes} and assume that we  are given a radius $\varrho_\eps \in [\frac 12, \frac 34]$, a number   
$\displaystyle{0<\upkappa <\frac{\upmu_0}{2}}$, where $\upmu_0>0$ is the constant provided in \eqref{kiv}. 
 We introduce the subdomain $\Upsilon_\eps(\varrho_\eps, \upkappa)$ defined  by
   \begin{equation}
\label{lilu}
\begin{aligned}
\Upsilon_\eps(\varrho_\eps, \upkappa) &=\left\{ x \in \D^2 ( \varrho_\eps ) {\rm \ such \ that \  \ }   \vert  u_\eps(x) -\upsigma_i \vert  <\upkappa, \ \ 
{\rm\ for  \ some  \ } i=1\ldots q \ \right\} \\
&=\underset {i=1}{\overset {q} {\large  \cup}}   \Upsilon_{\eps, i}  (\varrho_\eps, \upkappa), 
\end{aligned}
\end{equation}
 where  we have set 
 $$
  \Upsilon_{\eps, i}(\varrho_\eps, \upkappa)=w_i^{-1} ([0,   \upkappa) \cap  
\D^2 ( \varrho_\eps )=\Upsilon_{\upsigma_i}(u_\eps,  \varrho_\eps,  \upkappa)=
\{x\in \D^2 ( \varrho_\eps ), \vert u_\eps-\upsigma_i \vert  \leq   \upkappa \}.
 $$
 Notice that sets of the above form   have  already  been introduced in \eqref{subsub} for general maps $u$ and are denotes there $\Upsilon_{\upsigma}(u,  \varrho,  \upkappa)$.
The set $\Upsilon_\eps(\varrho_\eps, \upkappa)$  corresponds  hence to a truncation of the domain 
  $\D^2(\varrho_\eps)$,  where points with values far from the set $\Sigma$  have been removed, whereas the set $\Upsilon_{\eps, i}(\varrho_\eps, \upkappa)$  corresponds  to a truncation of the domain 
  $\D^2(\varrho_\eps)$ where points with values far from the point $\upsigma_i \in \Sigma$  have been removed.
  
  The main result of the present  section is to establish  an estimate on the integral of the energy on the domain $\Upsilon_\eps(\varrho_\eps, \upkappa)$ in terms of  the integral of the potential  as well as boundary integrals.  As a matter of fact, we choose here a fixed value of $\upkappa$, namely
   \begin{equation}
     \label{menhir} 
     \upkappa=\upmu_1=\frac{\upmu_0}{4}.
     \end{equation} 
  However, many elements in the proof carry out for a full range of values of $\upkappa$, and will be used later in Subsection \ref{later}.
      
        \begin{proposition}
  \label{kappacity20}
  Let $u_\eps$ be a solution of \eqref{elipes} on $\D^2$, let $\rm L >0$ be given and assume that 
  \begin{equation}
  \label{belliard}
   \Vert u_\eps \Vert_{L^\infty} \leq {\rm L}.
   \end{equation}
Let  $\varrho_\eps \in [\frac12, \frac 34]$. We have, for some constant ${\rm K}_\Upsilon(\rm L)  >0$,  depending only on the potential  $V$  and $\rm L$, the inequality
  \begin{equation}
  \label{nonodesu}
   \int_{\Upsilon_\eps (\varrho_\eps, \upmu_1)} e_\eps(u_\eps)(x) {\rm d}x \leq  {\rm K_\Upsilon}({\rm L}) \left[ 
   \int_{\D^2(\varrho_\eps)} \frac{V(u_\eps)}{\eps} {\rm d} x  +\eps    \int_{\partial \D^2({\varrho_\eps})} e_\eps(u_\eps){\rm d} \ell
     \right].
   \end{equation}
  \end{proposition}

  The proof will be divided in several results  of independznt interest.   Firstly, since $u_\eps$ is smooth and in view of Sard's Lemma, the boundary $\partial \Upsilon_\eps (\varrho_\eps, \upkappa)$ is smooth and  a finite union of smooth curves \emph{for almost every $\upkappa$}, which we will assume throughout. Hence, for $i=i, \ldots, q$
 the set $\partial \Upsilon_{\eps, i}$ is an union of smooth curves intersecting the boundary $\partial \D^2(\varrho_\eps)$ transversally. For $i=i, \ldots, q$,  we define the curves $\Gamma_\eps^i$  and 
 $\Pi_\eps^i$ as   
  \begin{equation}
  \label{levelsets}
  \left\{
  \begin{aligned}
\Gamma_\eps^i(\varrho_\eps, \upkappa) & \equiv  \partial   \Upsilon_{\eps, i}  (\varrho_\eps, \upkappa)\cap  \D^2( \varrho_\eps)= w_i^{-1}(\upkappa)\cap  \D^2 ( \varrho_\eps)  
 \  {\rm \ for \ } i=1 \ldots   q,  \\
\Pi_\eps^i (\varrho_\eps, \upkappa)&\equiv     \Upsilon_{\eps, i}  (\varrho_\eps, \upkappa)\cap \partial \D^2( \varrho_\eps) =
\left[w_1^{-1}([0,\upkappa])\cap \partial  \D^2 ( \varrho_\eps) \right], 
  \end{aligned}
  \right. 
  \end{equation}
  so that
  \begin{equation}
  \partial   \Upsilon_{\eps, i}  (\varrho_\eps, \upkappa)=\Gamma_\eps^i(\varrho_\eps, \upkappa)\cup \Pi_\eps^i (\varrho_\eps, \upkappa).
  \end{equation}	
  In view of \eqref{lilu}, we introduce, for $i=1, \ldots, q$, the integral quantities
  \begin{equation}
  \label{Qi}
  \mathfrak Q^i_\eps(\varrho_\eps, \upkappa)= \int_{ \Upsilon_{\eps, i} (\varrho_\eps, \upkappa)}\eps \vert \nabla u_\eps \vert^2 +\eps^{-1} \nabla_u V(u_\eps)\cdot (u_\eps-\upsigma_i).
  \end{equation}	
  We first notice that:
  
  \begin{lemma}  
  \label{dugoin007}
  We have, for every $\upkappa\in [0, \upmu_0]$, the inequality
  \begin{equation}
  \label{dugoin}
   \int_{ \Upsilon_{\eps, i}(\varrho_\eps, \upkappa) }e_\eps(u_\eps) {\rm d} x \leq  \frac{2\lambda_{\rm max}}{\lambda_0} \mathfrak Q^i_\eps(\varrho_\eps, \upkappa).
  \end{equation} 
  \end{lemma}	
  \begin{proof}
   Since, by the definition of $\Upsilon_{\eps, i}$,  we have $\displaystyle{\vert u-\upsigma_i  \vert \leq \upkappa \leq \frac{\upmu_0}{2}}$,   we  are  in position to invoke estimates \eqref{glutz},  which yields, for $i \in \{1, \ldots, q\}$,
\begin{equation}
\label{duglandmoins}
 \frac {\lambda_0}{2\lambda_{\rm max}}V(u_\eps) \leq \frac{1}{2}\lambda_0\vert u_\eps -\upsigma_i \vert^2  \leq  \nabla V(u_\eps) \cdot (u_\eps-\upsigma_i)
{\rm \ on \ }  \Upsilon_{\eps, i} (\varrho_\eps, \kappa), 
\end{equation}
where  $\displaystyle{\lambda_{\rm max}=\sup \{\lambda_i^+, i=1, \ldots, q_i\}}$.
Multiplying the previous inequality by $2\lambda_{\rm max}\slash{\lambda_0} $ and integrating on $\Upsilon_{\eps, i}(\varrho_\eps, \upkappa)$, we are led to  
\begin{equation}
\label{dugoin001}
 \int_{ \Upsilon_{\eps, i}(\varrho_\eps, \upkappa) } \eps^{-1} V(u_\eps) {\rm d} x \leq  \frac{2\lambda_{\rm max}}{\lambda_0} 
  \int_{ \Upsilon_{\eps, i} (\varrho_\eps, \upkappa)} \eps^{-1} \nabla_u V(u_\eps)\cdot (u_\eps-\upsigma_i).
 \end{equation}
 The conclusion then follows from the definitions  of $e_\eps$ and $\mathfrak Q^i_\eps(\varrho_\eps, \upkappa)$.					
    \end{proof} 
  
  A simple integration by parts yields the following:
  
  \begin{lemma} 
  \label{qidonc}
  Assume that $0<\eps \leq 1$  and  that $u_\eps$ is a solution to \eqref{elipes} on $\D^2(1)$. Let $\varrho_\eps$ be in $[1\slash2,  3\slash4]$.  We have, for every $\upkappa\in [0, \upmu_0]$, the identity, for every $i=1, \ldots, q$
\begin{equation}
\label{qidonc0}
\mathfrak Q^i_\eps(\varrho_\eps, \upkappa)={ \eps}
\left[ \int_{ \Gamma_{\eps}^i(\varrho_\eps, \upkappa)}\upkappa  \frac{\partial \vert u_\eps-\upsigma_i\vert}{\partial \vec n} \rd \ell\\
  +\int_{ \Pi_{\eps}^i(\varrho_\eps, \upkappa)}\left \vert u-\upsigma_i \right \vert \left \vert   \frac{\partial \vert u_\eps-\upsigma_i\vert}{\partial \vec n} \right \vert \rd \ell  
  \right].
\end{equation}
  \end{lemma} 
  \begin{proof}
For $i=1, \ldots, q$, we multiply equation  \eqref{elipes}  by $\eps(u_\eps-\upsigma_i)$ and integrate by parts on the domain 
  $ \Upsilon_{\eps, i} (\varrho_\eps, \kappa)$. This yields, for $i=1, \ldots, q$
\begin{equation}
  \label{stokounette0}
  \begin{aligned}
\mathfrak Q^i_\eps(\varrho_\eps, \upkappa)&=   \int_{ \Upsilon_{\eps, i} (\varrho_\eps, \upkappa)}\eps \vert \nabla u_\eps \vert^2 +\eps^{-1} \nabla_u V(u_\eps)\cdot (u_\eps-\upsigma_i) \\
  &= \int_{ \partial \Upsilon_{\eps, i}(\varrho_\eps, \upkappa)} \eps  \frac{\partial u_\eps}{\partial \vec n} \cdot (u_\eps-\upsigma_i)\\
  &=\frac{\eps}{2}\int_{ \Gamma_{\eps}^i(\varrho_\eps, \upkappa)}  \frac{\partial \vert u_\eps-\upsigma_i\vert^2}{\partial \vec n} 
  +\frac{\eps}{2}\int_{ \Pi_{\eps}^i(\varrho_\eps, \upkappa)}  \frac{\partial \vert u_\eps-\upsigma_i\vert^2}{\partial \vec n}, 
  \end{aligned}
   \end{equation}
   which yields the desired result, since
   $\displaystyle{\frac{\partial \vert u_\eps-\upsigma_i\vert^2}{\partial \vec n}= 2\vert u_\eps-\upsigma_i\vert\frac{\partial \vert u_\eps-\upsigma_i\vert}{\partial \vec n}}$, so that 
 \begin{equation}
 \frac{\partial \vert u_\eps-\upsigma_i\vert^2}{\partial \vec n}=2\upkappa    \frac{\partial \vert u_\eps-\upsigma_i\vert}{\partial \vec n}
 {\rm  \ on \ }   \Gamma_{\eps}^i(\varrho_\eps, \upkappa).  
\end{equation}
\end{proof}

  \begin{remark}{\rm Notice that we have the inequality
  \begin{equation}
  \label{beam}
   \frac{\partial \vert u_\eps-\upsigma_i\vert}{\partial \vec n} \geq 0
 {\rm  \ on \ }   \Gamma_{\eps}^i(\varrho_\eps, \upkappa). 
  \end{equation} 
  Indeed,  by definition $\vert u_\eps-\upsigma_i\vert=\upkappa$ on $\Upsilon_{\eps, i}(\varrho_\eps, \upkappa)$, so that we are on a level set and the normal derivative $\vec {n} (\ell)$ is pointing  towards the outside.
  }
  \end{remark}
 The next result will also be used extensively  in Subsection \ref{later}: 
      
  \begin{lemma}
  \label{eliot}
  Assume that $0<\eps \leq 1$  and  that $u_\eps$ is a solution to \eqref{elipes} on $\D^2(1)$. Let $\varrho_\eps$ be in $[1\slash2,  3\slash4]$.  We have, for every $\upkappa\in [0, \upmu_0]$, the inequality
  \begin{equation}
  \label{klisko}
  \int_{\Upsilon_\eps(\varrho_\eps, \upkappa)} e_\eps(u_\eps)  {\rm d}x \leq   
  C\eps \left[ \,  \upkappa \, 
  \underset{i=1} {\overset {q}\sum}
   \int_{\Gamma_{\eps}^i(\varrho_\eps, \upkappa)} \frac{\partial \vert u_\eps(\ell)-\upsigma_i \vert}{\partial \vec n (\ell)} {\rm d}\ell+    \int_{\partial \D^2(\varrho_\eps)} e_\eps(u_\eps(\ell)){\rm d}\ell \right]. 
       \end{equation}
    where $C>0$ is some constant depending only on the potential $V$  and where $\vec n(\ell)$ denotes the unit vector normal to  $\Gamma_{\eps, i}\cup \Pi_{\eps, i}$ pointing in the direction increasing $\vert u_\eps-\upsigma_i \vert $. 
      \end{lemma}
\begin{remark}{\rm Let us emphasize  that in this statement, $\upkappa$ is not constrained by \eqref{menhir} and may actually take arbritrary small values.}
\end{remark}

\begin{proof}      The proof relies on a combination of the results  of Lemmas \ref{qidonc} and \ref{eliot}. We first estimate the second term on the r.h.s of \eqref{qidonc0}. Since by definition, we have   the inclusion  $\Pi_{\eps}^i(\varrho_\eps, \upkappa)\subset \S^1(\varrho_\eps)$, it follows that $ \vec n(\ell)=\vec e_r$ on $\Pi_{\eps}^i(\varrho_\eps, \upkappa)$, so that  
\begin{equation}
\label{plume}
\left\vert \frac{\partial \vert u_\eps-\upsigma_i\vert}{\partial \vec n}\right\vert =\left \vert \frac{\partial \vert u_\eps- \upsigma_i \vert}{\partial r} \right \vert 
\leq  \left \vert \frac{\partial u_\eps}{\partial r}\right\vert
\leq \vert \nabla u_\eps \vert, 
  {\rm \ on \ } \Pi^i_\eps (\varrho_\eps, \upkappa).
\end{equation}
On the other hand, in view of Proposition \ref{potto},  as well as the fact that $\vert u_\eps(\ell)-\upsigma_i\vert\leq \upkappa\leq \upmu_0$ on $\Pi^i_\eps (\varrho_\eps, \upkappa)$,  we have 
\begin{equation}
\label{goudron}
\vert u_\eps-\upsigma_i\vert \leq \frac{2}{\lambda_0} \sqrt{V(u_\eps)}  {\rm \ on \ } \Pi^i_\eps (\varrho_\eps, \upkappa).
\end{equation}
Combining \eqref{plume} with \eqref{goudron} and integrating on $\Pi_\eps (\varrho_\eps, \upkappa)$, we obtain the estimate 
    \begin{equation}
  \label{ducruet2}
  \begin{aligned}
   \int_{ \Pi_{\eps}^i(\varrho_\eps, \upkappa)} \vert u_\eps-\upsigma_i\vert. \vert  \frac{\partial \vert u_\eps-\upsigma_i\vert}{\partial \vec n}\vert \rd \ell
    &
 \leq \frac{2}{\lambda_0} \int_{\Pi_{\eps}^i(\varrho_\eps, \upkappa)} \sqrt{V(u_\eps)}. \vert \nabla u_\eps \vert \rd \ell \\
 &\leq  \frac{2}{\lambda_0}\int_{\S^1(\varrho_\eps)}e_\eps (u) \rd \ell, 
 \end{aligned}
  \end{equation}
   where,  for the second  inequality, we used Lemma \ref{ab0} and the fact that  $\Pi_{\eps}^i(\varrho_\eps, \upkappa)\subset \S^1(\varrho_\eps)$. Combining \eqref{ducruet2} with  \eqref{qidonc0} and \eqref{dugoin}, we obtain the desired conclusion \eqref{klisko} for the choice of constant 
   $C= \frac{2\lambda_{\rm max}}{\lambda_0}(1+\frac{2}{\lambda_0})$.

\end{proof}

Our next results allows to obtain, for a suitable choice of $\upkappa$,  a bound on the first term on the right hand side of \eqref{klisko}:

\begin{lemma}
\label{toubon}
 Assume that $0<\eps \leq 1$  and  that $u_\eps$ is a solution to \eqref{elipes} on $\D^2(1)$.  Let $\varrho_\eps \in [\frac12, \frac 34]$.
There exists some number   $\displaystyle{\tilde \upmu_\eps \in [\frac{\upmu_0}{4}, \frac{\upmu_0}{2}]}$ such that 
  \begin{equation}
  \label{totobene}
 \eps  \int_{\Gamma_{\eps, i}(\varrho, \tilde \upmu_\eps)} \frac{\partial \vert u_\eps-\upsigma_i \vert }{\partial \vec n(\ell)} {\rm d } \ell 
 \leq \eps \int_{\Gamma_{\eps, i}(\varrho, \tilde \upmu_\eps)} \vert \nabla u_\eps \vert  {\rm d } \ell
  \leq  \frac{1}{\upmu_0 } \E_\eps(u, \Upxi (u_\eps,  \varrho_\eps)), 
    \end{equation} 
where   $\Upxi (u_\eps,  \varrho_\eps)$ is defined in \eqref{nablalala}.
\end{lemma}

\begin{proof} We invoke first  Lemma \ref{claudius} with the choices $r=\varrho_\eps$ and  $u=u_\eps$.  This yields  directly a number 
  $\displaystyle{\tilde \upmu_\eps \in [\frac{\upmu_0}{4}, \frac{\upmu_0}{2}]}$ such that \eqref{totobene} is satisfied, so that the proof is complete.
 \end{proof} 

\noindent
{\it Proof of Proposition \ref{kappacity20} completed}.    We combine \eqref{klisko}  for $\upkappa=\tilde \upmu_\eps$ with \eqref{totobene}. This yields
 \begin{equation}
  \label{klisko00}
  \int_{\Upsilon_\eps(\varrho_\eps, \tilde \upmu_\eps )} e_\eps(u_\eps)  {\rm d}x \leq   
  C\eps \left[ \,  \, 
   \frac{ \tilde \upmu_\eps}{\upmu_0 }  E_\eps (u_\eps,  \Upxi (u_\eps,  \varrho_\eps))+    \int_{\partial \D^2(\varrho_\eps)} e_\eps(u_\eps(\ell)){\rm d}\ell \right]. 
       \end{equation}
 On the other hand, it follows from assumption \eqref{belliard}  and Lemma \ref{bornepote} that
\begin{equation}
 \label{portos}
e_\eps(u_\eps)  \leq  C_{\rm  pt } \left(\rm L \right) \frac{V(u_\eps)}{\eps} \ 
{\rm  \ on \ } \Upxi_\eps \left(u_\eps, \frac 34\right)\supset  \Upxi \left(u_\eps,  \varrho_\eps\right),  
 \end{equation}
so that  
  \begin{equation}
  \label{grumeaux}
  \begin{aligned}
  E_\eps (u_\eps,  \Upxi (u_\eps,  \varrho_\eps))&=\int_{\Upxi (u_\eps,  \varrho_\eps)}e_\eps(u_\eps) \rd x \leq  C_{\rm  pt } \left(\rm L \right) \int_{\Upxi (u_\eps,  \varrho_\eps)}\frac{V(u_\eps)}{\eps} \rd x \\
  &\leq C_{\rm  pt } \left(\rm L \right) \int_{\D^2(\varrho_\eps)}\frac{V(u_\eps)}{\eps} \rd x.
  \end{aligned}
   \end{equation}
   Combining \eqref{grumeaux} with \eqref{klisko00},  we obtain \eqref{nonodesu} for ${\rm K}_\Upsilon({\rm L})=C. C_{\rm  pt } \left(\rm L \right)$.
    
    \subsection{Refined estimates on  level sets close to $\Sigma$ }  
    \label{later}

    Whereas we obtained in Proposition \ref{kappacity} an energy estimate on  a \emph{fixed level} set $\Upsilon_\eps(\varrho_\eps,  \upmu_1)$, we derive here an energy estimate on the set $\Upsilon_\eps (\varrho_\eps, \upkappa)$ allowing the value of $\upkappa$ to vary and in particular  to  be small. This will be completed at the cost of an additional assumption. Indeed, we will assume that   there exists an 
  element $\upsigmam \in \Sigma$ such that 
   \begin{equation}
  \label{kappacite}
   \vert u_\eps-\upsigmam \vert  < \upkappa  {\rm \ on \ } \partial \D^2(\varrho_\eps).
  \end{equation}
    The main result of this subsection is: 
        
        \begin{proposition}
  \label{kappacity}
  Let $u_\eps$ be a solution of \eqref{elipes} on $\D^2$,  $M>0$,   $0<\upkappa <\frac{\upmu_0}{4}$ and $\varrho_\eps \in [\frac12, \frac 34]$.
  Assume that \eqref{kappacite} is satisfied and that 
  \begin{equation}
  \label{abritel}
  E(u_\eps)\leq M.
  \end{equation}
  We have, for some constant ${\rm C}_\Upsilon (M) >0$,  depending only on the potential  $V$ and on $M$, 
  \begin{equation}
  \label{nonodes}
   \int_{\Upsilon_\eps (\varrho_\eps, \upkappa)} e_\eps(u_\eps)(x) {\rm d}x \leq  {\rm C_\Upsilon}(M) \left[ 
 \upkappa  \int_{\D^2(\varrho_\eps)} \frac{V(u_\eps)}{\eps} {\rm d} x  +\eps    \int_{\partial \D^2(\varrho_\eps)} e_\eps(u_\eps){\rm d} \ell
     \right].
   \end{equation}
  \end{proposition}
  
 Of major  importance in estimate \eqref{nonodes} is the \emph{presence of the term $\upkappa$ in front of the integral of the  potential}, so that the energy on 
 $\Upsilon_\eps (\varrho_\eps, \upkappa)$ grows essentially at most linearly with respect to $\upkappa$. Proposition \ref{kappacity} will be used in the proof of the clearing-out result, so that we will use it for \emph{small} $M$. 
 
   We may assume without loss of generality that $\upsigmam=\upsigma_1$, so that it follows from   assumption \eqref{kappacite} that 
\begin{equation}
\label{versailles}
\vert u_\eps(\ell)-\upsigma_1\vert < \upkappa  {\rm \  for \ }  \ell \in \partial \D^2 ( \varrho_\eps). 
\end{equation}

We deduce from  inequality  \eqref{versailles} that  
$\displaystyle { \partial \D^2 ( \varrho_\eps)  \subset \overline{\Upsilon_{\eps, 1}(\varrho_\eps, \upkappa)}}$, 
and that, for 
   $i=2, \ldots, q$, we have 
   $$\partial \D^2 ( \varrho_\eps) \cap  \partial   \Upsilon_{\eps, i} (\varrho_\eps, \upkappa)=\emptyset.$$ 
In particuler, we have the identities 
 \begin{equation}
  \left\{
  \begin{aligned}
\Gamma_\eps^i(\varrho_\eps, \upkappa) & \equiv  \partial   \Upsilon_{\eps, i}  (\varrho_\eps, \upkappa)= w_i^{-1}(\upkappa)\cap  \D^2 ( \varrho_\eps)  
 \  {\rm \ for \ } i=2 \ldots   q,  \\
 \Pi_\eps^i(\varrho_\eps, \upkappa) &=\emptyset,  \  {\rm \ for \ } i=2 \ldots   q,  \\
\Gamma_\eps^1 (\varrho_\eps, \upkappa)&\equiv  \partial   \Upsilon_{\eps, 1}  (\varrho_\eps, \upkappa)\setminus \partial \D^2( \varrho_\eps) =
\left( \left[w_1^{-1}(\upkappa)\cap  \D^2 ( \varrho_\eps) \right] \right)\setminus \partial \D^2( \varrho_\eps) {\rm \ and \ } \\
\Pi_\eps^1(\varrho_\eps, \upkappa)&=\partial \D^2(\varrho_\eps).
  \end{aligned}
  \right. 
  \end{equation}				
       As for Proposition \ref{kappacity20}, we will  deduce Proposition \ref{kappacity} from Lemma \ref{eliot}. For that purpose, we will make use of a \emph{new ingredient}, given by the following monotonicity formula:  
  
  \begin{lemma}  
  \label{repro}  
  Let  $\upkappa_1 \geq \upkappa_0 \geq \upkappa$ be given. If $u_\eps$ satisfies condition \eqref{kappacite},  then we have, for $i=1, \ldots, q$,   the inequality
   \begin{equation}
     \label{sting}
 0\leq  \int_{\Gamma_{\eps}^i(\varrho_\eps, \upkappa_0)} \frac{\partial \vert u_\eps(\ell)-\upsigma_i \vert }{\partial \vec n(\ell)} {\rm d} \ell \leq 
   \int_{\Gamma_{\eps}^i(\varrho_\eps, \upkappa_1)}  \frac{\partial \vert u_\eps(\ell)-\upsigma_i \vert }{\partial \vec n (\ell)}{\rm d} \ell.
\end{equation}
  \end{lemma}

  \begin{proof}  
  The proof involves again Stokes formula, now on the domain 
  $$
  \mathcal C(\upkappa_0, \upkappa_1)=\Upsilon_{\eps, i} (\varrho_\eps, \upkappa_1)\setminus  
 \Upsilon_{\eps, i} (\varrho_\eps, \upkappa_0).
  $$
 It follows from assumption \eqref{kappacite} that  
 $$
 \overline{  \mathcal C(\upkappa_0, \upkappa_1)} \cap \partial \D^2(\varrho_\eps)=\emptyset, 
$$
 so that 
 $$\partial \mathcal C(\upkappa_0, \upkappa_1)=  \partial  \Upsilon_{\eps, i} (\varrho_\eps, \upkappa_1)\cup
\partial  \Upsilon_{\eps, i} (\varrho_\eps, \upkappa_0).
$$
We multiply the equation \eqref{elipes}  by $\displaystyle{ \frac{u_\eps-\upsigma_i} {\vert u_\eps-\upsigma_i \vert }}$ which is well defined   on $  \mathcal C(\upkappa_0, \upkappa_1)$ and integrate by parts.  Since, on $\Gamma_{\eps, i}(\varrho_\eps, \upkappa)$, we have 
$$
\frac{\partial u_\eps}{\partial  \vec n} \cdot \frac{u_\eps-\upsigma_i} {\vert u_\eps-\upsigma_i \vert }=\frac{\partial (u_\eps-\upsigma_i)}{\partial  \vec  n} \cdot \frac{u_\eps-\upsigma_i} {\vert u_\eps-\upsigma_i \vert }= \frac{\partial \vert u_\eps-\upsigma_i \vert }{\partial \vec n}, 
$$
whereas on  $\mathcal C(\upkappa_0, \upkappa_1)$, we have 
\begin{equation*}
\begin{aligned}
\nabla u_\eps\cdot \nabla \left( \frac{u_\eps-\upsigma_i} {\vert u_\eps-\upsigma_i \vert }\right)&=
\nabla( u_\eps-\upsigma_i) \cdot \nabla\left (\frac{u_\eps-\upsigma_i} {\vert u_\eps-\upsigma_i \vert }\right) \\
&=\frac{1}{\vert u_\eps-\upsigma_i \vert} \vert \nabla( u_\eps-\upsigma_i)  \vert^2 +
\left[\nabla( u_\eps-\upsigma_i) \cdot( u_\eps-\upsigma_i)\right] \cdot \nabla (\frac{1}{\vert u_\eps-\upsigma_i\vert})\\
&=\frac{1}{\vert u_\eps-\upsigma_i \vert} \left[\vert \nabla( u_\eps-\upsigma_i)  \vert^2-\left \vert \nabla \vert u_\eps-\upsigma_i \vert \right \vert ^2\right], 
\end{aligned}
\end{equation*}
integration by parts  thus   yields
 \begin{equation}
  \label{stokounette}
  \begin{aligned}
  \int_{\Gamma_{\eps, i}(\varrho_\eps, \upkappa_1)} \frac{\partial \vert u_\eps-\upsigma_i \vert }{\partial \vec n}-
   \int_{\Gamma_{\eps, i}(\varrho_\eps, \upkappa_0)}  \frac{\partial \vert u_\eps-\upsigma_i \vert }{\partial \vec n} 
  &=
  \int_{ \mathcal C(\upkappa_0, \upkappa_1)} \frac{1}{\vert u-\upsigma_i \vert  }\left[\vert \nabla u_\eps \vert^2  -  
   \left \vert \nabla \vert u_\eps-\upsigma_i \vert \right \vert ^2\right]\\
+  & \int_{\mathcal C(\upkappa_0, \upkappa_1)} \eps^{-2} \nabla_u V(u_\eps)\cdot \frac{(u_\eps-\upsigma_i)}{\vert u-\upsigma_i \vert}.
  \end{aligned}
  \end{equation}
 Since 
 $$\displaystyle{\vert \nabla u_\eps \vert^2  -  
   \left \vert \nabla \vert u_\eps-\upsigma_i \vert \right \vert ^2}=\vert \nabla (u_\eps-\upsigma_i) \vert^2  -  
   \left \vert \nabla \vert u_\eps-\upsigma_i \vert \right \vert ^2\geq 0,$$  
   it follows that the r.h.s of inequality \eqref{stokounette} is positive.   Hence, we deduce \eqref{sting}.

  \end{proof}
  
   \begin{lemma} Assume that $0<\eps \leq 1$  and  that $u_\eps$ is a solution to \eqref{elipes} which  satisfies  \eqref{kappacite} and \eqref{abritel}. Then, there exits a constant $\rm C(M)>0$ depending only on $V$ and $M$ such that have 
  \label{denada}
  \begin{equation}
  \label{denada1}
     0\leq  \eps \int_{\Gamma_{\eps, i}(\varrho_\eps, \upkappa)} \frac{\partial \vert u_\eps-\upsigma_i \vert }{\partial \vec n(\ell)} {\rm d } \ell 
   \leq  
   C(M) \int_{\D^2(\varrho_\eps)}\frac{V(u)}{\eps} {\rm d} x   \leq C (M)\mathbb V(u_\eps, \D^2(\frac34)),  
     \end{equation}
     where, for a point $\ell \in \Gamma_\eps$, $\vec n (\ell)$  denotes the unit vector  perpendicular to $\Gamma_\eps$  and oriented in the direction  which  increases $\vert u-\upsigma_i \vert$. 
  \end{lemma}
  \begin{proof}
  By Lemma \ref{toubon}, there exists a number 
  $\displaystyle{\tilde \upmu_\eps \in [\frac{\upmu_0}{4}, \frac{\upmu_0}{2}]}$ such that 
  \begin{equation}
  \label{totobene33}
 \eps  \int_{\Gamma_{\eps, i}(\varrho, \tilde \upmu_\eps)} \frac{\partial \vert u_\eps-\upsigma_i \vert }{\partial \vec n(\ell)} {\rm d } \ell 
 \leq \eps \int_{\Gamma_{\eps, i}(\varrho, \tilde \upmu_\eps)} \vert \nabla u_\eps \vert  {\rm d } \ell
  \leq  \frac{1}{\upmu_0 } \E_\eps(u, \Upxi (u_\eps,  \varrho_\eps))
    \end{equation} 
 On the level set $\Upxi (u_\eps,  \varrho_\eps)$, we may however bound point-wise the energy in terms of the potential, as stated in Lemma \ref{bornepote}, inequality \eqref{portos2}. This yields by integration
  $$
 \E_\eps(u, \Upxi (u_\eps, \varrho_\eps)) \leq {\rm C_T} (M) \mathbb V(u_\eps, \Upxi  (u_\eps, \varrho_\eps) ). 
 $$
Combining the two previous inequalities, we obtain
 \begin{equation}
 \label{eau}
  \eps  \int_{\Gamma_{\eps, i}(\varrho, \tilde \upmu_\eps)} \frac{\partial \vert u_\eps-\upsigma_i \vert }{\partial \vec n(\ell)} {\rm d } \ell 
\leq \frac{\rm C_T(M)} {\upmu_0}  \mathbb V(u_\eps, \Upxi_\eps (u, \varrho_\eps)).
 \end{equation}
 On the other hand, we invoke  to Lemma \ref{repro} with the  $\upkappa_1=\tilde \upmu_\eps$ and $\upkappa_0=\upkappa$ to deduce  that
$$
 \int_{\Gamma_{\eps, i}(\varrho,  \upkappa)} \frac{\partial \vert u_\eps-\upsigma_i \vert }{\partial \vec n(\ell)} {\rm d } \ell  \leq \int_{\Gamma_{\eps, i}(\varrho, \tilde \upmu_\eps)} \frac{\partial \vert u_\eps-\upsigma_i \vert }{\partial \vec n(\ell)} {\rm d } \ell, 
 $$
 which together with \eqref{eau} leads  to   the desired result \eqref{denada1}. 
  \end{proof}
  
  \begin{proof}[Proof of Proposition \ref{kappacity} completed] We go back to Lemma \ref{eliot}  and   combine \eqref{klisko} with \eqref{denada1}:  This yields   
 the  desired inequality  \eqref{nonodes}.
      \end{proof}

      \subsection{Bounding  the total energy by the  integral of the potential}
      The main result of the present paragraph is the following result:

   \begin{proposition}
 \label{bolero}
  Let $u_\eps$ be a solution of \eqref{elipes} on $\D^2$ and let ${\rm L}>0$ be given and assume that 
   \begin{equation}
 \label{lajonie}
 \Vert u_\eps \Vert_{L^\infty (\D^2(\frac 45))}\,  \leq  {\rm L}.
   \end{equation}
  There exists some constant   ${\rm K_{\rm pot}}(\rm L)$ depending only on $V$ and $\rm L$ such that
 \begin{equation}
 \label{dupondt}
 \int_{\D^2(\frac 12)} e_\eps(u_\eps)(x) {\rm d}x \leq  {\rm K_{\rm pot}}({\rm L}) \left[ 
   \int_{\D^2(\frac34)} \frac{V(u_\eps)}{\eps} {\rm d} x  +\eps    \int_{ \D^2 \setminus \D^2(\frac 34) } e_\eps(u_\eps){\rm d} x. 
     \right].
 \end{equation}
 \end{proposition}  
 In  the context of the present paper, the main contribution of the r.h.s  of inequality \eqref{dupondt}  is given by the potential terms, so that Proposition \ref{bolero}  yields an estimate of the energy by the integral of  potential, provided  the later is sufficiently small, according to assumption \eqref{lajonie}.

 \smallskip
Before turning to the proof of Proposition \ref{bolero}, we observe, as a preliminary remark, that the result of proposition \ref{bolero} is, at first sight,  rather close to the result of Proposition  \ref{kappacity20}.  However, let us emphasize that  estimate  \eqref{nonodes} yields only an energy bound  only  for the domain where the  value of \emph{$u_\eps$ is close to one of the wells} whereas \eqref{dupondt} yield an estimate for the full domain $\D^2(1\slash2)$.

\smallskip
The first step in the proof of Proposition \ref{bolero} is:

\begin{lemma}
\label{bornepoting}
Let $\varrho_\eps \in [\frac 12, \frac34]$,   let $u_\eps$ be a solution of \eqref{elipes} on $\D^2$ and 
  assume that \eqref{lajonie} is satisfied.  We have, for some constant $\rm C_{\rm pot} ({\rm L}) >0$ depending only on the potential  $V$ and the value of  $\rm L$  such that
    \begin{equation}
  \label{nonodesing}
   \int_{\D^2(\varrho_\eps)} e_\eps(u_\eps)(x) {\rm d}x \leq  {\rm C_{\rm pot}}({\rm L}) \left[ 
   \int_{\D^2(\varrho_\eps)} \frac{V(u_\eps)}{\eps} {\rm d} x  +\frac{\eps}{4}    \int_{\partial \D^2(\varrho_\eps)} e_\eps(u_\eps){\rm d} \ell
     \right].
   \end{equation}
  \end{lemma}

\begin{proof}  We observe first  that 
\begin{equation}
\label{fruc}
\D^2(\varrho_\eps)=\Upxi(\varrho_\eps) \cup \Upsilon_\eps( \varrho_\eps, \frac{\upmu_0}{4}).
\end{equation}
In view  of Lemma \ref{bornepote},  we have
$$
 \int_{\Upxi(\varrho_\eps)} e_\eps(u_\eps){\rm d} x  \leq {\rm C}_{\rm T}({\rm L}) \int_{\Upxi(\varrho_\eps)} \frac{V(u_\eps)}{\eps} 
 {\rm d}x, 
$$
whereas Proposition \ref{kappacity20} yields
 \begin{equation*}
    \int_{\Upsilon_\eps (\varrho_\eps, \frac{\upmu_0}{4})} e_\eps(u_\eps)(x) {\rm d}x \leq  {\rm K_\Upsilon} (L)\left[ 
 \int_{\D^2(\varrho_\eps)} \frac{V(u_\eps)}{\eps} {\rm d} x  +\eps    \int_{\partial \D^2(\varrho_\eps)} e_\eps(u_\eps){\rm d} \ell
     \right].
   \end{equation*}
The proof  of \eqref{nonodesing}  then follows straightforwardly from our first observation  \eqref{fruc}. 
\end{proof}

\begin{proof}[Proof of Proposition \ref{bolero} completed]
As usual,  a mean-value argument allows us to 
choose  some radius $\displaystyle{\varrho_\eps \in [\frac 12, \frac 34]}$  such that 
\begin{equation}
\label{zebulon}
\int_{\partial \D^2(\varrho_\eps)}   e_\eps(u_\eps) {\rm d} \ell \leq 
8\int_{\D^2(\frac 34) \setminus \D^2(\frac 12)}  e_\eps(u_\eps) {\rm d} x.
\end{equation}
Combining with Lemma \ref{bornepoting}, we are led to
 \begin{equation}
  \label{gaston}
  \begin{aligned}
   \int_{\D^2(\frac 12)} e_\eps(u_\eps)(x) {\rm d}x &\leq \int_{\D^2(\varrho_\eps)} e_\eps(u_\eps)(x) {\rm d}x  \\
 &  \leq  {\rm C_{\rm pot}} \left[ 
   \int_{\D^2(\varrho_\eps)} \frac{V(u_\eps)}{\eps} {\rm d} x  +\frac{\eps}{4}    \int_{\partial \D^2(\varrho_\eps)} e_\eps(u_\eps){\rm d} \ell
     \right] \\
     &\leq  {\rm C_{\rm pot}} \left[ 
   \int_{\D^2(\frac 34)} \frac{V(u_\eps)}{\eps} {\rm d}x+  \eps    \int_{\D^2(\frac 34)} e_\eps(u_\eps){\rm d} \ell \right].
        \end{aligned}
   \end{equation}
The proof  of Proposition \ref{bolero} is hence  complete. 
 \end{proof}

 We will also invoke  the following variant of Proposition \ref{bolero}:
 
  \begin{proposition}
 \label{borneo}
  Let $u_\eps$ be a solution of \eqref{elipes} on $\D^2$,  let ${\rm M}>0$ be given and assume that \eqref{abritel} holds.
    There exists some constant   ${\rm C_{\rm pot}}(M)$ depending only on $V$ and $\rm M$ such that
 \begin{equation}
 \label{dupont}
 \int_{\D^2(\frac 12)} e_\eps(u_\eps)(x) {\rm d}x \leq  {\rm C_{\rm pot}}({\rm M}) \left[ 
   \int_{\D^2(\frac34)} \frac{V(u_\eps)}{\eps} {\rm d} x  +\eps    \int_{ \D^2 \setminus \D^2(\frac 12) } e_\eps(u_\eps){\rm d} x. 
     \right].
 \end{equation}
 \end{proposition}  

\begin{proof}  If $u_\eps$ satisfies \eqref{abritel}, then it follows from Lemma \ref{princours0}
 \begin{equation}
 \label{lajonie11}
 \Vert u_\eps \Vert_{L^\infty (\D^2(\frac 45))}\,  \leq  {\rm L_M}\equiv   {5 {\rm C}_{\rm unf}}{\rm M} +2 \sup \{ \vert \sigma \vert^2, \upsigma \in \Sigma \}.
   \end{equation}
Invoking Proposition \ref{bolero}, inequality \eqref{dupont} follows with
$$ {\rm C_{\rm pot}}({\rm M})= {\rm K_{\rm pot}}({\rm L}_{\rm M})= {\rm K_{\rm pot}}\left (   {5 {\rm C}_{\rm unf}}{\rm M} +2 \sup \{ \vert \sigma \vert^2, \upsigma \in \Sigma \}\right).
$$
\end{proof}

    In the course of the paper, we will invoke the  scaled versions of Proposition \ref{bolero} and \ref{borneo}. Given 
 $\varrho>\eps>0$ and $x_0 \in \Omega$, we consider a solution $u_\eps$ on $\Omega$ and assume it satisfies the bound \eqref{lajonie} or the bound 
 \begin{equation}
 \label{lajoiebis}
 \E(u_\eps,\D^2(x_0, \varrho)) \leq M\varrho, 
  \end{equation}
then,  thanks to the relations \eqref{scalingv}, we have the scaled version of \eqref{dupondt} or \eqref{dupont} respectively, namely
\begin{equation}
 \label{dupondtbis}
 \int_{\D^2(x_0, \frac \varrho 2)} e_\eps(u_\eps) {\rm d}x \leq  {\rm K_{\rm pot}}(L) \left[ 
   \int_{\D^2(x_0,\frac{3\varrho}{4})} \frac{V(u_\eps)}{\eps} {\rm d} x  +\frac{\eps}{\varrho}    
   \int_{ \D^2(x_0, \varrho) \setminus \D^2(x_0, \frac  {\varrho}2) } e_\eps(u_\eps){\rm d} x 
     \right],
 \end{equation}
and
\begin{equation}
 \label{dupontbis}
 \int_{\D^2(x_0, \frac \varrho2)} e_\eps(u_\eps) {\rm d}x \leq  {\rm C_{\rm pot}}(M) \left[ 
   \int_{\D^2(x_0,\frac {3\varrho}{4})} \frac{V(u_\eps)}{\eps} {\rm d} x  +\frac{\eps}{\varrho}    
   \int_{ \D^2(x_0, \varrho) \setminus \D^2(x_0, \frac\varrho2) } e_\eps(u_\eps){\rm d} x 
     \right].
 \end{equation}
 These relations lead to:
 
 \begin{proposition}
 \label{lheure} 
Let  $M_0>0$ and $\eps>0$  be given.   Let $u_\eps$ be a solution of \eqref{elipes} on $\Omega$ such that  $E_\eps (u_\eps) \leq M_0$, and  $x_0\in \Omega$ and $\varrho>\eps >0$ such that $\D^2(x_0, \varrho) \subset \Omega$. Then, we have 
\begin{equation*}
\label{demanger}
 \int_{\D^2(x_0, \frac \varrho 2)} e_\eps(u_\eps) {\rm d}x  \leq {\rm K}_V\left ({\rm dist}(x_0, \partial \Omega)\right) \left[ 
   \int_{\D^2(x_0,\frac{3\varrho}{4})} \frac{V(u_\eps)}{\eps} {\rm d} x  +\frac{\eps}{\varrho}    
   \int_{ \D^2(x_0, \varrho) \setminus \D^2(x_0, \frac  {\varrho}2) } e_\eps(u_\eps){\rm d} x, 
     \right].
\end{equation*}
where the constant  ${\rm K}_V\left ({\rm dist}(x_0, \partial \Omega)\right)$ depends only on $V$, $M_0$ and   ${\rm dist}(x_0, \partial \Omega)$, 
 \end{proposition}
 
\begin{proof}
Since $\D^2(x_0, \varrho) \subset \Omega$, we have 
$\displaystyle{{\rm dist}\left(\D^2(x_0, \frac{4\varrho}{5}), \partial \Omega)\right)\geq \frac{\varrho}{5}}$. It therefore follows from Lemma \ref{princours} that
$$
\Vert u \Vert_{L^\infty(\D^2(x_0, \frac{4\varrho}{5}))}
\leq  {\rm L}_0 \equiv \frac{20 {\rm C}_{\rm unf}M_0}{{\rm dist}(x_0, \partial \Omega)} +2 \sup \{ \vert \sigma \vert^2, \upsigma \in \Sigma \}.
$$
The conclusion then follows directly from \eqref{dupontbis} with the choice 
$ \displaystyle{{\rm K}_V\left ({\rm dist}(x_0, \partial \Omega)\right)={\rm K}_{\rm pot} ({\rm L}_0)}$.
\end{proof}
 
\subsection{Bounds energy by integrals on    external domains}
      Our next result paves the way for the proof of Theorem \ref{bordurer}. As there, we consider a open subset $\mathcal U$ of $\Omega$ and define $\mathcal U_\delta$ and $\mathcal V_\delta$ according to \eqref{Udelta}. 

  \begin{proposition}
  \label{pave}
let $u_\eps$ be a solution of \eqref{elipes} on $\Omega$, $\mathcal U$ be an open bounded subset of $\Omega$ and 
$1>\delta>\eps >1>0$ be given such that $\mathcal U_\delta \subset \Omega$. Assume that 
 \begin{equation}
 \label{carpediem}
 \int_{\mathcal V_\updelta} 
      e_{\eps}(u_{\eps}) \, {\rm d} x  \leq {\rm  K}_{\rm ext} (\mathcal  U, \delta), 
 \end{equation}
  where  ${\rm  K}_{\rm ext} (\mathcal  U, \delta)>0$ denotes some constant depending possibly on $\mathcal U$ and $\delta$. 
 Then, we have the bound,   for some constant  ${\rm C}_{\rm ext} (\mathcal  U, \delta)$  depending possibly on $\mathcal U$ and $\delta$
\begin{equation}
\label{borges}
 \int_{\mathcal U_{\frac{\delta}{4}} }e_\eps (u_\eps) {\rm d} x \leq {\rm C}_{\rm ext} (\mathcal  U, \delta)
  \left(
  \int_{\mathcal V_\updelta} 
      e_{\eps}(u_{\eps})  + \eps \int_{\mathcal U_\delta} e_\eps(u_\eps) {\rm d} x.
\right)
 \end{equation}
 \end{proposition}
 \begin{proof}  The proof combines Proposition \ref{borneo}, Proposition \ref{pascapit} with a standard covering by disks. 
We first bound the potential on the set $\mathcal U_{\frac{\delta}{2}}$  thanks to  of Proposition \ref{pascapit}, which yields 
 \begin{equation}
     \label{pascap360}
      \frac{1}{\eps} \int_{\mathcal U_{\frac{\updelta}{2}} }V(u_\eps) {\rm d} x \leq  C(U, \updelta)  \int_{\mathcal V_\updelta} e_\eps(u_\eps) {\rm d}x\leq  C(U, \updelta)  {\rm  K}_{\rm ext} (\mathcal  U, \delta). 
     \end{equation}
   In inequality \eqref{pascap360}, we have  assumed  that the bound \eqref{carpediem} is fullfilled for some constant  ${\rm  K}_{\rm ext} (\mathcal  U, \delta)$, which we choose now as
   \begin{equation}
   \label{thechoice} 
  {\rm  K}_{\rm ext} (\mathcal  U, \delta) = \frac{{\rm  K}_{\rm pot}(M_0) \delta}{8C(\mathcal U, \delta)}.
   \end{equation}
    Inequality  \eqref{pascap360} then yields 
    \begin{equation}
    \label{pascap361}
     \frac{1}{\eps} \int_{\mathcal U_{\frac{\updelta}{2}} }V(u_\eps) {\rm d} x  \leq  \frac{\delta}{8} K_{\rm pot}(M_0). 
    \end{equation}
This bound will allow us  to apply inequality \eqref{lajoiebis} on disks of radius $\frac{\delta}{8}$ covering $\mathcal U_{\frac{\delta}{4}}$. In this direction, we   claim that there exists  a finite collections of disks 
  $\displaystyle{\left\{\D^2\left(x_i, \frac{\delta}{8}\right)\right\}_{i \in I}}$  such that
  \begin{equation}
  \label{rez}
  \mathcal U_{\frac{\delta}{4}} \subset  \underset{i \in I} \cup \D^2\left(x_i, \frac{\delta}{8}\right)  {\rm \ and \ } x_i \in  \overline{\mathcal U_{\frac{\delta}{4}}}, {\rm \ for \ any \ } i \in I.
  \end{equation}
   Indeed, such a collections may be obtained invoking the  collection of  disks $\displaystyle{\left\{D^2\left(x, \frac{\delta}{8}\right)\right\}}$  with $x \in \overline{\mathcal U_{\frac{\delta}{4}}}$ and then extracting a finite subcover thanks to Lebesgue's Theorem. Notice that we also have 
 \begin{equation}
 \label{chausson}
  \underset{i \in I} \cup \D^2\left(x_i, \frac{\delta}{4}\right)  \subset \mathcal U_{ \frac{\delta}{2}}.
 \end{equation}
 On each of  the disks $\D^2\left(x_i, \frac{\delta}{4}\right)$, we  have, thanks to  \eqref{pascap361}
 $$
  \frac{1}{\eps} \int_{\D^2(x_i, {\frac{\delta}{4}} )}V(u_\eps) {\rm d} x  \leq  \frac{\delta}{8} K_{\rm pot}(M_0),
  $$
  so that we may apply the scaled version \eqref{dupontbis} of Proposition \ref{borneo}  on the disk 
  $\D^2(x_i, \frac{1}{4} \delta)$: 
   This  yields  the estimate
   \begin{equation*}
 \int_{\D^2(x_i, \frac 18\delta )} e_\eps(u_\eps)(x) {\rm d}x \leq  {\rm C_{\rm pot}} \left[ 
   \int_{\D^2(x_i,\frac3{16}\delta )} \frac{V(u_\eps)}{\eps} {\rm d} x  +\frac{\eps}{\delta}    
   \int_{ \D^2(x_i, \frac{\delta}{4})} e_\eps(u_\eps){\rm d} x 
     \right].
 \end{equation*}
 Adding these relations for $i \in I$ and invoking relations \eqref{rez} and \eqref{chausson}
 we are led to
 \begin{equation}
  \int_{\mathcal U_{\frac{\delta}{4}}} e_\eps(u_\eps)(x) {\rm d}x \leq \sharp (I) 
   {\rm C_{\rm pot}} \left[ 
   \int_{\mathcal U_{\frac{\delta}{2}}} \frac{V(u_\eps)}{\eps} {\rm d} x  +\frac{\eps}{\delta}  
    \int_{\mathcal U_{\frac{\delta}{2}} } e_\eps(u_\eps){\rm d} x 
     \right].
 \end{equation}
 Invoking again the first inequality in \eqref{pascap360} we may bound the potential term on the right hand side,  so that   we obtain 
$$
\int_{\mathcal U_{\frac{\delta}{4}}} e_\eps(u_\eps)(x) {\rm d}x \leq \sharp (I) 
   {\rm C_{\rm pot}} \left[  C(U, \updelta)  \int_{\mathcal V_\updelta} e_\eps(u_\eps) {\rm d}x
     +\frac{\eps}{\delta}  
    \int_{\mathcal U_{\frac{\delta}{2}} } e_\eps(u_\eps){\rm d} x 
     \right].
$$
This inequality finally leads to the conclusion \eqref{borges}.
 \end{proof}

 \section {Proof of the energy decreasing property}
 \label{challengedata}
  The purpose of this section is to provide a proof to  Proposition \ref{sindec}, which is a major step in the proofs of the main theorems of the paper.

  \subsection{An improved estimate of the energy on level sets}
  In this paragraph,   we consider again for given $0<\eps \leq 1$ a solution  $u_\eps: \D^2 \to \R^k$  to \eqref{elipes} and  specify the result of Proposition \ref{kappacity} for special choices of $\upkappa$ and $\varrho_\eps$. More precisely, we choose
  \begin{equation}
  \label{cabrovski}
  \varrho_\eps=\mathfrak r_\eps  {\rm \ and \ } \upkappa_\eps=C_{\rm bd} \sqrt{\E_\eps(u_\eps)}, 
  \end{equation}
   where $\frac 3 4 \leq \mathfrak r_\eps \leq 1$ is  the radius introduced in  
  subsection \ref{radamel}, Lemma \ref{moyenne}  for   the choice $\displaystyle{r_1=1, r_0=\frac 34}$ and where the constant  ${\rm K}_{\rm bd}$ is choosen as 
  \begin{equation}
  \label{Cbd}
  C_{\rm bd}=\sup \left\{ 2{C_{\rm unf}}, \sqrt{\frac{1}{16\sqrt{\lambda_0}}}\right\},
  \end{equation}
   $C_{\rm unf}$ being  the constant provided in Lemma \ref{valli}. With this choice, we have 
   \begin{equation}
   \label{minos}
   \upkappa_\eps^2 \geq \frac{1}{16\sqrt{\lambda_0}}E_\eps(u_\eps), 
   \end{equation}
so that the bound \eqref{camembert} is  satisfied for $\upkappa=\upkappa_\eps$. 
  We notice, in view of \eqref{bornuni2},  the definition \eqref{cabrovski} of $\upkappa_\eps$ and the definition  \eqref{Cbd}  of $C_{\rm bd}$, that there exists some element $\upsigmam \in \Sigma$ such that 
  \begin{equation}
      \label{bornuni20}
     \vert u(\ell)-\upsigmam \vert \leq  2{\rm C}_{\rm unf} \sqrt{{\E}_\eps(u_\eps, \D^2)})\leq \upkappa_\eps,  \  \   {\rm \ for \ all \ } \ell \in 
     \S^1(\tilde {\mathfrak r}_\eps),
      \end{equation}
      so that condition \eqref{kappacite} is automatically fullfilled in view of our choice  our choices of parameters.
   The main result of this subsection is  the following:

  \begin{proposition}  
  \label{pasmain}
Assume that $0<\eps \leq 1$ and that $u_\eps$ is a solution of \eqref{elipes} on $ \D^2$.   There exists a constant ${\rm K}_{\Upsilon}>0$ such 
 \begin{equation}
 \label{smalto}
  \int_{\Upsilon_\eps( \mathfrak r_\eps, \upkappa_\eps)} e_\eps(u_\eps)(x) {\rm d}x \leq  {\rm K}_{\Upsilon} \left[ \left(\int_{\D^2} e_\eps(u_\eps)(x) {\rm d}x\right)^{\frac32} 
  +
  \eps \int_{\D^2} e_\eps(u_\eps)(x) {\rm d}x\right]. 
  \end{equation}
  \end{proposition}    
   
  \begin{proof} Notice first that   the  result \eqref{smalto} is non trivial only when the energy is small, otherwise it is obvious, for a suitable choice of constant. We introduce therefore  the smallness condition  on the energy
    \begin{equation}
  \label{smallo}
  \int_{\D^2} e_\eps(u_\eps){\rm d}x \leq \upnu_1\equiv  \frac{\upmu_0^2}{4C_{\rm bd}^2}, 
  \end{equation}
    and distinguish two cases.
   
   \medskip
  \noindent
   {\bf Case 1}: {\it   Inequality \eqref{smallo} does not hold, that is} $\E(u_\eps) \geq \upnu_1$. In this case  \eqref{smalto} is straightforwardly satisfied, provided we choose the constant $K_{\Upsilon}$ sufficiently large so that
      $$ {\rm K}_{\Upsilon} \geq \frac{1}{\sqrt{\upnu_1}}. 
       $$
     Indeed, we obtain, since \eqref{smallo} is not  satisfied, 
      \begin{equation}
      \begin{aligned}
      {\rm K} _{\Upsilon}  \left(\int_{\D^2} e_\eps(u_\eps)(x) {\rm d}x\right)^{\frac32} &\geq  {\rm K}_{\Upsilon} (\upnu_1)^{\frac 12} \int_{\D^2} e_\eps(u_\eps)(x) {\rm d}x \\
      &\geq \int_{\D^2} e_\eps(u_\eps)(x) {\rm d}x  
      \geq  \int_{\Upsilon_\eps( \mathfrak r_\eps, \upkappa_\eps)} e_\eps(u_\eps)(x) {\rm d}x.
      \end{aligned}
      \end{equation}

      \medskip
     \noindent
   {\bf Case 2}: {\it Inequality \eqref{smallo} does  hold}.     Since assumption \eqref{kappacite} is satisfied  for $\varrho_\eps=\mathfrak r_\eps$  thanks to \eqref{bornuni20}, we are in position to apply Proposition \ref{kappacity}.   It yields 
       \begin{equation}
  \label{nonodesf}
   \int_{\Upsilon_\eps( \mathfrak r_\eps, \upkappa_\eps)} e_\eps(u_\eps)(x) {\rm d}x \leq  {\rm C}_{\Upsilon} (\upnu_1)\left[ 
 \upkappa_\eps  \int_{\D^2(\mathfrak r_\eps)} \frac{V(u_\eps)}{\eps} {\rm d} x  +\eps    \int_{\partial \D^2(\mathfrak  r_\eps)} e_\eps(u_\eps){\rm d} \ell
     \right].
        \end{equation}
  We choose the constant ${\rm K}_{\Upsilon}$ so that 
  $${\rm K}_{\Upsilon} \geq  \sup \{  {\rm C}_{\Upsilon}(\upnu_1)C_{\rm bd}, \frac{1}{\sqrt{\upnu_1}},  1\}.  $$
   Inequality  \eqref{smalto}   then follows directly  from \eqref{nonodesf} in view of the definition  $\upkappa_\eps=C_{\rm bd} \sqrt{\E_\eps(u_\eps)}$ of $\upkappa_\eps$ and the fact that, by definition of the energy,  we have the point-wise inequality
   $\displaystyle{\frac{V(u_\eps)}{\eps} \leq e_\eps(u_\eps)}$. 
   \end{proof}

     At this stage,  we have already derived an inequality very close to \eqref{bien}, namely  inequality \eqref{smalto} of Proposition \ref{pasmain}. However it holds only on a domain where points on which the value of $\vert u_\eps-\upsigma_i \vert $ is large in some suitable sense have been removed. To go further and obtain an estimate on a full disk, we invoke improved estimates on the potential $V$ which are derived in the next subsection.     
     
     \subsection{Improved potential estimates}
   
   \begin{proposition}  
   \label{propsac}  Assume that $0<\eps \leq 1$ and that $u_\eps$ is a solution of \eqref{elipes} on $ \D^2$.
   There exists a constant $C_{\rm V}>0$  such that
   \begin{equation}
   \label{propsac0}
   \frac{1}{\eps} \int_{\D^2(\frac{5}{8})}V(u_\eps) {\rm d} x \leq   C_{\rm V} \left[ 
  \left(\int_{\D^2} e_\eps(u_\eps)(x) {\rm d}x\right)^{\frac32} 
  +
  \eps \int_{\D^2} e_\eps(u_\eps)(x) {\rm d}x
\right]. 
 \end{equation}  
 \end{proposition}
   \begin{proof}   The proof  combines the energy estimates  of Proposition \ref{pasmain},   the avering argument of Lemma  \ref{remoyen} together with the Pohozaev type potential estimate provided in Proposition \ref{pascap0}. 
   
   \smallskip
    We first apply Proposition \ref{remoyen} with the choice $\varrho=\mathfrak r_\eps$ and $\upkappa=\upkappa_\eps$, where  $\mathfrak r_\eps$ and $\upkappa_\eps$ have been defined in  \eqref{cabrovski}. Since in view of definitions \eqref{cabrovski}, \eqref{Cbd} and \eqref{minos}  the lower-bound \eqref{camembert} is verified for $\upkappa_\eps$, we may invoke Proposition \ref{remoyen} to assert that there exists some radius $\displaystyle{\uptau_\eps \in [\mathfrak r_\eps, \varrho]}$ such that 
   \begin{equation*}
   \label{rugby}
     \int_{\S^1(\uptau_\eps)}e_\eps(u_\eps){\rm  d } \ell  \leq   \frac{1}{\varrho_\eps-\frac{11}{16}} \, \E_\eps(u_\eps, \Upsilon( \uptau_\eps,  \upkappa_\eps))
     \leq   16 \, \E_\eps(u, \Upsilon_\eps(\tilde {\uptau}_\eps,  \upkappa_\eps)).
  \end{equation*}
Invoking inequality \eqref{smalto} of  Proposition   \ref{pasmain}, are led to 
  \begin{equation}
   \label{rugby2}
     \int_{\S^1(\uptau_\eps)}e_\eps(u_\eps ){\rm  d } \ell  \leq   16{\rm K}_{\Upsilon} \left[ \left(\int_{\D^2} e_\eps(u_\eps)(x) {\rm d}x\right)^{\frac32} 
  +
  \eps \int_{\D^2} e_\eps(u_\eps)(x) {\rm d}x\right]. 
 \end{equation}
On the other hand, thanks to Proposition \ref{pascap0}, we have    
   \begin{equation}
     \label{pascabul0}
      \frac{1}{\eps} \int_{\D^2( \uptau_\eps)} V(u_\eps){\rm d}x \leq  2\uptau_\eps \int_{\S^1(\uptau_\eps)} e_\eps(u_\eps) {\rm d}\ell\leq  
      2 \int_{\S^1(\uptau_\eps)} e_\eps(u_\eps) {\rm d}\ell.
       \end{equation}
Combining \eqref{rugby2} and \eqref{pascabul0} with the fact that $\uptau_\eps \geq \frac 58$, we derive \eqref{propsac0} with
$$C_{\rm V}=32 {\rm K}_\Upsilon.$$
The proof is  hence complete.    
   \end{proof}

\subsection{Proof of Proposition \ref{sindec} completed}

We introduce   first a new radius $\displaystyle{\tilde{\mathfrak r}_\eps\in [\frac {9}{16},\frac 5 8]}$ corresponding  to  the intermediate  radius defined in Lemma \ref{moyenne}  for   the choice  $\displaystyle{r_1=\frac{9}{16}, r_0=\frac7 8}$  so that it satisfies 
\begin{equation}
\label{igloo}
\int_{\S^1(\tilde {\mathfrak r}_\eps)}e_\eps(u){\rm  d } \ell  \leq 16 \, \E_\eps(u, \D^2(\frac 5 8)). 
\end{equation}
It follows  as above from Lemma \ref{valli}  that there exists   some  element $  \upsigman \in \Sigma$, possibly different from $\upsigmam$ defined in \eqref{bornuni20}, such that 
   \begin{equation}
      \label{bornuni30}
     \vert u(\ell)-\upsigman \vert \leq  4 {\rm C}_{\rm unf} \sqrt{ \E_\eps\left(u, \D^2(\frac5 8)\right)},  \  \   {\rm \ for \ all \ } \ell \in 
     \S^1( \tilde {\mathfrak r}_\eps).
      \end{equation}
In order to  apply Proposition \ref{bornepoting}, we introduce once more a smallness condition on the energy, namely
\begin{equation}
\label{smallitude}
\E_\eps( u_\eps) \leq \upeta_2\equiv \frac{ \upmu_0^2}{256  {\rm C}_{\rm unf}^2}.
\end{equation}
 We then  distinguish two cases:

\medskip
\noindent
{\bf Case 1:} {\it The smallness condition \eqref{smallitude} holds}. In this case, we have, in view of \eqref{bornuni30}
  \begin{equation*}
     \vert u(\ell)-\upsigman \vert \leq  4 {\rm C}_{\rm unf} \sqrt{\upeta_2}=\frac{\upmu_0}{4},  \  \   {\rm \ for \ all \ } \ell \in 
     \S^1( \tilde {\mathfrak r}_\eps), 
      \end{equation*}
 so that condition \eqref{kappacite} holds fo $\varrho_\eps=  \tilde {\mathfrak r}_\eps$
(with $\upsigmam$ replaced by $\upsigman$). We are therefore in position to apply Lemma \ref{bornepoting} on the disk 
$\D^2(\tilde{ \mathfrak r}_\eps)$,  which yields
 \begin{equation}
  \label{nonodesing0}
   \int_{\D^2( \tilde {\mathfrak r}_\eps)} e_\eps(u_\eps)(x) {\rm d}x \leq  {\rm C_{\rm pot}}({\rm L}_M )\left[ 
   \int_{\D^2(\frac 5 8)} \frac{V(u_\eps)}{\eps} {\rm d} x  +\eps    \int_{\partial \D^2(\tilde {\mathfrak r}_\eps)} e_\eps(u_\eps){\rm d} \ell
     \right], 
   \end{equation}
  where ${\rm L}_{M}$  is defined in \eqref{lajonie11}, so that 
  $\displaystyle{\Vert u_\eps \Vert_{L^\infty (\D^2(\frac 45))}\,  \leq  {\rm L_M}}$.
  Invoking Proposition \ref{propsac}  and inequality \eqref{igloo} we are hence led to 
\begin{equation*}
\begin{aligned}
\int_{\D^2( \tilde {\mathfrak r}_\eps)} e_\eps(u_\eps)(x) {\rm d}x  & \leq   {\rm C_{\rm pot}}(({\rm L}_{M}) C_{\rm V} 
  \left(\int_{\D^2} e_\eps(u_\eps)(x) {\rm d}x\right)^{\frac32} \\
 & + {\rm C_{\rm pot}}({\rm L}_{M}) \left( {\rm C}_{\rm V}  + 16  \right)
  \eps \int_{\D^2} e_\eps(u_\eps)(x) {\rm d}x,  
\end{aligned}
\end{equation*}
 which yields \eqref{bien}, fore a suitable choice of the constant ${\rm C}_{\rm dec}$.

 \bigskip
\noindent
{\bf Case 2:} {\it The smallness condition \eqref{smallitude}  \emph{does not } holds}. In this case, inequality \eqref{bien} is straightforwardly fullfilled, provided we choose
$${\rm C}_{\rm dec} \geq  \upeta_2^{-\frac 12}.$$
The proof is  hence complete in both cases.
\qed
\section{  Proof of the Clearing-out theorem}
\label{solde}
The purpose of this section  is to provide the proof of the clearing-out property stated in Theorem \ref{clearingoutth}.
a major step being   the uniform bound \eqref{benkon}.  We first introduce a very weak form of the clearing-out theorem.
\subsection{A very  weak form of the clearing-out}
The following result is classical in the field (see e.g. \cite{ilmanen, BBH}).

\begin{proposition}
\label{brioche}
   Let $u_\eps$ be a solution of \eqref{elipes} on $\D^2$ with $0< \eps \leq 4$. There exists a constant $\upeta_3>0$ such that if 
$\displaystyle{ \E_\eps (u) \leq \upeta_3 \eps}$, 
 then we have, for  some $\upsigma \in \Sigma$, the bound 
  $$\displaystyle{
  \vert u_\eps(0)-\upsigma  \vert  \leq \frac{\upmu_0}{2}.
  }$$
\end{proposition}
\begin{remark}
{\rm   In the scalar case, Lemma \ref{glupsitude} combined with the monotonicity formula for the energy directly yields 
the proof of Theorem \ref{clearingoutth}.
}
\end{remark}

\begin{proof} Assume that the bound $\displaystyle{ \E_\eps (u) \leq \upeta_3 \eps}$ holds, for some constant $\upeta_2$ to be determined later. 
Imposing first $\upeta_2\leq 1$, it follows from Proposition \ref{princours}  and Proposition \ref{classic} that there exists a constant $C_0>0$ depending only on $V$ such that 
\begin{equation*}
 \vert \nabla u_\eps (x) \vert \leq \frac{C_0}{\eps}   {\rm \ and \ } \vert u_\eps(x) \vert \leq C_0 ,  {\rm \ for \ } x \in \D^2(\frac{7}{8}).
\end{equation*}
Since the potential $V$ is smooth, and hence its gradient is bounded on the disc $\B^k(C_0)$, we deduce that there exists a constant $C_1$ such that 
\begin{equation}
\label{bbq}
\vert \nabla  V(u_\eps) (x) \vert \leq \frac{C_1}{\eps}    {\rm \ for \ } x \in \D^2(\frac{7}{8}).
\end{equation}
Since $\E_\eps (u_\eps) \leq \upeta_2 \eps$, we deduce from the definition of the energy that 
\begin{equation}
\label{lucien}
\int_{\D^2 (\frac{7}{8})} V(u_\eps(x)) {\rm d} x \leq \int_{\D^2} V(u_\eps(x)) {\rm d} x \leq \upeta_3 \eps^2. 
\end{equation}
We claim that 
\begin{equation}
\label{ouidada}
V(u_\eps(x)) \leq \upalpha_0 {\rm \ for \  any  \   } x \in \D^2(\frac{3}{4}).
\end{equation}
Indeed, assume by contradiction  that there exists  some $x_0 \in \D^2(\frac{3}{4})$ such that 
$\displaystyle{V(u(x_0))>\upalpha}$.  Invoking the gradient bound \eqref{bbq}, we deduce that 
$$
V(u_\eps(x) )\geq \frac{\upalpha_0}{2}  {\rm \ for \ }   x\in \D^2\left(x_0, \frac {\upalpha_0 \eps}{2C_1}\right).
$$
Without loss of generality, we may assume that  $C_1$ is chosen sufficiently large so that   
$\displaystyle{\frac {4\upalpha_0}{2C_1}\leq  \frac 18}$ and hence 
$\displaystyle{ \D^2\left(x_0, \frac {\upalpha_0 \eps}{2C_1}\right)\subset \D^2(\frac 78)}$. Integrating  \eqref{ouidada}, on the disk  
$\displaystyle{\D^2\left(x_0, \frac {\upalpha_0 \eps}{2C_1}\right)}$, we are led to 
\begin{equation*}
 \int_{\D^2(\frac 78)}  V(u_\eps(x)) {\rm d} x  \geq  \int_{\D^2(x_0, \frac {\upalpha_0 \eps}{2C_1})}  V(u_\eps(x)) {\rm d} x \geq 
\pi  \frac{\upalpha_0^3}{8C_1^2} \eps^2.
\end{equation*}
This yields a contradiction with \eqref{lucien}, provided we impose  the upper bound on $\upeta_3$ given by
\begin{equation}
\label{imposition}
\upeta_3\leq \pi  \frac{\upalpha_0^3}{8C_1^2} \eps^2, 
\end{equation}
and  established the claim \eqref{ouidada}.   To complete the proof, we may invoke  Lemma \ref{watson}   and the continuity of the map $u_\eps$ to   asserts that 
there exists some $\upsigma \in \Sigma$ such that
\begin{equation}
\label{presque}
 \vert u_\eps(x)-\upsigma \vert \leq \upmu_0  {\rm \ for \ any \ } x \in \D^2(\frac{3}{4}).
 \end{equation}
This yields almost estimate \eqref{benkon}, except that we   still have to replace $\upmu_0$ by  $\upmu_0 \slash 2$ on the right-hand side of \eqref{presque}. In  order to improve the constant, we merely rely on the  same  type of argument. Arguing as above by contradiction, let us assume that there exists a point $x_1\in \D^3(3\slash 4)$ such that 
\begin{equation}
\label{claimobil}
\vert u_\eps(x_1)-\upsigma \vert > \frac{\upmu_0}{2} {\rm \ and \ hence  \ } V(u_\eps(x_1))> \frac{\lambda_0 \,  \upmu_0^2}{16}, 
\end{equation}  
the second inequality in \eqref{claimobil} being a consequence of the  second statement in  Lemma \ref{watson}.  Invoking again  the gradient bound \eqref{bbq}, we deduce that 
$$V(u_\eps(x) )\geq \frac{\lambda_0 \,  \upmu_0^2}{32}, {\rm \ for \ }   x\in \D^2\left(x_1, \frac {\lambda_0 \upmu_0^2\eps}{32C_1}\right). $$
Integrating the previous inequality, we obtain 
\begin{equation*}
 \int_{\D^2(\frac 78)}  V(u_\eps(x)) {\rm d} x  \geq  \int_{\D^2(x_0, \frac {\upalpha_0 \eps}{2C_1})}  V(u_\eps(x)) {\rm d} x 
 \geq \pi  \frac{ \lambda_0^3 \upmu_0^5}{32768 C_1} \eps^2, 
\end{equation*}
 a contradiction with  \eqref{lucien}, provided we impose that $\upeta_2$  is sufficiently small.
\end{proof}

\subsection {Confinement near a well of the potential}   
Our next result is the main step in the proof of Theorem \ref{clearingoutth}. It shows that, if the energy is sufficiently small near, that $0$ takes values  inside a well of the potential.
 \begin{proposition}
 \label{benko0}
Let  $0<\eps \leq 1$ and  $u_\eps$  be a solution of 
  \eqref{elipes}  on $\D^2$. There exists a constant $\upeta_2>0$ such that 
  if 
 \begin{equation}
 \label{benkito}
 E_\eps(u_\eps, \D^2) \leq  \upeta_2
 \end{equation}
  then, we have, for  some $\upsigma \in \Sigma$, the bound 
  $\displaystyle{
  \vert u_\eps(0)-\upsigma  \vert  \leq \frac{\upmu_0}{2}.
  }$
 \end{proposition}

 The proof   of the Proposition \ref{benko0} relies  on   inequality  \eqref{bienscale}  of Proposition \ref{sindec}, a standard scaling argument combined  with an iteration procedure. 
 
 \medskip
 \noindent
 {\it   Step 1: A scaled version of inequality \eqref{bienscale}}. Set, for   $0<r\leq 1$, 
 $\displaystyle{\E_\eps(r)= \E_\eps \left(u_\eps, \D^2(r)\right)}$, and assume that
 \begin{equation}
  \label{atletico}
  E_\eps (r) \geq  \frac{\eps^2}{r}. 
  \end{equation}
   Then, we have 
   \begin{equation}
  \label{bienscale2}
 \E_\eps(\frac r 2)  \leq  2\Cdec
 \frac{  {\E_\eps (r)\ }^{\frac 32}
}{\sqrt{r}}, 
   {\rm \  \   \ provided  \ }  r\geq \eps. 
    \end{equation}
Indeed, scaling inequality \eqref{bienscale}, we obtain 
  \begin{equation}
  \label{bienscale3}
 \E_\eps(\frac r 2)  \leq  \Cdec
  \left[ \frac{1}{\sqrt{r}}
    {\E_\eps (r)\ }^{\frac 32}+ 
  \frac{ \eps}{r} \E_\eps (r)
    \right],
     {\rm  \ provided  \ }  r\geq \eps, 
    \end{equation}
 which yields \eqref{bienscale2}.
 
 \medskip
 \noindent
 {\it Step 2: The iteration procedure}.   
     We  consider the  sequence $(r_n)_{n \in \N}$ of decreasing  radii $r_n$  defined as   $\displaystyle{r_n=\frac{1}{2^n}}$,  for $n \in \N$,  and set
     $\displaystyle{\E^\eps_n=   \E_\eps(r_n)= \E_\eps(\frac {1}{2^n})}$, dropping the superscript in case this  induces  no ambiguity.  We  introduce the number
     \begin{equation}
     \label{benkita}
     n_\eps=\sup\left \{ n \in \N,   {\rm \ such \ that \ }   \E^\eps_n \geq   2^n \eps^2 {\rm \  and \ } r_n=\frac{1}{2^n} \geq \eps \right\}.
     \end{equation}
 If we impose that  $\upeta_2\leq 1$, then condition \eqref{benkito} implies that $\E_\eps (u_\eps)\leq 1$, so that $0$ belongs to the set of the r.h.s of \eqref{benkita}, which is hence not empty. On the other hand,    since $2^n$ tends to infinity as $n$ tends to infinity, and since the sequence $(\E_n)_{n \in \N}$ is non-increasing, hence  bounded by $\E^\eps_0$,   the set of the r.h.s of \eqref{benkita} is a finite set of sequential number  and the number $n_\eps$ is hence  a well-defined integer. In view of the definition of $n_\eps$, inequality \eqref{atletico} is straighforwardly satisfied for every $r_n< r_{n_\eps}$.  We deduce therefore from Step 1  and the definition of $r_n$ that we have   the inequality
 \begin{equation}
 \label{ratp}
 \E_{n+1} \leq 2{\sqrt 2}^n\Cdec \left(\E_{n} \right)^{\frac 32},   {\rm \ for \  } n=0, \ldots n_\eps.
 \end{equation}
  We introduce,  for $n\in \N$,  the number   $A_n= -\log E_n $.  Inequality  \eqref{ratp} for $E_n$ is turned into  the inequality for $A_n$ given by 
     \begin{equation}
     \label{turnlog}
     A_{n+1} \geq  \frac 3 2 A_n -\frac{(\log 2)}{2} \, n -\log (2\Cdec),   \   {\rm \ for \  } n=0, \ldots n_\eps.
     \end{equation}
   In order to study   the sequence $(A_n)_{n \in \N}$,  we will invoke  the next elementary result. 
   
   \begin{lemma} 
   \label{suites}
    Let $n_\star \in \N^*$,  $(a_n)_{n \in \N}$  and $(f_n)_{n \in \N}$ be two  sequences of numbers such that
   \begin{equation} 
   \label{suitineq}
   a_{n+1} \geq {\rm c}_0 \, a_n  -f_n, {\rm \ for \ all \ } n \in \N, n \leq  n_\star,  
   \end{equation}
    where ${\rm c}_0>1$  represents a  given constant.    Then we have  the inequality,
    \begin{equation}
    \label{suiteconc}
    a_n\geq {\rm c}_0^n \left(a_0 -\underset{k=0}{\overset {n}\sum  }\frac{1}{{\rm c}_0^{k+1}} f_k\right)  {\rm \ for \ }
    n \in \N^*n \leq n_\star.
       \end{equation}
  \end{lemma}
  We postpone the proof of Lemma \ref{suites} and complete first the proof of Proposition \ref{benko0}.
  
\medskip
\noindent
{\it Step 3: Choice of $\upeta_2$ and energy decay estimates}. 
We apply Lemma \ref{suites} to the sequences 
  $(A_n)_{n \in \N}$ and  $(f_n)_{n \in \N}$ with $\displaystyle{f_n= \frac{(\log 2)}{2} \, n+\log (2\Cdec)}$, for any $n \in \N$, so that inequality \eqref{suitineq} is satisfied with  $\displaystyle{{\rm c}_0=\frac 32}$ and $n_\star=n_\eps$.  Inequality \eqref{suiteconc} then yields, for $n=0, \ldots n_\eps,$
\begin{equation}
\label{meray0}
\begin{aligned}
A_n =-\log E_n &\geq   \left(\frac32 \right)^n \left[ 
\log  \left( \frac{1}{E_\eps(u_\eps) }\right)
- \gamma_0 \right] \\
&\geq \left(\frac32 \right)^n \left[ 
  -\log \upeta_2
- \gamma_0 \right]. \\
\end{aligned}
\end{equation}
Here we have used, for the second inequality,  assumption \eqref{benkito} and  we have set
$$\gamma_0  =\underset{k=0}{\overset {\infty}\sum} \left(\frac 23\right)^{k+1} (\frac{(\log 2)}{2} \, k+\log (2\Cdec) )<+\infty.$$
  We impose a first constraint on the constant $\upeta_2$,   namely   we impose
  \begin{equation}
  \label{firstconstra}
   \upeta_2 \leq \exp\left[- (1+\gamma_0)\right] {\rm \ so \ that \ }    -\log \upeta_2 \geq 1+\gamma_0,  
   \end{equation}
 It follows that inequality \eqref{meray0} yields, provided inequality \eqref{benkito} holds,   
 \begin{equation}
 \label{meray}
 \E_n \leq  \exp \left[-\left(\frac32 \right)^n\right]   \   {\rm \ for \  } n=0, \ldots n_\eps-1.  
  \end{equation}

\medskip
\noindent
{\it Step 4:  Estimating  $n_\eps$ and $r_{n_\eps}$}. 
It follows from \eqref{meray}  and the definition of $n_\eps$ that 
$$
\eps^2= \exp (2 \log \eps ) \leq r_n E_n= 2^{-n}  \E_n \leq \exp \left[-\left(\frac32 \right)^n-n \log 2 \right]  \   {\rm \ for \  } n=0, \ldots n_\eps, 
$$
so that we are led to  the inequality
$$\left(\frac32 \right)^{n_\eps} +n_\eps \log 2\leq  2 \vert \log \eps \vert, $$
 and  hence  
$$\left(\frac32 \right)^{n_\eps} \leq  2 \vert \log \eps \vert.$$
Taking the logarithm of both sides,  we obtain the  bound for $n_\eps$
$$ n_\eps \leq\frac{  \log (2\vert \log \eps \vert)}{\log 3-\log 2}.   $$
It  yields a lower bound for $r_{n_\eps}$, namely
\begin{equation}
\label{rateau}
\begin{aligned}
r_{n_\eps}=2^{-n_\eps}=\exp (-(\log 2)\,  n_\eps)
& \geq 
\exp \left(-\log (2\vert \log \eps \vert)\frac{\log 2} {\log 3- \log 2}\right)   \\
& \geq \exp \left(-\upgamma_1\log (2\vert \log \eps \vert)\right) \\
&\geq  (2  \vert \log \eps \vert)^{-\upgamma_1} .
\end{aligned}
\end{equation}
Here we have set 
$$\upgamma_1=\frac{\log 2} {\log 3- \log 2},   {\rm \ so \ that \ } 1\leq \upgamma_1\leq 2.$$
We notice that $\displaystyle{(2  \vert \log \eps \vert)^{-\upgamma_1} \underset{\eps\to 0}\ll \eps}$, so that there exists some $0<\eps_1\leq 1$ such that  
 
\begin{equation}
\label{rne}
r_{n_\eps} \geq   2\, \eps,  {\rm \ provided  \ }  0<\eps\leq \eps_1.
\end{equation}
Going back to  the definition of $n_\eps$, we deduce from \eqref{rne} and \eqref{rateau} that
\begin{equation}
\label{enfin}
\E_{n_\eps+1}^\eps \leq \eps^2  r_{n_\eps+1}^{-1}=2^{n_\eps+1} \eps^2
\leq 8\vert \log \eps \vert^{\upgamma_1} \eps^2,  {\rm \ if   \ }  0<\eps\leq \eps_1.
\end{equation}

\medskip
\noindent
{\it Step 5: change of scale}. We consider the scaled map
 $\tilde u_\eps$ and the scaled parameter $\tilde \eps \geq \eps$ defined by
 $$
  \tilde u_\eps(x)= u_\eps (r_{n_\eps+1} x),  {\rm \ for \ } x \in \D^2, {\rm \ 	and \ the \ scaled \ parameter \ } 
  \tilde \eps=r_{n_\eps+1}^{-1}\eps =2 r_{n_\eps}^{-1} \eps.
  $$
   Turning back to \eqref{scalingv},  we  are led to 
  the identity, for the energy
  $$
  \E_{ \tilde \eps}(\tilde u_{ \eps}) = r_{n_\eps+1}^{-1} \Eps (u_\eps, \D^2(r_{n_\eps+1}))=
  r_{n_\eps+1}^{-1}\E_{n_\eps+1}^\eps,$$
  so that,   in view of \eqref{rateau} and \eqref{enfin}, we have  
 \begin{equation}
 \label{chap}
 \left\{
 \begin{aligned}
  \E_{ \tilde \eps}(\tilde u_{ \eps})&\leq 16 \vert \log \eps \vert)^{2\upgamma_1}\,  \eps^2,  {\rm \ if \ } \eps\leq \eps_1, {\rm \ and \ } \\
    \frac{\E_{ \tilde \eps}(\tilde u_{ \eps})}{\tilde \eps}&=\frac{\Eps (u_\eps, \D^2(r_{n_\eps+1}))}{\eps}\leq  8\vert \log \eps \vert^{\upgamma_1} \eps,
  {\rm \ if \ } \eps\leq \eps_1.
  \end{aligned}
  \right.
  \end{equation}
  Since the map $s\to  \vert \log s  \vert^{\upgamma_1} s$ is decreasing on the interval $(0, e^{-\upgamma_1})$, assuming that the constant $\upeta_2$ is choosen to be sufficiently small, there exists a unique  number $\eps_2 \in(0, e^{-\upgamma_1})$, such that  
  \begin{equation}
  \label{epson}
8\vert   \log \eps_1  \vert^{\upgamma_1}\eps_1 \leq \upeta_2  {\rm \ and \ } \eps_2 \leq \eps_1.
  \end{equation}

  \medskip
  \noindent
  {\it Step 6: Proof of Proposition \ref{benko0} completed}.  We conclude invoking the weak clearing-out property stated in Proposition \ref{brioche}. For that purpose,  we distinguish two cases:
  
  \smallskip
  \noindent
 {\bf Case 1: $0<\eps \leq \eps_2$}. It follows in this case  from the definition \eqref{epson} of $\eps_2$ and the first inequality in \eqref{chap} that 
 \begin{equation}
 \label{kakuta}
 \left\{
 \begin{aligned}
 \E_{ \tilde \eps}(\tilde u_{ \eps})&\leq \upeta_2\,  \tilde \eps  {\rm \ and \ } \\
 \tilde \eps &\leq 1. 
  \end{aligned}
 \right. 
 \end{equation}
 In view of \eqref{kakuta}, we are hence  
  in position to apply Proposition \ref{brioche} to the map $\tilde u_\eps$  with parameter $\tilde \eps$:  Hence  there exists some point $\upsigma \in \Sigma$ such that 
 $$\vert \tilde u_\eps(0)-\upsigma \vert \leq \frac{\upmu_0}{2}. $$
  since  $u_\eps(0)= \tilde u_\eps(0)$ the conclusion of Proposition \ref{benko0} follows. 
  
   \medskip
  \noindent
 {\bf Case 2: $1 \geq \eps > \eps_2$}.  Here we apply directly Proposition \ref{brioche} to $u_\eps$. Besides \eqref{firstconstra} we impose the additional condition 
$\displaystyle{\upeta_2 \leq {\upeta_3}{\eps_2}}$ on $\upeta_3$,  so that we finally may choose the constant $\upeta_2$ as 
 \begin{equation}
 \label{seconstraint}
 \upeta_2=\inf\{ {\upeta_3}\, {\eps_2}, \exp [-(1+\upgamma_1), 1] \}. 
 \end{equation}
  With this choice, we have, in view of assumption \eqref{benkito}, for $\eps \geq \eps_2$, 
 $$ \Eps(u_\eps) \leq \upeta_1 \leq  {\upeta_2}{\eps_2}  \leq  \upeta_2 \eps.  $$
  Hence $u_\eps$ fullfills the assumptions of Proposition \ref{brioche}, so that its conclusion yields  again  the existence of an element $\upsigma \in \Sigma$ such that  $\displaystyle{\vert  u_\eps(0)-\upsigma \vert \leq \frac{\upmu_0}{2}.} $

  \medskip
  In both cases, we have hence established  the conclusion of Proposition \ref{benko0} so that the proof is complete. 
  \qed
  
  
  \smallskip
  In the course of the proof, we have used Lemma \ref{suites}, which has not been proved yet.
   
        \begin{proof} [Proof of Lemma \ref{suites}]
     We introduce, inspired by the method of variation of constant,  the sequence $(b_n)_{n \in \N}$ defined by $a_n={\rm c}_0^n  \, b_n$, for any $n \in \N$.  Substituting into \eqref{suitineq},  we obtain
     $${\rm c}_0^{k+1} b_{k+1} \geq {\rm c}_0^{k+1}b_k-f_k, {\rm \ for \ all \ } k \in \{0, \ldots, n_\star\},  $$
      so that 
      $$b_{k+1}-b_k \geq - \frac{1}{{\rm c}_0^{k+1}} f_k,  {\rm \ for \ all \ } k \in \{0, \ldots, n_\star\}.$$
     Let $n\in \N$, $n\leq n_\star$. Summing these relations for $k=0$ to $k=n-1$, we are led to
      $$b_n \geq b_0 -\underset{k=0}{\overset {n}\sum  }\frac{1}{{\rm c}_0^{k+1}} f_k=
   	a_0 -\underset{k=0}{\overset {n}\sum  }\frac{1}{{\rm c}_0^{k+1}} f_k,$$
   which, in view of the definition of $b_n$,  yields the desired conclusion \eqref{suiteconc}. 
     \end{proof}
     
     A direct consequence of Proposition \ref{benko0} is the following:
     
     \begin{corollary}
      \label{benko0}
Let  $0<\eps \leq 1$ and  $u_\eps$  be a solution of 
  \eqref{elipes}.  Set  $\displaystyle{ \upeta_1=\inf \{\frac{1}{8} \upeta_2, \frac{1}{8} \upeta_3\}}$  and assume that 
   \begin{equation}
 \label{benkitoni}
 E_\eps(u_\eps, \D^2) \leq 2\upeta_1.
 \end{equation}
  then, there exists   some $\upsigma \in \Sigma$ such that 
  \begin{equation}
  \label{}
  \vert u_\eps(x)-\upsigma  \vert  \leq \frac{\upmu_0}{2}, {\rm \ for \ any \ } x \in \D^2\left( \frac{3}{4}\right).
  \end{equation}
     \end{corollary} 
      \begin{proof}
  Let $\displaystyle{x_0 \in \D^2(\frac 3 4)}$  be an arbitrary point.  We consider the  scaled parameter $\tilde \eps= 4\eps$   and the scaled and translated map $\tilde u_\eps$ defined on $\D^2$ by 
     $$\tilde u_{\tilde \eps} (x)=u_\eps \left(x_0+ \frac{1}{4} x\right) {\rm \ for \ every \ } x \in \D^2, $$
      so that
      \begin{equation}
      \label{arnault}
       \E_{\tilde \eps}(\tilde u_{\tilde \eps})=4 \E_\eps\left (u_\eps, \D^2 (x_0, \frac 14)\right)
       \leq 4  \E_\eps (u_\eps) \leq 
      8  \upeta_1 \leq  \upeta_2, 
     \end{equation}
  where we have used assumption \eqref{benkitoni} and  the definition of $\upnu_1$ for the last  inequality. As above, we distinguish two cases. 
  
  \medskip
  \noindent
  {\bf Case 1: $\eps\leq \frac{1}{4}$}. In this case $\tilde\eps \leq 1$, so that, in view of \eqref{arnault}, we are in position to apply Proposition \ref{brioche}: It yields  an element $\upsigma_{x_0} \in \Sigma$, depending possibly on the point $x_0$,   such that  
  $$\displaystyle{\vert  \tilde u_\eps(0)-\upsigma_{x_0} \vert \leq \frac{\upmu_0}{2}.} $$
    Since 
  $\tilde u_{\tilde \eps} (0)=u_\eps (x_0)$, we conclude that
 \begin{equation}
 \label{upsigmax0}
  \vert u_\eps(x_0)-\upsigma_{x_0} \vert \leq \frac{\upmu_0}{2}. 
  \end{equation}
  Since inequality \eqref{upsigmax0} holds for \emph{any point} $x_0\in \D^2(3 \slash 4)$, a continuity argument shows  that the point $\upsigma_{x_0}$ does not depend on $x_0$, so that the proof  of Proposition \ref{benko0} is complete in Case 1. 
  
   \medskip
  \noindent
  {\bf Case  2: $1\geq \eps\geq \frac{1}{4}$}. In this case $1 \leq \tilde \eps \leq 4$. In view of the definition of $\upeta_1$, we have  $16 \upeta_1\leq \upeta_3$, 
  It then follows from assumption \eqref{benkitoni} that 
        \begin{equation}
      \label{arnault2}
       \E_{\tilde \eps}(\tilde u_\eps)=4 \E_\eps\left (u_\eps, \D^2 \left(x_0, \frac 14\right)\right)
       \leq 4  \E_\eps (u_\eps) \leq 
      8 \upeta_1
      \leq   {\upeta_3}\leq \upeta_3 \tilde \eps. 
     \end{equation}
Hence, we are once more  in position to apply Proposition \ref{brioche}, so that there exists  an element $\upsigma_{x_0} \in \Sigma$, depending possibly on the point $x_0$  such that  
  $\displaystyle{\vert  \tilde u_\eps(0)-\upsigma_{x_0} \vert \leq \frac{\upmu_0}{2}.} $  Since 
  $\tilde u_\eps (0)=u_\eps (x_0)$, we conclude that
  $$
  \vert u_\eps(x_0)-\upsigma_{x_0} \vert \leq \frac{\upmu_0}{2}. 
  $$
The proof of Corollary \ref{benko0} is hence complete. 

     \end{proof}
     \subsection{Proof of Theorem \ref{clearingoutth} completed}
      We have determined so far the value of $\upeta_1$ in Corollary \ref{benko0}, which as matter of fact provides the proof of \eqref{benkon}. 
    The only remaining unproved assertion is  the energy estimate \eqref{engie}, which we establish next.  For that purpose, we notice that the restriction of the map $u_\eps$ to the smaller disk $\D^2(3\slash4)$ takes values into one of the wells, so that the functionals behaves there as a convex functional.

    The proof is  parallel and actually much easier then our earlier energy estimate.  We first invoke Lemma \ref{moyenne} with 
    $\displaystyle{r_1=\frac 34}$  and  $\displaystyle{r_0= \frac{5}{8}}$: This yields a radius $\displaystyle{\mathfrak r_\eps\in [\frac 58, \frac 34]}$  and an element $\upsigma \in \Sigma$ such that  
    \begin{equation}
    \label{ret}
     \int_{\S^1(\mathfrak r_\eps)}e_\eps(u_\eps){\rm  d } \ell  \leq 8 \, \E_\eps(u, \D^2)) {\rm \ and \ }
         \int_{\S^1(\mathfrak r_\eps) }\vert u_\eps- \upsigma\vert \vert \nabla  u_\eps\vert \leq 16\sqrt{\lambda_0^{-1}} \E_\eps(u_\eps, \D^2)).
     \end{equation}
 We multiply the equation by $(u_\eps-\upsigma)$ and integrate on the disk $\D^2(\mathfrak r_\eps)$ which yields, as in \eqref{stokounette}   
     \begin{equation}
  \label{stokom}
  \begin{aligned}
  \int_{\D^2(\mathfrak r_\eps)} \eps \vert \nabla u_\eps \vert^2 +\eps^{-1} \nabla_u V(u_\eps)\cdot (u_\eps-\upsigma)
  &= \eps  \int_{\S^1(\mathfrak r_\eps)}  \frac{\partial u_\eps}{\partial r} \cdot (u_\eps-\upsigma). 
 \end{aligned}
  \end{equation} 
     We deduce from \eqref{ret}   that
     \begin{equation}
     \label{reti}
     \int_{\S^1(\mathfrak r_\eps)}  \frac{\partial u_\eps}{\partial r} \cdot (u_\eps-\upsigma)
  \leq     \int_{\S^1(\mathfrak r_\eps) }\vert u_\eps- \upsigma\vert \vert \nabla  u_\eps \vert
  \leq 16\sqrt{\lambda_0^{-1}} \E_\eps(u_\eps, \D^2)). 
     \end{equation}
 We use next  the fact that, in view of assertion \eqref{benkon}, we have  $\displaystyle{\vert u_\eps-\upsigma  \vert   \leq \frac{\upmu_0}{2}}$ on the disk 
 $\D^2(\mathfrak r_\eps)$.  Arguing  as in \eqref{duglandmoins}, we have  the point-wise inequality
     \begin{equation}
     \label{retina}
     \eps \vert \nabla u_\eps \vert^2 +\eps^{-1} \nabla_u V(u_\eps)\cdot (u_\eps-\upsigma)\geq  \frac{\lambda_0} {2\lambda_{\rm max}} e_\eps(u). 
     \end{equation}
  Combining \eqref{stokom}    with \eqref{retina} and \eqref{reti},  we obtain
  $$
   \int_{\D^2(\mathfrak r_\eps)} e_\eps(u_\eps){\rm d}x \leq  {16 }{\lambda_0^{-\frac 32}}\lambda_{\rm max}\, \eps \E_\eps(u, \D^2)), 
  $$
 Which yields the energy estimate \eqref{engie}  choosing $\displaystyle{\Cnrg= {16}{\lambda_0^{-\frac 32}}\lambda_{\rm max}}$. The proof of Theorem \ref{clearingoutth} is hence complete. 

\bigskip

     \begin{center}{ \Large \bf
Part III:  Analysis of the limiting sets and measures}
\end{center}
\addcontentsline{toc}{section}{Part III: Analysis of the limiting sets and measures}


 \section{Properties of the concentration set $\mathfrak S_\star$}
 \label{sectionsix}
 The purpose of this section is to provide the proof of  assertion i) of  Theorem \ref{maintheo}.  We start with the proof of Theorem \ref{claire}, the clearing-out property for the measure $\upnu_\star$.
\subsection{Proof of Theorem \ref{claire}} 
Recall that $\upnu_\star$  is the weak limit of the measure $\upnu_{\eps_n}$ defined in \eqref{mesure} by $\upnu_\eps=e_\eps(u_\eps) \rd x$, so that 
$\E_\eps (u, \D^2(x_0, r))=\upnu_\eps (\D^2(x_0, r))=\upnu_\eps (\overline{\D^2(x_0, r)})$. Let  $x_0\in \Omega$ and $r>\rho>0$ be  such that $\D^2(x_0, r) \subset \Omega$. Since $\overline{\D^2(x_0, \rho)}$ is an closed set, we have,  by standard properties of weak convergence
\begin{equation}
\label{miseritude}
\underset{n \to +\infty} \limsup \,  \upnu_{\eps_n} (\D^2(x_0, \rho))  \leq \upnu_{\star} (\overline{\D^2(x_0, \rho)}) \leq \upnu_{\star} ({\D^2(x_0, r)})
\end{equation} 
Next, let  $x_0$ and $r>0$ are such that 
   $$\upnu_\star (\D^2(x_0, r))<\upeta_1\, r.$$
 It follows from \eqref{miseritude}  that, for  given $\rho < r$,  there exists some $n(\rho)\in \N$ such that, if $n \geq n(\rho)$,  then  we have 
 \begin{equation}
 \label{miserere}
 \upnu_{\eps_n} (\D^2(x_0, \rho)) \leq \frac{5}{4} \upeta_1 r. 
 \end{equation}
 We choose $\displaystyle{\rho=\frac{8r}{9}}$. We obtain, inserting in \eqref{miserere}, 
 \begin{equation}
 \label{nobis}
  \upnu_{\eps_n} (\D^2(x_0, \rho)) =\upnu_{\eps_n} (\D^2(x_0, \frac{8r}{9})) \leq \frac{5}{4}.\frac{8r}{9}  \upeta_1 =\frac{10}{9}\upeta_1<2\upeta_1. 
 \end{equation}

   Hence,  for sufficiently large $n$,  we are in position to apply Theorem \ref{clearingoutth}, so that
 \begin{equation}
 \label{peccata}
 \begin{aligned}
  \upnu_{\eps_n} \left( \D^2\left(x_0, \frac{ 5r}{9}\right)\right)= \upnu_{\eps_n} \left( \D^2\left(x_0, \frac{ 5\rho}{8}\right)\right)& \leq \Cnrg \, \frac{\eps_n}{r} E_{\eps_n} \left(u_{\eps_n}, \D^2\left(x_0, \rho\right)\right) \\
 & \leq  \frac{5}{4} \eps_n\upeta_1   \to 0 {\rm \ as \ } n \to + \infty.
\end{aligned}
\end{equation}
Letting $n \to +\infty$,  it follows that $\displaystyle{\upnu_\star \left(\D^2(x_0,\frac{r}{2})\right)=0}$ and the proof is complete. 
 \subsection{Elementary consequences of the clearing-out property}
We present here some simple consequences of the definition of $\mathfrak S_\star$,  as well as of the clearing out property stated in Proposition \ref{danes}.

  \begin{proposition} 
  \label{proptheo1}
  The set $\mathfrak S_\star$ is  a closed subset of $\Omega$.
 \end{proposition}
 \begin{proof} It suffices to prove that its complement, the set $\mathfrak U_\star=\Omega\setminus \mathfrak S_\star$ is an \emph{open  subset}  of $\Omega$. This property is actually a direct consequence of the clearing out  property stated in Theorem \ref{claire}. Indeed let $x_0$ be an arbitrary point in $\mathfrak U_\star$. It follows from the definition  \eqref{mathfrakSstar} of $\mathfrak S_\star$ that  $\theta_\star (x_0) <\upeta_1$,  so that there exists some radius 
 $r_0>0$ such that $\D^2(x_0, r_0) \subset \Omega$ and such that 
  $$ \upnu_\star (\D^2(x_0, r_0) ) <r_0\upeta_1.$$
  In view   of Theorem \ref{claire}, we  deduce  that
  $\displaystyle{ \upnu_\star (\D^2(x_0, \frac{r_0}{2}) )=0.}$
  Hence, for any point $\displaystyle{x\in \D^2(x_0, \frac{r_0}{4})}$, we  have $\theta_\star (x)=0$ and  therefore
  $$  \D^2(x_0, \frac{r_0}{4}) \subset \mathfrak U_\star. $$
Hence,  $\mathfrak U_\star$ is an open set.
 \end{proof}

 \begin{proposition}
 \label{danes}
 The set $\mathfrak S_\star$ has finite one-dimensional Hausdorff dimension. There exist a  constant ${\rm C_H} >0$ depending only on the potential $V$ such that 
 $$\mathcal H^1 (\mathfrak S_\star)\leq {\rm C_H} M_0.$$
 \end{proposition}
 
 \begin{proof} The proof relies on a standard covering argument. Let 
$0<\rho<\frac{1}{4}$ be given, and  consider the set 
$$\displaystyle{\Omega_\rho=\{ x \in \Omega, {\rm dist}(x, \partial \Omega) \geq \rho\}}.$$
 Next let $0<\delta<\rho\slash 4$ be given. Consider  the points $x_i$ on a uniform  square lattice of $\R^2$, with nearest neighbour at  distance 
$\displaystyle{\frac \delta 2}$.   We obtain for a subfamily $I$    a standard finite  covering of $\Omega_\rho$  of size $\delta$, that is such that
$$
\Omega_\rho \subseteq \underset{ j \in I}\cup \D^2\left(x_j,\delta\right) \,  \  \text{and} \ 
\D^2\left(x_i,\frac{\delta}{2}\right) \cap \D^2\left(x_j,\frac{\delta}{2}\right) = \emptyset
\text{ for } i\neq j \in I.
$$
We introduce  then the set of indices 
$$\displaystyle{
I_\delta = \left\{ i  \in I, \text{ such \ that \  } \D^2(x_i,\delta)\cap \mathfrak S_\star \neq \emptyset\right\}, }$$
 so that 
given any arbitrary index  $i\in I_\delta,$ there exists a point  $y_i \in \mathfrak S_\star \cap
\D^2(x_i,\delta)$. It follows from the definition of $\mathfrak S_\star$ that 
\begin{equation}
\label{davy}
\theta_\star (y_i) \geq \upeta_1.
\end{equation}
We claim that, for any $0<r\leq \delta$, we have 
\begin{equation}
\label{pret}
\upnu_\star(\D^2(y_i, r )) \geq  \upeta_1 \,  r.
\end{equation}
Indeed, if \eqref{pret} were not true, then  we would be in position to apply Theorem  \ref{claire}, which would imply  that 
$\displaystyle{\upnu_\star(B(y_i, \frac{\delta}{2} ))=0}$, and  hence that $\theta_\star (y_i)=0,$ a contradiction which \eqref{davy}. Inequality  \eqref{pret} is therefore established. Since $\D^2(y_i, \delta) \subset \D^2(x_i, 2\delta)$, we deduce from \eqref{pret} that 
\begin{equation}
\label{preterit}
\upnu_\star(\D^2(x_i, 2\delta )) \geq  \upeta_1 \,  \delta.
\end{equation}
Snce the points $x_i$  are on a uniform grid, we notice that  a given point $x \in \R^2$ belongs to at most 25 distinct  balls of the collection $\D^2(x_i, 2 \delta)$. We have therefore
 \begin{equation}
 \sharp (I_\delta) \upeta_1\delta \leq  {\underset {i \in I_\delta } \sum} \upnu_\star \left( \D^2(x_i, 2 \delta)\right) \leq  25\upnu_\star (\Omega)\leq 25M_0.
   \end{equation}
It follows therefore that
$$\sharp (I_\delta) \delta \leq \frac{25 M_0} {\upeta_1}.$$
Therefore, letting $\delta\to 0$, we deduce, as a consequence of the  definition of the  one-dimensional Hausdorff measure that 
$$
\mathcal H^1 (\mathfrak S_\star \cap\Omega_\rho) \leq \underset { \delta \to 0}\liminf \, 2\, \sharp (I_\delta) \delta  \leq 
 \frac{50M_0} {\upeta_1}.
$$
We conclude letting $\rho \to 0$, choosing $\displaystyle{{\rm C_H}=\frac{50}{\upeta_1}  }$. 
\end{proof}
\subsection{ Proof of Theorem \ref{bordurer}}
Theorem \ref{bordurer}  is a direct consequence of  Proposition \ref{pave} which has actually been taylored for this purpose.      
 Indeed,  since $\upnu_\star(\mathcal V_\updelta)=0$,  we have the convergence
  $$
  \int_{\mathcal V_\updelta} 
      e_{\eps_n}(u_{\eps_n}) \, {\rm d}\, x \to  0 {\rm \ as \ } n \to +\infty, 
  $$    
   so that condition \eqref{carpediem} is fullfilled for $\eps=\eps_n$ and the map $u_{\eps_n}$, provided $n$ is sufficiently large, say larger than some given value $n_0$. We are therefore in position to conclude, thanks to Proposition \ref{pave}, provided $n\geq n_0$ is sufficiently large,  that 
 \begin{equation*}
\label{borges2}
\begin{aligned}
 \int_{\mathcal U_{\frac{\delta}{4}} }e_{\eps_n} (u_{\eps_n}) {\rm d} x 
 &\leq {\rm C}_{\rm ext} (\mathcal  U, \delta)
  \left(
  \int_{\mathcal V_\updelta} 
      e_{\eps_n}(u_{\eps_n}){\rm d}x  + \eps_n \int_{\mathcal U_\delta} e_{\eps_n}(u_{\eps_n}) {\rm d} x 
\right)\\
& \leq   {\rm C}_{\rm ext} (\mathcal  U, \delta)
  \left(
  \int_{\mathcal V_\updelta} 
      e_{\eps_n}(u_{\eps_n}){\rm d}x  + \eps_n {\rm M}_0 \right).
\end{aligned}
 \end{equation*}
It follows that 
$$
 \int_{\mathcal U_{\frac{\delta}{4}} }e_{\eps_n} (u_{\eps_n}) {\rm d} x  \to 0 {\rm \  as \ } n \to +\infty,
$$ 
 so that the proof is complete. 
      \qed

\subsection{ Connectedness properties of  $\mathfrak S_\star$ }
 The purpose of the present  section is, among other things,  to provide the proof of Proposition \ref{connective}.  Given $r>0$ and $x_0\in \Omega$  such that $\D^2(x_0, 2r) \subset \Omega$, we consider the  closed  set 
 $$\mathfrak S_{\star, \varrho}=  \mathfrak S_{\star, \varrho}(x_0)\equiv\mathfrak S_\star \cap \overline{\D^2(x_0, \varrho)} {\rm \ for \  } \varrho \in [0, 2r). $$ 


 The proof of Proposition \ref{connective} relies on several intermediate properties we present next. 

 \begin{proposition} 
 \label{lesconti}
 Let $r>0$ and $x_0\in \Omega$ be as above.
  The closed set 
  \begin{equation}
  \label{theunion}
  \mathfrak Q_{\star, r}(x_0)=\mathfrak S_{\star, r}(x_0)\cup \S^2(x_0, r)
  \end{equation}
   is  a continuum, that is,   it is compact and  connected. 
 \end{proposition}

 \noindent
 {\it Proof}.  The proof of compactness  of $ \mathfrak Q_{\star, r}(x_0)$  is a straightforward consequence of Proposition \ref{proptheo1}, since both sets composing the union \eqref{theunion} are compact. The  proof of  connectedness of $ \mathfrak Q_{\star, r}(x_0)$ is more involved, and strongly  relies  on Theorem \ref{bordurer}, as we will see next.  In order to invoke Theorem \ref{bordurer},  a first step is to   approximate $\mathfrak S_{\star, r}$ by  sets $\mathfrak S_{\updelta, r}$ with a simpler structure.
 
 \medskip
 \noindent
 {\it Definition of the approximating sets $\mathfrak S_{\updelta, r}$}. These sets are defined  using a \emph{Besicovitch covering} of $\mathfrak S_{\star, r}$. 
 Let   
 $$\updelta_{x_0, r}={\rm dist}(\D^2(x_0, r), \partial \Omega)>0.$$
  For given $0<\updelta<\updelta_{x_0, r}$, we consider the covering of $\mathfrak S_{\star, r}$ by the collection of open  disks 
 $\displaystyle{\{\D^2(x_0, \updelta)\}_{x \in \mathfrak S_{\star, r}}}$, which is obviously a covering of $\mathfrak S_{\star, r}$, and actually a Besicovitch covering.  We may therefore invoke  Besicovitch covering theorem to  assert that there exists 
 a  universal constant $\mathfrak p$, depending only on the dimension $N=2$,  and  $\mathfrak p$ families of  points 
 $\{x_{i_1}\}_{i_1\in A_1}$,  $\{x_{i_2}\}_{i_2\in A_1}, \ldots,  \{x_{i_{\mathfrak  p}}\}_{i_{\mathfrak p}\in A_{\mathfrak p}}$, 
  such that $x_i \in \mathfrak S_{\star, r}(x_0)$, for any $i\in  A\equiv A_1\cup A_2 \ldots\cup A_{\mathfrak p}$, 
  \begin{equation}
  \label{labbe}
  \mathfrak S_{\star, r} \subset \mathfrak V_{\updelta, r}\equiv  \underset {\ell=1} {\overset {\mathfrak p}\cup} \, 
  \left( { \underset{i_\ell  \in A_\ell} \cup} \D^2(x_{i_\ell}, \updelta)
  \right) = \underset{i \in A } \cup \D^2(x_{i}, \updelta), 
  \end{equation}
   and such that the balls in each collection $\{\D^2(x_i, \updelta)\}_{i \in A_\ell}$ are disjoint, that is, for any 
   $\ell=1, \ldots, \mathfrak p$, we have 
   \begin{equation}
   \label{zikovich}
   \D^2(x_i, \updelta) \cap \D^2(x_j, \updelta)=\emptyset {\rm \ for \ } i\not=j {\rm \ with \ }  i, j \in A_\ell.
   \end{equation}
   As a consequence of the above constructions, a point $x \in \mathfrak   V_{\updelta, r}$, where $\mathfrak   V_{\updelta, r}$ is defined in \eqref{labbe}, belongs to at most $\mathfrak p$ distinct disks of the collection 
   $\{\overline{\D^2(x_{i}, \updelta)}\}_{i\in A}$.
   We define the set $\mathfrak S_{\updelta, r}$ as the closure  of the set $\mathfrak V_{\updelta, r}$ that is 
 $$ 
 \mathfrak S_{\updelta, r}\equiv \overline{\mathfrak V_{\updelta, r}}= \underset {\ell=1} {\overset {\mathfrak p}\cup} \, 
  { \underset{i_\ell  \in A_\ell} \cup}\overline{\D^2(x_{i_\ell}, \updelta)},  
  $$
     Notice that, by construction,  the total number $\sharp (A)$ of distinct disks  is finite. Actually, we have the bound
   \begin{equation}
   \label{sharpa}
    \sharp (A) \leq  \frac{4\mathfrak p {r^2}}{\updelta^2}. 
   \end{equation} 
   Indeed, since the famille of balls  $\{\D^2(x_{i_\ell}, \updelta)\}_{i\in A_\ell}$ are disjoint disks of radius $\updelta$ which are included in a ball of radius $2r$, we have 
   $$
   \sharp (A_\ell) \leq  \frac{ {4r^2}}{\updelta^2}  {\rm \ for \ } \ell=1, \ldots, \mathfrak p,
   $$
    so that \eqref{sharpa} follows by summation.
    
   We next consider  the set 
   $$
   \mathfrak Q_{\updelta, r}= \mathfrak S_{\updelta, r} \, \cup \S^2(x_0, r)
   $$
 and    its  distinct connected components  $\{\mathfrak T^k_{\updelta, r}\}_{_{k \in \mathcal J_\updelta }} $.  In view of the structure of $\mathfrak T_{\updelta, r}$,  which is an union of $\sharp (A)$ disks with a circle,  the total  number of connected components $\sharp {\mathcal J_\updelta}$ is  finite and actually bounded by $\sharp(A)+1$, hence the number on the right hand side of inequality \eqref{sharpa} plus one. As a matter of  fact,  we claim 
  
\begin{equation}
\label{fauconnier}
{\it    The  \ set \   } \mathfrak Q_{\updelta, r}   {\it  \ is  \ simply  \ connected, \ so  \ that \  }  \sharp(\mathcal J_\updelta)=1.
\end{equation}
   
   \medskip
   \noindent
 {\it Proof of the claim  \eqref{fauconnier}. }
   We assume by contradiction that $\mathfrak Q_{\updelta, r}$ has at least two distinct connected components and denote  by 
   $\mathfrak Q^1_{\updelta, r}$  the connected component which contains the circle $\S^1(x_0, r)$. Let 
   $\mathfrak Q^2_{\updelta, r}$ be a connected component distinct  from $\mathfrak Q^1_{\updelta, r}$, and set
   $$
   \upbeta\equiv  \inf\left\{ {\rm dist} (\mathfrak Q^2_{\updelta, r}, \mathfrak Q^j_{\updelta, r}),
    j \in \mathcal J_\updelta, j\not =2\right\}>0.
   $$
     We consider the open set 
     $$ \mathcal U=\left\{ x\in \R^2, {\rm dist }\left(x, \mathfrak Q^2_{\updelta, r}\right)<\frac{\upbeta}{4}\right\}
      \subset  \D^2(x_0, r) \setminus { \underset{ j \in \mathcal J_\updelta \setminus \{2\}}  \cup}\mathfrak Q^j_{\updelta, r}, 
      $$
      so that using the notation \eqref{Udelta}, we have
     $$
     \mathcal U_{\frac{\upbeta}{4}}=\left\{ x\in \R^2, {\rm dist }\left(x, \mathcal  U\right)<\frac{\upbeta}{4}\right\}  \subset 
      \D^2(x_0, r) \setminus { \underset{ j \in \mathcal J_\updelta \setminus \{2\}}  \cup}\mathfrak Q^j_{\updelta, r}
     $$
    and 
    \begin{equation}
    \label{udelta2}
    \mathcal V_{\frac{\upbeta}{4}}\equiv     \mathcal U_{\frac{\upbeta}{4}} \setminus \mathcal U \subset 
    \left\{ x \in \R^2, \frac{\upbeta}{4}\leq   {\rm dist }\left(x, \mathfrak Q^2_{\updelta, r}\right)\leq \frac{\upbeta}{2};
    \right\}
    \end{equation}
    combining \eqref{udelta2}  with the definition of $\upbeta$, we obtain
 \begin{equation}
 \label{udelta3}
 \mathcal V_{\frac{\upbeta}{4}} \cap \mathfrak S_{\star}=\emptyset  {\rm \ and \ }
 \upnu_\star\left(\mathcal V_{\frac{\upbeta}{4}}\right)=0.
 \end{equation}
     We are therefore in position to apply Theorem \ref{bordurer}  to assert that
   $ \displaystyle{\upnu_\star (\mathcal U)=0.}$
     However, since  by definition $\mathfrak Q^2_{\updelta, r}  \subset  \mathcal U $, it follows that 
     $\mathcal U \cap \mathfrak S_\star \not = \emptyset,$ so that $\upnu_\star(\mathcal U)>0$. We have hence reached a contradiction, which establishes the proposition.
       
 \begin{proof}[Proof of Proposition \ref{lesconti} completed]   It follows from the definition of $\mathfrak S_{\updelta, r}$ that
  $${\rm dist} (  \mathfrak Q_{\updelta, r} ,\mathfrak Q_{\star, r} ) \leq \updelta,  {\rm \ where \ } \mathfrak Q_{\star, r}=  \mathfrak S_{\updelta, r} \, \cup \S^2(x_0, r),$$
 so that  $\mathfrak Q_{\updelta, r}$ converges as $\updelta \to 0$ to $\mathfrak Q_{\star, r}$ in the Hausdorff metric.  Since for every $\updelta$, the set $\mathfrak S_{\updelta, r}$ is a continuum,    it then follows (see e.g. \cite{falconer}, Theorem 3.18)   that the Hausdorff limit $\mathfrak Q_{\star, r}$ is also a continuum and the proof is complete. 
 \end{proof}     
      
 We deduce as a consequence of Proposition \ref{lesconti}:
 
 \begin{corollary}  
 \label{lesconti2}
 The set $\mathfrak Q_{\star, r}$ is arcwise connected. 
 \end{corollary}    
   
  \begin{proof}  Indeed, any continuum with finite one-dimensional Hausdorff dimension is arcwise connected, see e.g \cite{falconer}, Lemma 3.12, p 34. 
  \end{proof}

    \subsubsection{Proof of Proposition \ref{connective}}
      Invoking Fubini's theorem together with a mean value argument, we may choose some radius 
    $r_0 \in [r, 2r)$ such that  the number of points in $\mathfrak S_\star \cap \partial \D^2(x_0, r_0)$ is finite, more precisely
    $$
 m_0\equiv    \sharp \left( \mathfrak S_\star \cap \partial \D^2(x_0, r_0)\right)  \leq  \frac{ {\rm C}_{\rm H}}{r} M_0, 
    $$
   where we have used estimate \eqref{herbert} of the $\mathcal H^1$ measure of $\mathfrak S_\star$.  We may hence write
   \begin{equation} 
   \label{vacherol}
   \mathfrak S_\star \cap \partial \D^2(x_0, r_0)=\{a_1, \ldots, a_{m_0}\}.
   \end{equation}
    Next, we claim  that for any point $y \in \mathfrak S_{\star, r_0}$, there exists a  continuous path $p: [0, 1] \mapsto   \mathfrak S_{\star, r_0}$ connecting  the point $y$  to one of the points $a_1, \ldots, a_{m_0}$, that is  such that 
   \begin{equation}
   \label{lebof}
   p(0)=y {\rm \ and \ }  p(1)\in \{a_1, \ldots, a_{m_0}\}.
   \end{equation}
   \noindent
  {\it Proof of the claim \eqref{lebof}}.
If $\vert y-x_0 \vert =r_0$, then $y \in \mathfrak S_\star \cap \partial \D^2(x_0, r_0)$, and it therefore suffices to choose $p(s)=y$, for all $s \in [0, 1]$.    Otherwise, since, in view of Corollary  \ref{lesconti2}  applied at $x_0$ with radius $r_0$, the set  $\mathfrak S_{\star, r_0}\cup  \partial \D^2(x_0, r_0) $ is path-connected, there exists a continuous path $ \tilde p: [0, 1]\to \mathfrak S_{\star, r_0}\cup  \partial \D^2(x_0, r_0)$ such that 
   $$\tilde p(0)=y  {\rm \ and \ }  \tilde  p(1) \in  \partial \D^2(x_0, r_0).$$
 By continuity, there exists some number $s_0 \in [0, 1]$  such that
 $$ \vert   \tilde p(s)  \vert <r_0, {\rm \ for \ } 0\leq s<s_0  {\rm \ and \ } \vert   \tilde p(s_0)  \vert=r_0.
 $$
 It follows that 
 $$\tilde p(s_0)  \in \mathfrak S_\star \cap \partial \D^2(x_0, r_0)=\{a_1, \ldots, a_{m_0}\}.$$
 We then set 
 $$ p(s)=\tilde p(s), {\rm \ for \ } 0\leq s<s_0,   {\rm \ and \ } p(s)=\tilde p(s_0),  {\rm \ for \ }  s_0 \leq s\leq 1,$$
 and verify that $p$ has the desired property, so that the proof of the claim is complete.

 \medskip
 \noindent
  {\it Proof of  Proposition \ref{connective} completed}.  It follows from the claim \eqref{lebof} that any point  $y \in \mathfrak S_{\star, r_0}$ is connected to one of the points $a_1, \ldots, a_{m_0}$ given in \eqref{vacherol}. Hence $\mathfrak S_{\star, r_0}$ has at most $m_0$ connected components and the proof is complete. 
\qed 

       
       \subsection{Rectifiability of $\mathfrak S_\star$}
       In this section,  we prove:
       
       \begin{theorem}
       \label{rectifiable}
        The set $\mathfrak S_\star$ is rectifiable. 
       \end{theorem}

       \begin{proof}  The result is actually an immediate consequence of Proposition \ref{lesconti} and the fact that any 1-dimensional continuum is rectifiable, a result due to  Wazewski and independently Besicovitch (see e.g \cite{falconer}, Theorem 3.12). Indeed, given any $x_0 \in \Omega$,  $r>0$  such that $\D^2(x_0,r)\subset \Omega$, the set $\mathfrak S_{\star, r} \cup  \S^2(x_0, r)$ is a continuum, hence  rectifiable in view of the result quoted above, and hence so is the set $\mathfrak S_{\star, \frac{r}{2}}$. Since rectifiability is a local property,  the conclusion follows.
      \end{proof}
      
      \subsection{Proof of Theorem \ref{maintheo} completed}
      \label{proofmain}
      All statements in Theorem \ref{maintheo} have been obtained so far. Indeed, assertions i) follows combining several result in Section \ref{sectionsix}, namely Proposition \ref{proptheo1}, Proposition \ref{danes}, Proposition \ref{lesconti}, Proposition \ref{connective}  and Theorem  \ref{rectifiable}.

      \subsection{On the tangent line at regular points  of $\mathfrak S_\star$}
      In this subsection, we provide the proof to Proposition \ref{tangentfort}. It relies on the following Lemma, which is actually a weaker statement:
      
      \begin{lemma}
      \label{ratus}
       Let $x_0$ be a regular point of $\mathfrak S_\star$ and   $\vec e_{x_0}$ be a unit tangent vector  to $\mathfrak S_\star$ at $x_0$. Given any $\uptheta>0$ there exists  a radius $R_{\rm cone}(\uptheta, x_0)$ such that 
 \begin{equation}
 \label{radius2}
\mathfrak S_\star \cap \left(  \D^2\left(x_0, \uptau \right) \setminus \D^2\left(x_0, \frac{\uptau}{2}\right)\right)
 \subset   
  \mathcal C_{\rm one}\left(x_0, \vec e_{x_0}, \uptheta  \right), 
  {\rm  \  for \   any \  }  0<\uptau \leq  R_{\rm cone}(\uptheta, x_0).
 \end{equation}
      \end{lemma}      
      
\begin{proof}  Since we have the inclusion 
 $$\mathcal C_{\rm one}\left(x_0, \vec e_{x_0}, \uptheta  \right)  \subset  \mathcal C_{\rm one}\left(x_0, \vec e_{x_0}, \uptheta'  \right)$$
for  $<0\leq \uptheta \leq \uptheta'$, it suffices to establish the statement for $\uptheta$ arbitrary small. 
For  a given regular point $x_0$ of $\mathfrak S_\star$, we may invoke the convergence \eqref{tangent}  to assert that  there exists some 
$r_1> 0$ such that for $0<\uptau \leq r_1$  we have 
 \begin{equation}
 \label{tangenti}
 \mathcal H^1 \left(\mathfrak S_\star \cap  \D^2\left(x_0, 2\uptau \right) 
   \setminus   \mathcal C_{\rm one}\left(x_0, \vec e_{x_0}, \frac{\uptheta}{2}  \right)\right)\leq   
 \frac{\theta \uptau}{8}.
 \end{equation}
 We set
 $$A(x_0, \uptau, \uptheta)=\left(  \mathfrak S_\star \cap   \D^2\left(x_0, \uptau \right)
 \right)
 \setminus  \left(   \mathcal C_{\rm one}\left(x_0, \vec e_{x_0}, \uptheta  \right)
 \cup 
 \D^2\left(x_0, \frac{\uptau}{2}\right) 
 \right), 
 $$
 and  have to prove that $A(x_0, \uptau, \uptheta)$  is empty, if $\tau$ is sufficiently small.   We assume by contradiction that 
 $A(x_0, \uptau, \uptheta)\not =\emptyset$  for small $\tau$, and will show that we obtain a contradiction 
We have, in view of the definition of $A(x_0, \uptau, \uptheta)$
\begin{equation}
\label{quesaisje}
A(x_0, \uptau, \uptheta) \cap  \mathcal C_{\rm one}\left(x_0, \vec e_{x_0}, \frac{\uptheta}{2}  \right)=\emptyset {\rm \ and  \ }
\mathcal H^1 \left( A(x_0, \uptau, \uptheta) \right)\leq   
 \frac{\theta \uptau}{8}. 
\end{equation}
 we notice that, if  $A(x_0, \uptau, \uptheta)$ is not empty, then we have
\begin{equation*}
\left\{
\begin{aligned}
{\rm dist } \left( A(x_0, \uptau, \uptheta),  \mathcal C_{\rm one}\left(x_0, \vec e_{x_0}, \frac{\uptheta}{2}  \right)
\right)  &\geq   \frac{\uptau} {2} \sin \left(  \frac{\uptheta}{2}\right) \\
{\rm dist } \left( A(x_0, \uptau, \theta),\partial \D^2(x_0, 2\uptau) \right)
&\geq \uptau, 
\end{aligned}
\right.
\end{equation*}
so that, if $\uptheta>0$ is sufficiently small
\begin{equation}
\label{distus}
{\rm dist } \left( A(x_0, \uptau, \uptheta), \mathcal C_{\rm one}\left(x_0, \vec e_{x_0}, \frac{\uptheta}{2}  \right) 
\cup \partial \D^2(x_0, 2\uptau) 
\right)  \geq  \frac{\uptau} {2} \sin \left(  \frac{\uptheta}{2}\right).
\end{equation}
Since, by assumption, the set  $A(x_0, \uptau, \uptheta)$ is not empty, there  exists some  point $x_1 \in A(x_0, \uptau, \uptheta)$.  Swe consider  the set $\mathfrak Q_{\star, 2\uptau}(x_0)\equiv\mathfrak S_\star \cup  \partial \D^2\left(x_0, 2\uptau \right)$ introduced in \eqref{theunion}. In view of Propostion \ref{lesconti} and Corollary \ref{lesconti2}, the set $\mathfrak Q_{\star, 2\uptau}(x_0)$ is path-connected: Hence,   there exists a continuous path $p$ joining $x_1$ to some  point $x_2\in \partial \D^2(x_0, 2\uptau)$  which stays inside $\mathfrak S_{\star, 2\uptau}(x_0)$. On the other hand, since $x_1\in \D^2(x_0, \uptau)$ the length $\mathcal H^1(p)$ of this path is larger than $\uptau$.  We claim that 
\begin{equation}
\label{grouinox}
p \, \cap  \mathcal C_{\rm one}\left(x_0, \vec e_{x_0}, \frac{\uptheta}{2}  \right) \not =\emptyset. 
\end{equation}
Otherwise, indeed, $p$ would be a path inside $\mathfrak S_\star \cap  \D^2\left(x_0, 2\uptau \right) 
   \setminus   \mathcal C_{\rm one}\left(x_0, \vec e_{x_0}, \frac{\uptheta}{2}  \right)$. Since its length is larger then $\uptau$, this would contradict \eqref{tangenti}.  Next, combining  \eqref{grouinox}  and \eqref{distus}, we obtain 
   $$\mathcal H^1\left (p\cap \mathcal C_{\rm one}\left(x_0, \vec e_{x_0}, \frac{\uptheta}{2}  \right)  \right) \geq  \frac{\uptau} {2} \sin \left(  \frac{\uptheta}{2}\right)\underset{\uptheta \to 0} \sim\frac{ \uptau \uptheta} {4}. 
   $$
   Since $p$ is a path inside $\mathfrak S_{\star, 2\uptau}(x_0)$ this contradicts \eqref{tangenti}, provided $\uptheta$ is chosen sufficiently small. This completes the proof of the Lemma, choosing $R_{\rm cone}(\uptheta, x_0)=r_1$. 
   \end{proof}      
      
      \begin{proof} [Proof of Proposition \ref{tangentfort} completed] Given $\uptau <R_1$,  we apply Lemma \ref{ratus},  the sequence of radii 
      $(\uptau_k)_{k \in \N}$ given by 
      $$\uptau_k=\frac{\uptau}{2^k}  {\rm \ for \  } k \in \N,  $$
      so that 
      \begin{equation*}
 \mathfrak S_\star \cap 
\left(  \D^2\left(x_0, \uptau_k \right) \setminus \D^2\left(x_0, \uptau_{k+1}\right) \right)
 \subset      \mathcal C_{\rm one}\left(x_0, \vec e_{x_0}, \uptheta  \right), 
 {\rm  \  for \   any \  }  k \in \N.
 \end{equation*}
We take the union of these sets  on theft hand side, we obtain
\begin{equation*}
\mathfrak S_\star \setminus \{x_0\} =\underset {k \in \N}  \cup\mathfrak S_\star \cap 
\left(  \D^2\left(x_0, \uptau_k \right) \setminus \D^2\left(x_0, \uptau_{k+1}\right) \right)
\subset \mathcal C_{\rm one}\left(x_0, \vec e_{x_0}, \uptheta  \right).
\end{equation*} 
This yields the result. 
      \end{proof}
   \section{Behavior near points in $\mathfrak S_\star \setminus \mathfrak E_\star$}
   \label{goodpoints}
   In this section, we analyze more precisely the behavior of the measures $\upzeta_\star$ and $\upmu_{\star, i, j}$ in  the vicinity of \emph{good} points, that is points $x_0$ in $\mathfrak S_\star \setminus \mathfrak E_\star$, in particular points having the Lebesgue property for the absolutely continuous part of the measure.  One of our main goals is to provide the proof to 
  Proposition \ref{discrepvec6} and Lemma \ref{starac}. The results in this section also pave the way to the proof of Theorem \ref{segmentus} provided in Section \ref{segmentino}.

\subsection{The limiting Hopf differential} 
    The Hopf differential  
   $$\omega_\eps\equiv \eps \left( \vert (u_\eps)_{x_1} \vert^2-\vert (u_\eps)_{x_2} \vert^2 -2i (u_\eps)_{x_1}\cdot (u_\eps)_{x_2}\right) $$
   defined in \eqref{hopfique} has turned out to be  a central tool in our analysis so far. We combine it in the present subsection with \emph{the rectifiability} properties and Proposition \ref{tangentfort} to derive new properties near good points.  Recall that we have defined  $\omega_\star$  in \eqref{schmil0} as
   \begin{equation*}
\omega_\star=\left( \upmu_{\star, 1, 1} -\upmu_{\star, 2, 2}\right) -2i\upmu_{\star, 1, 2}.
\end{equation*}
 So that, in view of the definition \eqref{boulga2} of the measures  $\upmu_{\star, i, j}$, we have
   \begin{equation}
   \label{gloubi}
   \omega_{\eps_n} \rightharpoonup  \omega_\star,  {\rm \ in \ the \ sense \ of \  measures \ on  \  } \Omega,  {\rm \ as \ } n \to + \infty.
   \end{equation}

   
   \subsection{The limiting differential relation for $\omega_\star$ and $\upzeta_\star$} 
  In this paragraph, we provide a prove to Lemma \ref{holomorphitude}. First,  passing to the limit in \eqref{canardwc2}, we are  led to:

\begin{lemma}  
\label{scratch}
Let $(u_{\eps_n})_{n \in \N}$ be a sequence of solutions to \eqref{elipes} on $\Omega$ with $\eps_n \to 0$ as $n \to +\infty$ and assume that \eqref{naturalbound} holds. Let $\omega_\star$ and $\upzeta_\star$ be the bounded measures on $\Omega$ given by \eqref{gloubi} and \eqref{boulga} respectively. Then, we have, in the sense of distributions 
  \begin{equation}
 \label{canardwcstar}
  \mathrm{Re}\left ( \left \langle \omega_\star,  \frac{\partial X}{\partial
\bar{z} }\right \rangle_{_{\mathcal D'(\Omega), \mathcal D(\Omega)}} \right)
=
\left \langle  {\upzeta_\star}, \, \mathrm {Re} \left(\frac{\partial X}{\partial
{z} }\right)\right \rangle_{_{\mathcal D'(\Omega), \mathcal D(\Omega)}},  {\rm \ for  \ any \ } X\in C_0^\infty (\Omega, \C).
\end{equation}
 \end{lemma}  
 Lemma \ref{scratch}  is actually  our  main tool in the rest of the discussion, and will be used with vector fields $X$ of  various  types.

 \begin{proof}[Proof of Lemma \ref{holomorphitude}]
 Using $iX$ as test function in \eqref{canardwc} and the fact that $\mathrm{Re}(iz)=-\mathrm{Im}(z)$ for any complex number $z \in \C$,  we obtain likewise
  \begin{equation}
 \label{canardwcstar2}
\mathrm{Im} \left( \left \langle\left(  \omega_\star, \,  \frac{\partial X}{\partial
\bar{z} }\right)\right \rangle_{_{\mathcal D'(\Omega), \mathcal D(\Omega)}}\right)
=
2 \left \langle {\upzeta_\star}, \, \mathrm {Im} \left(\frac{\partial X}{\partial
{z} }\right)\right \rangle_{_{\mathcal D'(\Omega), \mathcal D(\Omega)}}\,  {\rm \ for  \ any \ } X\in C_0^\infty (\Omega, \C).
\end{equation}
Combining \eqref{canardwcstar} and \eqref{canardwcstar2}, we are hence led to the simple  identity  
 \begin{equation}
 \label{canardwcstar3}
\left \langle  \omega_\star, \,  \frac{\partial X}{\partial
\bar{z} }\right \rangle_{_{\mathcal D'(\Omega), \mathcal D(\Omega)}}
=2 \left \langle {\upzeta_\star}, \,\frac{\partial X}{\partial
{z} }\right \rangle_{_{\mathcal D'(\Omega), \mathcal D(\Omega)}}, 
 {\rm \ for  \ any \ } X\in C_0^\infty (\Omega, \C),
\end{equation}
which yields \eqref{distributude} in the sense of distributions.
\end{proof}

  We describe next some  additional properties of the measures $\omega_\star$ et $\upzeta_\star$, mostly bases on Lemma \ref{scratch}, choosing various kinds of test vector fields $\vec X$.
 Whereas we have used so far mainly vector fields yielding dilatations of  the domain (see e.g. Lemma \ref{poho}), we consider  also vector fields of different nature. Given a point 
   $x_0=(x_{0, 1}, x_{0, 2})\in  \Omega$, $\rho>0$ such that $\D^2(x_0, 2\rho) \subset \Omega$,   the fields we  will consider in the next paragraphs are  of the form
\begin{equation}
\label{sheraton00}
 \vec X_{f}(x_1, x_2)=  f_1(x_1)f_2(x_2) \be_j=i f_1(x_1)f_2(x_2), {\rm \ with \ } j=1, 2.
\end{equation}
where,  $f_i$ represents, for $i=1, 2$ an arbitrary function in $C_c^\infty\left (\left(x_{0, i}-\rho, x_{0, i}+\rho\right)\right)$.   These vector fields have hence support on  the square $Q_\rho(x_0)$,   defined by
  \begin{equation}
  \label{cubuc}
  Q_\rho(x_0)   = \mathcal I_r(x_{0, 1}) \times  \mathcal I_\rho(x_{0, 2}),{\rm \ where \   }  \mathcal I_\rho(s)=[s-\rho, s+\rho]=\B^1(s, r), {\rm \ for \ } s>0, 
 \end{equation}
  We consider  also  the subset $\mathcal R_\rho(x_0)$ of $Q_\rho(x_0)$ given by 
  \begin{equation}
  \label{nubuc}
  \mathcal R_\rho(x_0)\equiv \mathcal I_\rho(x_{0, 1}) \times \mathcal I_{\frac {3\rho}{4}}(x_{0, 2})\subset  Q_\rho(x_0),
  \end{equation}
  so that $Q_\rho(x_0) \setminus \mathcal R_\rho(x_0)$ is the union of two disjoint rectangles
  $$
   Q_\rho(x_0) \setminus \mathcal R_\rho(x_0)= \left(\mathcal I_\rho(x_{0, 1}) \times (x_{0, 2}+\frac{3\rho}{4}, x_{0,2}+\rho) \right) \cup 
    \left(\mathcal I_\rho(x_{0, 1}) \times( x_{0, 2}-\rho, 
     x_{0,2}-\frac{3\rho}{4}) \right).
  $$
In several places, we will assume that the following conditions holds
\begin{equation}
\label{michelin}
\upnu_\star (\overline{Q_\rho(x_0) \setminus \mathcal R_\rho(x_0)}) =0, 
\end{equation}
which means that the measure $\upnu_\star$ concentrates, locally,   in a neighborhood of the segment $(x_0-\rho\be_1, x_0+\rho\be_1)$.
\subsection{Projecting the measures on the tangent line}
\label{dezingue}
 In the above framework, the $\be_1$ direction plays a distinguished role:  Integrating   various quantities with respect to the $x_2$-variable, we obtain    one-dimensional quantities, treated as measures on the interval $ {\mathcal I_{\rho_0}} (x_{0, 1})= (x_{0, 1}-\rho_0, x_{0, 1}+\rho_0)$. Using appropriate test functions, relation \eqref{canardwcstar}   is then  turned into a differential equation.   
 
 Given a Radon measure  
 $\upupsilon$ on $Q_\rho(x_0)$,  and a test function $\varphi \in C_c(Q_\rho(x_0), \C)$,  we define the Radon measure  $(\varphi\upupsilon)^{x_1}=\bP_\sharp(\varphi\upupsilon)$  defined on $\mathcal I_\rho(x_{0, 1})$ as follows:  For any Borel set $A$ of $\mathcal I_\rho(x_{0, 1})$, we have 
\begin{equation*}
(\varphi \upupsilon)^{x_1}(A)=(\varphi\upupsilon) \left (\bP^{-1}(A)\cap Q_\rho(x_0)\right)=\varphi \upupsilon \left ((A \times \R)\cap Q_r(x_0)\right).
\end{equation*} 
so that 
\begin{equation}
\label{lecrochet}
\langle \upupsilon, \varphi \rangle=(\varphi\upupsilon)(Q_\rho(x_0))=\int_{Q_\rho(x_0)}\varphi \rd \upupsilon=\int_{\mathcal I_\rho(x_0)} \rd (\varphi \upupsilon)^{x_1}.
\end{equation}
We mainly will make use of  test functions $\varphi$ of the form 
\begin{equation}
\label{produit}
\varphi(x_1, x_2)=g_1(x_1) g_2(x_2),
\end{equation}
 where $g_1$ and $g_2$ are defined on the intervals $\mathcal I_\rho(x_{0, 1})$  and $\mathcal I_\rho(x_{0, 1})$ respectively. If $\varphi$ is of the form 
 \eqref{produit}, then \eqref{lecrochet} becomes 
 \begin{equation}
 \label{capitaine}
 \begin{aligned}
\left \langle \upupsilon, \varphi  \right \rangle_{\mathcal D'(Q_\rho(x_0)),\mathcal D(Q_\rho(x_0))}
 &=\int_{\mathcal I_\rho(x_0)} g_1(x_1)\rd (g_2(x_2)\upupsilon)^{x_1} \\
&=\left  \langle (\rd g_2(x_2)\upupsilon)^{x_1}, g_1\right\rangle_{\mathcal D'(\mathcal I_\rho(x_{0, 1})),\mathcal D(\mathcal I_\rho(x_{0, 1}))}.
 \end{aligned}
 \end{equation}
  In  the case   whre $\upupsilon (\overline{Q_\rho(x_0) \setminus \mathcal R_\rho(x_0)}) =0$ and $g_2(s)=1$ for $s \in \mathcal I_{\frac 34 \rho} (x_{0, 2})$, 
  then we have $g(x_2)\upupsilon=\upupsilon$, so that 
  identity  \eqref{capitaine} becomes
  \begin{equation}
   \label{capitaine3}
\left \langle \upupsilon, \varphi  \right \rangle =\int_{\mathcal I_\rho(x_{0,1})}  g_1(x_1)  \rd \upupsilon^{x_1}.
 \end{equation}
We will make use of this formulas in several places for a  corresponding formulas for the Radon measures $\tilde \upmu_{\star, i, j}$, for $i=1, 2$, $\upnu_\star$, and $\upzeta_\star$ and also  related  measures, obtained by multiplication and sums of the previous ones.
  \subsection{Some quantities of interest} 
  
  The measures $\mathbbmss L_{x_0, \rho}$,  $\bN_{x_0, \rho}$, defined on $\mathcal I_\rho(x_0)$ as well as  the measures  
  $\tilde \upmu_{\star, i, j}^{x_1}$ 
  already introduced in the  introduction in \eqref{caminare} and correspond to the description in the previous paragraph. Our computations will also involve some auxiliary  "moment " measures, defined for, $k \in \N$, by 
\begin{equation}
\label{micheline}
\left\{
\begin{aligned}
 \mathbbmss J_{k, x_0,\rho}& \equiv  \mathbbmss J_{k,\rho} =\mathbbmtt P_{\sharp} \left( (x_2-x_{0,2})^k\tilde \upmu_{\star, 1, 2}\right) \\
 \mathbbmss L_{k, x_0, \rho}&\equiv \mathbbmss L_{k,\rho}=\mathbbmtt P_{\sharp} \left( (x_2-x_{0,2})^k\left [  2 \tilde  \upzeta_{\star, i, j}-\tilde  \upmu_{\star, 1, 1} +\tilde  \upmu_{\star, 2, 2} \right ] \right) \\
  \mathbbmss N_{k, x_0, \rho}(s)&\equiv \mathbbmss N_{k,\rho}=
 \mathbbmtt P_{\sharp} \left((x_2-x_{0,2})^k \left[2\tilde \upzeta_{\star, i, j} +\tilde  \upmu_{\star, 1, 1} -\tilde \upmu_{\star, 2, 2}\right] \right).
 \end{aligned}
\right.
\end{equation}
With this notation (dropping the subscript $x_0$),  we have 
$\tilde \upmu_{\star, 1, 2}^{x_1}=\bJ_0$, $\bL_\rho=\bL_{0, \rho} $ and  $\bN_\rho=\bN_{0, \rho} $. 
 We also consider the measures, for $k \in \N$, 
\begin{equation}
\label{worldwide}
\bH_{k, x_0, \rho}(s)=\frac{1}{4}(\bN_{k,\rho}+\bL_{k, \rho})=\bP_{\sharp} \left( (x_2-x_{0,2})^k\tilde \upzeta_\star \right).
\end{equation}
The main result of this section is: 

\begin{proposition} 
\label{letitwave}
Assume that \eqref{michelin} holds.  Then, the measures $\bL_{x_0, \rho}$ and $\bJ_{x_0, \rho}$ are proportional to the Lebesgue measure  on 
${\mathcal I_\rho} (x_{0, 1})$. Moreover, we have the differential relations
\begin{equation}
\label{zoran}
\left\{
\begin{aligned}
\frac{\rd }{\rd s} {\mathbbmss J}_{k,\rho}&=k {\mathbbmss N}_{k-1,\rho}  {\rm \ in \ } \mathcal D'((x_{0, 1}-\rho, x_{0, 1}+\rho)),\\
-\frac{\rd }{\rd s} \bL_{k, \rho}&=k\bJ_{k-1,\rho}  {\rm \ in \ } \mathcal D'((x_{0, 1}-\rho, x_{0, 1}+\rho)).\\
\end{aligned}
\right.
\end{equation}
\end{proposition}
In the case $k=1$, we obtain hence the relations
\begin{equation}
\label{zoran0}
\left\{
\begin{aligned}
\frac{\rd }{\rd s} {\mathbbmss J}_{1,\rho}&=\bN_\rho, {\rm \ in \ } \mathcal D'((x_{0, 1}-\rho, x_{0, 1}+\rho)) {\rm \ and \ }  \\
-\frac{\rd }{\rd s} \bL_{1,\rho}&= \bJ_\rho {\rm \ in \ } \mathcal D'((x_{0, 1}-\rho, x_{0, 1}+\rho)) .
\end{aligned}
\right.
\end{equation}

Notice the following consequence of Proposition \ref{letitwave}:

\begin{corollary}
\label{becool}
For any $k \in \N$, the measures $\bJ_{k, \rho}$ and $ \bL_{k, \rho}$ are absolutely continuous with respect to the Lebesgue measure $\rd x_1$. Hence there exist measurable functions 
$J_{k, \rho}$ and $ L_{k, \rho}$  on $\mathcal I_\rho(x_{0, 1})$ such that
\begin{equation}
\label{becool1} 
\bJ_{k, \rho}=J_{k, \rho} \rd x_1  {\rm \ and  \  }\bL_{k, \rho}=L_{k, \rho}\rd x_1.
\end{equation} 
Moreover, the functions $J_{k, \rho}$ and $L_{k, \rho}$ are bounded on $\mathcal I_\rho(x_{0, 1})$.
\end{corollary}

\begin{proof} The result is an immediat consequence of the fact that the measures ${\mathbbmss N}_{k-1,r}$  and ${\mathbbmss J}_{k-1,r}$ are bounded, so that, $\bJ_{k, r}$ and $ \bL_{k, r}$ represent $BV$ functions on $\mathcal I_\rho(x_0)$, and hence are bounded.
\end{proof}


\medskip
The proof of Proposition \ref{letitwave} corresponds to the use of different kinds of vector fields of the form \eqref{sheraton00}  in \eqref{canardwcstar}that  we will  describe next n details.The proof of Proposition \ref{letitwave} is  completed in Subsection \ref{cecomplet}.

  \subsection{Shear vector fields}
 We use in this section vector fields of the form \eqref{sheraton00}, specifying $j=2$. More precisely, we consider here vector fields of the form 
 
 \begin{equation}
 \label{sheraton}
 \vec X_{f}(x_1, x_2)=  f_1(x_1)f_2(x_2) \be_2=i f_1(x_1)f_2(x_2).
 \end{equation}
A short computation shows that
   \begin{equation}
   \left\{
   \label{lagaffe}
   \begin{aligned}
    \frac{\partial X_f}{\partial z}&=\frac{1}{2}f_1(x_1)f'_2 (x_2) +\frac i2 f_1'(x_1)f_2 (x_2),\\
     \frac{\partial X_f}{\partial \bar{z}}&=-\frac{1}{2}f_1(x_1)f'_2(x_2) +\frac i2 f_1'(x_1)f_2(x_2), \\
   \end{aligned}
   \right.
   \end{equation}
   and hence
  \begin{equation}
  \label{lagaffe2}
  \left\{
  \begin{aligned}
    {\upzeta_\star}\, \mathrm {Re} \left(\frac{\partial X_f}{\partial
{z} }\right)  &=   \frac{1}{2}f_1(x_1)f'_2(x_2) {\upzeta_\star}\, {\rm \ and \ } \\
\mathrm{Re}\left(  \omega_{\star} \frac{\partial X_f}{\partial
\bar{z} }\right)&=-\frac{\mathrm{Re}(\omega_\star)}{2}f_1(x_1)f'_2 (x_2)-
\frac{{\rm Im} (\omega_\star)}{2}f_1'(x_1)f (x_2).
\end{aligned}
  \right.
   \end{equation}
Identity  \eqref{canardwcstar} then becomes 
 \begin{equation}
 \label{canardobis}
   \left \langle \left (\mathrm{Re}(\omega_\star)+2\upzeta_\star \right), \, f'_2(x_2)  f_1(x_1)\right \rangle-
  {\rm Im}  \left \langle\omega_\star,  f (x_2)\,  f_1'(x_1)\right \rangle_{_{\mathcal D'(\Omega), \mathcal D(\Omega)}}=0.
\end{equation}
In view of \eqref{capitaine},  we have 
\begin{equation}
\label{boudii}
\left\{
\begin{aligned}
 \left \langle \left (\mathrm{Re}(\omega_\star)+2\upzeta_\star \right), f'_2(x_2)  f_1(x_1)\right \rangle&=
  \int_{\mathcal I_r(x_{0,1})}   f_1(s)  \rd  \left  [ f'_2(x_2)\left (\mathrm{Re}(\tilde \omega_\star)+2\tilde\upzeta_\star \right)\right]^{x_1} \\
  {\rm \ and \ }&\\
 {\rm Im}  \left \langle\omega_\star,  f (x_2)\,  f_1'(x_1)\right \rangle_{_{\mathcal D'(\Omega), \mathcal D(\Omega)}}&=-
 \int_{\mathcal I_r(x_{0,1})}f_1'(s)  \rd [f(x_2)\tilde \upmu_{\star, i, j}]^{x_1}.
 \end{aligned}
\right.
\end{equation}
 We choose, in this subsection  as functions $f_1,  f_2$ in  \eqref{sheraton}   $f_1=f$, where $f$ is an arbitrary function in $C_\infty(\mathcal I_\rho(x_0))$ and,  for $f_2$, a function of the form 
$$\displaystyle{f_2(x_2)=\chi( \frac{ x_2-x_{0, 2}}{\rho})}, $$ where $\chi$ is a
  non-negative given  smooth plateau function such that 
 \begin{equation}
 \label{varphiness}
 \chi(s)=1, {\rm \ for \ } s \in [-\frac  34, \frac 34],  {\rm \ and \ } \varphi (s)=0,  {\rm \ for \ } \vert s \vert \geq 1.
\end{equation}
In particular, we have
 \begin{equation}
 \label{duffy}
 f_2'(x_2)=0, {\rm \ if \ } \vert x_2-x_{0,2} \vert \leq \frac{3\rho}{4}, 
 \end{equation}
  Such a vector field  corresponds somewhat to  \emph{shear vector field}.   Using   these shear vector fields,  as test vector fields in \eqref{canardwcstar},  we obtain:

\begin{proposition}
\label{scratchiness}
Assume that \eqref{michelin} holds. Then the measure $\bJ_r$ defined on $\overset{\circ}  { \mathcal I_r} (x_0)$ by \eqref{micheline}
is proportional to the Lebesgue measure, that is $\bJ_r=J_{0, r} \rd x$, for some number $J_{0, r} \in \R$.
\end{proposition}
\begin{proof}
We first show  that, for   any function $f \in C_c^\infty(\mathcal I_r(x_{0, 1}))$, we have 
\begin{equation}
\label{scratchitude}
\langle \tilde \upmu_{\star, 1, 2}, \, f' _1(x_1)\rangle_{_{\mathcal D(\mathcal I),\mathcal D'(\mathcal I)}}=0.
\end{equation}
Indeed, identity \eqref{scratchitude} follows  combining \eqref{canardobis} and  \eqref{duffy}, together with the   fact that $\upnu_\star (\overline{Q_r(x_0) \setminus \mathcal R_r(x_0)}) =0$, so that ${f'_2}$ vanishes on the support of  vanish on the support  $\mathrm{Re}(\omega_\star)+2\upzeta_\star$. It follows that have
$$
 f'_2(x_2)\left (\mathrm{Re}(\tilde \omega_\star)+2\tilde\upzeta_\star \right)=0 {\rm \ and \ therefore \ } 
  \left( f'_2(x_2)\left (\mathrm{Re}(\tilde \omega_\star)+2\tilde\upzeta_\star \right)\right)^{x_1} =0.
$$
In view of the first identity in \eqref{boudii},  the  first term on the left hand side of \eqref{canardobis} vanishes,  so that 
\begin{equation}
\label{torpedo}
 {\rm Im}  \left \langle\omega_\star,  f (x_2)\,  f_1'(x_1)\right \rangle_{_{\mathcal D'(\Omega), \mathcal D(\Omega)}}=0.
\end{equation}  
We notice that $\omega_\star f (x_2)=-2  \upmu_{\star, 1, 2}{\bf 1}_{Q_r(x_0)} =-2 \tilde \upmu_{\star, 1, 2}$, so that we are led to the identity
$$
\langle  \tilde \upmu_{\star, 1, 2},  \,f_2(x_2) f'(x_1)\rangle_{_{\mathcal D'(\Omega), \mathcal D(\Omega)}} =0.
$$
We invoke now identity in \eqref{boudii}, together with the fact that 
$
[f(x_2)\tilde \upmu_{\star, i, j}]^{x_1}=\tilde \upmu_{\star, i, j}^{x_1}=\bJ_\rho, 
$
 to deduce  from \eqref{torpedo} that 
 \begin{equation}
\langle \bJ, f'\rangle_{\mathcal D'(\mathcal I_\rho(x_{0, 1}), \mathcal D(\mathcal I_\rho(x_{0, 1}))} =\int_{\mathcal I_r(x_{0,1})}f'(s)  \rd \bJ_\rho=0, 
 \end{equation}
 We have hence, in the sense of distributions
$$\frac{\rd} {\rd  s} \bJ_\rho=0.$$
A classical results in distribution theory  then shows  that $\bJ_r$ is proportional to the uniform  Lebesgue measure. 
\end{proof}
\subsubsection{Stretching vector fields}
In this subsection, we assume that
   $f_1=f$, where $f$ is an arbitrary function in $C_\infty(\mathcal I_\rho(x_0))$ as above, and,  that $f_2$ is given by, for $k \in N^\star$, by  
   $$\displaystyle{f_2(x_2)=(x-x_{0,2})^k\chi( \frac{ x_2-x_{0, 2}}{\rho})}, $$ 
   where $\chi$ is a
  non-negative given  smooth plateau function such that \eqref{varphiness} holds. With this choice, we have now
  \begin{equation}
 \label{duffy2}
 f'(x_2)=k((x-x_{0,2})^{k-1} {\rm \ if \ } \vert x_2-x_{0,2} \vert \leq \frac{3\rho}{4}.
 \end{equation}
  Combining as above  \eqref{canardobis} and  \eqref{duffy2}, we obtain:

 \begin{lemma}
 \label{demeaux}
  Assume that \eqref{michelin} holds.  We have, for any function $f$  in $C_\infty(\mathcal I_\rho(x_0))$
\begin{equation*}
\label{demeaux}
\left \langle  {\bf 1}_{Q_r}  k(x_2-x_{0,2})^{k-1}(\left(  \mathrm{Re} (\omega_\star)+2\upzeta_\star\right), f(x_1)\right \rangle  +
\left \langle {\bf 1}_{Q_r}   (x_2-x_{0,2})^k {\rm Im } (\omega_\star), f'(x_1)
\right \rangle=0.
\end{equation*}
\end{lemma}
The identity of Lemma \ref{demeaux}  becomes, using  \eqref{micheline}, \eqref{capitaine} and \eqref{capitaine}, and arguing as in the proof of Proposition \ref{scratchiness}
$$
\left \langle {k \mathbbmss N}_{k-1}, f \right \rangle +\left \langle \mathbbmss J_k, f' \right \rangle=0, {\rm \ for \  } f \in C^\infty_c ( (x_{0, 1}-r, x_{0, 1}+r), 
$$
so that, in  the sense of distributions 
\begin{equation}
\label{scratchodon}
\frac{\rd }{\rd s} {\mathbbmss J}_k=k {\mathbbmss N}_{k-1}  {\rm \ in \ } \mathcal D'((x_{0, 1}-r, x_{0, 1}+r)), {\rm \ for \ } k \in \N^\star. 
\end{equation}
\subsection {Dilation vector fields}
\label{dilatons}
We use here   as test vector fields in \eqref{canardwcstar}, vector fields of  the form
$$
\vec X_d(x_1, x_2)=f_1(x_1) f_2(x_2)\be_1.
$$
Computations similar to \eqref{lagaffe} yield
  \begin{equation*}
   \left\{
   \label{lagaffe22}
   \begin{aligned}
    \frac{\partial X_d}{\partial z}&=\frac{1}{2}f_1'(x_1)f_2 (x_2) -\frac i2 f_1(x_1)f'_2 (x_2),\\
     \frac{\partial X_d}{\partial \bar{z}}&=\frac{1}{2}f'_1(x_1)f_2(x_2) +\frac i2 f_1'(x_1)f_2(x_2), \\
   \end{aligned}
   \right.
   \end{equation*}
Relation \eqref{canardwcstar} then becomes 
 \begin{equation}
 \label{canardoter}
  \left \langle   \left(\mathrm{Re}(\omega_\star)-2\upzeta_\star \right),  f_2(x_2)\, f'_1(x_1)\right \rangle_{_{\mathcal D'(\Omega), \mathcal D(\Omega)}}+
  \left \langle{\rm Im} \left(\omega_\star\right),  f_2'(x_2)\,  f_1(x_1)\right \rangle_{_{\mathcal D'(\Omega), \mathcal D(\Omega)}}=0, 
 \end{equation}
 Arguing as for \eqref{boudii}, we obtain the relations 
 \begin{equation}
 \label{amerzone}
\left\{
\begin{aligned}
 \left \langle \left (\mathrm{Re}(\omega_\star)-2\upzeta_\star \right), f_2(x_2)  f'_1(x_1)\right \rangle&=
  \int_{\mathcal I_r(x_{0,1})}   f'_1(s)  \rd  \left  [ f_2(x_2)\left (\mathrm{Re}(\tilde \omega_\star)-2\tilde\upzeta_\star \right)\right]^{x_1} \\
  {\rm \ and \ }&\\
 {\rm Im}  \left \langle\omega_\star,  f' (x_2)\,  f_1(x_1)\right \rangle_{_{\mathcal D'(\Omega), \mathcal D(\Omega)}}&=-
 \int_{\mathcal I_r(x_{0,1})}f_1(s)  \rd [f'_2(x_2)\tilde \upmu_{\star, i, j}]^{x_1}.
 \end{aligned}
\right.
\end{equation}

  We next choose test vectors fields $\vec X_d$, with $f_2$   of the form 
 $\displaystyle{f_2(x_2)=\chi( \frac{ x_2-x_{0, 2}}{\rho})}$, so that \eqref{duffy} holds.  With this choice, we obtain 
 $$[f'_2(x_2)\tilde \upmu_{\star, i, j}]^{x_1}=0 {\rm \ and \ } 
  \left(f_2(x_2)\left (\mathrm{Re}(\tilde \omega_\star)-2\tilde\upzeta_\star \right)\right)^{x_1}=-\bL_\rho.
 $$
 Inserting into \eqref{canardoter}, we derive the relation
 \begin{equation}
 \langle \bL_\rho, f'\rangle=0, {\rm \ for  \ any \ } f_1 \in C_c(\mathcal I_\rho (x_{0, 1})).
 \end{equation}
 Arguing as in the proof  of Proposition \ref{scratchiness}, we derive from \eqref{canardoter} and \eqref{duffy}  that:
  \begin{proposition}
\label{scratchinou}
Assume that  \eqref{michelin} holds. Then, the measure  $\bL_{x_0, \rho}$  	   defined on $  { \mathcal I_\rho} (x_0)$ by \eqref{micheline}
is proportional to the uniform  Lebesgue  measure, that is $\bL_r=L_{0, \rho}\,  \rd x$, for some number $L_{0, r} \in \R$.
\end{proposition}

 We finally use test vectors fields $\vec X_d$, with $f_2$   of the form 
 $$\displaystyle{f_2(x_2)=(x_2-x_{0,2})^k\varphi( \frac{ x_2-x_{0, 2}}{\rho})}, k \in \N^\star$$
  so that \eqref{duffy2} holds.  Inserting into \eqref{canardoter}, and setting $f=f_1$, we are led to 
 $$
\left \langle  k \bJ_{k-1}, f\right \rangle_{_{\mathcal D'(\mathcal I_\rho), \mathcal D'(\mathcal I_\rho)}} +
\left \langle \bL_k(s), f'\right \rangle_{_{\mathcal D'(\mathcal I_\rho), \mathcal D(\mathcal I_\rho)}}=0, {\rm \ for \  } f \in C^\infty_c ( (x_{0, 1}-\rho, x_{0, 1}+\rho). 
$$
Hence, we have, in the sense of distributions, for $k \in \N^\star$
\begin{equation}
\label{scratchodon2}
-\frac{\rd }{\rd s} \bL_k=k\bJ_{k-1}  {\rm \ in \ } \mathcal D'((x_{0, 1}-\rho, x_{0, 1}+\rho)).
\end{equation}
\subsection{Proof of Proposition \ref{letitwave} completed}
\label{cecomplet}
The proof of Proposition \ref{letitwave} follows combining Proposition \ref{scratchiness}, Proposition \ref{scratchinou}, together with identities  \eqref{scratchodon} and \eqref{scratchodon2}.
\subsection{Behavior near regular points}
We specify  Proposition \ref{letitwave} to  regular points.

\subsubsection{Property \eqref{michelin} is satisfied near regular points} 
  We have:

\begin{proposition}  
\label{assos}
Assume that $x_0 \in \mathfrak S_\star \setminus \mathfrak A_\star$.   Then, there exists $\rho_0>0$ such that property \eqref{michelin} is satisfied. Consequently,  the measures $\bL_{x_0, \rho_0}$ and $\bJ_{x_0, \rho_0}$ are proportional to the Lebesgue measure  on ${\mathcal I_{\rho_0}} (x_{0, 1})$ and  the  differential relations \eqref{zoran} hold for $\rho=\rho_0$.
\end{proposition} 

Let $x_0\in \Omega$ and $r>0$ be such that $\D^2(x_0, r)\subset \Omega$. We assume that $x_0$ is  a regular point of $\mathfrak S_\star$ and  choose the orthonormal basis so that   $\be_1=\vec e_{x_0}$ is a unit tangent vector  to $\mathfrak S_\star$ at $x_0$.   In view of   Proposition \ref{tangentfort},   we have,  for  any  $\uptheta \in [0, \frac{\pi}{2}]$  and 
  $0 <\varrho \leq R_{\rm cone}(\uptheta, x_0)$
  \begin{equation}
  \label{lecoq}
 \mathfrak S_\star \cap  \D^2\left(x_0, \varrho\right) \subset  \mathcal C_{\rm one}\left(x_0, \vec e_{x_0}, \uptheta  \right)=\mathcal C_{\rm one}\left(x_0, \be_1, \uptheta  \right), 
   \end{equation} 
Since we have  $\displaystyle{Q_{\frac{\varrho}{\sqrt 2}}(x_0) \subset \D^2\left(x_0, \varrho\right)}$,  we obtain,  for $ 0\leq r \leq \rho_0\equiv {\sqrt{2}}^{-1}R_{\rm cone}(\uptheta, x_0)$
\begin{equation}
\label{lecoq2}
\mathfrak S_\star \cap  Q_r(x_0) \subset  \mathcal C_{\rm one}\left(x_0, \be_1, \uptheta \right). 
\end{equation}
 Specifying \eqref{lecoq2}  with  $\displaystyle{\uptheta=\frac {\pi}{8}}$, we obtain, for 
$0\leq r \leq \rho_0\equiv {\sqrt{2}}^{-1}R_{\rm cone}(\frac {\pi}{8}, x_0)$
\begin{equation}
  \label{lecoq}
 \mathfrak S_\star \cap  Q_r(x_0) \subset  \mathcal C_{\rm one}\left(x_0, \vec e_{x_0}, \frac{\pi}{8} \right)
 \cap Q_r(x_0)
  \subset \mathcal R_r(x_0). 
   \end{equation} 
    It follows that,  if $r\leq \rho_0$,  then   \eqref{michelin} holds. In particular, we are in position  to  apply Proposition \ref{letitwave} at the point $x_0$. This yields immediately,  for the number  $\rho_0>0$ provided by the discussion above,  the fact that the functions   $L_{x_0, \rho_0}$ and $J_{x_0, \rho_0}$ are constant on the interval $(x_{0, 1}-\rho_0, x_{0, 1}+\rho_0)$,  and relations  \eqref{zoran} hold.  The proof of the proposition is hence complete.
   \qed

\medskip

\subsubsection{Some additional properties near regular points}
  We derive next  some additional properties for regular points, in connection with the singular part of the measures.  We introduce therefore the set
  $$
  B_\star=\{s \in \R {\rm \ such  \  that \ }  (\{s\} \times \R) \cap \mathfrak B_\star\not = \emptyset\}=\bP(\mathfrak B_\star).  $$
 where $\mathfrak  B_\star$ is defined in \eqref{perpendicularite} and represents the set where the singular part of the measures concentrates.
 Notice that, since $\mathcal H^1(\mathfrak B_\star)=0$,  the  Lebesgue measure of the set $B_\star$ vanishes likewise.
 Recall,  in view of Corollary \ref{becool}, that  we have
$ \bJ_{1, r}=J_{1, r} \rd x_1$ and $\bL_{1, r}=L_{1, r}\rd x_1$, where the function $L_{1, r}$ and $J_{1, r}$ are bounded.  We have first:
  
  \begin{lemma}
 Let $x_0=(x_{0, 1}, x_{0, 2})$ be  a regular point in  $\mathfrak S_\star\setminus \mathfrak A_\star$ and  let $\rho_0$ be given by Proposition \ref{assos}.  Let $\uptheta \in [0, \frac{\pi}{8}]$.  We have,  for any $r\leq  \frac12\inf\{R_{\rm cone}(\uptheta, x_0), \rho_0\},$
  \begin{equation}
  \label{conan}
  \left\{ 
  \begin{aligned}
 \int_{x_{0, 1}-2r}^{x_{0, 1}+2r} \vert J_{1,\rho_0} (s) \vert  \rd s  \    &\leq   4r\sin \uptheta \,  \upnu_\star^{ac} \left(\D^2(x_0, 2r)\right) {\rm \ and \ } \\ 
    \int_{x_{0, 1}-2r}^{x_{0, 1}+2r} \vert  L_{1,\rho_0}(s)  \vert  \rd s     &\leq   8r\sin \uptheta \,  \upnu_\star^{ac} \left(\D^2(x_0, 2r)\right).
\end{aligned} 
\right.
\end{equation}
\end{lemma}

\begin{proof}  
 If $2r \leq  R_{\rm cone}(\uptheta, x_0)$, it follows from \eqref{radius}   that we have   $\upnu_\star  (\mathcal R_{2r}(x_0)\setminus\mathcal C_{\rm one}\left(x_0, \vec e_{x_0}, \uptheta  \right))=0$. On the other hand,   we have
$$
\vert x_2-x_{0, 2}\vert \leq 2r \sin \uptheta  {\rm \ for  \ } x=(x_1, x_2) \in  \mathcal R_r(x_0)\cap \mathcal C_{\rm one}\left(x_0, \vec e_{x_0}, \uptheta  \right).
$$
Multipling by ${\rm Im } (\omega_\star)$ and integrating  on the set  $\mathcal R_{2r}(x_0)\setminus B_\star \times \R$, we are led to 
\begin{equation}
\label{couvrefeu}
\begin{aligned}
\int_{\mathcal R_{2r}(x_0)\setminus B_\star \times \R }  \rd   \left   \vert  \upmu_{\star, 1, 2} (\omega_\star)
 \left( x_2-x_{0, 2}\right) \right \vert  & \leq   4r\sin \uptheta \,  \upnu_\star \left(\D^2(x_0, 2r) \setminus B_\star \times \R\right) \\
 &\leq 4r\sin \uptheta \, \upnu_\star^{ac} \left(\D^2(x_0, 2r)\right).
\end{aligned}
\end{equation}
 For the last inequality, we  invoke   the fact that we have the inclusion  $\D^2(x_0, 2r) \setminus B_\star \times \R\subset \D^2(x_0, 2r) \setminus \mathfrak  B_ \star$,  so that
  $$\upnu_\star \left(\D^2(x_0, 2r) \setminus B_\star \times \R\right)
  \leq \upnu_\star \left(\D^2(x_0, 2r) \setminus \mathfrak B_\star \right) 
  = \upnu_\star^{ac} \left(\D^2(x_0, 2r)\right).$$
 Since, by  definition  $\bJ_{1, r}=\bP_\sharp( \tilde  \upmu_{\star, i, 2}(x_2-x_{0,2}))$, we have hence, in view of\eqref{citrus}
\begin{equation}
\label{lalveole}
 \int_{(x_{0, 1}-2r, x_{0, 1}+2r)\setminus B_\star} \vert J_{1,\rho_0} (s) \vert  \rd s \leq \int_{\mathcal R_{2r}(x_0)\setminus B_\star \times \R }  \rd   \left   \vert  \upmu_{\star, 1, 2} (\omega_\star)
 \left( x_2-x_{0, 2}\right) \right \vert. 
 \end{equation}
Combining \eqref{couvrefeu},\eqref{lalveole} together with the fact that $B_\star$ has zero Lebesgue measure and the function $J_{1, \rho_0}$ is bounded, thus integrable,    we deduce the first inequality in  \eqref{conan}.  The second is established invoking similar arguments. 
 \end{proof}
 
 \begin{lemma} 
 \label{credible}
 Let $x_0=(x_{0, 1}, x_{0, 2})$ be  a regular point  $\mathfrak S_\star\setminus \mathfrak A_\star$, and and  let $\rho_0$ be given by Proposition \ref{assos}.  Let $\uptheta \in [0, \frac{\pi}{8}]$. 
 For any $\displaystyle{0<r <\frac{1}{2}\inf\left \{R_{\rm cone}(\theta, x_0), \rho_0)\right\}}$, there exists some $\varrho_r \in [r, 2r]$  such that 
\begin{equation}
\label{compens}
\left\{
\begin{aligned}
 \left  \vert  \int_{x_{0,1}-\varrho_r}^{x_{0, 1}+\varrho_r}  J_{\rho_0}(s) \rd s \right \vert  &\leq 16\sin \uptheta \,  \upnu_\star^{ac} \left(\D^2(x_0, 2r)\right)  {\rm  \ and \ } \\
 \left  \vert  \int_{x_{0,1}-\varrho_r}^{x_{0, 1}+\varrho_r}  \rd \bN_{\rho_0} \right \vert  &\leq 8\sin \uptheta \,  \upnu_\star^{ac} \left(\D^2(x_0, 2r)\right).
 \end{aligned}
\right.
\end{equation}
 \end{lemma}
 \begin{proof} 
 The proof of \eqref{compens} follows from \eqref{conan} integrating  the differential equations \eqref{zoran} for $k=1$. Indeed,  for almost every  $\varrho \in [r, 2r]$, 
$x_{0,1}-\varrho$ and $x_{0,1}+\varrho$ are Lebesgue points of $J_{1, \rho_0}$,  $L_{1, \rho_0}$, $J_{\rho_0}$ and the absolutely continuous part of 
$\bN_{\rho_0}$. We choose next  a sequence of test functions  $\{\psi_m\}_{m \in \N}$ such that   such that  $0\leq \psi_m \leq 1$, for any $m \in \N$, and
  \begin{equation}
  \label{pascool}
   \psi_m {\underset {m \to +\infty} \to}  {\bf 1}_{(x_{0,1}-\varrho, {x_{0, 1}+\varrho})}  {\rm \ in \ } L^1(\mathcal I_r(x_{0, 1})).
   \end{equation}
   In view of  the differential equation \eqref{conan} we have, for any $m \in \N$
   \begin{equation}
   \left\{
   \begin{aligned}
   \int_{x_{0,1}-\varrho}^{x_{0, 1}+\varrho}  \psi_m(s) J_{\rho_0}(s) \rd s&= -\int_{x_{0,1}-\varrho}^{x_{0, 1}+\varrho}  \psi'_m(s) L_{1,\rho_0} (s) \rd s {\rm \ and \ } \\
   \int_{x_{0,1}-\varrho}^{x_{0, 1}+\varrho}\psi_m  \rd \bN_{\rho_0}&=\int_{x_{0,1}-\varrho}^{x_{0, 1}+\varrho}  \psi'_m(s) J_{1,\rho_0} (s) \rd s.
   \end{aligned}
   \right.
   \end{equation}
   Passing to the limit, we obtain, using the Lebesgue properties of the points  $x_{0,1}-\varrho$ and $x_{0,1}+\varrho$
 \begin{equation}
 \label{varant}
 \begin{aligned}
 \int_{x_{0,1}-\varrho}^{x_{0, 1}+\varrho}  \rd \bN_{\rho_0}&=J_{1,\rho_0} (x_{0, 1}+\varrho)-J_{1,\rho_0}(x_{0, 1}-\varrho) {\rm \ and \ } \\
 \int_{x_{0,1}-\varrho}^{x_{0, 1}+\varrho}  J_{\rho_0}(s) \rd s&=L_{1,\rho_0} (x_{0, 1}-\varrho)-L_{1,\rho_0}(x_{0, 1}+\varrho).
 \end{aligned} 
 \end{equation}
Next, we use   a mean value argument to deduce that  there exists some  number  $\varrho_r \in [r, 2r]$,  such that $x_{0,1}-\varrho_r$ and $x_{0,1}+\varrho_r$ are Lebesgue points of $J_{1, \rho_0}$,  $L_{1, \rho_0}$, $J_{\rho_0}$ and the absolutely continuous part of $\bN_{\rho_0}$ and such that 
\begin{equation}
\label{labelette}
\left\{
\begin{aligned}
\vert J_{1,\rho_0} (x_{0, 1}+\varrho_r) \vert + \vert J_{1,\rho_0}(x_{0, 1}-\varrho_r) \vert  &\leq \frac{2}{r}  \int_{x_{0, 1}-2r}^{x_{0, 1}+2r} \vert J_{1,\rho_0} (s) \vert  \rd s  \\
\vert L_{1,\rho_0} (x_{0, 1}+\varrho_r) \vert + \vert L_{1,\rho_0}(x_{0, 1}-\varrho_r) \vert  &\leq \frac{2}{r}  \int_{x_{0, 1}-2r}^{x_{0, 1}+2r} \vert L_{1,\rho_0} (s) \vert  \rd s  
\end{aligned} 
\right.
\end{equation}
Combining \eqref{varant}, \eqref{labelette} with \eqref{conan}, we obtain the desired result.
   \end{proof}


   \subsection{Behavior  near Lebesgue points: Proofs to  Proposition \ref{discrepvec6} and Lemma \ref{starac}}  
Recall that, at this stage we already know that,  if $x_0 \in \mathfrak S_\star \setminus \mathfrak A_\star$,  in view of Propositions \ref{scratchiness} and \ref{scratchinou}. 
$$ \bL_{x_0, \rho_0}=L_{0, \rho_0} \rd x_1  {\rm \ and \ }\bP_\sharp( \tilde  \upmu_{\star, i, 2})=J_{0, \rho_0} \rd x_1, $$
where $L_{0, \rho_0} \in \R$ and $J_{0, \rho_0} \in \R$. We derive here additional properties in the case $x_0 \not \in \mathfrak E_\star$, leading eventually to the proof of Proposition \ref{discrepvec6}.

\subsubsection{ Additional properties of $J_{x_0, \rho_0}$ and $\bN_{x_0, \rho_0}$ at Lebesgue points}  
 Let $x_0 \in \mathfrak S_\star$ and $\rho_0>0$. We impose in this paragraph the additional condition that $x_0  \not \in \mathfrak E_\star$, i.e. $x_0$ is a regular point, which is not on the support of the singular part, and is moreover a Lebesgue point for the densities  of  the absolutely continuous part for all measures of interest.   More precisely, this means that 
  \begin{equation}
 \label{themute}
 \left\{
 \begin{aligned}
& \underset {r\to 0}  \lim   \,  \frac{1} {r}   \int_{\mathfrak S_\star\cap \D^2(x_0, r)} \left \vert  {\Uptheta_\star}(\tau)-{\Uptheta_\star} (x_0) \right \vert    \rd \tau=0 \\
&\underset {r\to 0}  \lim   \,  \frac{1} {r}   \int_{\mathfrak S_\star\cap \D^2(x_0, r)} \left \vert  {\tbe_\star}(\tau)-{\tbe_\star} (x_0) \right \vert    \rd \tau=0,   {\rm \ and \ }\\
&\underset {r\to 0}  \lim   \,  \frac{1} {r}   \int_{\mathfrak S_\star\cap \D^2(x_0, r)} \left \vert  {\bm_{\star, i, j}}(\tau)-{\bm_{\star, i, j} } (x_0) \right \vert    \rd \tau=0
 \end{aligned}
 \right.
  \end{equation}
  As  a first direct consequence, we deduce that, for some constant $K=K(x_0)>0$   depending on $x_0$, we have 
   \begin{equation}
  \label{auxiliary}
  \upnu_\star^{ac} \left(\D^2(x_0, r)\right)   \leq K r  {\rm \ for \ any \ } 0<r<R, 
\end{equation}
  and also that 
   \begin{equation}
 \label{themute2}
 \left\{
 \begin{aligned}
& \underset {r\to 0}  \lim   \,  \frac{1} {r}   \int_{\mathfrak S_\star\cap \D^2(x_0, r)}  {\Uptheta_\star}(\tau)\rd \tau={\Uptheta_\star} (x_0),  \\
&\underset {r\to 0}  \lim   \,  \frac{1} {r}   \int_{\mathfrak S_\star\cap \D^2(x_0, r)}   {\tbe_\star}(\tau)\rd \tau={\tbe_\star} (x_0),   {\rm \ and \ }\\
&\underset {r\to 0}  \lim   \,  \frac{1} {r}   \int_{\mathfrak S_\star\cap \D^2(x_0, r)} \  {\bm}_{\star, i, j}(\tau)\rd \tau={\bm}_{\star, i, j } (x_0).
\end{aligned}
 \right.
  \end{equation}
  At this stage, we already know that $J_{\rho_0}$ is a constant map. Concerning $\bN_{\rho_0}$ we may decompose this measure on $\mathcal I_{\rho_0} (x_{0, 1})$ as a sum of an absolutely continuous part and a singular part
  $$
  \bN_{\rho_0}=\bN_{\rho_0}^{ac} +\bN_{\rho_0}^{s} {\rm \ with \ } \bN_{\rho_0}^{ac}\ll \rd x_1 {\rm \ and \ }  \bN_{\rho_0}^{s} \perp \bN_{\rho_0}^{ac}, 
  $$
so that  there exists a set $F_{\rho_0} \subset \mathcal I_{\rho_0}(x_{0, 1})$ such that 
$\mathcal H^1(F_{\rho_0})=0$ and  $\bN_{\rho_0}^{s} ( \mathcal I_{\rho_0}(x_{0, 1}\setminus F_{\rho_0}))=0$,  and a measurable function $N_{ \rho_0}$ defined on $\mathcal I_{\rho_0} (x_{0, 1})$ such that $\bN_{\rho_0}^{ac}=N_{\rho_0} \rd x_1$.
In this setting,   the functions $L_{\rho_0}, N_{\rho_0}$ and $J_{\rho_0}$ on $\mathcal I_{\rho_0}(x_{0, 1})$ are related to the functions $\Uptheta_\star$ and $\bm_{\star, i, j}$, for $i, j=1, 2$ defined on $\mathfrak S_\star$ by \eqref{gdansk0}  by the following result.
  
  \begin{proposition}
  \label{constantin}
    Let $x_0 \in \mathfrak S_\star\setminus \mathfrak E_\star$  and $\rho_0>0$ be given by Proposition \ref{assos} so that \eqref{michelin} holds for $\rho=\rho_0$. Choose the orthonormal basis so that   $\be_1=\vec e_{x_0}$ is a unit tangent vector  to $\mathfrak S_\star$ at $x_0$.   Then, $x_{0, 1} \not \in F_{\rho_0}$ and is a Lebesgue point for $N_{\rho_0}$ and $J_{\rho_0}(x_0)$.
      We have the identities, at the point $x_0$, 
     \begin{equation}
     \label{lacata}
     \left\{
     \begin{aligned}
   N_{\rho_0}(x_{0, 1})&=2\Uptheta_{\star}(x_0)-\bm_{\star, 2, 2}(x_0)+\bm_{\star, 1, 1}(x_0),   \\ 
   J_{\rho_0}(x_{0, 1})&=\bm_{\star, 1, 2}(x_0)  {\rm \ and \ } \\
   L_{\rho_0}(x_{0, 1})&=2\Uptheta_{\star}(x_0)-\bm_{\star, 1, 1}(x_0)+\bm_{\star, 2, 2}^{ac}(x_0).
 \end{aligned}
  \right.
     \end{equation}
  \end{proposition} 
  
  \begin{proof}    
 We  go back to the definition \eqref{bordeaux} of $\mathfrak E_\star$. Since $x_0\not \in \mathfrak E_\star$, and hence  $x_0 \not \in \mathfrak B_\star$ (see \eqref{perpendicularite}), we have  by definition of the set $\mathfrak B_\star$
\begin{equation}
\label{exemplaire}
D_\lambda (\upnu_\star)(x_0)= \underset {r \to 0} \lim \frac{\upnu_\star \left(\D^2(x_0, r)\right)}{\lambda \left(\D^2(x_0, r)\right)} <+ \infty {\rm \ and \ } 
D_\lambda (\upnu_\star^{s})(x_0)= \underset {r \to 0} \lim \frac{\upnu_\star \left(\D^2(x_0, r)\right)}{\lambda \left(\D^2(x_0, r)\right)}=0,
\end{equation}
where $\lambda$ represents the one-dimensional Hausdorff measure on $\mathfrak S_\star$.
On the other hand, since $x_0$ is a regular point, we have, in view of \eqref{densitusone}
$$
 \underset {r \to 0} \lim \frac{\lambda \left(\D^2(x_0, r)\right)}{2r}=1, 
$$
 so that 
 \begin{equation}
\label{mevanwi}
D_\lambda (\upnu_\star)(x_0)=D_\lambda (\upnu_\star^{ac})(x_0)=\underset {r \to 0} \lim \frac{\upnu_\star \left(\D^2(x_0, r)\right)}{2r}<+\infty. 
 \end{equation}
 Turning to the measure $\tilde \upnu_\star^{x_1}$, we have 
 $\displaystyle{\upnu_\star^{x_1}(\mathcal I_r(x_{0, 1}))=\tilde \upnu_\star \left( \mathcal I_r(x_{0, 1}) \times  \mathcal I_r(x_{0, 2})\right).}$   In view of Proposition \ref{tangentfort}, given $ \uptheta>0$, we have, for  
 $r\leq R_{\rm cone}(\uptheta, x_0)$, the inclusion 
$\mathfrak S_\star \cap  \D^2\left(x_0, r\right) \subset  \mathcal C_{\rm one}\left(x_0, \vec e_{x_0}, \uptheta  \right)$. On the other hand, we have also  the chain of inclusions
\begin{equation}
\label{cadran}
 \D^2(x_0, r)\subset \left(\mathcal I_r(x_{0, 1}) \times  \mathcal I_r(x_{0, 2})\right) \cap \mathcal C_{\rm one}\left(x_0, \vec e_{x_0}, \uptheta  \right) \subset \D^2(x_0, \frac{r}{\cos \uptheta}),
\end{equation}
so that combining the previous relations, we are led to the bounds
\begin{equation}
\label{encadrement}
 \upnu_\star \left(\D^2(x_0, r)\right) \leq \upnu_\star^{x_1}(\mathcal I_r(x_{0, 1}))\leq \upnu_\star \left(\D^2(x_0, \frac{r}{\cos \uptheta})\right).
\end{equation}
Letting $\uptheta$ and $r$ go to zero, we deduce from \eqref{exemplaire} and \eqref{encadrement}  the identity
$$
\underset {r\to 0}  \lim \, \frac{\upnu_\star^{x_1}(\mathcal I_r(x_{0, 1}))}{2r}=D_\lambda (\upnu_\star)(x_0)=D_\lambda (\upnu_\star^{ac})(x_0)<+\infty,
$$
and similarily, for $i, j=1, 2$
\begin{equation*}
\label{caltech}
\left \{
\begin{aligned}
\underset {r\to 0}  \lim \, \frac{\upzeta_\star^{x_1}(\mathcal I_r(x_{0, 1}))}{2r}&=D_\lambda (\upzeta_\star^{ac})(x_0)
=\Uptheta_\star({x_0})  {\rm \ and \ } \\
\underset {r\to 0}  \lim \, \frac{\upmu_{\star, i, j}^{x_1}(\mathcal I_r(x_{0, 1}))}{2r}&=
D_\lambda (\upmu_{\star, i, j}^{ac})(x_0)=\bm_{\star, i, j}(x_0)<+\infty.
\end{aligned}
\right.
\end{equation*}
It follows that, in view of the  definition \eqref{caminare} of $\bN_{\rho_0}$, we have  
\begin{equation*}
\label{dickens}
\begin{aligned}
\underset {r\to 0}  \lim \, \frac{\bN_{\rho_0}(\mathcal I_r(x_{0, 1}))}{2r}&=2D_\lambda (\upzeta_\star^{ac})(x_0)-D_\lambda(\upmu_{2,2}^{ac})(x_0)+D_\lambda (\upmu_{1, 1}^{ac})(x_0)\in \R \\
&=2\Uptheta_{\star}^{ac}(x_0)-\bm_{\star, 2, 2}^{ac}(x_0)+\bm_{\star, 1, 1}^{ac}(x_0).
\end{aligned}
\end{equation*}
We deduce  that $x_{0, 1} \not \in F_{\rho_0}$ and that we have 
$$
\underset {r\to 0}  \lim \, \frac{\bN_{\rho_0}^{ac}(\mathcal I_r(x_{0, 1}))}{2r}=\underset {r\to 0}  \lim \, \frac{\bN_{\rho_0}(\mathcal I_r(x_{0, 1}))}{2r}=2\Uptheta_{\star}^{ac}(x_0)-\bm_{\star, 2, 2}^{ac}(x_0)+\bm_{\star, 1, 1}^{ac}(x_0)..
$$
 We prove using similar arguments  that $x_{0, 1}$ is a Lebesgue point for the map $N_{\rho_0}$, so that the first identity in \eqref{lacata} is established. Turning to the maps  $J_{\rho_0}$ and $L_{\rho_0}$ we observe that, since  these maps are constant, $x_{0, 1}$ is obviously a Lebesgue point for them. The two last identities in \eqref{lacata} are established using the same  arguments.
\end{proof}

We compute next $J_{\rho_0}(x_0)$ and $N_{\rho_0}(x_0)$ in a different way.
  
   \begin{proposition}
   \label{diff}
   Let $x_0$ and $\rho_0>0$ be as in Proposition \ref{constantin}. We have 
    \begin{equation}
    \label{diff0}
    \left \{
    \begin{aligned}
    &J_{x_0, \rho_0}(s)=0  {\rm \ for \ } s \in  (x_{0, 1}-\rho_0, x_{0, 1}+\rho_0)   {\rm \ and  \  }\\
     &N_{x_0, \rho_0}(x_{0, 1})=0. 
    \end{aligned}
    \right. 
    \end{equation}
\end{proposition}
  In order to proof Proposition \ref{diff}, we rely on an  intermediate result:

  \begin{lemma}
  \label{carmelide}
   Let $x_0 \in \mathfrak S_\star\setminus \mathfrak E_\star$  and $\rho_0>0$ be given by  Proposition \ref{assos}. Choose the orthonormal basis so that   $\be_1=\vec e_{x_0}$ is a unit tangent vector  to $\mathfrak S_\star$ at $x_0$.  For $<r<\rho_0$,  let $\varrho_r>0$ be given by Lemma \ref{credible}. Then,  we have 
 \begin{equation}
    \label{caroual20}
  \underset{r\to 0} \lim   \,  \frac{1}{2\varrho_r} \int_{x_{0,1}-\varrho_r}^{x_{0, 1}+\varrho_r}  \rd \bN_r(s)=0 {\rm \ and \ }  
   \underset{r\to 0}  \lim  \,   \frac{1}{2\varrho_r} \int_{x_{0,1}-\varrho_r}^{x_{0, 1}+\varrho_r}  J_r(s) \rd s=0.
   \end{equation}
\end{lemma}
\begin{proof} 
    For any given $\uptheta \in [0, \frac{\pi}{8}] $, and $0<r\leq \inf\{ \rho_0, \frac{1}{2}R_{\rm cone}(\uptheta, x_0)\}$, we deduce,  combining   \eqref{auxiliary} with   \eqref{varant}, that,  \begin{equation}
  \label{caroual} 
  \left  \vert  \int_{x_{0,1}-\varrho_r}^{x_{0, 1}+\varrho_r}  \rd \bN_r \right \vert + 
   \left  \vert  \int_{x_{0,1}-\varrho_r}^{x_{0, 1}+\varrho_r}  J_r(s) \rd s \right \vert\leq 24\sin \uptheta \,  \upnu_\star^{ac} \left(\D^2(x_0, 2r)\right)
   \leq 48 Kr \, \sin \uptheta,  
   \end{equation} 
   so that 
   \begin{equation}
   \label{caroual2}
     \left  \vert \frac{1}{2\varrho_r} \int_{x_{0,1}-\varrho_r}^{x_{0, 1}+\varrho_r}  \rd \bN_r  \right \vert + 
   \left  \vert \frac{1}{2\varrho_r} \int_{x_{0,1}-\varrho_r}^{x_{0, 1}+\varrho_r}  J_r(s) \rd s \right \vert\leq 48K\sin \uptheta.
   \end{equation}
   We first  let $r\to 0$, so that $\varrho_r  \to 0$ as $r \to 0$, and then   
   let  $\uptheta \to 0$ in \eqref{caroual2}, which yields \eqref{caroual}.
   \end{proof}
 
  \noindent
 {\it Proof of Proposition \ref{diff} completed}. We first consider $J_{\rho_0}$. 
  We already know that the function $J_{\rho_0}$ is constant on $\mathcal I_{\rho_0}(x_{0, 1})$, so that 
 $$\displaystyle{ \frac{1}{2\varrho_r} \int_{x_{0,1}-\varrho_r}^{x_{0, 1}+\varrho_r}  J_r(s) \rd s}=J_r(x_{0, 1}),$$
 we deduce therefore  from  the second relation in \eqref{caroual20} that $J_{\rho_0}(x_{0, 1})=0$.  
 
 \smallskip
 We turn to $N_{\rho_0}$.  Since $x_0\not \in F_{\rho_0}$, we have $D_\lambda(\bN_{\rho_0}^s)(x_{0, 1})=0$,  that is 
 $$\underset{r\to 0} \lim   \,  \frac{\bN_{\rho_0}^{s}(\mathcal I_{\varrho_r}(x_{0, 1}))}{2\varrho_r}=0.$$
 Combining with the first identity in \eqref{caroual20}, we are led to  
 \begin{equation}
   \underset{r\to 0} \lim   \,  \frac{1}{2\varrho_r} \int_{x_{0,1}-\varrho_r}^{x_{0, 1}+\varrho_r}  N_{\rho_0}(s)\rd s=  \underset{r\to 0} \lim   \,  \frac{1}{2\varrho_r} \int_{x_{0,1}-\varrho_r}^{x_{0, 1}+\varrho_r}  \rd \bN_r^{ac}(s)=0.  
 \end{equation}  
 Since $x_{0, 1}$ is a Lebesgue point for $N_{\rho_0}$, we derive that $N_{\rho_0}(x_{0, 1})=0$, so that the proof is complete.

 \subsubsection{Proof of Proposition  \ref{discrepvec6} completed}
Since $x_0 \in  \mathfrak S_\star\setminus \mathfrak E_\star$, we are in position to apply Propositions \ref{constantin} and \ref{diff}. 
Combining \eqref{lacata} with \eqref{caroual20}, we obtain \eqref{discrepvec66} and the proof is complete.
\qed
\subsubsection{Change of orthonormal basis for the Hopf differential }
  \label{cdc}
 Recall that we have assumed in Proposition \ref{discrepvec6} that the orthonormal basis is chosen so that $\be_1$ is tangent to $\mathfrak S_\star$ at $x_0$.  However, the  definition of the Hopf differential clearly depends on the choice of coordinates, and we will need to compute it in various basis, for instance a moving frame on $\mathfrak S_\star$ or a frame related to polar coordinates.  For that purpose, and for a given angle $\theta \in \R$, let $(\vec {\bf  e}_1^{\,  \theta} , \be_2^{\, \theta})$  be a new orthonormal basis such that 
 \begin{equation}
 \label{cdc2}
 \left\{
 \begin{aligned}
 \betheta_1&=\cos \theta \, \be_1+\sin\theta \, \be_2\\
 \betheta_2&=-\sin \theta  \, \be_1+\cos \theta \, \be_2. 
 \end{aligned}
 \right.
 \end{equation}
Let $(x_{\theta,1}, x_{\theta, 2})=(\cos \theta \, x_1+\sin \theta \,  x_2, -\sin \theta\, x_1+ \cos \theta\, x_2)$ denote the  coordinates related  to the new basis and $\omega_{\eps, \theta} $ the corresponding Hopf differential. Then, for any map $u: \Omega \to \R^2$, we have the identities
$u_{x_{\theta, 1}} =u_{x_1} \cos \theta +u_{x_2} \sin \theta$ and $u_{x_{\theta, 2}} =-u_{x_1} \sin  \theta +u_{x_2} \cos \theta$, so that 
\begin{equation}
\label{transformer}
\left\{
\begin{aligned}
 \vert u_{x_{\theta, 1} }\vert^2-\vert u_{x_{\theta, 2} }\vert^2&=\cos 2 \theta \left ( \vert u_{x_1} \vert^2-\vert u_{x_2} \vert^2 \right)  +2\sin 2 \theta \,  u_{x_1}\cdot u_{x_2} \\
2 u_{x_{\theta, 1} }\cdot u_{x_{\theta, 2}}&= -\sin 2 \theta \left ( \vert u_{x_1} \vert^2-\vert u_{x_2} \vert^2 \right) +2 \cos  2 \theta u_{x_1}\cdot u_{x_2}.
\end{aligned}
\right.
\end{equation}
We are hence led to the transformation law
\begin{equation}
\label{transformer2}
\left\{
\begin{aligned}
\omega_{\eps, \theta} =(\cos 2 \theta +i \sin 2  \theta) \omega_\eps &=\exp (2 i \theta) \omega_\eps {\rm \ and \ } \\
\omega_{\star, \theta} =(\cos 2 \theta +i \sin 2  \theta) \omega_\star&=\exp (2 i \theta) \omega_\star.
\end{aligned}
\right.
\end{equation}
 It follows in particular from the above relations that, if the limits \eqref{gloubi} and \eqref{boulga} exist for a given orthonormal  basis, then they exist also for any other one. 
\subsubsection{Proof of Lemma \ref{starac} completed}
In view of \eqref{gdansk0}, we may write, in the basis $(\be_1, \be_2)$
\begin{equation}
\label{dantzig}
\omega_\star^{ac}=\left((\bm_{\star, 1, 1}-\bm_{\star, 2, 2})-2i \bm_{1, 2} \right)\rd \lambda.
\end{equation}
Next let $x_0 \in \mathfrak S_\star\setminus \mathfrak E_\star$, $\vec e_{x_0}$ be  a tangent vecteur at $x_0$ to $\mathfrak S_\star$, so that the angle of $\be_1$ with  $\vec e_{x_0}$  is given by $\upgamma_{\star}(x_0)\in [-\pi\slash2, \pi\slash2]$. In view of the notation \eqref{cdc2}, we have 
$\vec e_{x_0}=\be_1^{\, \upgamma_{\star}(x)}$. It follows from \eqref{transformer2} that 
\begin{equation}
\label{laberie}
\omega_{\star, \upgamma_\star(x)}^{ac}=\exp (2 i \upgamma_{\star}) \omega_\star^{ac}=\exp (2 i \upgamma_{\star}) \left((\bm_{\star, 1, 1}-\bm_{\star, 2, 2})-2i \bm_{1, 2} \right)\rd \lambda.
\end{equation}
Appying Proposition \ref{discrepvec6} at $x_0$ in the basis  $\left(\be_1^{\, \upgamma_{\star}(x)}, \be_2^{\, \upgamma_{\star}(x)}\right)$,  we are led to the identity
\begin{equation*}
\exp (-2 i \upgamma_{\star}) \left((\bm_{\star, 1, 1}(x_0)-\bm_{\star, 2, 2}(x_0))-2i \bm_{1, 2}(x_0) \right)
=-2\Uptheta_\star(x_0), 
\end{equation*}
so that, for any $x \in \mathfrak S_\star \setminus \mathfrak E_\star$, we have
\begin{equation*}
(\bm_{\star, 1, 1}(x)-\bm_{\star, 2, 2}(x))-2i \bm_{1, 2}(x)=-2\exp (2 i \upgamma_{\star}(x)) \Uptheta(x).
\end{equation*}
Going back to \eqref{dantzig}, we obtain hence that
$$
\omega_\star^{ac}=-2\exp (-2 i \upgamma_{\star}(x))\Uptheta\rd \lambda=-2\exp (-2 i \upgamma_{\star})\upzeta_\star^{ac}.
$$
The proof is hence complete.

\section{Monotonicity for $\upzeta_\star$ and its consequence}

The purpose of present section is to establish Proposition \ref{psycho}.
\subsection{Proof of Lemma \ref{passing} }
Since $\upzeta_{\eps_n} \rightharpoonup \upzeta_\star$, as  $n \to +\infty$, weakly in the sense of measures,   we have for any  Borel  set $A$ such that $\upzeta_\star (\partial A)=0$, the convergence $\upzeta_{\eps_n} (A)\to \upzeta_\star (A)$, as $n \to +\infty$. Since 
$\upnu_\star (\partial \D^2(x_0,r))=0$ for almost every $r \in (0, \rho)$, we have 
\begin{equation}
\left\{
\begin{aligned}
&\upzeta_{\eps_n} (\D^2(x_0,r_i)) \underset{ n \to + \infty} \to \upzeta_\star (\D^2(x_0,r_i)) {\rm \ for \  almost \ every \ } r_i \in (0, \rho), \
 i=0, 1 {\rm \ and \ } \\
&\int_{D^2(x_0,r_1)\setminus \D^2(x_0,\rho_0)} \frac {1}{r} \rd \mathscr N_\eps  
\underset{ n \to + \infty}\to 
\int_{D^2(x_0,r_1)\setminus \D^2(x_0,\rho_0)} \frac {1}{r} \rd \mathscr N_\star.
\end{aligned}
\right.
\end{equation}
Passing to the limit in \eqref{degrace}, we obtain the identity \eqref{monomaniaque}.
\subsection{First properties of $\mathscr N_\star$} 
Let $\upmu_{\star, \theta, \theta}$ and  $\upmu_{\star, r, r}$ be defined by \eqref{radicaux} on $\D^2(x_0, \rho)$.  We denote by $\upmu_{\star, \theta, \theta}^{ac}$ and  $\upmu_{\star, r, r}^{ac}$  the absolutely continuous parts of these measures with respect to the 
$\mathcal H^1$-Hausdorff  measure $\lambda$  on $\mathfrak S_\star\cap \D^2(x_0,\rho)$. We prove in this subsection:

\begin{lemma}
\label{dingo}
We have  the relations 
\begin{equation}
\label{dingo2}
\mathscr N_\star^{ac} =\left(2\upzeta_\star^{ac}-r^{-2}\upmu_{\star, \theta , \theta }^{ac}+\upmu_{\star, r, r}^{ac}\right)\rest \D^2(x_0,\rho)=
4\sin^2 \left( \upgamma_{\star}-\theta\right)\left(\upzeta_\star^{ac}\rest \D^2(x_0,\rho)\right)\geq 0,
\end{equation}
where $\theta$ denotes the polar angle with respect to the $x_0$.
\end{lemma}
\begin{remark}
{\rm  
Let $\nabla r$ denote the gradient of the function $r: (x_1, x_2)\mapsto \sqrt{x_1^2+x_2^2}$, so that  
$\nabla r(x)=(x_1\slash r, x_2 \slash r)$. For given 
$x\in (\mathfrak S_\star \setminus \mathfrak E_\star)\cap \D^2(x_0,\rho)$, we denote by
$\nabla^\perp r(x)$, the projection of $\nabla r(x)$ onto the  othogonal line to the tangent to $\mathfrak S_\star$ at the point $x$.  We compute
$$ \vert \nabla^\perp r(x) \vert =\vert \sin\left( \upgamma_{\star}(x)-\theta\right)\vert. $$
Formula \eqref{dingo2} can therefore be rewritten as 
\begin{equation}
\label{dingo3}
\mathscr N_\star^{ac}= \vert \nabla^\perp r \vert ^2 \upzeta^{ac}_\star \geq 0.
\end{equation}
}
\end{remark}
\begin{proof}[Proof of Lemma \ref{dingo}]
We may write
\begin{equation}
\label{moulin}
\upmu_{\star, r, r}^{ac}=\bm_{\star, r, r} \rd \lambda  {\rm \ and \ }
r^2\upmu_{\star, \theta , \theta }^{ac}=r^2\bm_{\star, \theta , \theta } \rd \lambda, 
\end{equation}
where, similar to \eqref{radicaux}, we have set, for $x\in (\mathfrak S_\star \setminus \mathfrak E_\star) \cap \D^2(x_0,\rho)$, 
\begin{equation}
\label{radicaux2}
\left\{
\begin{aligned}
\bm_{\star, r, r}(x)&=\cos^2 \theta (x)  \bm_{\star, 1, 1}(x) + \sin^2 \theta(x) \bm_{\star, 2, 2}(x)+2\sin\theta \cos \theta (x)\bm_{\star, 1, 2} (x) \\
r^{-2} \bm_{\star, \theta, \theta}(x)&=\sin^2 \theta (x)  \bm_{\star, 1, 1}(x)+
 \cos^2 \theta(x) \bm_{\star, 2, 2}(x)-2\sin \theta(x)  \cos \theta(x)\bm_{1, 2}(x).
 \end{aligned}
\right.
 \end{equation}
where $(r, \theta)$ denote   the polar coordinates of $x=(x_1,x_2)$, so that $x_{1}-x_{0,1}=r\cos \theta$ and $x_{2}-x_{0,2}=r \sin \theta$.
We have, in view of Proposition \ref{starac} and relations \eqref{transformer2}
\begin{equation}
\label{molina}
\omega_{\star, \theta}^{ac}=-2\exp (2 i (\upgamma_{\star}(x)-\theta))\upzeta_\star^{ac}.
\end{equation}
Since  $\omega_{\star, \theta_0}^{ac}$ is absolutely continuous with respect to $\rd \lambda$, we may write
 $\omega_{\star, \theta_0}^{ac}=\mathbbm w_{\star, \theta_0}\rd \lambda$, where $\mathbbm w_{\star, \theta}$ is a function on $\mathfrak S_\star\cap \D^2(x_0,\rho)$. Concerning the measure  $\mathscr N_\star^{ac}$, we have
 \begin{equation}
 \label{mordred} 
 \mathscr N_\star^{ac}=\left(2\Uptheta-r^{-2}\bm_{\star, \theta , \theta }  +\bm_{\star, r, r} \right)\rd \lambda.
 \end{equation}
It follows from the definitions \eqref{radicaux2} and \eqref{moulin}, that  we have the identity
\begin{equation}
\label{moulinsart}
 \mathbbm w_{\star, \theta}(x)=\left (\bm_{\star, r, r}(x)-r^{-2} \bm_{\star, \theta, \theta}(x)\right) -
 2i r^{-1} \bm_{\star, r, \theta}(x).
 \end{equation}
Combining \eqref{molina} and \eqref{moulinsart}, we are hence led to 
\begin{equation*}
\bm_{\star, r, r}(x)-r^{-2} \bm_{\star, \theta, \theta}(x)=-2\cos \left( \upgamma_{\star}(x)-\theta\right))\Uptheta_\star(x), 
\end{equation*}
so that 
\begin{equation*}
\begin{aligned}
\left(2\Uptheta(x)-r^2\bm_{\star, \theta , \theta }(x)  +\bm_{\star, r, r}(x) \right)&
=2(1-\cos \left( \upgamma_{\star}(x)-\theta\right) \Uptheta_\star(x) \\
&=4\sin^2 \left( \upgamma_{\star}(x)-\theta\right)\Uptheta_\star(x).
\end{aligned}
\end{equation*}
Going back to  \eqref{mordred}, we deduce that 
\begin{equation*}
\mathscr N_\star^{ac}=4\sin^2 \left( \upgamma_{\star}-\theta\right)\upzeta_\star^{ac},  
\end{equation*} 
so that \eqref{dingo2} is established.
\end{proof}
\subsection{Integrating on growing  disks}
 Let $\rho>\delta>0$ be given. We introduce and study  in this section  the functions $V$, $F$ and $G_\delta$   defined on the interval 
 $[\delta, \rho]$, by 
 \begin{equation}
 \label{vetg}
 \left\{
 \begin{aligned}
 V(r)&=\upzeta_\star (\D^2 (x_0,r)), \, F(r)=\frac{V(r)}{r}  {\rm \ and \ } \\
 G_\delta(r)&=\int_{D^2(x_0,r)\setminus \D^2(x_0,\delta)} \frac 1 r \rd \mathcal N_\star, {\rm \ for \ } 
 x \in [\delta, \rho].
 \end{aligned}
 \right.
 \end{equation}
 The three  functions  defined in \eqref{vetg} are clearly bounded on the interval  $[\delta, r]$, since 
  $0\leq V(r)  \leq V(\rho)$,  $0\leq F(r) \leq F(\rho) \slash \delta$ and  
  $\vert G_\delta (r) \vert \leq 2\upnu_\star(\D^2(x_0,\rho)) \slash \delta.$ Moreover, the function $V$ is clearly non-decreasing. 
  We will show below that these functions have bounded variation. 
  
   In order to relate these functions and their derivatives  to the measures on $\D^2(x_0, \rho)$ introduced so far,  we  have to eliminate the polar angle $\theta$. For that purpose, we  consider the map $\Uppi:\D^2(x_0,\rho)\setminus \{0\} \to (0, \rho)$ defined  by 
$$
\Uppi(x)=r=\sqrt{(x_1-x_{0,1})^2+ (x_2-x_{0,2})^2}, {\rm \ for \ }  x=(x_1, x_2)\in \D^2(x_0,\rho),
$$
 so that  $\Uppi^{-1}(\varrho)=\S^1(x_0,\varrho)$, for any $\varrho \in (0, r]$.  We define the measures 
$\check \upzeta_{\star}$  and $\check {\mathcal N_\star}$ on $[\delta, \rho)$ by
$$ \check \upzeta_{\star}=\Uppi_\sharp (\upzeta_\star) {\rm \ and \ } \check {\mathcal N_\star}= \Uppi_\sharp (\mathcal N_\star).$$
More precisely, for any Borel subset of $(\delta, \rho)$, we have 
  \begin{equation}
  \label{defcheck} 
 \check \upzeta_{\star}(A)= \upzeta_\star\left (\Uppi^{-1}(A)\right) {\rm \ and \ } \check {\mathcal N_\star}(A)=\left (\Uppi^{-1}(A)\right).
\end{equation}

  We first  have:

  \begin{lemma}
   The function $V$ and $G_\delta$ have bounded variation. We  have
   \begin{equation}
   \label{leclate}
   \frac{\rd }{\rd r} V= \check \upzeta_{\star}\geq 0, \,
   {\rm \ and \ }   \frac{\rd }{\rd r} G_\delta=r^{-1} \check {\mathcal N_\star} {\rm \ in  \ the \ sense \ of \ distributions} \
    \mathcal D'(\delta, \rho).
   \end{equation}
  \end{lemma}
\begin{proof}  We first observe that, as a consequence of the definition \eqref{vetg} ef \eqref{defcheck}, we have the identities
$$V(r)=\check \upzeta_\star (0, r)=\int_0^r \rd  \check \upzeta_\star=\int_0^\rho {\bf 1}_{(0, r)} \rd \check \upzeta_\star.$$
 The desired result \eqref{leclate} is then a direct consequence of Fubini's Theorem. Indeed, let $\varphi\in C_{\rm c} (\delta, \rho)$. We have
\begin{equation}
\label{leluc}
\begin{aligned}
\int_\delta^\rho \varphi'(r) V(r) \rd r&=\int_0^\rho  \varphi'(r) \left[ \int_0^\rho \rd  \check \upzeta_\star\right] \rd r 
=\iint_{(0, \rho)\times (0, \rho)}\varphi'(r) {\bf 1}_{(0, r)} \rd  \check \upzeta_\star\rd r \\
&=\int_{(0, \rho)}\left[\int_0^\rho \varphi'(r) {\bf 1}_{(0, r)} \rd r\right] \rd  \check \upzeta_\star=
-\int_{(0, \rho)}\varphi(r) \rd  \check \upzeta_\star,
\end{aligned}
\end{equation} 
Which establishes the first identity in \eqref{leluc}. The second in proved using the same argument. Finalement, since $\check \upzeta_\star$ and $r^{-1}\check {\mathscr N}_\star$ are bounded measures, it follows that the functions $V$ and $G_\delta$ have bounded variation. 
\end{proof}

  For the proof of Proposition \ref{psycho}, we will make use of the fact that the derivative of $F$ may be written in two different ways, as stated in the next Lemma. 
\begin{lemma} 
\label{doppio}
 The function $F$ has bounded variation. We have the identities
\begin{equation}
\label{derivef}
 \frac{\rd }{\rd r} F= \frac {1}{r}\check \upzeta_{\star}-\frac{1}{r^2} V=\frac{1}{r}\check {\mathcal N_\star},  {\rm \ in  \ the \ sense \ of \ distributions}
 \   \mathcal D'(\delta, \rho).
\end{equation}
\end{lemma}
\begin{proof}   The first identity in \eqref{derivef} corresponds to the Leibnitz rule applied to the product $\displaystyle{F=\frac{V}{r}}$  of the  measure  $V$, handled as a distribution on $(\delta, \rho)$, by  the smooth function
$\displaystyle{r \mapsto \frac{1}{r}}$. It yields
$$  \frac{\rd }{\rd r} F=-\frac{V}{r^2}+\frac{1}{r}  \frac{\rd }{\rd r} V, {\rm \ in \ the \ sense \ of \ distributions},  $$
so that  the first identity in \eqref{derivef} follows, in view of the first identity in \eqref{leluc}.  

\smallskip
For the second identity, we invoke Lemma \ref{passing}, which asserts that, for almost every $r \in (\delta, \rho)$, we have 
\begin{equation}
\label{monomaniaque2}
F(r)-F(\delta)=
\int_{\D^2(x_0,r)\setminus \D^2(x_0,\delta)}\frac{1}{4r} \rd \mathscr N_\star=G_\delta(r).
 \end{equation}
Taking the derivative, in the sense of distributions,   of this identity, the second identity in \eqref{derivef}  then follows from the second identity in \eqref{leluc}.
\end{proof}
 \subsection{Refined analysis of the derivative of $F$: Proof of Proposition \ref{psycho}}
 In this subsection, we make use   of the two different forms of the derivative  $F'=r^{-1} \check {\mathscr N}_\star$ provided by Lemma \ref{doppio}, in order  to show that this distribution  is actually a  non-negative measure. 
 We first have: 

\begin{lemma}  
\label{chiche}
Set  $\mathbbmtt B_\rho=\Uppi^{-1}\left( \mathfrak E_\star  \cap \D^2(\rho) \right)$.  We have   $\mathcal H^1(\mathbbmtt B_\rho)=0$ and 
\begin{equation}
\label{prems}
\check {\mathscr N}_\star \rest  \left( (0, \rho) \setminus \mathbbmtt B_\rho\right) \geq 0. 
\end{equation}
\end{lemma}
\begin{proof}   Since $\mathcal H^1( \mathfrak E_\star)=0$, we deduce that $\mathcal H^1(\mathbbmtt B_\rho)=0$.  Recall that 
\begin{equation}
\label{goufinet}
\mathscr N_\star= \mathscr N_\star^{ac} {\rm \ on \ } \D^2(\rho) \setminus \mathfrak E_\star,
\end{equation}
whereas in view of Lemma \ref{dingo}, we have $\mathscr N_\star^{ac}\geq 0$. Combining this inequality with \eqref{goufinet} we obtain
$$ \mathscr N_\star \rest \left( \D^2(x_0,\rho) \setminus \mathfrak E_\star\right) \geq 0.$$
 In view of the definition of 
$\check {\mathscr N}_\star$, we obtain hence \eqref{prems}.
\end{proof}

It remains to study $\check {\mathscr N}_\star \rest \mathbbmtt B_\rho$.  We have:

\begin{lemma}
\label{chiche2} 
The restriction of $\check {\mathscr N}_\star$ to  $\mathbbmtt B_\rho$ is non-negative, i.e. 
\begin{equation}
\label{chiche3}
\check {\mathscr N}_\star \rest \mathbbmtt B_\rho \geq 0.
\end{equation} 
\end{lemma}
 \begin{proof}Recall, that,  in view of Lemma \ref{doppio}, we have in the sense of distributions 
 \begin{equation}
 \label{doppiace}
 \check {\mathscr N}_\star=\check \upzeta_\star-\frac{V}{r}  {\rm \ in \ }  \mathcal D'(\delta, \rho).
 \end{equation}
 Since both sides of \eqref{doppiace}  are  bounded measures,  the identity in \eqref{doppiace} is also an identity of measures.   Since $V$ is a bounded function, it follows from the fact that $\mathbbmtt B_\rho$ has vanishing one-dimensional Lebesgue measure that 
 $$\frac{V}{r}\rest \mathbbmtt B_\rho=0 {\rm \ and \  hence  \  } 
 \check {\mathscr N}_\star \rest \mathbbmtt B_\rho= \check \upzeta_\star\rest \mathbbmtt B_\rho \geq 0. 
 $$
 \end{proof}

\noindent
{\it Proof of Proposition \ref{psycho} completed}. Combining \eqref{prems} and \eqref{chiche3}, we obtain that 
  $$\ \check {\mathscr N}_\star \geq 0 {\rm \ on \ }  (0, \rho). $$  
  Since $F'=r^{-1} \check {\mathscr N}_\star$, we deduce that $F'\geq 0$ on $(0, \rho)$,  so that $F$ is non-decreasing.  Inequality \eqref{monomaniaque} follows. The other statements of Proposition \ref{psycho}  are then straightforward, so that the  proof is complete.
\qed
 \subsection{Proofs of Theorems \ref{absolute} and \ref{discrepvec}} 
 Recall that at this stage we already know, thanks to Proposition \ref{psycho} that the measure $\upzeta_\star$ is absolutely continuous with respect to the measure $\rd \lambda$. We next derive the same statement for  the measure $\upnu_\star$, thanks to a comparison with the measure $\upzeta_\star$ relying on our PDE analysis developed in Part II.
 
 \subsubsection{ An upper bound for the measure $\upnu_\star$}
In follows from the very definition of the measures $\upzeta_\star$ and $\upnu_\star$ that we have the inequality 
$\upzeta_\star \leq\upnu_\star$. Indeed, we have for every  $ \eps>0$, the  straightforward inequality $\upzeta_\eps \leq\upnu_\eps$.  We next present a reverse inequality:

\begin{lemma}
\label{cornifle}
Let $x_0 \in  \Omega$ and $r>0$ be such that $\D^2(x_0, r) \subset \Omega$. Then we have 
\begin{equation}
\label{cornichon}
\upnu_\star \left ( \D^2(x_0, \frac{r}{2} \right)  \leq {\rm K}_{\rm V} (\rd (x_0)) \upzeta_\star \left ( \D^2(x_0, r )\right), 
\end{equation}
where $\rd (x_0)={\rm dist}(x_0, \partial\Omega)$ and where the constant ${\rm K}_{V}>0$ depends only on $V$, $M_0$ and $\rd (x_0)$.
\end{lemma}
 
 \begin{proof} The result is an immediate consequence of Proposition \ref{lheure}. Indeed, for $n \in \N$, we have    
the inequality
\begin{equation*}
\upnu_{\eps_n} \left(\D^2(x_0, \frac r2) \right) \leq {\rm K}_V\left ({\rm dist}(x_0, \partial \Omega)\right) 
\left[ \upzeta_{\eps_n}  \left(\D^2(x_0, \frac {3r}{4}) \right)+
 \frac{\eps_n}{r} \upnu_{\eps_n} \left(\D^2(x_0, \frac r2) \right)
\right].
\end{equation*}
Letting $n \to + \infty$, we are led to the inequality
\begin{equation*}
\upnu_{\star} \left(\D^2(x_0, \frac r2) \right) \leq {\rm K}_V\left ({\rm dist}(x_0, \partial \Omega)\right)  \upzeta_{\star}  \left(\overline{\D^2(x_0, \frac {3r}{4})} \right),
 \end{equation*} 
which yields \eqref{cornichon}.
 \end{proof}

An immediate consequence is:
\begin{corollary}
\label{griboeuf} 
The measure $\upnu_\star$ is absolutely continuous with respect to the measure $\rd \lambda=\mathcal H^1\rest \mathfrak S_\star$.
Moreover, we have,  writing $\upnu_\star=\tbe_\star \rd \lambda$, for $ \lambda$-almost every $x \in \mathfrak  S_ \star$, 
\begin{equation}
\label{griv}
\tbe_\star (x) \leq {\rm K}_V\left ({\rm dist}(x_0, \partial \Omega)\right)\Uptheta (x).
\end{equation} 
\end{corollary}
\begin{proof}
We have, for  every $x_0 \in \mathfrak S_\star$, the identity
\begin{equation}
\label{grivache}
\begin{aligned}
\overline{D}_\lambda (\upnu_\star)(x_0)&\equiv\underset {r \to 0} \limsup  \frac{\upnu_\star \left(\D^2(x_0, r)\right)}{2r}=
\underset {r \to 0} \limsup \frac{\upnu_\star \left(\D^2(x_0, \frac r 2)\right)}{r} \\
& \leq {\rm K}_{\rm V} (\rd (x_0)) \underset {r \to 0} \limsup\frac{ \upzeta_\star \left ( \D^2(x_0, r )\right)}{r}=
2 {\rm K}_{\rm V} (\rd (x_0))\Uptheta_\star(x_0),
\end{aligned}
\end{equation}
 where we used Lemma \ref{cornifle} for the second line. It follows that $\overline{D}_\lambda (\upnu_\star)(x_0)$ is locally bounded for every $x_0\in \Omega$, so that $\upnu_\star$ is absolutely continuous with respect to $\lambda$.  Since
 $$\tbe(x_0)=\overline{D}_\lambda (\upnu_\star)(x_0),$$
  for $\lambda$-almost every $x_0 \in \mathfrak S_\star$, \eqref{griv} follows from \eqref{grivache}. 
\end{proof} 
\subsubsection{Proof of Theorem \ref{absolute} } 
 In view of Proposition \ref{psycho}, we know that $\upzeta_\star$ is absolutely continous with respect to $\lambda$, whereas the same conclusion holds for $\upnu_\star$, in view of Corollary \ref{griboeuf}.  All inequalities in \eqref{ondanse} follow from either \eqref{griv} or \eqref{dansitus}, except the first one, namely  $\upeta_1\leq \tbe(x)$. The later inequality is a  is a consequence of the clearing-out theorem, Theorem \ref{claire},  and the definition \eqref{mathfrakSstar} of $\mathfrak S_\star$.
 \subsubsection{Proof of Theorem \ref{discrepvec}}
 Theorem \ref{discrepvec6} is an immediat consequence of Proposition \ref{discrepvec6} combined with the fact that all measures are absolutely continuous with respect to the measure $\mathcal H^1\rest \mathfrak S_\star$ (so that the singular parts actually vanish).

\section{Proof of Theorem \ref{segmentus}}
\label{segmentino} 
The argument consists,  for a large part, in revisiting  the analysis provided in Section \ref{goodpoints}, taking however   now into account the fact that all measures at stake are absolutely continuous with respect to $\mathcal H^1\rest \mathfrak S_\star$.   We first present several observations   which are relevant   for  the proof. In particular,  combining Lemma \ref{starac} with Theorem \ref{absolute}, we obtain, for $\omega_\star=\left( \upmu_{\star, 1, 1} -\upmu_{\star, 2, 2}\right) -2i\upmu_{\star, 1, 2}$ 
\begin{equation*}
\label{starac33}
\omega_\star=-2\exp (-2 i \upgamma_{\star})\upzeta_\star=-2 (\cos 2\upgamma_\star-i \sin 2\upgamma_\star) \upzeta_\star, 
\end{equation*}
so that 
\begin{equation}
\label{starac33}
 \upmu_{\star, 1, 1} -\upmu_{\star, 2, 2}=-2(\cos 2\upgamma_\star)\upzeta_\star 
 {\rm  \ and  \ } \upmu_{\star, 1, 2}=(\sin 2\upgamma_\star) \upzeta_\star.
 \end{equation} 
We will   make use of  these identities in several relations the obtained in Section \ref{goodpoints}.

\subsection{Preliminary observations}

 \begin{lemma} 
 \label{pudding}  Given any orthonormal basis $(\be_1, \be_2)$, 
  we have the relations
  \begin{equation}
  \label{pudding0}
   2\upzeta_\star-\upmu_{\star, 2, 2}+ \upmu_{\star, 1, 1}=4 \sin ^2 \upgamma_\star \upzeta_\star\geq 0.
 \end{equation}
\end{lemma}
\begin{proof}    The  proof is an  immediat consequence of \eqref{starac33} since 
$2\upzeta_\star-\upmu_{\star, 2, 2}+ \upmu_{\star, 1, 1}=2(1-\cos^2\upgamma_\star)\upzeta_\star.$ 
\end{proof}
Next, consider  a point  $x_0\in \mathfrak S_\star\setminus \mathfrak E_\star$, so that a tangent exists, and we assume moreover  that the orthonormal basis $(\be_1, \be_2)$ is chosen so that $\be_1=\vec e_{x_0}$.

\begin{lemma} 
\label{bidoche}
Let  $x_0\in \mathfrak S_\star\setminus \mathfrak E_\star$ and   $\rho_0>0$  be  the number   provided by Proposition \ref{assos}. Then the function $J_{1, \rho_0}$ defined on $\mathcal I_{\rho_0}(x_{0, 1})$  by   identity \eqref{becool1} in Corollary \ref{becool} is non-decreasing.
\end{lemma}  
\begin{proof}
Let $\bP$ be the orthogonal projection onto the line the tangent line ${\rm D}_{x_0}^1=\{ x_0+s\be_1, s\in \R \}$. Recall that, in view of the definition \eqref{caminare}, we have $\bN_{x_0, \rho_0}=\bP_\sharp (2\tilde \upzeta_\star  + \tilde \upmu_{\star, 1, 1} -\tilde \upmu_{\star, 2, 2})$
so that it follows from  \eqref{pudding0}  that
\begin{equation}
\label{pudding1} 
 \bN_{x_0, \rho_0}\geq 0.
\end{equation}
The conclusion is that an immediate consequence of the first  differential relation in \eqref{zoran} for $k=1$.
\end{proof}

For $s \in \mathcal I_{\rho_0} (x_{0,1})$, we  introduce the set $\Lambda (s)= \bP^{-1} (s) \cap  Q_{\rho_0}(x_0)$, the set of points in the square $Q_{\rho_0}$ whose orthogonal projection onto the line  $x_0+\R \be_1$ is the point $(s, x_{0,1})$. Let $\mathcal Z(s)=\sharp (\Lambda (s))$ be  the numbers of elements in 
$\Lambda(s)$. An important step in the proof is to prove that $\mathcal Z(s)=1$.  Since $\mathfrak S_\star$ is  connected, we have 
\begin{equation}
\Lambda (s) \not =\emptyset {\rm \ and \ hence \ } \mathcal Z(s) \geq 1 {\rm \ for \ } s \in \mathcal I_{\rho_0}(x_{0, 1}).
\end{equation}
We provide next a few simple observations. 

\begin{lemma}
\label{glupsitude} For almost every $s \in  \mathcal I_{\rho_0} (x_{0,1})$, the number $\mathcal Z(s)$ is finite.   
If $\Lambda(s)$ is finite, then we have, for $k \in \N$  
\begin{equation}
\label{gallarditude}
J_{k, \rho_0}(s)=2\underset{a(s)=(x_{0,1}, a_2(s))   \in \Lambda(s)}\sum  \left( x_{0,2}-a_2(s) \right)^k\sin (\upgamma_\star(a(s)) \Uptheta (a(s)).
\end{equation}
\end{lemma}

\begin{proof} 
Invoking again  \eqref{starac33}, we have $ \tilde \upmu_{\star, 1, 2} =\sin (2\upgamma_\star )\tilde \upzeta_\star=
{\bf 1}_{Q_{\rho_0} (x_0)}\sin (2\upgamma_\star )\Uptheta_\star \rd \lambda$.  in view of the definition of  $J_{1, \rho_0}$, we have
\begin{equation}
\label{groote}
J_{k, \rho_0}\rd s=\bJ_{k, \rho_0}=\bP_\sharp\left(  \left(x_{0,2}-x_2\right)^k \tilde \upmu_{\star, 1, 2}\right)=
\bP_\sharp\left(\left(x_{0,2}-x_2\right)^k{\bf 1}_{Q_{\rho_0} (x_0)}\sin (2\upgamma_\star )\Uptheta_\star \rd \lambda \right).
\end{equation}
If $\Lambda(s)$ is finite, and given any point $a(s)\in \Lambda(s)$,  we may find some arbitrary small number $\delta>0$ such find 
$(\mathfrak S_\star \cap \D^2(x_0, \delta))\cap \Lambda (s)=\{a(s)\}$. If $a(s)  \not \in \mathfrak E_\star$, then the angle of the tangent to $\mathfrak S_\star$ at the point $a(s)$  with the vector $\be_1$ is $\upgamma(a(s))$ so that, if $\upgamma(a(s))\not =\pm \pi \slash 2$, then we have   
\begin{equation}
\label{jesappelle}
\frac{\rd\bP_\sharp\left({\bf 1}_{\D^2 (a(s))} \rd \lambda\right)}
{\rd s} =\frac{1}{\cos(\upgamma(a(s))}.
\end{equation} 
  Since $\sin (2\upgamma(a(s)))=\sin (\upgamma(a(s))). \cos (\upgamma(a(s)))$, the conclusion follows combining \eqref{groote} and \eqref{jesappelle}.
\end{proof}

\begin{lemma}
\label{unicitude}
Let  $s \in  \mathcal I_{\rho_0} (x_{0,1})$ be such that $\mathcal Z(s)=1$. Then, we have $J_{k, \rho_0}(s)=0$, for any $k \in \N$.
\end{lemma}
\begin{proof} In view of the assumption of Lemma \ref{unicitude},  $\Lambda(s)$ contains a unique element $a(s)=(x_{0, 1}+s, a_2(s))$, so that 
$$
J_{k, \rho_0}(s)= \left( x_{0,2}-a_2(s) \right)^k\sin (\upgamma_\star(a(s)) \Uptheta (a(s)=\left( x_{0,2}-a_2(s)\right)^kJ_{0, \rho_0}(s).
$$
In view of Proposition \ref{diff} we have $J_{0, \rho_0}(s)=0$, for any $s\in [x_{0,1}-\rho_0, x_{0,1}+ \rho_0]$  (see the first identity in \eqref{diff0}), so that  the conclusion follows.

\end{proof}
 
 Next consider  for $0<\rho \leq \rho_0$, the set  $\mathcal G(\rho)=\{ s \in [x_{0, 1}-\rho, x_{0, 1}+\rho], {\rm \ such  \ that  \ }   \mathcal Z(s)=1\}$. We have: 
 
 \begin{lemma} 
 \label{densitudine}
 There exists $0<\rho_1\leq \rho_0$, such that we have the upper bound
 $$
 \vert \mathcal G(\rho) \vert \geq \frac{5\rho}{3}, {\rm \ for \ any \ } 0<\rho \leq \rho_1.
 $$
 \end{lemma} 
 \begin{proof} We first notice that, since $\bP_\star$ is a contraction,  that  for any $\rho\leq \rho_0$, we have 
 \begin{equation}
 \label{extrait} 
 \int_{x_{0,1}-\rho}^{x_{0,1}+\rho} \mathcal Z(s) \rd s \leq \mathcal H^1(\mathfrak S_\star \cap Q_{\rho} (x_0)) 
 \leq  \mathcal H^1(\mathfrak S_\star \cap \D^2 (x_0, \frac{r}{\cos\frac{\pi}{8}})\leq  
  \mathcal H^1(\mathfrak S_\star \cap \D^2 (x_0, \frac{10r}{9}),
 \end{equation}
 where we used \eqref{cadran}.
 On the other hand,  in view of  \eqref{densitusone}, there exists some $0< \varrho_1\leq \rho_0$, such that, for $\rho \leq \varrho_1$, we have
 \begin{equation}
  \label{densitusone2}
  \mathcal H^1 (\mathfrak S_\star(\D^2(x_0, \rho))\leq \frac{21\rho}{10}, 
  \end{equation}
  Combining \eqref{extrait} and \eqref{densitusone2}, we obtain hence,  for $\displaystyle{\rho \leq \rho_1\equiv\frac{9}{10} \varrho_1}$,
  \begin{equation}
  \label{graet}
   \int_{x_{0,1}-\rho}^{x_{0,1}+\rho} \mathcal Z(s) \rd s \leq \frac{21\rho} {9}=\frac{7\rho} {3}.
   \end{equation}
   We introduce the set $\mathcal K(\rho)=\{ s \in [x_{0,1}-\rho, x_{0,1}+\rho], {\rm \ such \ that \  } \mathcal Z(s) \geq 2\}$.  We have 
  \begin{equation}
  \label{niche}
  \int_{x_{0, 1}-\rho}^{x_{0, 1}+\rho} \mathcal Z(s) \rd s=
  \int_{\mathcal G(\rho) } \mathcal Z(s) \rd s+ \int_{\mathcal K(\rho) } \mathcal Z(s) \rd s
   \geq \vert \mathcal G(\rho) \vert + 2 \vert \mathcal K (\rho)\vert=  2\rho +\vert \mathcal K (\rho)\vert.
   \end{equation}
   Combining \eqref{graet} and \eqref{niche}, we deduce that $\displaystyle {\vert \mathcal K (\rho)\vert\leq \frac{\rho}{3}}$  and the conclusion follows.
     \end{proof}
 \subsection{Proof of Theorem \ref{segmentus} completed} 
For the sake of simplicity, we assume that the origin has been chosen so that $x_0=0$. We  first apply Lemma \ref{densitudine}, so that there exists $0<\rho_1\leq \rho_0$ such that $\vert \mathcal G(\rho_1) \vert\geq 5\slash3\rho_1$. Hence there exists two numbers 
$\rho_2^+>0$ and $\rho_2^->0$ such that
$$-\rho_1\leq -\rho_2^-\leq -\frac{2\rho_1}{3} <0<\frac{2\rho_1}{3} \leq \rho_2^+ \leq \rho_1  {\rm \ and  \  such  \ that \ }  \{\rho_2^+, -\rho_2^-\} \subset \mathcal G(\rho).$$
Since $\mathcal Z(-\rho_2^-)=\mathcal Z(\rho_2^+)=1$, we may  apply Lemma \ref{unicitude} to $-\rho_2^-$ et $\rho_2^+$ to assert that
$$\displaystyle{J_{1, \rho_0}(-\rho_2^-)=J_{1, \rho_0}(\rho_2^+)=0.}$$
Since, in view of Lemma \ref{bidoche}, the function $J_{1, \rho_0}$ is \emph{monotone}  on $\mathcal I_{\rho_0}$, we deduce that 
\begin{equation}
\label{viandox}
 J_{1, \rho_0}(s)=0  {\rm \ on \ } [-\rho_2^-, \rho_2^+]  {\rm \ and  \ hence  \ }  \bN_{\rho_0}=\frac{\rd }{\rd s} J_{1, \rho_0}=0 {\rm \ in \ } 
\mathcal D'( [-\rho_2^-, \rho_2^+]).
 \end{equation}
 It follows from the second identity  in \eqref{viandox}, the definition \eqref{caminare} of $\bN_{\rho_0}$ and \eqref{pudding0},  that  the restriction of the measure $(2\tilde \upzeta_\star  + \tilde \upmu_{\star, 1, 1} -\tilde \upmu_{\star, 2, 2})$  to $\mathcal I_{\rho_0}(0) \times [-\rho_2^-, \rho_2^+]$   vanishes.   This implies, that, for any $k \in \N$, we have 
 \begin{equation}
 \label {bkn}
 \bN_{k, \rho_0}\rest   ]-\rho_2^-, \rho_2^+[=0, 
 \end{equation}
 where $\bN_{k, \rho_0}$ is defined in \eqref{micheline}.  In view of the first differential equation in \eqref{zoran}, we have hence
 $$\displaystyle{\frac{\rd  }{\rd s }{\bJ}_{k-1, \rho_0}=0} {\rm \ on \ }
 ]-\rho_2^-, \rho_2^+[.$$
   Since 
 $J_{k, \rho_0}(0)=0$,  for $k \geq 1$, it follows that
 \begin{equation}
 \label{jk}
 J_{k, \rho_0}(s)=0  {\rm \ for \ every \ } s \in ]-\rho_2^-, \rho_2^+[.
 \end{equation}
 Similarly, invoking the second  relation  in \eqref{zoran}, that is $-\frac{\rd }{\rd s} \bL_{k,r}=k\bJ_{k-1,r}$, \eqref{jk} and the fact that $L_{k, \rho_0}(0)=0$,
 we deduce that
 \begin{equation}
 \label{bkn2}
 L_{k, \rho_0}=0  {\rm \ for \ every \ } s \in ]-\rho_2^-, \rho_2^+[, {\rm \ for \ } k \in \N^*.
 \end{equation}
 Combining \eqref{bkn} and \eqref{bkn2} with \eqref{worldwide} we deduce that (since $x_0=0$)
 \begin{equation}
\label{worldwide2}
\left\{
\begin{aligned}
\frac 14 L_{0} \rd s\rest   ]-\rho_2^-, \rho_2^+[&=\bP_{\sharp} \left( \tilde \upzeta_\star \right)\rest   ]-\rho_2^-, \rho_2^+[  {\rm \ for \  } k=0  {\rm \ and \ }  \\
\frac{1}{4}(\bN_{k,r}+\bL_{k, r})\rest   ]-\rho_2^-, \rho_2^+[&=\bP_{\sharp} \left( x_2^k\, \tilde \upzeta_\star \right)\rest   ]-\rho_2^-, \rho_2^+[=0, 
{\rm \ for \ }  k \in \N^*.
\end{aligned}
\right.
\end{equation}
The first identity in \eqref{worldwide2} shows hence that $\bP_{\sharp} \left( \tilde \upzeta_\star \right)$ is constant on $]-\rho_2^-, \rho_2^+[$.  Next, we choose $k=2$. The second  identity in \eqref{worldwide2} implies that $x_2^k\,  \upzeta_\star=0$ on $]-\rho_2^-, \rho_2^+[\times \mathcal I_{\rho_0} (0)$. these relations yield 
\begin{equation}
\label{cooper} 
\upzeta_\star\rest  \left (\left( ]-\rho_2^-, \rho_2^+[\setminus \{0\}\right)\times \mathcal I_{\rho_0} (0)\right)=0.  
\end{equation} 
We choose  next $r_0=\inf \{ \rho_2^-,\rho_2^+\} >0$. Combining \eqref{cooper} with the first relation in \eqref{worldwide2} we are led to
\begin{equation}
\label{droitaubut}
\upzeta_\star \rest \D^2(r_0)=L_0 \rd \ell {\rm \ where \ } \rd \ell {\rm \ is \ the \ Lebesgue \ measure \ on  } (-r_0, r_0).
\end{equation}
  Invoking Theorem \ref{absolute}, we complete the proof of Theorem \ref{segmentus}.
 \section{Proof of Theorem \ref{varifoltitude}}
  Inserting   identities  \eqref{starac33}  into the system \eqref{cauchise}, we are led to  the system of first-order equations 
 \begin{equation}
\label{riemannise}
\left\{
\begin{aligned}
-\frac{\partial}{\partial x_2} \left[(\sin 2\upgamma_\star)\, \upzeta_\star\right]&=\frac{\partial}{\partial x_1} 
\left[ \left(1 +\cos 2\upgamma_\star\right)\upzeta_\star\right] {\rm \ and \  } \\
-\frac{\partial}{\partial x_1} \left[(\sin 2\upgamma_\star) \upzeta_\star \right] &=\frac{\partial}{\partial x_2}
\left[\left(1-\cos 2\upgamma_\star\right)\upzeta_\star\right]. 
 \end{aligned}
\right.
\end{equation}
We are going to show next that these relations are equivalent, in the sense of  distributions, to \eqref{nary}. For that purpose, let 
$\vec X=(X_1, X_2)$ be a vector-field in $C_c^\infty (\Omega, \R^2)$.  We have, for any $x \in \mathfrak S\setminus \mathfrak E_\star$, 
since   by definition $\vec e_{x_0}=\cos \upgamma(x_0)\be_1+\sin  \upgamma(x_0)\be_2$   
\begin{equation*}
\begin{aligned}
{\rm div}_{_{T_x\mathfrak S_\star}}\vec X(x)&=
\left (\vec e_x\cdot \vec {\nabla}\vec X(x)\right) \cdot\vec e_{x} \\
&=\left(  \cos \upgamma_\star(x)\frac{\partial \vec X}{\partial x_1}(x)+ 
\sin \upgamma_\star(x)\frac{\partial \vec X}{\partial x_2}(x)\right)\cdot 
\left (\cos \upgamma_\star(x)\be_1+\sin \upgamma_\star(x)\be_2 \right) \\
&=\cos^2 \upgamma_\star (x)\frac{\partial  X_1}{\partial x_1}(x)+\sin^2 \upgamma_\star(x))\frac{\partial  X_2}{\partial x_2}(x)\\
&+\sin \upgamma_\star(x) \cos \upgamma_\star(x) \left[ \frac{\partial  X_2}{\partial x_1}(x)+
+\frac{\partial  X_1}{\partial x_2}(x) \right].
\end{aligned}
\end{equation*}
 Using this computation, we may expand relation \eqref{nary} as
 \begin{equation}
 \left\langle  \upzeta_\star,\cos^2\upgamma_\star   \frac{\partial  X_1}{\partial x_1}+\sin^2\upgamma_\star   \frac{\partial  X_2}{\partial x_2} 
 + \sin \upgamma_\star \cos \upgamma_\star\left[ \frac{\partial  X_2}{\partial x_1}+
\frac{\partial  X_1}{\partial x_2} \right]\right \rangle =0
 \end{equation}
Integrating by parts in the sense of distributions, we obtain hence, for every $X_1 \in C_c^\infty (\Omega, \R)$ and any 
$X_2\in C_c(\Omega, \R)$, the relation 
\begin{equation*}
\left \langle \frac{\partial}{\partial x_1} (\cos^2 \upgamma_\star \upzeta_\star)+ \frac{\partial}{\partial x_2} (\sin \upgamma_\star \cos \upgamma_\star\,  \upzeta_\star), X_1  \right\rangle+ 
\left \langle \frac{\partial}{\partial x_2} (\sin^2 \upgamma_\star \upzeta_\star)+ \frac{\partial}{\partial x_1} (\sin \upgamma_\star \cos \upgamma \upzeta_\star), X_2  \right\rangle=0.
\end{equation*} 
 Since $X_1$ and $X_2$ can  be chosen independently, we are led to the system, in the sense of distributions,
  \begin{equation}
\label{fuchsise}
\left\{
\begin{aligned}
-\frac{\partial}{\partial x_2} \left[(\sin\upgamma_\star \cos \upgamma_\star, \upzeta_\star\right]&=\frac{\partial}{\partial x_1} 
\left[ \left(\cos^2\upgamma_\star\right)\upzeta_\star\right]  {\rm \ and \  } \\
-\frac{\partial}{\partial x_1} \left[(\sin \upgamma_\star \cos \upgamma_\star) \upzeta_\star \right] &=\frac{\partial}{\partial x_2}
\left[\left(\sin^2\upgamma_\star\right)\upzeta_\star\right]. 
 \end{aligned}
\right.
\end{equation}
Since $2\sin\upgamma_\star \cos \upgamma_\star=\sin 2\upgamma_\star$, $1+\cos2\upgamma_\star=2\cos^2\upgamma_\star$  and 
$1-\cos^2\upgamma_\star=2\sin^2\upgamma_\star$, we verify that \eqref{fuchsise} is equivalent to \eqref{riemannise}, so that the system
\eqref{cauchise} is equivalent to \eqref{nary}. The varifold ${\bf V}(\mathfrak S_\star, \Uptheta_\star)$ is hence stationary. The proof of Theorem \ref{varifoltitude} is complete.


\end{document}